\documentclass[12pt]{amsart}

\usepackage{dsfont}
\usepackage{amsxtra}
\usepackage{amsmath}
\usepackage{amscd}
\usepackage{amssymb}
\usepackage{amsfonts}
\usepackage{mathrsfs}
\usepackage{ytableau}
\usepackage{amsthm}
\usepackage{color}
\usepackage{upgreek }
\usepackage{tikz}
\usepackage{float}
\usetikzlibrary{matrix,arrows,decorations.pathmorphing,automata, matrix, positioning, calc, shapes.multipart}
\usepackage{tikz-cd}
\usepackage{enumerate}

\usepackage{blkarray}
\usepackage{multirow}
\usepackage{mathtools}
\usepackage{graphicx}

\usepackage{verbatim}
\usepackage{dsfont}
\usepackage{enumitem}
\usepackage[cmtip,arrow,all]{xy}
\usepackage[mathscr]{eucal}
\usepackage{footmisc}
\usepackage{stmaryrd}
\usepackage{amscd}
\usepackage{latexsym}
\usepackage{amscd}
\usepackage{graphicx}
\usepackage{tikz}
\usepackage{enumitem}
\usepackage{txfonts}
\usepackage{setspace}
\usepackage{wasysym}
\usepackage[breaklinks=true]{hyperref}

\usepackage{imakeidx}
\makeindex

\tikzset{wipe/.style={white,line width=4pt}}

\numberwithin{itemcounter}{subsection}

\theoremstyle{plain}

\newtheorem{theorem}{Theorem}[section]
\newtheorem{lemma}[theorem]{Lemma}
\newtheorem{definition-lemma}[theorem]{Definition-Lemma}

\newtheorem{proposition}[theorem]{Proposition}

\newtheorem{corollary}[theorem]{Corollary}
\theoremstyle{definition}
\newtheorem{definition}[theorem]{Definition}
\theoremstyle{remark}
\newtheorem{remark}[theorem]{Remark}
\newtheorem{example}[theorem]{Example}
\numberwithin{equation}{section}

\def\bbC{\mathbb{C}}

\def\bbF{\mathbb{F}}

\def\bbH{\mathbb{H}}

\def\bbK{\mathbb{K}}

\def\bbN{\mathbb{N}}
\def\bbO{\mathbb{O}}

\def\bbZ{\mathbb{Z}}

\def\k{\Bbbk}
\def\KK{\mathcal K}
\def\kk{\kappa}

\def\sfT{\mathsf{T}}
\def\sfH{\mathsf{H}}
\def\scrI{{I}}

\def\scrP{\mathscr{P}}

\def\fraks{\mathfrak{s}}
\def\frakl{\mathfrak{l}}

\def\frakS{\mathfrak{S}}

\def\frakU{\mathfrak{U}}

\def\calC{\mathcal{C}}

\def\calE{\mathcal{E}}
\def\calF{\mathcal{F}}

\def\calI{\mathcal{I}}

\def\calL{\mathcal{L}}

\def\calO{\mathcal{O}}
\def\calP{\mathcal{P}}

\def\frakl{\mathfrak{l}}

\def\bfG{\mathbf{G}}

\def\bfO{\mathbf{O}}

\def\simto{\overset{\sim}\to}

\def\O{\mathrm{O}}
\def\SO{\mathrm{SO}}
\def\Sp{\mathrm{Sp}}

\def\Irr{{\operatorname{Irr}\nolimits}}

\def\sp{{\operatorname{sp}\nolimits}}

\def\GL{\mathrm{GL}}
\def\GU{\mathrm{GU}}
\def\diag{\operatorname{diag}\nolimits}

\def\C{\operatorname{C}\nolimits}
\def\M{{\operatorname{M}\nolimits}}

\def\R{\operatorname{R}\nolimits}
\def\X{\operatorname{X}\nolimits}
\def\Y{\operatorname{Y}\nolimits}
\def\Z{\operatorname{Z}\nolimits}
\def\W{\operatorname{W}\nolimits}

\def\up{{\operatorname{\,up}\nolimits}}

\def\rank{{\operatorname{rank}\nolimits}}
\def\min{{\operatorname{min}\nolimits}}

\def\Hom{\operatorname{Hom}\nolimits}

\def\End{\operatorname{End}\nolimits}
\def\Res{\operatorname{Res}\nolimits}

\def\Ind{\operatorname{Ind}\nolimits}

\def\Id{\operatorname{Id}\nolimits}
\def\id{\operatorname{id}\nolimits}

\def\wt{\operatorname{wt}\nolimits}

\def\mod{\operatorname{-mod}\nolimits}

\def\ct{{{\text{ct}}}}

\def\Heis{\mathcal{H}eis}

\def\RX{\mathrm{RX}}
\def\LX{\mathrm{LX}}

\def\Lcross{\mathrm{LT}}
\def\MixLcross{\mathrm{LH}}
\def\Rcross{\mathrm{RT}}

\def\Proj{\mathrm{Proj}}

\def\up{\uparrow}
\def\down{\downarrow}
\def\thickdown{\pmb\downarrow}
\def\diam{{\color{white}\scriptstyle\vardiamondsuit}\hspace{-1.493mm}\scriptstyle\diamondsuit}
\def\thickup{\pmb\uparrow}

\def\Heis{\mathcal{H}eis}

\def\symHeis{\underline{\mathcal{H}eis}}

\renewcommand{\sl}{\mathfrak{sl}}

\newcommand{\g}{\mathfrak{g}}
\newcommand{\hh}{\mathfrak{h}}

\def\Hom{\operatorname{Hom}}

\def\wt{{\operatorname{wt}}}
\def\End{{\operatorname{End}}}

\def\UU{{\mathfrak U}}

\def\C{{\mathbb C}}

\def\Z{{\mathbb Z}}
\def\N{{\mathbb N}}
\def\R{{\mathcal R}}

\def\dt{{\color{white}\bullet}\!\!\!\circ}
\def\bull{{\scriptstyle\bullet}}

\newcommand*{\tuple}[1]
{\boldsymbol{#1}}

\makeatletter
\newcommand{\unit}{\mathds{1}}
\newcommand{\qc}{p}

\newcommand{\Fqb}{\overline{\mathbb{F}}_q}

  \def\anticlock{\begin{tikzpicture}[baseline=-.9mm]
		\filldraw[white] (0,0) circle (1.72mm);
		\draw[-] (0,-0.18) to[out=180,in=-102] (-.178,0.02);
		\draw[-] (-0.18,0) to[out=90,in=180] (0,0.18);
		\draw[-] (0.18,0) to[out=-90,in=0] (0,-0.18);
		\draw[<-] (0,0.18) to[out=0,in=90] (0.18,0);
	\end{tikzpicture}\,}
\def\redanticlock{\begin{tikzpicture}[baseline=-.9mm,red]
		\filldraw[white] (0,0) circle (1.72mm);
		\draw[-] (0,-0.18) to[out=180,in=-102] (-.178,0.02);
		\draw[-] (-0.18,0) to[out=90,in=180] (0,0.18);
		\draw[-] (0.18,0) to[out=-90,in=0] (0,-0.18);
		\draw[<-] (0,0.18) to[out=0,in=90] (0.18,0);
	\end{tikzpicture}\,}

 \def\redclock{\begin{tikzpicture}[baseline=-.9mm,red]
		\filldraw[white] (0,0) circle (1.72mm);
		\draw[-] (0,-0.18) to[out=180,in=-90] (-.18,0);
		\draw[->] (-0.18,0) to[out=90,in=180] (0,0.18);
		\draw[-] (-0.02,0.178) to[out=12,in=90] (0.18,0);
		\draw[-] (0.18,0) to[out=-90,in=0] (0,-0.18);
	\end{tikzpicture}\,}

\newcommand{\darkg}{\color{green!70!black}}
\tikzset{darkg/.style={green!70!black}}

\def\dott{{\color{white}\bullet}\!\!\!\circ}

\def\dott{{\color{white}\bullet}\!\!\!\circ}
\definecolor{darkblue}{HTML}{000000}
\def\red#1{{\color{red} #1}}
\def\blue#1{{\color{blue} #1}}
\def\green#1{{\color{green} #1}}
\def\up{\uparrow}
\def\down{\downarrow}

\DeclareMathOperator*{\ot}{\otimes}

\newcommand{\coloredbubble}[2]{
\begin{tikzpicture}[baseline = 1.25mm]
  \draw[->] (0.2,0.2) to[out=90,in=0] (0,.4);
  \draw[-] (0,0.4) to[out=180,in=90] (-.2,0.2);
  \draw[-] (-.2,0.2) to[out=-90,in=180] (0,0);
  \draw[-] (0,0) to[out=0,in=-90] (0.2,0.2);
  \node at (0.2,0.2) {$\dt$};
  \node at (0.4,0.2) {$\scriptstyle{#1}$};
  \node at (0.0,-.2) {$\scriptstyle{#2}$};
\end{tikzpicture}
}

\begin{document}
	\title{Big categorification on towers of classical groups and wreath product groups}

    \author{Xin Huang}
\address[Xin Huang]{School of Mathematics and Statistics, Central China Normal University, Wuhan 430079, China}
\email{xinhuang@mails.ccnu.edu.cn}

    \author{Pengcheng Li}
\address[Pengcheng Li]{Department of Mathematics and New Cornerstone Science Laboratory, The University of Hong Kong, Hong Kong}
\email{pcli25@hku.hk}

	\author{Yaolong Shen}
\address[Yaolong Shen]{School of Mathematical Sciences, Key Laboratory of MEA (Ministry of Education)
\& Shanghai Key Laboratory of PMMP, East China Normal University, Shanghai 200241,
China}
\email{yaolongshen3555@gmail.com}


\begin{abstract}
We develop a uniform framework for ``big'' categorification of representation categories of towers of finite classical groups and wreath product groups. We construct actions of symmetric products of Heisenberg categories, quantum in the finite classical group case and degenerate in the wreath product case. These actions lead to categorical actions of suitable symmetric products of Kac--Moody 2-categories, and hence to actions of large Lie algebras on Grothendieck groups. In characteristic zero, the resulting actions control the full graded centers of the group algebras through diagrammatic central elements, and the associated colored weight functions separate all irreducible ordinary characters. We also obtain modular block descriptions for wreath product groups through categorification.
\end{abstract}

	\maketitle
\tableofcontents

	\pagestyle{myheadings}

	\markboth{}{}

\section{Introduction}

The study of towers of groups and their representation categories has long
played a central role in representation theory. A paradigmatic example is the
tower of symmetric groups. For each $n \geq 0$, the symmetric group
$\mathfrak S_n$ embeds naturally into $\mathfrak S_{n+1}$ as the subgroup
fixing the element $n+1$, giving rise to the ascending chain
\[
    \mathfrak S_0 \subset \mathfrak S_1 \subset \mathfrak S_2 \subset \cdots
    \subset \mathfrak S_n \subset \mathfrak S_{n+1} \subset \cdots .
\]

Let $\k$ be an algebraically closed field. A fundamental object associated to this tower is the additive category
\[
    \k \mathfrak S_\bullet\text{-mod}
    :=
    \bigoplus_{n \geq 0} \k \mathfrak S_n\text{-mod},
\]
where $\k \mathfrak S_n\text{-mod}$ denotes the category of
finite-dimensional $\k\mathfrak S_n$-modules. In characteristic $0$, the Grothendieck group of $\k \mathfrak S_\bullet\mod$ is isomorphic to the Fock space. Under this
identification, suitable summands of the induction and restriction functors
give rise to the standard action of the infinite-dimensional Lie algebra
$\mathfrak{gl}_\infty$ \cite{OV96}. 
In characteristic $\ell>0$, the modular branching functors, namely the
$i$-induction and $i$-restriction functors, induce an action of the affine Lie
algebra $\widehat{\mathfrak{sl}}_\ell$ on
\(
    \mathbb C\otimes_{\mathbb Z}
    K_0\bigl(\k\mathfrak S_\bullet\text{-mod}\bigr).
\)
This Grothendieck group is isomorphic to the irreducible highest weight module
$L(\Lambda_0)$ of $\widehat{\mathfrak{sl}}_\ell$; see \cite{LLT}.
Thus, the tower of symmetric groups provides a basic
example of a categorical realization of representations of 
infinite-dimensional Lie algebras.

This example naturally motivates the study of analogous structures associated
with other infinite towers of groups. Two important families are wreath
product groups and finite classical groups. 

Let $H$ be a finite group. For each $n\geq 0$, the wreath product
\[
H_n:=H\wr\mathfrak S_n=H^n\rtimes\mathfrak S_n
\]
is the semi-direct product in which $\mathfrak S_n$ acts on $H^n$ by permuting
the factors. The standard embeddings of symmetric groups induce a tower
\[
H_0\subset H_1\subset H_2\subset\cdots.
\]
Wreath product groups generalize symmetric groups and play a crucial role in the representation theory of finite groups, as they often serve as building blocks for more complex group structures and their representations exhibit rich combinatorial patterns; see for example \cite[Ch.~4]{JK81},\cite[Appendix~B]{Mac95},\cite{Pus99}.

Finite classical groups -- including general linear groups $\GL_n$, unitary groups $\GU_n$, symplectic groups $\Sp_{2n}$, and orthogonal groups $\O^\pm_n$ -- form another important family of group towers. These groups arise naturally in linear algebra and geometry, and their tower structures are defined by embedding each group into a larger one of the same type by extending the underlying vector space. Specifically, we have the following towers:
\[
\cdots \subset\GL_{n-1}\subset \GL_n\subset \GL_{n+1}\subset \cdots ;
\]
\[
\cdots \subset \GU_{n-2}\subset \GU_{n}\subset \GU_{n+2}\subset \cdots ;
\]
\[
\cdots \subset \Sp_{2n-2}\subset \Sp_{2n}\subset \Sp_{2n+2}\subset \cdots ;
\]
and
\[
\cdots \subset \O^\pm_{n-2}\subset \O^\pm_{n}\subset \O^\pm_{n+2}\subset \cdots .
\]

The representation theory of finite classical groups, and in particular the finite general linear groups $\mathrm{GL}_n(q)$, has been extensively studied through a wide range of distinct approaches over the past decades.
For example, for the ordinary representation of $\GL_n(q)$:
\begin{itemize}
    \item Green  gave the first complete classification and explicit construction of all irreducible characters of $\mathrm{GL}_n(q)$ via combinatorial parabolic induction, establishing the celebrated Green polynomial character formula \cite{Green55}.
    \item Zelevinsky developed a unified Hopf algebraic framework for the representation theory of $\mathrm{GL}_n(q)$, connecting it to Hall algebras and the combinatorics of symmetric functions \cite{Zel81}.
    \item Deligne and Lusztig pioneered a geometric approach, constructing all irreducible representations of finite reductive groups via $\ell$-adic cohomology of Deligne--Lusztig varieties, which provides a uniform geometric interpretation of the full character theory of $\mathrm{GL}_n(q)$ \cite{DL76}.
    \item More recently, Jing and Wu initiated a vertex algebraic approach to the representation theory of $\mathrm{GL}_n(q)$ \cite{JW23}.
\end{itemize}

In this paper, we study the representation categories of {\bf all} finite classical
groups and wreath product groups uniformly from the perspective of categorification. For each of the
classical towers considered above, we form the additive category
\[
    \k G_\bullet\text{-mod}
    :=
    \bigoplus_m \k G_m\text{-mod}
\]
where $\k$ is an appropriate coefficient ring, as specified below.
 The inclusions in the tower give rise to induction and restriction functors
between the summands of this category. 
Categorification provides a natural language for studying such structures. A
landmark result in this direction is the work of Chuang and Rouquier
\cite{CR08}, who constructed categorical $\mathfrak{sl}_2$-actions on
categories of unipotent representations of the finite general linear groups
$\GL_n(q)$. Their construction, however, concerns distinguished subcategories
of unipotent representations rather than the full representation categories.

Our first main result is a uniform construction of categorical Heisenberg and Kac--Moody actions on \(\k G_\bullet\)-mod for the towers considered in this paper. More specifically, for the tower of finite general linear groups in characteristic 0, we construct a categorical
action of a Kac--Moody 2-category on
$\k\GL_\bullet\text{-mod}$, and the corresponding Kac--Moody algebra acts on its Grothendieck group. For all the
other finite classical groups and wreath product groups, we also establish analogous categorification results in a uniform way.

Our uniform approach is based on the Heisenberg categorification of \cite{BSW20a}.  More specifically, we study symmetric products of degenerate and quantum Heisenberg categories and establish a Heisenberg-to-Kac--Moody construction for these symmetric products.  We prove that any abelian module category over such a symmetric product, subject to suitable finiteness conditions, naturally carries the structure of a 2-representation of a Kac--Moody 2-category associated with a disjoint union of type~A quivers.

We then apply this general framework to \(\k G_\bullet\)-mod for the towers of groups considered in this paper. The colors in the resulting Heisenberg actions come from the primitive building blocks in the corresponding induction--restriction theory. For example, for finite general linear groups, these colors are indexed by cuspidal irreducible representations, the primitive building blocks in Harish--Chandra theory. The corresponding functors are realized by Harish--Chandra induction and restriction with respect to these cuspidal representations. This generalizes the tower of symmetric groups, where the single primitive block is the trivial representation of \(\mathfrak S_1\), giving the usual one-colored Heisenberg action.

For the finite classical groups, fix a prime $\ell$ with $\ell\nmid q(q-1)$ and the coefficient triple $(\mathbb K,\mathcal O,\mathbb F)$ used in \S\ref{subsec:cus}. The quantum Heisenberg action is defined for $\k\in\{\mathbb K,\mathcal O,\mathbb F\}$, while the Kac--Moody $2$-action is defined for $\k\in\{\mathbb K,\mathbb F\}$. For the wreath-product towers in characteristic $\ell>0$, assume $\ell\nmid |H|$.

In this way, we obtain the following theorem.
\begin{theorem}[Theorem~\ref{Thm:quantumheisenberg},~\ref{ThmB}, ~\ref{Thm:doubleheisenberg} and \ref{thm:Lieactwreath}.]
\label{introthm1}
Let $G_\bullet$ be one of the above towers of groups, with the coefficient assumptions just stated.
\begin{enumerate}
    \item The category $\k G_\bullet\text{-mod}$ admits an action of a
    suitable symmetric product of Heisenberg categories.

    \item If $\k$ is a field, the category $\k G_\bullet\text{-mod}$ admits an action of a
    suitable symmetric product of Kac--Moody 2-categories.
\end{enumerate}
\end{theorem}

This theorem places several existing constructions in a common framework. For finite classical groups, it
unifies the categorification program initiated by Chuang and Rouquier
\cite{CR08}, the categorical actions on unipotent representations of finite
classical groups constructed by Dudas, Varagnolo, and Vasserot
\cite{DVV17,DVV19}, and the recent work on quadratic unipotent representations
and finite classical groups by Li, Liu, and Zhang \cite{LLZ} and Li, Shan, and
Zhang \cite{LSZ25}. %The point of the present work is to formulate these phenomena in terms of symmetric products of Heisenberg categories and the associated Kac--Moody 2-representations, and to apply this framework uniformly to the full representation categories of the group towers considered above.
For wreath product groups, it establishes a 2-categorical lift of the results in \cite{WW08} in the sense of Khovanov and Lauda \cite{KL10} and Rouquier \cite{R08}. 

%In \cite{WW08}, Wan and Wang study the modular representation theory of wreath product groups via wreath Hecke algebras. They show that at the level of Grothendieck groups, there exists a product of Kac-Moody actions. Our results on wreath product groups can be viewed as a 2-categorical lift of their work: we realize the categorical Kac-Moody action in the sense of Khovanov and Lauda \cite{KL10} and Rouquier \cite{R08}, rather than merely at the Grothendieck group level.

Beyond Theorem~\ref{introthm1}, we also study several applications of the
categorical actions constructed above. We first consider the problem of
distinguishing irreducible ordinary characters of finite classical groups. The
categorical actions give rise to colored weight functions, obtained by
evaluating colored bubbles on irreducible objects. These functions provide
explicit invariants of irreducible characters. Our second main result states that these invariants are complete.
\begin{theorem} \label{introthm2}
{\rm (Theorem~\ref{thm:invariants} and \rm Theorem~\ref{cwf})} In characteristic $0$,
for a finite classical tower the colored weight functions
\[
    \{\bbO^\Gamma(u)(-)\}_{\Gamma\in \calI(\bbK)}
\]
form a complete set of invariants for the irreducible ordinary characters. For a wreath-product tower $H\wr\mathfrak S_\bullet$, the analogous complete family is indexed by $\chi\in\Irr(H)$.
\end{theorem}

Recall that the ordinary irreducible characters of a finite classical group $G$ are partitioned into Lusztig series
\[
   \Irr(G)=\coprod_{(s)}\mathcal E(G,s),
\]
where $s$ runs over semisimple conjugacy classes of the dual group $G^*$. Lusztig's Jordan decomposition gives a bijection between $\mathcal E(G,s)$ and the unipotent characters of the corresponding centralizer $C_{G^*}(s)$; see \cite{DL76,DM}. For finite classical groups, semisimple conjugacy classes and their centralizers are described in terms of elementary divisors: writing the semisimple parameter in its primary decomposition, the corresponding factors of $C_{G^*}(s)$ are indexed by the elementary divisors $\Gamma$. We use these elementary divisors, or the corresponding duality orbits when appropriate, as the colors in our categorical construction.

The centers of group algebras for symmetric groups, wreath products, and finite general linear groups have been studied through many approaches, notably by Farahat--Higman \cite{FH59}, Wang \cite{W04b}, and Wan--Wang \cite{WW19}. Our colored bubble generators provide an independent categorical realization of the full graded centers appearing in the towers considered above. %\textcolor{red}{It is by no means obvious from the definition of bubbles alone that all colored bubbles generate the full center of the group algebra of a finite classical group.}

\begin{theorem}
{\rm (Theorem~\ref{Thm:centersurjective} and Theorem~\ref{cwf})} \label{introthm3}
In characteristic $0$, the evaluated colored bubbles generate the full center in every rank. More precisely, for each $n$, the coefficients of
\[
\big\{\bbO^\Gamma(u)\big\}_{\Gamma\in\calI(\bbK)}
\]
generate $Z(\bbK G_n)$, and, for the wreath-product tower $H\wr\mathfrak S_\bullet$, the coefficients of
\[
\big\{\bbO^\chi(u)\big\}_{\chi\in\Irr(H)}
\]
generate $Z(\bbK H_n)$.
\end{theorem}
Theorems~\ref{introthm2} and~\ref{introthm3} together establish that our categorical quantum Heisenberg and Kac--Moody actions in characteristic 0 are ``biggest". More precisely, the colored weight functions arising from the action separate all ordinary irreducible characters, while the colored bubbles generate the full center of the corresponding group algebras in each fixed rank.

For wreath product groups, in characteristic zero, the degenerate
symmetric-product Heisenberg action gives rise to an analogous family of
colored weight functions. Using the Jucys--Murphy elements and the
wreath-product version of the Okounkov--Vershik approach developed by Pushkarev \cite{Pus99} and further studied by Mishra--Srinivasan \cite{MS16}, building on the symmetric-group method of Okounkov--Vershik \cite{OV96}, we identify
these functions with the content functions of the multipartitions labeling
irreducible representations. In particular, they form a complete set of
invariants for irreducible modules over characteristic-zero wreath product
algebras.

Finally, we consider the modular representation theory of wreath product
groups under the standing assumption $\ell\nmid |H|$. The Kac--Moody
2-representation obtained from the symmetric-product Heisenberg action
identifies the block decomposition of
\(
    K_0(\bbF H_n\text{-mod})
\)
with the simultaneous decomposition by the component sizes and the Kac--Moody weight spaces. Equivalently, two
irreducible $\bbF H_n$-modules lie in the same block precisely when the
corresponding multipartitions have the same component sizes and the same componentwise $\ell$-cores. Thus, in the modular
setting, the colored weight functions together with the component sizes give complete invariants for blocks of
wreath product groups as well. Using categorical Kac--Moody actions, we also recover the modular branching results studied by Tsuchioka \cite{Tsu07} and Wan--Wang \cite{WW08}.

The structure of this article is as follows: In Section~\ref{threecats}, we recall the diagrammatic categories that will be used throughout the paper, including Heisenberg categories and Kac--Moody 2-categories. In
Sections~\ref{sec:symproduct} and~\ref{sec:HeistoKac}, we study symmetric
products of degenerate and quantum Heisenberg categories and prove that, under suitable finiteness assumptions, their module categories give rise to 2-representations of Kac--Moody 2-categories of type~A. In
Section~\ref{sec:bigquantum}, we apply this framework to the representation categories of finite classical groups, introduce colored weight functions and prove their completeness as invariants for ordinary irreducible characters. In Section~\ref{sec:degenerate} we discuss applications to the (modular) representation theory of wreath product groups.
    
\subsection*{Acknowledgment}
We thank Professor Alistair Savage for explaining the relation between the wreath-product action in Section~\ref{sec:degenerate} and Frobenius Heisenberg categorification, and in particular for pointing out the normalization issue corrected in Theorem~\ref{Thm:expofbubbles}.
XH is supported by NSFC (No. 12501024, 12471016) and China Postdoctoral Science Foundation (No. GZC20252006, 2025T001HB).
YS is partially supported by the Science and Technology Commission of Shanghai Municipality (No. 22DZ2229014).
This work was partially supported by the New Cornerstone Science Foundation through the New Cornerstone Investigator Program awarded to Professor Xuhua He.

\section{Three diagrammatic categories}\label{threecats}

Throughout the article, $\kk$ is an algebraically closed field and $z \in \kk$ is a parameter.
We refer to the cases $z\neq 0$ and $z = 0$ as the
{\em quantum} and {\em degenerate} cases, respectively.
For use in the quantum case, we choose a root $q$ of the polynomial
$x^2-zx-1$, so that $z=q - q^{-1}$. By convention, if $z=0$, then we take $q=1$.

We recall the three diagrammatic categories used below, following the
notation of \cite[\S3]{BSW20a}.
\subsection{Kac--Moody 2-category}
\label{sec:KMCdef}
The Kac--Moody 2-category was introduced by Khovanov--Lauda \cite{KL10}
and Rouquier \cite{R08}.  We use only Cartan type~$A$.  Let $I$ carry a
fixed-point-free automorphism $i\mapsto i^+$, with inverse $i\mapsto i^-$.
The associated quiver has components of type $A_\infty$ or
$A_{\qc-1}^{(1)}$, $\qc\geq2$.  Its generalized Cartan matrix is
$a_{i,i}=2$ and
$a_{i,j}=-\delta_{i^+,j}-\delta_{i,j^+}$ for $i\ne j$.  Let $\g$ be the
Kac--Moody Lie algebra over $\C$ defined by this matrix.  It is a direct sum
of copies of $\sl_\infty(\C)$ and $\widehat{\sl}_{\qc}(\C)'$, according to
the infinite and finite components of the quiver.

Let $\hh$ be the Cartan subalgebra of $\g$ with
basis
$\{h_i\:|\:i \in I\}$.
The {\em weight lattice} $X$ of $\g$ is the abelian
subgroup of $\hh^*$ generated by
the fundamental weights $\{\Lambda_j\:|\:j \in I\}$ defined from
$\langle h_i,\Lambda_j \rangle = \delta_{i,j}$.
We have the set of {\em dominant weights}
$$
X^+ := \bigoplus_{i \in I} \N \Lambda_i
= \big\{\lambda \in X\:\big|\:\langle h_i, \lambda \rangle \geq 0
\text{ for all $i \in I$ and $\textstyle\sum_{i \in I}\langle
  h_i,\lambda\rangle < \infty$}\big\}.
$$
Let $\alpha_i := \sum_{j \in I} a_{i,j} \Lambda_j \in X$ be the
$i$-th simple root.
Note that
we have  $\sum_{i \in I_0} \alpha_i = 0$ for each finite component
$I_0$ of $I$.
Let $Y := \sum_{i \in I} \Z \alpha_i \subseteq X$.

The  Kac--Moody 2-category
 $\UU=\UU(\g)$
is the strict $\kk$-linear 2-category
with objects $X$,
generating
1-morphisms
$E_i 1_\lambda = \substack{\thickup \\ {\scriptscriptstyle i}}{\scriptstyle \color{gray}\lambda}:\lambda \rightarrow \lambda+\alpha_i$ and
$F_i 1_\lambda= \substack{{\scriptscriptstyle
    i}\\\thickdown}{\scriptstyle \color{gray}\lambda}:\lambda \rightarrow \lambda-\alpha_i$ for
$i \in I$ and $\lambda \in X$,
and generating
2-morphisms
\begin{align*}
\mathord{
\begin{tikzpicture}[baseline = -2]
	\draw[->,thick] (0.08,-.15) to (0.08,.3);
      \node at (0.08,0.05) {$\bull$};
   \node at (0.08,-.3) {$\scriptstyle{i}$};
\end{tikzpicture}
}
{\color{gray}\scriptstyle\lambda}
&:E_i 1_\lambda \Rightarrow E_i 1_\lambda,
&
\mathord{
\begin{tikzpicture}[baseline = 0]
	\draw[<-,thick] (0.3,0.2) to[out=-90, in=0] (0.1,-0.1);
	\draw[-,thick] (0.1,-0.1) to[out = 180, in = -90] (-0.1,0.2);
    \node at (-0.1,.35) {$\scriptstyle{i}$};
\end{tikzpicture}
}\:\,
{\color{gray}\scriptstyle\lambda}
&:1_\lambda \Rightarrow F_i E_i 1_\lambda,
&
\mathord{
\begin{tikzpicture}[baseline = -2]
	\draw[<-,thick] (0.3,-0.1) to[out=90, in=0] (0.1,0.2);
	\draw[-,thick] (0.1,0.2) to[out = 180, in = 90] (-0.1,-0.1);
    \node at (-0.1,-.25) {$\scriptstyle{i}$};
\end{tikzpicture}
}
\:\,{\color{gray}\scriptstyle\lambda}
&:E_i F_i 1_\lambda \Rightarrow 1_\lambda,\\
\mathord{
\begin{tikzpicture}[baseline = -2]
	\draw[->,thick] (0.18,-.15) to (-0.18,.3);
	\draw[->,thick] (-0.18,-.15) to (0.18,.3);
   \node at (-0.18,-.3) {$\scriptstyle{j}$};
   \node at (0.18,-.3) {$\scriptstyle{i}$};
\end{tikzpicture}
}
{\color{gray}\scriptstyle\lambda}
&:E_j E_i 1_\lambda \Rightarrow E_i E_j 1_\lambda,&
\mathord{
\begin{tikzpicture}[baseline = 0]
	\draw[-,thick] (0.3,0.2) to[out=-90, in=0] (0.1,-0.1);
	\draw[->,thick] (0.1,-0.1) to[out = 180, in = -90] (-0.1,0.2);
    \node at (0.3,.35) {$\scriptstyle{i}$};
\end{tikzpicture}
}\!
{\color{gray}\scriptstyle\lambda}
&:1_\lambda \Rightarrow E_i F_i 1_\lambda,
&
\mathord{
\begin{tikzpicture}[baseline = -2]
	\draw[-,thick] (0.3,-0.1) to[out=90, in=0] (0.1,0.2);
	\draw[->,thick] (0.1,0.2) to[out = 180, in = 90] (-0.1,-0.1);
    \node at (0.3,-.25) {$\scriptstyle{i}$};
\end{tikzpicture}
}
\!{\color{gray}\scriptstyle\lambda}
&:F_i E_i 1_\lambda \Rightarrow 1_\lambda.
\end{align*}
For convenience, we also define the sideways crossings
\begin{align*}
\mathord{
\begin{tikzpicture}[baseline = 0]
	\draw[<-,thick] (0.28,-.3) to (-0.28,.4);
	\draw[->,thick] (-0.28,-.3) to (0.28,.4);
   \node at (-.3,-.43) {$\scriptstyle{j}$};
   \node at (-.3,.55) {$\scriptstyle{i}$};
   \node at (0.5,0.05) {$\color{gray}\scriptstyle{\lambda}$};
\end{tikzpicture}
}
&:=
\mathord{
\begin{tikzpicture}[baseline = 0]
	\draw[->,thick] (0.3,-.5) to (-0.3,.5);
	\draw[-,thick] (-0.2,-.2) to (0.2,.3);
        \draw[-,thick] (0.2,.3) to[out=50,in=180] (0.5,.5);
        \draw[->,thick] (0.5,.5) to[out=0,in=90] (0.9,-.5);
        \draw[-,thick] (-0.2,-.2) to[out=230,in=0] (-0.6,-.5);
        \draw[-,thick] (-0.6,-.5) to[out=180,in=-90] (-0.9,.5);
   \node at (.3,-.65) {$\scriptstyle{j}$};
   \node at (-.9,.65) {$\scriptstyle{i}$};
   \node at (1.1,0) {$\color{gray}\scriptstyle{\lambda}$};
\end{tikzpicture}
}\:,&
\mathord{
\begin{tikzpicture}[baseline = 0]
	\draw[->,thick] (0.28,-.3) to (-0.28,.4);
	\draw[<-,thick] (-0.28,-.3) to (0.28,.4);
   \node at (.3,-.43) {$\scriptstyle{j}$};
   \node at (.3,.55) {$\scriptstyle{i}$};
   \node at (0.5,0.05) {$\color{gray}\scriptstyle{\lambda}$};
\end{tikzpicture}
}
&:=
\mathord{
\begin{tikzpicture}[baseline = 0]
	\draw[<-,thick] (0.3,.5) to (-0.3,-.5);
	\draw[-,thick] (-0.2,.2) to (0.2,-.3);
        \draw[-,thick] (0.2,-.3) to[out=130,in=180] (0.5,-.5);
        \draw[-,thick] (0.5,-.5) to[out=0,in=270] (0.9,.5);
        \draw[-,thick] (-0.2,.2) to[out=130,in=0] (-0.6,.5);
        \draw[->,thick] (-0.6,.5) to[out=180,in=-270] (-0.9,-.5);
   \node at (-0.3,-.65) {$\scriptstyle{j}$};
   \node at (.9,.65) {$\scriptstyle{i}$};
   \node at (1.1,0) {$\color{gray}\scriptstyle{\lambda}$};
\end{tikzpicture}
}\:,
\end{align*}

The generating 2-morphisms are subject to the following
relations:
\begin{align}\label{KM1}
\mathord{
\begin{tikzpicture}[baseline = 1mm]
	\draw[<-,thick] (0.25,.55) to (-0.25,-.15);
	\draw[->,thick] (0.25,-.15) to (-0.25,.55);
  \node at (-0.25,-.3) {$\scriptstyle{j}$};
   \node at (0.25,-.3) {$\scriptstyle{i}$};
  \node at (.3,.25) {$\color{gray}\scriptstyle{\lambda}$};
      \node at (-0.13,0.01) {$\bull$};
\end{tikzpicture}
}
-
\mathord{
\begin{tikzpicture}[baseline = 1mm]
	\draw[<-,thick] (0.25,.55) to (-0.25,-.15);
	\draw[->,thick] (0.25,-.15) to (-0.25,.55);
  \node at (-0.25,-.3) {$\scriptstyle{j}$};
   \node at (0.25,-.3) {$\scriptstyle{i}$};
  \node at (.3,.25) {$\color{gray}\scriptstyle{\lambda}$};
      \node at (0.13,0.38) {$\bull$};
\end{tikzpicture}
}
&=
\mathord{
\begin{tikzpicture}[baseline = 1mm]
 	\draw[<-,thick] (0.25,.55) to (-0.25,-.15);
	\draw[->,thick] (0.25,-.15) to (-0.25,.55);
  \node at (-0.25,-.3) {$\scriptstyle{j}$};
   \node at (0.25,-.3) {$\scriptstyle{i}$};
  \node at (.3,.25) {$\color{gray}\scriptstyle{\lambda}$};
      \node at (-0.13,0.38) {$\bull$};
\end{tikzpicture}
}
-
\mathord{
\begin{tikzpicture}[baseline = 1mm]
	\draw[<-,thick] (0.25,.55) to (-0.25,-.15);
	\draw[->,thick] (0.25,-.15) to (-0.25,.55);
  \node at (-0.25,-.3) {$\scriptstyle{j}$};
   \node at (0.25,-.3) {$\scriptstyle{i}$};
  \node at (.3,.25) {$\color{gray}\scriptstyle{\lambda}$};
      \node at (0.13,0.01) {$\bull$};
\end{tikzpicture}
}
=
\delta_{i,j}
\mathord{
\begin{tikzpicture}[baseline = -.5mm]
 	\draw[->,thick] (0.08,-.3) to (0.08,.4);
	\draw[->,thick] (-0.28,-.3) to (-0.28,.4);
   \node at (-0.28,-.45) {$\scriptstyle{j}$};
   \node at (0.08,-.45) {$\scriptstyle{i}$};
 \node at (.28,.06) {$\color{gray}\scriptstyle{\lambda}$};
\end{tikzpicture}
},\\
\label{KM2}
\mathord{
\begin{tikzpicture}[baseline = 2.5mm]
	\draw[->,thick] (0.28,.4) to[out=90,in=-90] (-0.28,1.1);
	\draw[->,thick] (-0.28,.4) to[out=90,in=-90] (0.28,1.1);
	\draw[-,thick] (0.28,-.3) to[out=90,in=-90] (-0.28,.4);
	\draw[-,thick] (-0.28,-.3) to[out=90,in=-90] (0.28,.4);
  \node at (-0.28,-.45) {$\scriptstyle{j}$};
  \node at (0.28,-.45) {$\scriptstyle{i}$};
   \node at (.43,.4) {$\color{gray}\scriptstyle{\lambda}$};
\end{tikzpicture}
}
&=
\left\{
\begin{array}{ll}
0\hspace{40mm}&\text{if $j=i$,}\\
\mathord{
\begin{tikzpicture}[baseline = -.5mm]
	\draw[->,thick] (0.08,-.3) to (0.08,.4);
	\draw[->,thick] (-0.28,-.3) to (-0.28,.4);
   \node at (-0.28,-.45) {$\scriptstyle{j}$};
   \node at (0.08,-.45) {$\scriptstyle{i}$};
   \node at (.3,-.05) {$\color{gray}\scriptstyle{\lambda}$};
     \node at (0.08,0.05) {$\bull$};
\end{tikzpicture}
}
-\mathord{
\begin{tikzpicture}[baseline = -.5mm]
	\draw[->,thick] (0.08,-.3) to (0.08,.4);
	\draw[->,thick] (-0.28,-.3) to (-0.28,.4);
   \node at (-0.28,-.45) {$\scriptstyle{j}$};
   \node at (0.08,-.45) {$\scriptstyle{i}$};
   \node at (.3,-.05) {$\color{gray}\scriptstyle{\lambda}$};
      \node at (-0.28,0.05) {$\bull$};
\end{tikzpicture}
}
&\text{if $j^- =i\neq j^+$,}\\
\mathord{
\begin{tikzpicture}[baseline = -.5mm]
	\draw[->,thick] (0.08,-.3) to (0.08,.4);
	\draw[->,thick] (-0.28,-.3) to (-0.28,.4);
   \node at (-0.28,-.45) {$\scriptstyle{j}$};
   \node at (0.08,-.45) {$\scriptstyle{i}$};
   \node at (.3,-.05) {$\color{gray}\scriptstyle{\lambda}$};
      \node at (-0.28,0.05) {$\bull$};
\end{tikzpicture}
}
-\mathord{
\begin{tikzpicture}[baseline = -.5mm]
	\draw[->,thick] (0.08,-.3) to (0.08,.4);
	\draw[->,thick] (-0.28,-.3) to (-0.28,.4);
   \node at (-0.28,-.45) {$\scriptstyle{j}$};
   \node at (0.08,-.45) {$\scriptstyle{i}$};
   \node at (.3,-.05) {$\color{gray}\scriptstyle{\lambda}$};
     \node at (0.08,0.05) {$\bull$};
\end{tikzpicture}
}
&\text{if $j^- \neq i= j^+$,}\\
 2 \mathord{
\begin{tikzpicture}[baseline = -.5mm]
	\draw[->,thick] (0.08,-.3) to (0.08,.4);
	\draw[->,thick] (-0.28,-.3) to (-0.28,.4);
   \node at (-0.28,-.45) {$\scriptstyle{j}$};
   \node at (0.08,-.45) {$\scriptstyle{i}$};
   \node at (.3,-.05) {$\color{gray}\scriptstyle{\lambda}$};
     \node at (0.08,0.05) {$\bull$};
      \node at (-0.28,0.05) {$\bull$};
\end{tikzpicture}
}
-
\mathord{
\begin{tikzpicture}[baseline = -.5mm]
	\draw[->,thick] (0.08,-.3) to (0.08,.4);
	\draw[->,thick] (-0.28,-.3) to (-0.28,.4);
   \node at (-0.28,-.45) {$\scriptstyle{j}$};
   \node at (0.08,-.45) {$\scriptstyle{i}$};
   \node at (.3,-.05) {$\color{gray}\scriptstyle{\lambda}$};
     \node at (-0.28,0.15) {$\bull$};
     \node at (-0.28,-.1) {$\bull$};
\end{tikzpicture}
}
-\mathord{
\begin{tikzpicture}[baseline = -.5mm]
	\draw[->,thick] (0.08,-.3) to (0.08,.4);
	\draw[->,thick] (-0.28,-.3) to (-0.28,.4);
   \node at (-0.28,-.45) {$\scriptstyle{j}$};
   \node at (0.08,-.45) {$\scriptstyle{i}$};
   \node at (.3,-.05) {$\color{gray}\scriptstyle{\lambda}$};
     \node at (0.08,0.15) {$\bull$};
     \node at (0.08,-.1) {$\bull$};
\end{tikzpicture}
}
\:\:\:\:\;&\text{if $j^- =i=j^+$,}\\
\mathord{
\begin{tikzpicture}[baseline = -.5mm]
	\draw[->,thick] (0.08,-.3) to (0.08,.4);
	\draw[->,thick] (-0.28,-.3) to (-0.28,.4);
   \node at (-0.28,-.45) {$\scriptstyle{j}$};
   \node at (0.08,-.45) {$\scriptstyle{i}$};
   \node at (.3,.05) {$\color{gray}\scriptstyle{\lambda}$};
\end{tikzpicture}
}&\text{otherwise,}
\end{array}
\right. \\
\label{KM3}
\mathord{
\begin{tikzpicture}[baseline = .5mm]
	\draw[<-,thick] (0.45,.8) to (-0.45,-.4);
	\draw[->,thick] (0.45,-.4) to (-0.45,.8);
        \draw[-,thick] (0,-.4) to[out=90,in=-90] (-.45,0.2);
        \draw[->,thick] (-0.45,0.2) to[out=90,in=-90] (0,0.8);
   \node at (-0.45,-.6) {$\scriptstyle{k}$};
   \node at (0,-.6) {$\scriptstyle{j}$};
  \node at (0.45,-.6) {$\scriptstyle{i}$};
   \node at (.5,-.1) {$\color{gray}\scriptstyle{\lambda}$};
\end{tikzpicture}
}
\!\!-
\!\!\!
\mathord{
\begin{tikzpicture}[baseline = .5mm]
	\draw[<-,thick] (0.45,.8) to (-0.45,-.4);
	\draw[->,thick] (0.45,-.4) to (-0.45,.8);
        \draw[-,thick] (0,-.4) to[out=90,in=-90] (.45,0.2);
        \draw[->,thick] (0.45,0.2) to[out=90,in=-90] (0,0.8);
   \node at (-0.45,-.6) {$\scriptstyle{k}$};
   \node at (0,-.6) {$\scriptstyle{j}$};
  \node at (0.45,-.6) {$\scriptstyle{i}$};
   \node at (.5,-.1) {$\color{gray}\scriptstyle{\lambda}$};
\end{tikzpicture}
}
&=
\left\{
\begin{array}{ll}
\displaystyle
\mathord{
\begin{tikzpicture}[baseline = -.5mm]
	\draw[->,thick] (0.44,-.3) to (0.44,.4);
	\draw[->,thick] (0.08,-.3) to (0.08,.4);
	\draw[->,thick] (-0.28,-.3) to (-0.28,.4);
   \node at (-0.28,-.5) {$\scriptstyle{i}$};
   \node at (0.08,-.5) {$\scriptstyle{j}$};
   \node at (0.44,-.5) {$\scriptstyle{i}$};
  \node at (.6,-.1) {$\color{gray}\scriptstyle{\lambda}$};
\end{tikzpicture}
}
&\text{if $j^- = i=k \neq j^+$}\\
\displaystyle
-\mathord{
\begin{tikzpicture}[baseline = -.5mm]
	\draw[->,thick] (0.44,-.3) to (0.44,.4);
	\draw[->,thick] (0.08,-.3) to (0.08,.4);
	\draw[->,thick] (-0.28,-.3) to (-0.28,.4);
   \node at (-0.28,-.5) {$\scriptstyle{i}$};
   \node at (0.08,-.5) {$\scriptstyle{j}$};
   \node at (0.44,-.5) {$\scriptstyle{i}$};
  \node at (.6,-.1) {$\color{gray}\scriptstyle{\lambda}$};
\end{tikzpicture}
}
&\text{if $j^-\neq i=k= j^+$,}\\
2\displaystyle
\mathord{
\begin{tikzpicture}[baseline = -0.5mm]
	\draw[->,thick] (0.44,-.3) to (0.44,.4);
	\draw[->,thick] (0.08,-.3) to (0.08,.4);
	\draw[->,thick] (-0.28,-.3) to (-0.28,.4);
   \node at (-0.28,-.5) {$\scriptstyle{i}$};
   \node at (0.08,-.5) {$\scriptstyle{j}$};
   \node at (0.44,-.5) {$\scriptstyle{i}$};
  \node at (.6,-.1) {$\color{gray}\scriptstyle{\lambda}$};
     \node at (0.08,0.05) {$\bull$};
\end{tikzpicture}
}
\!\!-\!
\displaystyle
\mathord{
\begin{tikzpicture}[baseline = -.5mm]
	\draw[->,thick] (0.44,-.3) to (0.44,.4);
	\draw[->,thick] (0.08,-.3) to (0.08,.4);
	\draw[->,thick] (-0.28,-.3) to (-0.28,.4);
   \node at (-0.28,-.5) {$\scriptstyle{i}$};
   \node at (0.08,-.5) {$\scriptstyle{j}$};
   \node at (0.44,-.5) {$\scriptstyle{i}$};
  \node at (.6,-.1) {$\color{gray}\scriptstyle{\lambda}$};
     \node at (-0.28,0.05) {$\bull$};
\end{tikzpicture}
}\!\!-\!
\mathord{
\begin{tikzpicture}[baseline = -.5mm]
	\draw[->,thick] (0.44,-.3) to (0.44,.4);
	\draw[->,thick] (0.08,-.3) to (0.08,.4);
	\draw[->,thick] (-0.28,-.3) to (-0.28,.4);
   \node at (-0.28,-.5) {$\scriptstyle{i}$};
   \node at (0.08,-.5) {$\scriptstyle{j}$};
   \node at (0.44,-.5) {$\scriptstyle{i}$};
  \node at (.6,-.1) {$\color{gray}\scriptstyle{\lambda}$};
     \node at (0.44,0.05) {$\bull$};
\end{tikzpicture}
}
&\text{if $j^-=i= k=j^+$,}\\
0&\text{otherwise,}
\end{array}
\right.\end{align}\begin{align}
\mathord{
\begin{tikzpicture}[baseline = -1mm]
  \draw[->,thick] (0.3,0) to (0.3,.4);
	\draw[-,thick] (0.3,0) to[out=-90, in=0] (0.1,-0.4);
	\draw[-,thick] (0.1,-0.4) to[out = 180, in = -90] (-0.1,0);
	\draw[-,thick] (-0.1,0) to[out=90, in=0] (-0.3,0.4);
	\draw[-,thick] (-0.3,0.4) to[out = 180, in =90] (-0.5,0);
  \draw[-,thick] (-0.5,0) to (-0.5,-.4);
   \node at (-0.5,-.55) {$\scriptstyle{i}$};
   \node at (0.5,0) {$\color{gray}\scriptstyle{\lambda}$};
\end{tikzpicture}
}
&=
\mathord{\begin{tikzpicture}[baseline=-1mm]
  \draw[->,thick] (0,-0.4) to (0,.4);
   \node at (0,-.55) {$\scriptstyle{i}$};
   \node at (0.2,0) {$\color{gray}\scriptstyle{\lambda}$};
\end{tikzpicture}
},\qquad\quad\qquad
\mathord{
\begin{tikzpicture}[baseline = 0mm]
  \draw[->,thick] (0.3,0) to (0.3,-.4);
	\draw[-,thick] (0.3,0) to[out=90, in=0] (0.1,0.4);
	\draw[-,thick] (0.1,0.4) to[out = 180, in = 90] (-0.1,0);
	\draw[-,thick] (-0.1,0) to[out=-90, in=0] (-0.3,-0.4);
	\draw[-,thick] (-0.3,-0.4) to[out = 180, in =-90] (-0.5,0);
  \draw[-,thick] (-0.5,0) to (-0.5,.4);
   \node at (-0.5,.55) {$\scriptstyle{i}$};
   \node at (0.5,0) {$\color{gray}\scriptstyle{\lambda}$};
\end{tikzpicture}
}
=
\mathord{\begin{tikzpicture}[baseline=0mm]
  \draw[<-,thick] (0,-0.4) to (0,.4);
   \node at (0,.55) {$\scriptstyle{i}$};
   \node at (0.2,0) {$\color{gray}\scriptstyle{\lambda}$};
\end{tikzpicture}
},
\label{KMrightadj}\end{align}
 plus the
{\em inversion relation} asserting that the following are isomorphisms:
\begin{align}
\mathord{
\begin{tikzpicture}[baseline = 0]
	\draw[<-,thick] (0.28,-.3) to (-0.28,.4);
	\draw[->,thick] (-0.28,-.3) to (0.28,.4);
   \node at (-0.31,-.45) {$\scriptstyle{i}$};
   \node at (-0.31,.55) {$\scriptstyle{j}$};
   \node at (.4,.05) {$\scriptstyle{\color{gray}\lambda}$};
\end{tikzpicture}
}
&:E_i F_j 1_\lambda \Rightarrow F_j E_i 1_\lambda
&&\text{if $j \neq i$,}\label{KMinverse1}\\
\left[
\begin{tikzpicture}[baseline = 0]
	\draw[<-,thick] (0.28,-.3) to (-0.28,.4);
	\draw[->,thick] (-0.28,-.3) to (0.28,.4);
   \node at (-0.32,-.4) {$\scriptstyle{i}$};
   \node at (-0.32,.5) {$\scriptstyle{i}$};
   \node at (.4,.05) {$\scriptstyle{\color{gray}\lambda}$};
\end{tikzpicture}
\begin{tikzpicture}[baseline = 0]
	\draw[<-,thick] (0.4,0.2) to[out=-90, in=0] (0.1,-.2);
	\draw[-,thick] (0.1,-.2) to[out = 180, in = -90] (-0.2,0.2);
    \node at (-0.2,.35) {$\scriptstyle{i}$};
  \node at (0.3,-0.25) {$\scriptstyle{\color{gray}\lambda}$};
\end{tikzpicture}
\:\cdots
\:
\begin{tikzpicture}[baseline = 0]
	\draw[<-,thick] (0.4,0.2) to[out=-90, in=0] (0.1,-.2);
	\draw[-,thick] (0.1,-.2) to[out = 180, in = -90] (-0.2,0.2);
    \node at (-0.2,.35) {$\scriptstyle{i}$};
  \node at (0.3,-0.25) {$\scriptstyle{\color{gray}\lambda}$};
      \node at (1.05,0) {$\scriptstyle{-\langle h_i,\lambda\rangle-1}$};
      \node at (0.38,0) {$\bull$};
\end{tikzpicture}
\!\right]
&
:E_i F_i 1_\lambda \oplus
1_\lambda^{\oplus -\langle h_i,\lambda\rangle}
\Rightarrow
 F_i E_i 1_\lambda&&\text{if $\langle h_i,\lambda\rangle \leq
  0$,}\notag
  \end{align}
\begin{align}
\left[\begin{array}{r}
\begin{tikzpicture}[baseline = 0]
	\draw[<-,thick] (0.28,-.3) to (-0.28,.4);
	\draw[->,thick] (-0.28,-.3) to (0.28,.4);
   \node at (-0.31,-.45) {$\scriptstyle{i}$};
   \node at (-0.31,.55) {$\scriptstyle{i}$};
   \node at (.4,.05) {$\scriptstyle{\color{gray}\lambda}$};
\end{tikzpicture}
\\
\begin{tikzpicture}[baseline = 0]
	\draw[<-,thick] (0.4,0) to[out=90, in=0] (0.1,0.4);
	\draw[-,thick] (0.1,0.4) to[out = 180, in = 90] (-0.2,0);
    \node at (-0.2,-.15) {$\scriptstyle{i}$};
  \node at (0.3,0.5) {$\scriptstyle{\color{gray}\lambda}$};
\end{tikzpicture}
\\\vdots\:\:\:\:\\
\!\!\!\!\begin{tikzpicture}[baseline = 0]
	\draw[<-,thick] (0.4,0) to[out=90, in=0] (0.1,0.4);
	\draw[-,thick] (0.1,0.4) to[out = 180, in = 90] (-0.2,0);
    \node at (-0.2,-.15) {$\scriptstyle{i}$};
  \node at (0.3,0.5) {$\scriptstyle{\color{gray}\lambda}$};
      \node at (-0.7,0.2) {$\scriptstyle{\langle h_i,\lambda\rangle-1}$};
      \node at (-0.15,0.2) {$\bullet$};
\end{tikzpicture}
\end{array}
\right]
&:
E_i F_i 1_\lambda \Rightarrow
F_i E_i 1_\lambda \oplus 1_\lambda^{\oplus \langle h_i,\lambda\rangle}
&&\text{if $\langle h_i,\lambda\rangle \geq
  0$}.\label{KMinverse3}
\end{align}

Moreover, $\UU(\g)$ is strictly pivotal, so that
we can introduce downward dots and crossings by taking right and/or left mates of the upward ones.

\subsection{Degenerate Heisenberg categories} Recall from \cite{B18} that
the {\em degenerate Heisenberg category} $\Heis_k$ is the strict $\kk$-linear monoidal category
		generated by objects
		$E=\up$ and $F=\down$
		and
		morphisms
		\begin{align*}
			\mathord{
				\begin{tikzpicture}[baseline = 0]
					\draw[->] (0.08,-.3) to (0.08,.4);
					\node at (0.08,0.05) {$\dt$};
				\end{tikzpicture}
			}
			&:E\rightarrow E,
			&\mathord{
				\begin{tikzpicture}[baseline = 1mm]
					\draw[<-] (0.4,0.4) to[out=-90, in=0] (0.1,0);
					\draw[-] (0.1,0) to[out = 180, in = -90] (-0.2,0.4);
				\end{tikzpicture}
			}&:\unit\rightarrow F\otimes E
			\:,
			&\mathord{
				\begin{tikzpicture}[baseline = 1mm]
					\draw[<-] (0.4,0) to[out=90, in=0] (0.1,0.4);
					\draw[-] (0.1,0.4) to[out = 180, in = 90] (-0.2,0);
				\end{tikzpicture}
			}&:E \otimes F\rightarrow\unit\:,
			\mathord{
				\begin{tikzpicture}[baseline = 0]
					\draw[->] (0.28,-.3) to (-0.28,.4);
					\draw[->] (-0.28,-.3) to (0.28,.4);
				\end{tikzpicture}
			}&:E\otimes E \rightarrow E \otimes E
		\end{align*}
		subject to certain relations.
		To record these relations, we denote $n\geq 0$ dots on a string instead by labelling a single dot
		with the multiplicity $n$. We also introduce
		the sideways crossings
		\begin{align*}
			\mathord{
				\begin{tikzpicture}[baseline = 0]
					\draw[<-] (0.28,-.3) to (-0.28,.4);
					\draw[->] (-0.28,-.3) to (0.28,.4);
				\end{tikzpicture}
			}
			&:=
			\mathord{
				\begin{tikzpicture}[baseline = 0]
					\draw[->] (0.3,-.5) to (-0.3,.5);
					\draw[-] (-0.2,-.2) to (0.2,.3);
					\draw[-] (0.2,.3) to[out=50,in=180] (0.5,.5);
					\draw[->] (0.5,.5) to[out=0,in=90] (0.9,-.5);
					\draw[-] (-0.2,-.2) to[out=230,in=0] (-0.6,-.5);
					\draw[-] (-0.6,-.5) to[out=180,in=-90] (-0.9,.5);
				\end{tikzpicture}
			}\:,&
			\mathord{
				\begin{tikzpicture}[baseline = 0]
					\draw[->] (0.28,-.3) to (-0.28,.4);
					\draw[<-] (-0.28,-.3) to (0.28,.4);
				\end{tikzpicture}
			}
			&:=
			\mathord{
				\begin{tikzpicture}[baseline = 0]
					\draw[<-] (0.3,.5) to (-0.3,-.5);
					\draw[-] (-0.2,.2) to (0.2,-.3);
					\draw[-] (0.2,-.3) to[out=130,in=180] (0.5,-.5);
					\draw[-] (0.5,-.5) to[out=0,in=270] (0.9,.5);
					\draw[-] (-0.2,.2) to[out=130,in=0] (-0.6,.5);
					\draw[->] (-0.6,.5) to[out=180,in=-270] (-0.9,-.5);
				\end{tikzpicture}
			}\:,
		\end{align*}
		Then the relations are
		\begin{align}\label{dAHA}
			\mathord{
				\begin{tikzpicture}[baseline = -1mm]
					\draw[->] (0.28,0) to[out=90,in=-90] (-0.28,.6);
					\draw[->] (-0.28,0) to[out=90,in=-90] (0.28,.6);
					\draw[-] (0.28,-.6) to[out=90,in=-90] (-0.28,0);
					\draw[-] (-0.28,-.6) to[out=90,in=-90] (0.28,0);
				\end{tikzpicture}
			}&=
			\mathord{
				\begin{tikzpicture}[baseline = -1mm]
					\draw[->] (0.18,-.6) to (0.18,.6);
					\draw[->] (-0.18,-.6) to (-0.18,.6);
				\end{tikzpicture}
			}\:,
			\quad\qquad\mathord{
				\begin{tikzpicture}[baseline = -1mm]
					\draw[<-] (0.45,.6) to (-0.45,-.6);
					\draw[->] (0.45,-.6) to (-0.45,.6);
					\draw[-] (0,-.6) to[out=90,in=-90] (-.45,0);
					\draw[->] (-0.45,0) to[out=90,in=-90] (0,0.6);
				\end{tikzpicture}
			}
			=
			\mathord{
				\begin{tikzpicture}[baseline = -1mm]
					\draw[<-] (0.45,.6) to (-0.45,-.6);
					\draw[->] (0.45,-.6) to (-0.45,.6);
					\draw[-] (0,-.6) to[out=90,in=-90] (.45,0);
					\draw[->] (0.45,0) to[out=90,in=-90] (0,0.6);
				\end{tikzpicture}
			}\:,&
			\qquad
			\mathord{
				\begin{tikzpicture}[baseline = -1mm]
					\draw[<-] (0.25,.3) to (-0.25,-.3);
					\draw[->] (0.25,-.3) to (-0.25,.3);
					\node at (-0.12,-0.145) {$\dt$};
				\end{tikzpicture}
			}
			&=
			\mathord{
				\begin{tikzpicture}[baseline = -1mm]
					\draw[<-] (0.25,.3) to (-0.25,-.3);
					\draw[->] (0.25,-.3) to (-0.25,.3);
					\node at (0.12,0.135) {$\dt$};
			\end{tikzpicture}}
			+\:\mathord{
				\begin{tikzpicture}[baseline = -1mm]
					\draw[->] (0.08,-.3) to (0.08,.3);
					\draw[->] (-0.28,-.3) to (-0.28,.3);
				\end{tikzpicture}
			}
			\:,\\
   \label{eq:rightadj}
			\mathord{
				\begin{tikzpicture}[baseline = -.8mm]
					\draw[->] (0.3,0) to (0.3,.4);
					\draw[-] (0.3,0) to[out=-90, in=0] (0.1,-0.4);
					\draw[-] (0.1,-0.4) to[out = 180, in = -90] (-0.1,0);
					\draw[-] (-0.1,0) to[out=90, in=0] (-0.3,0.4);
					\draw[-] (-0.3,0.4) to[out = 180, in =90] (-0.5,0);
					\draw[-] (-0.5,0) to (-0.5,-.4);
				\end{tikzpicture}
			}
			&=
			\mathord{\begin{tikzpicture}[baseline=-.8mm]
					\draw[->] (0,-0.4) to (0,.4);
				\end{tikzpicture}
			}\:,
			&\mathord{
				\begin{tikzpicture}[baseline = -.8mm]
					\draw[->] (0.3,0) to (0.3,-.4);
					\draw[-] (0.3,0) to[out=90, in=0] (0.1,0.4);
					\draw[-] (0.1,0.4) to[out = 180, in = 90] (-0.1,0);
					\draw[-] (-0.1,0) to[out=-90, in=0] (-0.3,-0.4);
					\draw[-] (-0.3,-0.4) to[out = 180, in =-90] (-0.5,0);
					\draw[-] (-0.5,0) to (-0.5,.4);
				\end{tikzpicture}
			}
			&=
			\mathord{\begin{tikzpicture}[baseline=-.8mm]
					\draw[<-] (0,-0.4) to (0,.4);
				\end{tikzpicture}
			}\:
		\end{align}
  plus the inversion relation, namely, that the
following
is an isomorphism in the additive envelope:
\begin{align}
\label{dMackey1}
\left[
\begin{array}{r}
\mathord{
\begin{tikzpicture}[baseline = 0]
	\draw[<-] (0.28,-.3) to (-0.28,.3);
	\draw[->] (-0.28,-.3) to (0.28,.3);
   \end{tikzpicture}
}\\
\mathord{
\begin{tikzpicture}[baseline = 1mm]
	\node at (0,.6){$\phantom.$};
\draw[<-] (0.4,0) to[out=90, in=0] (0.1,0.4);
	\draw[-] (0.1,0.4) to[out = 180, in = 90] (-0.2,0);
\end{tikzpicture}
}\\
\mathord{
\begin{tikzpicture}[baseline = 1mm]
	\node at (0,.6){$\phantom.$};
	\draw[<-] (0.4,0) to[out=90, in=0] (0.1,0.4);
	\draw[-] (0.1,0.4) to[out = 180, in = 90] (-0.2,0);
      \node at (-0.15,0.2) {$\dt$};
\end{tikzpicture}
}\\
\vdots\:\:\:\:\\
\!\!\!\!\mathord{
\begin{tikzpicture}[baseline = 1mm]
	\node at (0,.5){$\phantom.$};
	\draw[<-] (0.4,0) to[out=90, in=0] (0.1,0.4);
	\draw[-] (0.1,0.4) to[out = 180, in = 90] (-0.2,0);
     \node at (-0.5,0.2) {$\scriptstyle{k-1}$};
      \node at (-0.15,0.2) {$\dt$};
\end{tikzpicture}
}\:\end{array}
\right]
&:
E \otimes F\:\rightarrow\:
F \otimes E \oplus \unit^{\oplus k}&&
\text{if $k \geq
  0$},\\
\left[\:
\mathord{
\begin{tikzpicture}[baseline = 0]
	\draw[<-] (0.28,-.3) to (-0.28,.3);
	\draw[->] (-0.28,-.3) to (0.28,.3);
\end{tikzpicture}
}\:\:\:
\mathord{
\begin{tikzpicture}[baseline = -0.9mm]
	\draw[<-] (0.4,0.2) to[out=-90, in=0] (0.1,-.2);
	\draw[-] (0.1,-.2) to[out = 180, in = -90] (-0.2,0.2);
\end{tikzpicture}
}
\:\:\:
\mathord{
\begin{tikzpicture}[baseline = -0.9mm]
	\draw[<-] (0.4,0.2) to[out=-90, in=0] (0.1,-.2);
	\draw[-] (0.1,-.2) to[out = 180, in = -90] (-0.2,0.2);
      \node at (0.38,0) {$\dt$};
\end{tikzpicture}
}
\:\:\:\cdots
\:\:
\mathord{
\begin{tikzpicture}[baseline = -0.9mm]
	\draw[<-] (0.4,0.2) to[out=-90, in=0] (0.1,-.2);
	\draw[-] (0.1,-.2) to[out = 180, in = -90] (-0.2,0.2);
     \node at (0.83,0) {$\scriptstyle{-k-1}$};
      \node at (0.38,0) {$\dt$};
\end{tikzpicture}
}
\right]
&:E \otimes F \oplus
\unit^{\oplus (-k)}
\:\rightarrow\:
 F \otimes E&&\text{if $k \leq
  0$}.\label{dMackey2}
\end{align}
The resulting category then contains unique morphisms
$\:\begin{tikzpicture}[baseline = .75mm]
	\draw[-] (0.3,0.3) to[out=-90, in=0] (0.1,0);
	\draw[->] (0.1,0) to[out = 180, in = -90] (-0.1,0.3);
\end{tikzpicture}\:$
and
$\:\begin{tikzpicture}[baseline = .75mm]
	\draw[-] (0.3,0) to[out=90, in=0] (0.1,0.3);
	\draw[->] (0.1,0.3) to[out = 180, in = 90] (-0.1,0);
\end{tikzpicture}\:$; cf. \cite{BSW20a}. Moreover, the degenerate Heisenberg category $\Heis_k$ can been shown to be strictly pivotal as well.

  \subsection{Quantum Heisenberg categories}
We also recall the quantum Heisenberg categories introduced in \cite[Definition 2.2]{BSW20b}.
In the quantum case, $z\neq 0$ and there is an additional invertible parameter $t$. Thus, in the quantum case, we will work over the
ground ring $$\KK=\kk[t,t^{-1}].$$
    The quantum Heisenberg category $\Heis_k(z,t)$ is the strict $\KK$-linear monoidal category generated by objects
$E=\up$ and $F=\down$
and the following
morphisms:
\begin{align}\label{qHgens}
\mathord{
\begin{tikzpicture}[baseline = 0]
	\draw[->] (0.08,-.3) to (0.08,.4);
      \node at (0.08,0.05) {$\dt$};
\end{tikzpicture}
}
&:E\rightarrow E,
&\mathord{
\begin{tikzpicture}[baseline = 1mm]
	\draw[<-] (0.4,0.4) to[out=-90, in=0] (0.1,0);
	\draw[-] (0.1,0) to[out = 180, in = -90] (-0.2,0.4);
\end{tikzpicture}
}&:\unit\rightarrow F\otimes E
\:,
&\mathord{
\begin{tikzpicture}[baseline = 1mm]
	\draw[<-] (0.4,0) to[out=90, in=0] (0.1,0.4);
	\draw[-] (0.1,0.4) to[out = 180, in = 90] (-0.2,0);
\end{tikzpicture}
}&:E\otimes F\rightarrow\unit\:,
\mathord{
\begin{tikzpicture}[baseline = 0]
	\draw[->] (0.28,-.3) to (-0.28,.4);
	\draw[-,white,line width=4pt] (-0.28,-.3) to (0.28,.4);
	\draw[->] (-0.28,-.3) to (0.28,.4);
\end{tikzpicture}
}&:E\otimes E \rightarrow E\otimes E
\:
\end{align}
The first and last generator of \ref{qHgens}, the dot and the positive crossing, are required to be invertible. Hence it makes sense to label dots by an arbitrary integer.
We denote the inverse of the positive crossing
by
\begin{align*}
\mathord{
\begin{tikzpicture}[baseline = -.5mm]
	\draw[->] (-0.28,-.3) to (0.28,.4);
	\draw[-,line width=4pt,white] (0.28,-.3) to (-0.28,.4);
	\draw[->] (0.28,-.3) to (-0.28,.4);
\end{tikzpicture}
}&:E\otimes E\rightarrow E \otimes E
\:,
\end{align*}
and call this the negative crossing.
Thus, we have that
\begin{align*}
\mathord{
\begin{tikzpicture}[baseline = -1mm]
	\draw[-] (0.28,-.6) to[out=90,in=-90] (-0.28,0);
	\draw[->] (-0.28,0) to[out=90,in=-90] (0.28,.6);
	\draw[-,line width=4pt,white] (-0.28,-.6) to[out=90,in=-90] (0.28,0);
	\draw[-] (-0.28,-.6) to[out=90,in=-90] (0.28,0);
	\draw[-,line width=4pt,white] (0.28,0) to[out=90,in=-90] (-0.28,.6);
	\draw[->] (0.28,0) to[out=90,in=-90] (-0.28,.6);
\end{tikzpicture}
}&=
\mathord{
\begin{tikzpicture}[baseline = -1mm]
	\draw[->] (0.18,-.6) to (0.18,.6);
	\draw[->] (-0.18,-.6) to (-0.18,.6);
\end{tikzpicture}
}
=
\mathord{
\begin{tikzpicture}[baseline = -1mm]
	\draw[->] (0.28,0) to[out=90,in=-90] (-0.28,.6);
	\draw[-,line width=4pt,white] (-0.28,0) to[out=90,in=-90] (0.28,.6);
	\draw[->] (-0.28,0) to[out=90,in=-90] (0.28,.6);
	\draw[-] (-0.28,-.6) to[out=90,in=-90] (0.28,0);
	\draw[-,line width=4pt,white] (0.28,-.6) to[out=90,in=-90] (-0.28,0);
	\draw[-] (0.28,-.6) to[out=90,in=-90] (-0.28,0);
\end{tikzpicture}
}\:.
\end{align*}
We also need the sideways crossings, both positive and negative,
\begin{align*}
\mathord{
\begin{tikzpicture}[baseline = -.5mm]
	\draw[->] (-0.28,-.3) to (0.28,.4);
	\draw[line width=4pt,white,-] (0.28,-.3) to (-0.28,.4);
	\draw[<-] (0.28,-.3) to (-0.28,.4);
\end{tikzpicture}
}&:=
\mathord{
\begin{tikzpicture}[baseline = 0]
	\draw[->] (0.3,-.5) to (-0.3,.5);
	\draw[line width=4pt,-,white] (-0.2,-.2) to (0.2,.3);
	\draw[-] (-0.2,-.2) to (0.2,.3);
        \draw[-] (0.2,.3) to[out=50,in=180] (0.5,.5);
        \draw[->] (0.5,.5) to[out=0,in=90] (0.8,-.5);
        \draw[-] (-0.2,-.2) to[out=230,in=0] (-0.6,-.5);
        \draw[-] (-0.6,-.5) to[out=180,in=-90] (-0.85,.5);
\end{tikzpicture}
}\:,
&\mathord{
\begin{tikzpicture}[baseline = -.5mm]
	\draw[<-] (0.28,-.3) to (-0.28,.4);
	\draw[line width=4pt,white,-] (-0.28,-.3) to (0.28,.4);
	\draw[->] (-0.28,-.3) to (0.28,.4);
\end{tikzpicture}
}&:=
\mathord{
\begin{tikzpicture}[baseline = 0]
	\draw[-] (-0.2,-.2) to (0.2,.3);
	\draw[-,line width=4pt,white] (0.3,-.5) to (-0.3,.5);
	\draw[->] (0.3,-.5) to (-0.3,.5);
        \draw[-] (0.2,.3) to[out=50,in=180] (0.5,.5);
        \draw[->] (0.5,.5) to[out=0,in=90] (0.8,-.5);
        \draw[-] (-0.2,-.2) to[out=230,in=0] (-0.6,-.5);
        \draw[-] (-0.6,-.5) to[out=180,in=-90] (-0.85,.5);
\end{tikzpicture}
}\:.
\end{align*}
The other defining relations are :
\begin{align*}
\mathord{
\begin{tikzpicture}[baseline = -.5mm]
	\draw[->] (0.28,-.3) to (-0.28,.4);
	\draw[line width=4pt,white,-] (-0.28,-.3) to (0.28,.4);
	\draw[->] (-0.28,-.3) to (0.28,.4);
\end{tikzpicture}
}-\mathord{
\begin{tikzpicture}[baseline = -.5mm]
	\draw[->] (-0.28,-.3) to (0.28,.4);
	\draw[line width=4pt,white,-] (0.28,-.3) to (-0.28,.4);
	\draw[->] (0.28,-.3) to (-0.28,.4);
\end{tikzpicture}
}&=
z\:\mathord{
\begin{tikzpicture}[baseline = -.5mm]
	\draw[->] (0.18,-.3) to (0.18,.4);
	\draw[->] (-0.18,-.3) to (-0.18,.4);
\end{tikzpicture}
}\:,
&\mathord{
\begin{tikzpicture}[baseline = -1mm]
	\draw[->] (0.45,-.6) to (-0.45,.6);
        \draw[-] (0,-.6) to[out=90,in=-90] (-.45,0);
        \draw[-,line width=4pt,white] (-0.45,0) to[out=90,in=-90] (0,0.6);
        \draw[->] (-0.45,0) to[out=90,in=-90] (0,0.6);
	\draw[-,line width=4pt,white] (0.45,.6) to (-0.45,-.6);
	\draw[<-] (0.45,.6) to (-0.45,-.6);
\end{tikzpicture}
}
&=
\mathord{
\begin{tikzpicture}[baseline = -1mm]
	\draw[->] (0.45,-.6) to (-0.45,.6);
        \draw[-,line width=4pt,white] (0,-.6) to[out=90,in=-90] (.45,0);
        \draw[-] (0,-.6) to[out=90,in=-90] (.45,0);
        \draw[->] (0.45,0) to[out=90,in=-90] (0,0.6);
	\draw[-,line width=4pt,white] (0.45,.6) to (-0.45,-.6);
	\draw[<-] (0.45,.6) to (-0.45,-.6);
\end{tikzpicture}
}
\:,
\qquad\qquad
\mathord{
\begin{tikzpicture}[baseline = -.5mm]
	\draw[->] (-0.28,-.3) to (0.28,.4);
      \node at (-0.16,-0.15) {$\dt$};
	\draw[-,line width=4pt,white] (0.28,-.3) to (-0.28,.4);
	\draw[->] (0.28,-.3) to (-0.28,.4);
\end{tikzpicture}
}=
\mathord{
\begin{tikzpicture}[baseline = -.5mm]
	\draw[->] (0.28,-.3) to (-0.28,.4);
	\draw[line width=4pt,white,-] (-0.28,-.3) to (0.28,.4);
	\draw[->] (-0.28,-.3) to (0.28,.4);
      \node at (0.145,0.23) {$\dt$};
\end{tikzpicture}
}\:,\\
\mathord{
\begin{tikzpicture}[baseline = -.8mm]
  \draw[->] (0.3,0) to (0.3,.4);
	\draw[-] (0.3,0) to[out=-90, in=0] (0.1,-0.4);
	\draw[-] (0.1,-0.4) to[out = 180, in = -90] (-0.1,0);
	\draw[-] (-0.1,0) to[out=90, in=0] (-0.3,0.4);
	\draw[-] (-0.3,0.4) to[out = 180, in =90] (-0.5,0);
  \draw[-] (-0.5,0) to (-0.5,-.4);
\end{tikzpicture}
}
&=
\mathord{\begin{tikzpicture}[baseline=-.8mm]
  \draw[->] (0,-0.4) to (0,.4);
\end{tikzpicture}
}\:,
&\mathord{
\begin{tikzpicture}[baseline = -.8mm]
  \draw[->] (0.3,0) to (0.3,-.4);
	\draw[-] (0.3,0) to[out=90, in=0] (0.1,0.4);
	\draw[-] (0.1,0.4) to[out = 180, in = 90] (-0.1,0);
	\draw[-] (-0.1,0) to[out=-90, in=0] (-0.3,-0.4);
	\draw[-] (-0.3,-0.4) to[out = 180, in =-90] (-0.5,0);
  \draw[-] (-0.5,0) to (-0.5,.4);
\end{tikzpicture}
}
&=
\mathord{\begin{tikzpicture}[baseline=-.8mm]
  \draw[<-] (0,-0.4) to (0,.4);
\end{tikzpicture}
}\:,\end{align*}
plus the {\em inversion relation} asserting that the following is
invertible:
\begin{align}
\label{qinver1}
\left[\!\!\!\!\!
\begin{array}{r}
\mathord{
\begin{tikzpicture}[baseline = 0]
	\draw[->] (-0.28,-.3) to (0.28,.3);
	\draw[-,line width=4pt,white] (0.28,-.3) to (-0.28,.3);
	\draw[<-] (0.28,-.3) to (-0.28,.3);
   \end{tikzpicture}
}\\
\mathord{
\begin{tikzpicture}[baseline = 1mm]
	\draw[<-] (0.4,0) to[out=90, in=0] (0.1,0.4);
      \node at (-0.15,0.45) {$\phantom\bullet$};
	\draw[-] (0.1,0.4) to[out = 180, in = 90] (-0.2,0);
\end{tikzpicture}
}\\
\mathord{
\begin{tikzpicture}[baseline = 1mm]
	\draw[<-] (0.4,0) to[out=90, in=0] (0.1,0.4);
	\draw[-] (0.1,0.4) to[out = 180, in = 90] (-0.2,0);
      \node at (-0.15,0.45) {$\phantom\bullet$};
      \node at (-0.15,0.2) {$\dt$};
\end{tikzpicture}
}\\\vdots\:\:\;\\
\mathord{
\begin{tikzpicture}[baseline = 1mm]
	\draw[<-] (0.4,0) to[out=90, in=0] (0.1,0.4);
	\draw[-] (0.1,0.4) to[out = 180, in = 90] (-0.2,0);
     \node at (-0.52,0.2) {$\scriptstyle{k-1}$};
      \node at (-0.15,0.42) {$\phantom\bullet$};
      \node at (-0.15,0.2) {$\dt$};
\end{tikzpicture}
}
\end{array}
\right]
&:
E \otimes F \:\rightarrow\:
F \otimes E \oplus \unit^{\oplus k}
&&\text{if $k \geq
  0$,}\\
\left[\:
\mathord{
\begin{tikzpicture}[baseline = 0]
	\draw[<-] (0.28,-.3) to (-0.28,.3);
	\draw[-,line width=4pt,white] (-0.28,-.3) to (0.28,.3);
	\draw[->] (-0.28,-.3) to (0.28,.3);
\end{tikzpicture}
}\:\:\:
\mathord{
\begin{tikzpicture}[baseline = -0.9mm]
	\draw[<-] (0.4,0.2) to[out=-90, in=0] (0.1,-.2);
	\draw[-] (0.1,-.2) to[out = 180, in = -90] (-0.2,0.2);
\end{tikzpicture}
}
\:\:\:
\mathord{
\begin{tikzpicture}[baseline = -0.9mm]
	\draw[<-] (0.4,0.2) to[out=-90, in=0] (0.1,-.2);
	\draw[-] (0.1,-.2) to[out = 180, in = -90] (-0.2,0.2);
      \node at (0.38,0) {$\dt$};
\end{tikzpicture}
}
\:\:\:\cdots
\:\:\:
\mathord{
\begin{tikzpicture}[baseline = -0.9mm]
	\draw[<-] (0.4,0.2) to[out=-90, in=0] (0.1,-.2);
	\draw[-] (0.1,-.2) to[out = 180, in = -90] (-0.2,0.2);
     \node at (0.83,0) {$\scriptstyle{-k-1}$};
      \node at (0.38,0) {$\dt$};
\end{tikzpicture}
}
\right]
&:E \otimes F \oplus
\unit^{\oplus (-k)}
\:\rightarrow\:
 F \otimes  E&&
\text{if $k \leq 0$.}\label{qinver2}
\end{align}
The situation is slightly more delicate than in the degenerate case as
it is also necessary to impose one additional
relation:
\begin{itemize}
\item
If $k > 0$ we require that
$\begin{tikzpicture}[baseline = 1mm]
  \draw[<-] (0,0.4) to[out=180,in=90] (-.2,0.2);
  \draw[-] (0.2,0.2) to[out=90,in=0] (0,.4);
 \draw[-] (-.2,0.2) to[out=-90,in=180] (0,0);
      \node at (-0.2,0.2) {$\dt$};
      \node at (-.5,.2) {$\scriptstyle{-1}$};
  \draw[-] (0,0) to[out=0,in=-90] (0.2,0.2);
      \node at (0,0) {$\diam$};
\end{tikzpicture}
=-t^2 1_\unit$
where $\:\begin{tikzpicture}[baseline = .75mm]
	\draw[-] (0.3,0.3) to[out=-90, in=0] (0.1,0);
	\draw[->] (0.1,0) to[out = 180, in = -90] (-0.1,0.3);
      \node at (.1,0) {$\diam$};
\end{tikzpicture}\:$ is the last entry of the
inverse of the matrix (\ref{qinver1}).
\item
If $k < 0$ we require that
$\begin{tikzpicture}[baseline = 1mm]
  \draw[-] (0,0.4) to[out=180,in=90] (-.2,0.2);
  \draw[-] (0.2,0.2) to[out=90,in=0] (0,.4);
 \draw[-] (-.2,0.2) to[out=-90,in=180] (0,0);
  \draw[->] (0,0) to[out=0,in=-90] (0.2,0.2);
      \node at (-0.2,0.2) {$\dt$};
      \node at (-.5,.2) {$\scriptstyle{-1}$};
      \node at (0,.4) {$\diam$};
\end{tikzpicture}
=-t^{-2} 1_\unit$
where $\:\begin{tikzpicture}[baseline = .75mm]
	\draw[-] (0.3,0) to[out=90, in=0] (0.1,0.3);
	\draw[->] (0.1,0.3) to[out = 180, in = 90] (-0.1,0);
      \node at (0.1,.3) {$\diam$};
\end{tikzpicture}\:$
is the last entry of the inverse of the matrix (\ref{qinver2}).
\item
If $k=0$ there are two equivalent presentations
here:
if one picks
(\ref{qinver1}) the additional relation is
$\mathord{
\begin{tikzpicture}[baseline = 1.5mm]
	\draw[<-] (-0.2,.5) to[out=-90,in=90] (0.2,.05);
	\draw[-] (0.2,.05) to[out=-90, in=0] (0,-0.15);
	\draw[-] (0,-0.15) to[out = 180, in = -90] (-0.2,.05);
	\draw[-,line width=4pt,white] (0.25,.5) to[out=-90,in=90] (-0.2,.05);
	\draw[-] (0.2,.5) to[out=-90,in=90] (-0.2,.05);
\draw[-](-.2,.5) to [out=90,in=180] (0,.7);
\draw[-](.2,.5) to [out=90,in=0] (0,.7);
\end{tikzpicture}
}= \frac{1-t^{-2}}{z}\unit$
where
$\begin{tikzpicture}[baseline = -1mm]
	\draw[->] (0.2,-.2) to (-0.2,.2);
	\draw[-,line width=4pt,white] (-0.2,-.2) to (0.2,.2);
	\draw[<-] (-0.2,-.2) to (0.2,.2);
\end{tikzpicture}
:= \Big(
\begin{tikzpicture}[baseline = -1mm]
	\draw[->] (-0.2,-.2) to (0.2,.2);
	\draw[-,line width=4pt,white] (0.2,-.2) to (-0.2,.2);
	\draw[<-] (0.2,-.2) to (-0.2,.2);
   \end{tikzpicture}
\Big)^{-1}$, while for
(\ref{qinver2}) it is
$\mathord{
\begin{tikzpicture}[baseline = 1.5mm]
	\draw[-] (0.2,.05) to[out=-90, in=0] (0,-0.15);
	\draw[-] (0,-0.15) to[out = 180, in = -90] (-0.2,.05);
	\draw[->] (0.2,.5) to[out=-90,in=90] (-0.2,.05);
\draw[-](-.2,.5) to [out=90,in=180] (0,.7);
\draw[-](.2,.5) to [out=90,in=0] (0,.7);
	\draw[-,line width=4pt,white] (-0.2,.5) to[out=-90,in=90] (0.2,.05);
	\draw[-] (-0.2,.5) to[out=-90,in=90] (0.2,.05);
\end{tikzpicture}
}
=\frac{t^2-1}{z} 1_\unit$
where
$\mathord{
\begin{tikzpicture}[baseline = -1mm]
	\draw[<-] (-0.2,-.2) to (0.2,.2);
	\draw[-,line width=4pt,white] (0.2,-.2) to (-0.2,.2);
	\draw[->] (0.2,-.2) to (-0.2,.2);
   \end{tikzpicture}
} := \Big(\begin{tikzpicture}[baseline = -1mm]
	\draw[<-] (0.2,-.2) to (-0.2,.2);
	\draw[-,line width=4pt,white] (-0.2,-.2) to (0.2,.2);
	\draw[->] (-0.2,-.2) to (0.2,.2);
\end{tikzpicture}
\Big)^{-1}$.
\end{itemize}
The resulting category then contains unique morphisms
$\:\begin{tikzpicture}[baseline = .75mm]
	\draw[-] (0.3,0.3) to[out=-90, in=0] (0.1,0);
	\draw[->] (0.1,0) to[out = 180, in = -90] (-0.1,0.3);
\end{tikzpicture}\:$
and
$\:\begin{tikzpicture}[baseline = .75mm]
	\draw[-] (0.3,0) to[out=90, in=0] (0.1,0.3);
	\draw[->] (0.1,0.3) to[out = 180, in = 90] (-0.1,0);
\end{tikzpicture}\:$; see \cite[Lemma 4.3]{BSW20b}.

\section{Symmetric product of Heisenberg categories}
\label{sec:symproduct}
 We recall symmetric products of degenerate and quantum Heisenberg categories
 from \cite{BSW20b,BSW23} and give an upward presentation.
    \subsection{Degenerate case}
	As outlined in \cite{BSW20b,BSW23}, given strict $\kk$-linear monoidal categories $\mathcal{C}$ and $\mathcal{D}$,
   the {\em symmetric product} $\mathcal C \odot \mathcal D$ is the strict  $\kk$-linear monoidal
	category obtained from their free product by
	adjoining isomorphisms $\sigma_{X,Y}:X \otimes Y \stackrel{\sim}{\rightarrow} Y \otimes
	X$
	such that $\sigma_{Y,X} = \sigma_{X,Y}^{-1}$
	for each pair of objects $X \in \mathcal C$ and $Y \in \mathcal D$,
	subject also to the relations
	\begin{align}\label{symm1}
		\sigma_{X_1 \otimes X_2, Y} &= (\sigma_{X_1,Y} \otimes 1_{X_2}) \circ
		(1_{X_1} \otimes \sigma_{X_2,Y}),&
		\sigma_{X, Y_1 \otimes Y_2} &= (1_{Y_1} \otimes \sigma_{X,Y_2}) \circ
		(\sigma_{X, Y_1} \otimes 1_{Y_2})
	\end{align}
	\begin{align}\label{symm2}
		\sigma_{X_2,Y} \circ (f \otimes 1_Y)  &= (1_Y \otimes f) \circ
		\sigma_{X_1,Y},&
		\sigma_{X,Y_2} \circ (1_X \otimes g) &= (g \otimes 1_X)\circ \sigma_{X,Y_1}
	\end{align}
	for all $X, X_1,X_2 \in\mathcal C, Y, Y_1,Y_2 \in \mathcal D$ and
	$f\in \Hom_{\mathcal C}(X_1,X_2), g \in \Hom_{\mathcal
		D}(Y_1,Y_2)$.

	\smallskip

	For monoidal categories $\mathcal C , \mathcal D$ and $  \mathcal E$, we can always identify $(\mathcal C \odot \mathcal D) \odot \mathcal E$ with $\mathcal C \odot (\mathcal D \odot \mathcal E)$, and we will write it by $\mathcal C \odot \mathcal D \odot \mathcal E.$ Similarly, for monoidal categories $\mathcal C_1 ,\cdots, \mathcal C_d$, we can define  their symmetric product $\mathcal C_1\odot\cdots\odot \mathcal C_d$ accordingly. In this subsection, we will be mainly interested in the symmetric product of degenerate Heisenberg categories.

	Let $$\symHeis_d:=\Heis_{k_1}\odot\cdots \odot \Heis_{k_d}$$ be the symmetric product
	of $\Heis_{k_{\red{\alpha}}}$ for ${\red{\alpha}}=1,2,\cdots, d$.  Diagrammatically it is often convenient to use different colors to distinguish morphisms and objects in different copies of $\Heis_{k_{\red{\alpha}}}$.

		By definitions, the generating objects of $\symHeis_d$ are: $$F^{(\red\alpha)}=
$$ for each $\red\alpha\in \{1,\cdots, d\}$. By (\ref{symm1}), we know that $\sigma_{X\otimes Y,Z}$ is determined by $\sigma_{X,Z}$ and $\sigma_{ Y,Z}$, hence the
	generating morphisms of $\symHeis_d$ are $$%
	$$ for all $\red\alpha,\blue\beta\in \{1,\cdots, d\}$ with $\red\alpha\neq \blue\beta$. Since $\sigma_{Y,X} = \sigma_{X,Y}^{-1}$, we have the following relations:
	\begin{align}\label{mixedrelation1}
		\mathord{
			%
		}\:.
	\end{align}

In relations \eqref{mixedrelation3}--\eqref{3color}, an unoriented strand may
carry either admissible orientation; every relation preserves the
orientations of its strands.

	By (\ref{symm2}), we have
	\begin{align}\label{mixedrelation2}
		\mathord{
			%
		}\:.
	\end{align}

	If $d\geq 3$, then we also have the following relations

	\begin{align}\label{3color}
		\mathord{
			%
		}\:.
	\end{align}

	\subsection{An upward presentation of the symmetric product of Heisenberg categories}
	The following presentation uses only upward crossings among its generators.
		\begin{definition}\label{DoubleHeisenberg}
		The \emph{$d$-multiple degenerate Heisenberg category} $\symHeis_d'$ is  the strict $\kk$-linear monoidal category
		generated by the objects
		$$F^{(\red\alpha)}=%
				$$ for all $\red\alpha,\blue\beta\in \{1,\cdots, d\}$ with $\red\alpha\neq \blue\beta$,
					subject to the certain relations as follows.
					We require that
					$%
\:,$
					 satisfy all relations of the degenerate Heisenberg category $\Heis_{k_{\red{\alpha}}}.$
							In addition, for each $\red\alpha\neq \blue\beta$, $\red\alpha\neq \green\gamma$ and $\blue\beta\neq \green\gamma,$ we also have
							\begin{align}\label{mixedHecke}
								\mathord{
									%
								}\:.
							\end{align}

							We also introduce the sideways mixed crossings,
							\begin{align}\label{defcross}
								\mathord{
									%
								}\:,
							\end{align} then we require \begin{align}\label{mixedMackey}
								\mathord{
									%
								}\:.\end{align}
						\end{definition}

							\begin{theorem}\label{thm:equiv}
								The $d$-multiple Heisenberg category $\symHeis_d'$ is equivalent to $\symHeis_d$.
						\end{theorem}
					In the proof, thick strands denote morphisms in $\symHeis_d'$
					and thin strands denote morphisms in $\symHeis_d$.

					We define $\Theta:\symHeis_d'\to \symHeis_d$ by sending generating objects $$\Theta(%
$$ for each $\red\alpha\in \{1,\cdots, d\}$ and sending
					morphisms $$%
					$$ for all $\red\alpha,\blue\beta\in \{1,\cdots, d\}$ with $\red\alpha\neq \blue\beta$ to the corresponding generating morphisms$$%
					. $$

					We also define its inverse $\Omega:\symHeis_d\to \symHeis_d'$  by sending generating objects $$\Omega(%
$$ for each $\red\alpha\in \{1,\cdots, d\}$ to the corresponding generating objects  and sending $$%
					, $$
					morphisms for all $\red\alpha,\blue\beta\in \{1,\cdots, d\}$ with $\red\alpha\neq \blue\beta$ to the corresponding generating morphisms $$%
}.$$

					We need to check that
					\begin{itemize}
						\item $\Theta$ and $\Omega$ are well-defined;
						\item $\Theta\circ\Omega=\Id$ and $\Omega\circ \Theta=\Id.$
					\end{itemize}

					\begin{lemma}\label{mixedmackey}
						In $\symHeis_d$, we have \begin{align*}
							\mathord{
								%
							}\:.
						\end{align*}
					\end{lemma}
					\begin{proof}  For the first relation, since $%
$ are inverse to each other, it suffices to show that
						\begin{align*}\mathord{
								%
.
							}
						\end{align*} The second relation follows similarly. For the third relation, we first observe that
						\begin{align}
                        \label{middle}
                        \mathord{
								%
$ are inverse to each other, we have
						\begin{align*}%
}
						\end{align*}
						This concludes the proof.
					\end{proof}

					\begin{lemma}\label{Lemma:dot}
						In $\symHeis_d'$, the relations (\ref{mixedrelation2}) hold.
					\end{lemma}
					\begin{proof} We only prove the third relation as an example.\begin{align*}
							\mathord{
								%
.
							}
						\end{align*}
						This concludes the proof.
					\end{proof}
					\begin{lemma}
						In $\symHeis_d'$, we have the following relations:

						\begin{align}\label{pic}
							\mathord{
								%
							}
							\:.
						\end{align*}

					\end{lemma}
					\begin{proof} We only show the first one as an example.
						\begin{align*}
							\mathord{
								%
							}
						\end{align*}
							The rest follow similarly.
					\end{proof}

					\begin{lemma}\label{Lem:adj}In $\symHeis_d'$,  the relations (\ref{mixedrelation3}) hold.
					\end{lemma}
					\begin{proof} We only prove for one orientation as an example:
						\begin{align*}\mathord{
								%
}
						\end{align*}
						The rest follow similarly.
					\end{proof}
					\begin{lemma}
						In $\symHeis_d'$,		we have the following relations:

						\begin{align*}
							\mathord{
								%
							}\:,
						\end{align*}
					\end{lemma}

					\begin{proof} These relations follow from the first and third relations of   (\ref{mixedHecke}).
					\end{proof}

					\begin{lemma}\label{Lem:braid}In $\symHeis_d'$, the relations (\ref{twisbraid})-- (\ref{3color}) holds.
					\end{lemma}
					\begin{proof}
						It suffices to treat one orientation.  The defining relations give
						\begin{align}\label{xx}\mathord{
								%
							}
						\end{align*}

						On the other hand we see that
						\begin{align*}\mathord{
								%
							}
						\end{align*}
						This proves the lemma.
					\end{proof}

					\begin{proof}[Proof of Theorem~\ref{thm:equiv}]

							We first check that all of the defining relations (\ref{mixedHecke}) and (\ref{mixedMackey}) of $\symHeis_d'$  from Definition 2.2
							are satisfied in
							$\symHeis_d$, so that there is a strict $k$-linear monoidal functor $\Theta:\symHeis_d'\to \symHeis_d.$ The relations in (\ref{mixedHecke}) are satisfied.
							The relations (\ref{mixedMackey}) are deduced from Lemma \ref{mixedmackey}.

							To see that $\Omega$ is well defined, we must verify the relations (\ref{mixedrelation1})--(\ref{3color}). The relations  (\ref{mixedrelation1}) are checked in Lemma \ref{mixedmackey}.
							The relations  (\ref{mixedrelation2}) are checked in Lemma \ref{Lemma:dot}.
							The relations  (\ref{mixedrelation3}) are checked in Lemma \ref{Lem:adj}.
							The relations (\ref{twisbraid})--(\ref{3color}) are checked in Lemma \ref{Lem:braid}.

							It remains to check that $\Theta$ and $\Omega$ are inverse.
							The equality $\Omega\circ\Theta=\Id$ is immediate, while
							$\Theta\circ\Omega=\Id$ reduces to
							$$\Theta\circ\Omega(			\begin{tikzpicture}[baseline = -1mm]
								\draw[->,blue] (0.2,-.2) to (-0.2,.2);
								\draw[<-,red] (-0.2,-.2) to (0.2,.2);
								\node at (-0.2,-.34) {$\scriptstyle{\red\alpha}$};
								\node at (0.2,-.34) {$\scriptstyle{\blue\beta}$};
							\end{tikzpicture}
							)=\begin{tikzpicture}[baseline = -1mm]
								\draw[->,blue] (0.2,-.2) to (-0.2,.2);
								\draw[<-,red] (-0.2,-.2) to (0.2,.2);
								\node at (-0.2,-.34) {$\scriptstyle{\red\alpha}$};
								\node at (0.2,-.34) {$\scriptstyle{\blue\beta}$};
							\end{tikzpicture}
							,
							\Theta\circ\Omega(\begin{tikzpicture}[baseline = -1mm]
								\draw[<-,blue] (0.2,-.2) to (-0.2,.2);
								\draw[->,red] (-0.2,-.2) to (0.2,.2);
								\node at (-0.2,-.34) {$\scriptstyle{\red\alpha}$};
								\node at (0.2,-.34) {$\scriptstyle{\blue\beta}$};
							\end{tikzpicture})=\begin{tikzpicture}[baseline = -1mm]
								\draw[<-,blue] (0.2,-.2) to (-0.2,.2);
								\draw[->,red] (-0.2,-.2) to (0.2,.2);
								\node at (-0.2,-.34) {$\scriptstyle{\red\alpha}$};
								\node at (0.2,-.34) {$\scriptstyle{\blue\beta}$};
							\end{tikzpicture}\,,
							\Theta\circ\Omega(\begin{tikzpicture}[baseline = -1mm]
								\draw[<-,blue] (0.2,-.2) to (-0.2,.2);
								\draw[<-,red] (-0.2,-.2) to (0.2,.2);
								\node at (-0.2,-.34) {$\scriptstyle{\red\alpha}$};
								\node at (0.2,-.34) {$\scriptstyle{\blue\beta}$};
							\end{tikzpicture})=\begin{tikzpicture}[baseline = -1mm]
								\draw[<-,blue] (0.2,-.2) to (-0.2,.2);
								\draw[<-,red] (-0.2,-.2) to (0.2,.2);
								\node at (-0.2,-.34) {$\scriptstyle{\red\alpha}$};
								\node at (0.2,-.34) {$\scriptstyle{\blue\beta}$};
							\end{tikzpicture},$$
							which follow from Lemma \ref{mixedmackey}.
						\end{proof}

		\begin{corollary}
			The strict monoidal category $\symHeis_d$ is strictly pivotal.
		\end{corollary}

\subsection{Quantum case}
The same construction applies to quantum Heisenberg categories.  Retaining the
notation $\symHeis_d$, set
$$\symHeis_d:=\Heis_{k_1}(z_1,t_1)\odot\cdots \odot \Heis_{k_d}(z_d,t_d)$$
for the symmetric product of
$\Heis_{k_{\red{\alpha}}}(z_{\red{\alpha}},t_{\red{\alpha}})$,
$\red{\alpha}=1,\ldots,d$.
For the quantum presentation put $\KK_d=\kk[t_1^{\pm1},\ldots,t_d^{\pm1}]$ and extend scalars of each factor to $\KK_d$.

\begin{theorem}\label{dmultiHeisenberg}
		The category $\symHeis_d$ is equivalent to  the strict $\KK_d$-linear monoidal category
		generated by the objects
		$$F^{(\red\alpha)}=\begin{tikzpicture}[baseline = -1mm]
			\draw[->,red] (0.08,-.2) to (0.08,.2);
			\node at (0.08,-.3) {$\scriptstyle{\red\alpha}$};
		\end{tikzpicture}\,,\quad E^{(\red\alpha)}=\begin{tikzpicture}[baseline = -1mm]
			\draw[<-,red] (0.08,-.2) to (0.08,.2);
			\node at (0.08,-.3) {$\scriptstyle{\red\alpha}$};
		\end{tikzpicture}$$ for each $\red\alpha\in \{1,\cdots, d\}$
				and
				morphisms $$\begin{tikzpicture}[baseline = -1mm,red]
					\draw[->,red] (0.08,-.2) to (0.08,.2);
					\node at (0.08,0) {$\dott$};
					\node at (0.08,-.3) {$\scriptstyle{\red\alpha}$};
				\end{tikzpicture}\,,\:
				\begin{tikzpicture}[baseline = -1mm]
					\draw[->,red] (0.2,-.2) to (-0.2,.2);
					\draw[-,white,line width=4pt] (-0.2,-.2) to (0.2,.2);
					\draw[->,red] (-0.2,-.2) to (0.2,.2);
					\node at (0.2,-.3) {$\scriptstyle{\red\alpha}$};
					\node at (-0.2,-.3) {$\scriptstyle{\red\alpha}$};
				\end{tikzpicture}\,,\:
				\begin{tikzpicture}[baseline = .75mm]
					\draw[<-,red] (0.3,0.3) to[out=-90, in=0] (0.1,0);
					\draw[-,red] (0.1,0) to[out = 180, in = -90] (-0.1,0.3);
					\node at (0.3,0.38) {$\scriptstyle{\red\alpha}$};
				\end{tikzpicture}\,,\:
				\begin{tikzpicture}[baseline = .75mm]
					\draw[<-,red] (0.3,0) to[out=90, in=0] (0.1,0.3);
					\draw[-,red] (0.1,0.3) to[out = 180, in = 90] (-0.1,0);
					\node at (0.3,-.08) {$\scriptstyle{\red\alpha}$};
				\end{tikzpicture}\,,
				\:\begin{tikzpicture}[baseline = -1mm]
					\draw[->,blue] (0.2,-.2) to (-0.2,.2);
					\draw[->,red] (-0.2,-.2) to (0.2,.2);
					\node at (-0.2,-.34) {$\scriptstyle{\red\alpha}$};
					\node at (0.2,-.34) {$\scriptstyle{\blue\beta}$};
				\end{tikzpicture}
				$$ for all $\red\alpha,\blue\beta\in \{1,\cdots, d\}$ with $\red\alpha\neq \blue\beta$,
					subject to the relations as follows.
					We require that
					$\begin{tikzpicture}[baseline = -1mm,red]
						\draw[->,red] (0.08,-.2) to (0.08,.2);
						\node at (0.08,0) {$\dott$};
						\node at (0.08,-.28) {$\scriptstyle{\red\alpha}$};
					\end{tikzpicture}$,
					$\begin{tikzpicture}[baseline = -1mm]
						\draw[->,red] (0.2,-.2) to (-0.2,.2);
						\draw[-,white,line width=4pt] (-0.2,-.2) to (0.2,.2);
						\draw[->,red] (-0.2,-.2) to (0.2,.2);
						\node at (0.2,-.28) {$\scriptstyle{\red\alpha}$};
						\node at (-0.2,-.28) {$\scriptstyle{\red\alpha}$};
					\end{tikzpicture}\:,$
					$\begin{tikzpicture}[baseline = .75mm]
						\draw[<-,red] (0.3,0.3) to[out=-90, in=0] (0.1,0);
						\draw[-,red] (0.1,0) to[out = 180, in = -90] (-0.1,0.3);
						\node at (0.3,0.38) {$\scriptstyle{\red\alpha}$};
					\end{tikzpicture}\:,$
					$\begin{tikzpicture}[baseline = .75mm]
						\draw[<-,red] (0.3,0) to[out=90, in=0] (0.1,0.3);
						\draw[-,red] (0.1,0.3) to[out = 180, in = 90] (-0.1,0);
						\node at (0.3,-.08) {$\scriptstyle{\red\alpha}$};
					\end{tikzpicture}\:,$ satisfy all relations of quantum Heisenberg category $\Heis_{k_{\red{\alpha}}}(z_{\red{\alpha}},t_{\red{\alpha}}).$
							In addition, for each $\red\alpha\neq \blue\beta$, $\red\alpha\neq \green\gamma$ and $\blue\beta\neq \green\gamma,$ we also have
							\begin{align*}
								\mathord{
									\begin{tikzpicture}[baseline = -1mm]
										\draw[->,thin,red] (0.28,0) to[out=90,in=-90] (-0.28,.6);
										\draw[->,thin,blue] (-0.28,0) to[out=90,in=-90] (0.28,.6);
										\draw[-,thin,red] (-0.28,-.6) to[out=90,in=-90] (0.28,0);
										\draw[-,thin,blue] (0.28,-.6) to[out=90,in=-90] (-0.28,0);
										\node at (-0.28,-.7) {$\scriptstyle{\red\alpha}$};
										\node at (0.28,-.7) {$\scriptstyle{\blue\beta}$};
									\end{tikzpicture}
								}\,=&
								\mathord{
									\begin{tikzpicture}[baseline = -1mm]
										\draw[->,thin,blue] (0.18,-.6) to (0.18,.6);
										\draw[->,thin,red] (-0.18,-.6) to (-0.18,.6);
										\node at (-0.18,-.7) {$\scriptstyle{\red\alpha}$};
										\node at (0.18,-.7) {$\scriptstyle{\blue\beta}$};
									\end{tikzpicture}
								}\:,
								\:&
								\mathord{
									\begin{tikzpicture}[baseline = -.5mm]
										\draw[thin,red,->] (-0.28,-.3) to (0.28,.4);
										\node[red] at (-0.16,-0.15) {$\dott$};
										\draw[->,thin,blue] (0.28,-.3) to (-0.28,.4);
										\node at (-0.28,-.4) {$\scriptstyle{\red\alpha}$};
										\node at (0.28,-.4) {$\scriptstyle{\blue\beta}$};
									\end{tikzpicture}
								}\,=&
								\mathord{
									\begin{tikzpicture}[baseline = -.5mm]
										\draw[->,thin,blue] (0.28,-.3) to (-0.28,.4);
										\draw[thin,red,->] (-0.28,-.3) to (0.28,.4);
										\node[red] at (0.145,0.23) {$\dott$};
										\node at (-0.28,-.4) {$\scriptstyle{\red\alpha}$};
										\node at (0.28,-.4) {$\scriptstyle{\blue\beta}$};
									\end{tikzpicture}
								}\:,&
								\mathord{
									\begin{tikzpicture}[baseline = -1mm]
										\draw[<-,thin,red] (-0.45,.6) to (0.45,-.6);
										\draw[-,line width=4pt,white] (-0.45,-.6) to (0.45,.6);
										\draw[->,thin,red] (-0.45,-.6) to (0.45,.6);
										\draw[-,thin,blue] (0,-.6) to[out=90,in=-90] (-.45,0);
										\draw[->,thin,blue] (-0.45,0) to[out=90,in=-90] (0,0.6);
										\node at (-0.45,-.7) {$\scriptstyle{\red\alpha}$};
										\node at (0.45,-.7) {$\scriptstyle{\red\alpha}$};
										\node at (0,-.7) {$\scriptstyle{\blue\beta}$};
									\end{tikzpicture}
								}
								=&
								\mathord{
									\begin{tikzpicture}[baseline = -1mm]
										\draw[<-,thin,red] (-0.45,.6) to (0.45,-.6);
									\draw[-,line width=4pt,white] (-0.45,-.6) to (0.45,.6);
										\draw[->,thin,red] (-0.45,-.6) to (0.45,.6);
										\draw[-,thin,blue] (0,-.6) to[out=90,in=-90] (.45,0);
										\draw[->,thin,blue] (0.45,0) to[out=90,in=-90] (0,0.6);
										\node at (-0.45,-.7) {$\scriptstyle{\red\alpha}$};
										\node at (0.45,-.7) {$\scriptstyle{\red\alpha}$};
										\node at (0,-.7) {$\scriptstyle{\blue\beta}$};
									\end{tikzpicture}
								}\:,&
								\mathord{
									\begin{tikzpicture}[baseline = -1mm]
										\draw[->,thin,green] (0.45,-.6) to (-0.45,.6);
										\draw[<-,thin,red] (0.45,.6) to (-0.45,-.6);
										\draw[-,thin,blue] (0,-.6) to[out=90,in=-90] (-.45,0);
										\draw[->,thin,blue] (-0.45,0) to[out=90,in=-90] (0,0.6);
										\node at (-0.45,-.7) {$\scriptstyle{\red\alpha}$};
										\node at (0,-.7) {$\scriptstyle{\blue\beta}$};
										\node at (0.45,-.7) {$\scriptstyle{\darkg\gamma}$};
									\end{tikzpicture}
								}
								=&
								\mathord{
									\begin{tikzpicture}[baseline = -1mm]
										\draw[->,thin,green] (0.45,-.6) to (-0.45,.6);
										\draw[<-,thin,red] (0.45,.6) to (-0.45,-.6);
										\draw[-,thin,blue] (0,-.6) to[out=90,in=-90] (.45,0);
										\draw[->,thin,blue] (0.45,0) to[out=90,in=-90] (0,0.6);
										\node at (-0.45,-.7) {$\scriptstyle{\red\alpha}$};
										\node at(0,-.7) {$\scriptstyle{\blue\beta}$};
										\node at (0.45,-.7) {$\scriptstyle{\darkg\gamma}$};
									\end{tikzpicture}
								}\:.
							\end{align*}

							We also introduce the sideways mixed crossings,
							\begin{align*}
								\mathord{
									\begin{tikzpicture}[baseline = -.5mm]
										\draw[->,red] (-0.28,-.3) to (0.28,.4);
										\draw[<-,blue] (0.28,-.3) to (-0.28,.4);
										\node at (-0.28,-.4) {$\scriptstyle{\red\alpha}$};
										\node at (0.28,-.4) {$\scriptstyle{\blue\beta}$};
									\end{tikzpicture}
								}&:=
								\mathord{
									\begin{tikzpicture}[baseline = 0]
										\draw[->,red] (0.3,-.5) to (-0.3,.5);
										\draw[-,blue] (-0.2,-.2) to (0.2,.3);
										\draw[-,blue] (0.2,.3) to[out=50,in=180] (0.5,.5);
										\draw[->,blue] (0.5,.5) to[out=0,in=90] (0.8,-.5);
										\draw[-,blue] (-0.2,-.2) to[out=230,in=0] (-0.6,-.5);
										\draw[-,blue] (-0.6,-.5) to[out=180,in=-90] (-0.85,.5);\
										\node at (0.3,-.6) {$\scriptstyle{\red\alpha}$};
										\node at (0.8,-.6) {$\scriptstyle{\blue\beta}$};
									\end{tikzpicture}
								}\:,
								\:&\mathord{
									\begin{tikzpicture}[baseline = -.5mm]
										\draw[<-,red] (-0.28,-.3) to (0.28,.4);
										\draw[->,blue] (0.28,-.3) to (-0.28,.4);
										\node at (-0.28,-.4) {$\scriptstyle{\red\alpha}$};
										\node at (0.28,-.4) {$\scriptstyle{\blue\beta}$};
									\end{tikzpicture}
								}&:=
								\mathord{
									\begin{tikzpicture}[baseline = 0]
										\draw[-,red] (-0.2,.2) to (0.2,-.3);
										\draw[<-,blue] (0.3,.5) to (-0.3,-.5);
										\draw[-,red] (0.2,-.3) to[out=130,in=180] (0.5,-.5);
										\draw[-,red] (0.5,-.5) to[out=0,in=270] (0.8,.5);
										\draw[-,red] (-0.2,.2) to[out=130,in=0] (-0.5,.5);
										\draw[->,red] (-0.5,.5) to[out=180,in=-270] (-0.8,-.5);
										\node at (-0.8,-.6) {$\scriptstyle{\red\alpha}$};
										\node at (-0.3,-.6) {$\scriptstyle{\blue\beta}$};
									\end{tikzpicture}
								}\:,
							\end{align*} then we require \begin{align}\label{mixedMackeyquantum}
								\mathord{
									\begin{tikzpicture}[baseline = -1mm]
										\draw[->,thin,red] (0.28,0) to[out=90,in=-90] (-0.28,.6);
										\draw[-,thin,blue] (-0.28,0) to[out=90,in=-90] (0.28,.6);
										\draw[-,thin,red] (-0.28,-.6) to[out=90,in=-90] (0.28,0);
										\draw[<-,thin,blue] (0.28,-.6) to[out=90,in=-90] (-0.28,0);
										\node at (-0.28,-.7) {$\scriptstyle{\red\alpha}$};
										\node at (0.28,-.7) {$\scriptstyle{\blue\beta}$};
									\end{tikzpicture}
								}\,=&
								\mathord{
									\begin{tikzpicture}[baseline = -1mm]
										\draw[<-,thin,blue] (0.18,-.6) to (0.18,.6);
										\draw[->,thin,red] (-0.18,-.6) to (-0.18,.6);
										\node at (-0.18,-.7) {$\scriptstyle{\red\alpha}$};
										\node at (0.18,-.7) {$\scriptstyle{\blue\beta}$};
									\end{tikzpicture}
								}\:,\:&
								\mathord{
									\begin{tikzpicture}[baseline = -1mm]
										\draw[-,thin,blue] (0.28,-.6) to[out=90,in=-90] (-0.28,0);
										\draw[->,thin,blue] (-0.28,0) to[out=90,in=-90] (0.28,.6);
										\draw[<-,thin,red] (-0.28,-.6) to[out=90,in=-90] (0.28,0);
										\draw[-,thin,red] (0.28,0) to[out=90,in=-90] (-0.28,.6);
										\node at (-0.28,-.7) {$\scriptstyle{\red\alpha}$};
										\node at (0.28,-.7) {$\scriptstyle{\blue\beta}$};
									\end{tikzpicture}
								}\,=&
								\mathord{
									\begin{tikzpicture}[baseline = -1mm]
										\draw[->,thin,blue] (0.18,-.6) to (0.18,.6);
					\draw[<-,thin,red] (-0.18,-.6) to (-0.18,.6);
								\node at (-0.18,-.7) {$\scriptstyle{\red\alpha}$};
								\node at (0.18,-.7) {$\scriptstyle{\blue\beta}$};
					\end{tikzpicture}
								}\:.\end{align}
						\end{theorem}
\begin{proof}
    This theorem follows from a rerun of the proof of Theorem~\ref{thm:equiv} in the quantum case.
\end{proof}

  \begin{remark}
    For both the degenerate and quantum Heisenberg categories, the presentations involve certain inversion relations; see also \cite[Definition 1.1]{B18} and \cite[Definition 3.1]{BSW20b}. These presentations are easier to work with when constructing $\Heis_k$-module categories since it involves fewer
generators and relations. In fact, the relations \eqref{mixedMackeyquantum} can and should be interpreted as the following inversion relations (for any $\red{\alpha}\neq \blue{\beta}$)
\[ \begin{tikzpicture}[baseline = -.5mm]
										\draw[->,red] (-0.28,-.3) to (0.28,.4);
										\draw[<-,blue] (0.28,-.3) to (-0.28,.4);
										\node at (-0.28,-.4) {$\scriptstyle{\red\alpha}$};
										\node at (0.28,-.4) {$\scriptstyle{\blue\beta}$};
						\end{tikzpicture}: E^{(\red{\alpha})}\otimes F^{(\blue{\beta})}\to F^{(\blue{\beta})}\otimes E^{(\red{\alpha})} \text{ has an inverse }\begin{tikzpicture}[baseline = -.5mm]
										\draw[->,blue] (-0.28,-.3) to (0.28,.4);
										\draw[<-,red] (0.28,-.3) to (-0.28,.4);
										\node at (-0.28,-.4) {$\scriptstyle{\blue\beta}$};
										\node at (0.28,-.4) {$\scriptstyle{\red\alpha}$};
							\end{tikzpicture}: F^{(\blue{\beta})}\otimes E^{(\red{\alpha})}\to E^{(\red{\alpha})}\otimes F^{(\blue{\beta})}.
								 \]
\end{remark}

\section{From Heisenberg categories to Kac--Moody categorification}
\label{sec:HeistoKac}
Let $R=\kk$ in the degenerate case and $R=\KK$ in the quantum case, and put
$$\symHeis_d:=\Heis_{k_1}(z_1,t_1)\odot\cdots \odot
\Heis_{k_d}(z_d,t_d).$$
All factors are taken to be degenerate or all to be quantum.

Let $\mathfrak{Cat}_\kk$ be the strict $\kk$-linear 2-category of
$\kk$-linear categories.  A strict 2-representation of $\UU$ consists of
\begin{itemize}
    \item a family $\{\mathcal{R}_\lambda\}_{\lambda\in \UU}$ of $\kk$-linear categories,
    \item a strict $\kk$-linear 2-functor ${\bf R\colon \UU \to \mathfrak{Cat}_\kk}$ such that ${\bf R}(\lambda)=\mathcal{R}_\lambda$ for each $\lambda\in \UU$.
\end{itemize}
Let $\calC$ be a locally finite abelian $R$-linear category equipped with a
$\symHeis_d$-action.  The construction of \cite{BSW20a} extends to this
colored setting and gives a Kac--Moody 2-representation on $\calC$.

\subsection{Colored diagram calculus}\label{sec:colordiagam}
We use the colored version of \cite[\S4.1]{BSW20a}.  Let $L\in\calC$ be
irreducible.  For ${\red{\alpha}}=1,\ldots,d$, let
$m^{(\red{\alpha})}_L(u),n^{(\red{\alpha})}_L(u)\in R[u]$ be the monic
minimal polynomials of the endomorphisms
$\mathord{\begin{tikzpicture}[baseline = -1mm]
		\draw[->,red] (0.08,-.2) to (0.08,.2);
		\node[red] at (0.08,0) {$\dott$};
		\node at (0.08,-.3) {$\scriptstyle{\red\alpha}$};
		\draw[-,darkg,thick] (0.38,.2) to (0.38,-.2);
		\node at (0.55,0) {$\darkg\scriptstyle{L}$};
\end{tikzpicture}}$ and
$\mathord{\begin{tikzpicture}[baseline = -1mm]
		\draw[<-,red] (0.08,-.2) to (0.08,.2);
		\node at (0.08,-.3) {$\scriptstyle{\red\alpha}$};
		\node[red] at (0.08,0.02) {$\dott$};
		\draw[-,darkg,thick] (0.38,.2) to (0.38,-.2);
		\node at (0.55,0) {$\darkg\scriptstyle{L}$};
\end{tikzpicture}}$ respectively.
 There are {injective} homomorphisms
\begin{align*}
	R[u] / (m^{(\red{\alpha})}_L(u)) &\hookrightarrow \End_{\calC} (E^{(\red{\alpha})}L),
	&R[u] / (n^{(\red{\alpha})}_L(u)) &\hookrightarrow \End_{\calC} (F^{(\red{\alpha})}L),\\
	p(u) &\mapsto \mathord{\begin{tikzpicture}[baseline = -1mm]
			\draw[->,red] (0.08,-.3) to (0.08,.4);
			\node[red] at (0.08,0.05) {$\dott$};
			\node at (0.08,-.4) {$\scriptstyle{\red\alpha}$};
			\node at (-0.3,0.05) {$\scriptstyle p(x)$};
			\draw[-,darkg,thick] (0.45,.4) to (0.45,-.3);
			\node at (0.45,-.45) {$\darkg\scriptstyle{L}$};
		\end{tikzpicture}
	},
	&p(u) &\mapsto \mathord{
		\begin{tikzpicture}[baseline = -1mm]
			\draw[<-,red] (0.08,-.3) to (0.08,.4);
			\node[red] at (0.08,0.1) {$\dott$};
			\node at (0.08,-.4) {$\scriptstyle{\red\alpha}$};
			\node at (-0.3,0.1) {$\scriptstyle p(x)$};
			\draw[-,darkg,thick] (0.45,.4) to (0.45,-.3);
			\node at (0.45,-.45) {$\darkg\scriptstyle{L}$};
		\end{tikzpicture}
	}.\notag
\end{align*}
where $\mathord{
\begin{tikzpicture}[baseline = -1.5]
	\draw[red,-] (0.08,-.15) to (0.08,.3);
      \node[red] at (0.08,0.08) {$\dt$};
\node at (0.32,.08) {$\scriptstyle x^n$};
\end{tikzpicture}
}$  denotes a dot of multiplicity $n$; we do this also
for negatively
dotted bubbles using negative values of $n$.

Also, let $\epsilon^{(\red{\alpha})}_i(L)$ and $\phi^{(\red{\alpha})}_i(L)$ denote the multiplicities
of $i^{(\red{\alpha})} \in R$ as the root of the polynomials $m^{(\red{\alpha})}_L(u)$ and $n^{(\red{\alpha})}_L(u)$,
respectively.
By the Chinese remainder theorem, we have
\begin{align*}
	R[u] / (m^{(\red{\alpha})}_L(u)) &\cong \bigoplus_{i \in R} R[u] \big/ \big((u-i)^{\epsilon^{(\red{\alpha})}_i(L)}\big),&
	R[u] / (n^{(\red{\alpha})}_L(u)) &\cong \bigoplus_{i \in R} R[u] \big/
	\big((u-i)^{\phi^{(\red{\alpha})}_i(L)}\big).
\end{align*}
We define $E^{(\red{\alpha})}_i$ and $F^{(\red{\alpha})}_i$ to be the direct
summands of $E^{(\red{\alpha})}$ and $F^{(\red{\alpha})}$ such that $E^{(\red{\alpha})}_i L$ and $F^{(\red{\alpha})}_i L$ are the generalized $i$-eigenspaces of $\mathord{\begin{tikzpicture}[baseline = -1mm]
		\draw[->,red] (0.08,-.2) to (0.08,.2);
		\node at (0.08,0) {$\dott$};
		\node at (0.08,-.3) {$\scriptstyle{\red\alpha}$};
		\draw[-,darkg,thick] (0.38,.2) to (0.38,-.2);
		\node at (0.55,0) {$\darkg\scriptstyle{L}$};
	\end{tikzpicture}
}$ and
$\mathord{\begin{tikzpicture}[baseline = -1mm]
		\draw[<-,red] (0.08,-.2) to (0.08,.2);
		\node at (0.08,-.3) {$\scriptstyle{\red\alpha}$};
		\node at (0.08,0.02) {$\dott$};
		\draw[-,darkg,thick] (0.38,.2) to (0.38,-.2);
		\node at (0.55,0) {$\darkg\scriptstyle{L}$};
	\end{tikzpicture}
}$ for any  $L\in \calC$, respectively.
\smallskip

The identity endomorphisms of the functors $E^{(\red{\alpha})}_i$ and $F^{(\red{\alpha})}_i$ are illustrated by vertical strings in color $i^{(\red{\alpha})}$, as depicted in the first pair of diagrams below. The second pair of diagrams portrays the inclusions $E^{(\red{\alpha})}_i \hookrightarrow E^{(\red{\alpha})}$ and $F^{(\red{\alpha})}_i \hookrightarrow F^{(\red{\alpha})}$. Lastly, the projections $E^{(\red{\alpha})} \twoheadrightarrow E^{(\red{\alpha})}_i$ and $F^{(\red{\alpha})} \twoheadrightarrow F^{(\red{\alpha})}_i$ are represented in the final pair of diagrams.
\begin{align*}
	&\mathord{
		\begin{tikzpicture}[baseline = 0,red]
			\draw[->] (0.08,-.3) to (0.08,.4);
			\node[black] at (0.08,-0.45) {$\scriptstyle{i^{(\red{\alpha})}}$};
		\end{tikzpicture}
	}
	\!:E^{(\red{\alpha})}_i \Rightarrow E^{(\red{\alpha})}_i,
	&&
	\mathord{
		\begin{tikzpicture}[baseline = 0,red]
			\draw[<-] (0.08,-.3) to (0.08,.4);
			\node[black] at (0.08,0.55) {$\scriptstyle{i^{(\red{\alpha})}}$};
		\end{tikzpicture}
	} \!:F^{(\red{\alpha})}_i \Rightarrow F^{(\red{\alpha})}_i,
	&&
	\mathord{
		\begin{tikzpicture}[baseline = 0,red]
			\draw[->] (0.08,-.3) to (0.08,.4);
			\draw[black,-] (0.04,.05) to (0.12,.05);
			\node[black] at (0.08,-0.45) {$\scriptstyle{i^{(\red{\alpha})}}$};
		\end{tikzpicture}
	} \!:E^{(\red{\alpha})}_i \Rightarrow E^{(\red{\alpha})},
	&&
	\mathord{
		\begin{tikzpicture}[baseline = 0,red]
			\draw[<-] (0.08,-.3) to (0.08,.4);
			\draw[black,-] (0.04,.05) to (0.12,.05);
			\node[black] at (0.08,-0.45) {$\scriptstyle{i^{(\red{\alpha})}}$};
		\end{tikzpicture}
	} \!:F^{(\red{\alpha})}_i \Rightarrow F^{(\red{\alpha})},
\end{align*}
\begin{align*}
	\mathord{
		\begin{tikzpicture}[baseline = 0mm,red]
			\draw[->] (0.08,-.3) to (0.08,.4);
			\draw[black,-] (0.04,.05) to (0.12,.05);
			\node[black] at (0.08,0.55) {$\scriptstyle{i^{(\red{\alpha})}}$};
		\end{tikzpicture}
	} \!:E^{(\red{\alpha})} \Rightarrow E^{(\red{\alpha})}_i,
	&&
	\mathord{
		\begin{tikzpicture}[baseline = 0mm,red]
			\draw[<-] (0.08,-.3) to (0.08,.4);
			\draw[black,-] (0.04,.05) to (0.12,.05);
			\node[black] at (0.08,0.55) {$\scriptstyle{i^{(\red{\alpha})}}$};
		\end{tikzpicture}
	} \!:F^{(\red{\alpha})} \Rightarrow F^{(\red{\alpha})}_i.
\end{align*}
 We assign colors to each copy corresponding to the colors representing each ${\red{\alpha}}$. For example we also have the natural transformation
$\:\mathord{
	\begin{tikzpicture}[baseline = -0.8mm,red]
		\draw[->] (0.08,-.25) to (0.08,.35);
		\draw[black,-] (0.04,.15) to (0.12,.15);
		\draw[-,black] (0.04,-.05) to (0.12,-.05);
		\node at (0.3,0.05) {$\scriptstyle{i^{(\red{\alpha})}}$};
	\end{tikzpicture}
}:E^{(\red{\alpha})} \Rightarrow E^{(\red{\alpha})}
$
is the projection of $E^{(\red{\alpha})}$ onto its summand $E^{(\red{\alpha})}_i$, while
\begin{equation*}
	\mathord{
		\begin{tikzpicture}[baseline = 0,red]
			\draw[->] (0.08,-.3) to (0.08,.4);
			\draw[black,-] (0.04,.15) to (0.12,.15);
			\draw[-] (0.04,-0.05) to (0.12,-0.05);
			\node at (0.08,0.55) {$\scriptstyle{\alpha_j}$};
			\node[black] at (0.08,-.45) {$\scriptstyle{i^{(\red{\alpha})}}$};
		\end{tikzpicture}
	}
	= \delta_{i,j}
	\mathord{
		\begin{tikzpicture}[baseline = 0,red]
			\draw[->] (0.08,-.3) to (0.08,.4);
			\node[black] at (0.08,-.45) {$\scriptstyle{i^{(\red{\alpha})}}$};
		\end{tikzpicture}
	}\:\:.
\end{equation*}
It is also clear from the definition that the endomorphisms of
$E^{(\red{\alpha})}$ and $F^{(\red{\alpha})}$
defined by the dots restrict to endomorphisms of the
summands $E^{(\red{\alpha})}_i$
and $F^{(\red{\alpha})}_i$.
Representing these restrictions simply by drawing the dots on a string
colored by $i^{(\red{\alpha})}$, we have that
\begin{align}\label{sizzle}
\mathord{
\begin{tikzpicture}[red,baseline = -1mm]
	\draw[->] (0.08,-.3) to (0.08,.4);
      \node at (0.08,-0.05) {$\dt$};
	\draw[black,-] (0.04,.15) to (0.12,.15);
      \node[black] at (0.08,-0.45) {$\scriptstyle{i^{(\red{\alpha})}}$};
\end{tikzpicture}
}&=
\mathord{
\begin{tikzpicture}[red,baseline = -1mm]
	\draw[->] (0.08,-.3) to (0.08,.4);
      \node at (0.08,0.15) {$\dt$};
	\draw[-,black] (0.04,-.05) to (0.12,-.05);
      \node[black] at (0.08,-0.45) {$\scriptstyle{i^{(\red{\alpha})}}$};
\end{tikzpicture}
}\:\:,
&
\mathord{
\begin{tikzpicture}[red,baseline = -1mm]
	\draw[<-] (0.08,-.3) to (0.08,.4);
      \node at (0.08,-0.05) {$\dt$};
	\draw[black,-] (0.04,.15) to (0.12,.15);
      \node[black] at (0.08,-0.45) {$\scriptstyle{i^{(\red{\alpha})}}$};
\end{tikzpicture}
}&=
\mathord{
\begin{tikzpicture}[red,baseline = -1mm]
	\draw[<-] (0.08,-.3) to (0.08,.4);
      \node at (0.08,0.15) {$\dt$};
	\draw[-,black] (0.04,-.05) to (0.12,-.05);
      \node[black] at (0.08,-0.45) {$\scriptstyle{i^{(\red{\alpha})}}$};
\end{tikzpicture}
}\:\:,&
\mathord{
\begin{tikzpicture}[red,baseline = 1mm]
	\draw[->] (0.08,-.3) to (0.08,.4);
      \node at (0.08,-0.05) {$\dt$};
	\draw[black,-] (0.04,.15) to (0.12,.15);
      \node[black] at (0.08,0.55) {$\scriptstyle{i^{(\red{\alpha})}}$};
\end{tikzpicture}
}&=
\mathord{
\begin{tikzpicture}[red,baseline = 1mm]
	\draw[->] (0.08,-.3) to (0.08,.4);
      \node at (0.08,0.15) {$\dt$};
	\draw[-,black] (0.04,-.05) to (0.12,-.05);
      \node[black] at (0.08,0.55) {$\scriptstyle{i^{(\red{\alpha})}}$};
\end{tikzpicture}
}\:\:,
&
\mathord{
\begin{tikzpicture}[red,baseline = 1mm]
	\draw[<-] (0.08,-.3) to (0.08,.4);
      \node at (0.08,-0.05) {$\dt$};
	\draw[black,-] (0.04,.15) to (0.12,.15);
      \node[black] at (0.08,0.55) {$\scriptstyle{i^{(\red{\alpha})}}$};
\end{tikzpicture}
}&=
\mathord{
\begin{tikzpicture}[red,baseline = 1mm]
	\draw[<-] (0.08,-.3) to (0.08,.4);
      \node at (0.08,0.15) {$\dt$};
	\draw[-,black] (0.04,-.05) to (0.12,-.05);
      \node[black] at (0.08,0.55) {$\scriptstyle{i^{(\red{\alpha})}}$};
\end{tikzpicture}
}
\:\:.
\end{align}

Since the downwards dot is both the left and right mate of the upwards dot,
the adjunctions $(E^{(\red{\alpha})},F^{(\red{\alpha})})$ and $(F^{(\red{\alpha})},E^{(\red{\alpha})})$ induce adjunctions $(E^{(\red{\alpha})}_i, F^{(\red{\alpha})}_i)$
and $(F^{(\red{\alpha})}_i, E^{(\red{\alpha})}_i)$ for all $i^{(\red{\alpha})} \in \kk$.
We draw the units and counits of these adjunctions using cups and caps
colored by $i^{(\red{\alpha})}$.
Again, the various inclusions and projections commute with these morphisms:
\begin{align}
\begin{array}{llll}
\mathord{
\begin{tikzpicture}[red,baseline = 1mm]
	\draw[<-] (0.4,0.4) to[out=-90, in=0] (0.1,0);
	\draw[-,black] (-0.2,.188) to (-0.123,.22);
	\draw[-] (0.1,0) to[out = 180, in = -90] (-0.2,0.4);
      \node[black] at (-0.2,0.55) {$\scriptstyle{i^{(\red{\alpha})}}$};
\end{tikzpicture}
}
\:=
\mathord{
\begin{tikzpicture}[red,baseline = 1mm]
	\draw[<-] (0.4,0.4) to[out=-90, in=0] (0.1,0);
	\draw[-,black] (.402,.188) to (.325,.22);
	\draw[-] (0.1,0) to[out = 180, in = -90] (-0.2,0.4);
      \node[black] at (-0.2,0.55) {$\scriptstyle{i^{(\red{\alpha})}}$};
\end{tikzpicture}
}\:
,
\quad&
\mathord{
\begin{tikzpicture}[red,baseline = 1mm]
	\draw[-] (0.4,0.4) to[out=-90, in=0] (0.1,0);
	\draw[-,black] (-0.2,.188) to (-0.123,.22);
	\draw[->] (0.1,0) to[out = 180, in = -90] (-0.2,0.4);
      \node[black] at (-0.2,0.55) {$\scriptstyle{i^{(\red{\alpha})}}$};
\end{tikzpicture}
}
\;=\;
\mathord{
\begin{tikzpicture}[red,baseline = 1mm]
	\draw[-] (0.4,0.4) to[out=-90, in=0] (0.1,0);
	\draw[-,black] (.402,.188) to (.325,.22);
	\draw[->] (0.1,0) to[out = 180, in = -90] (-0.2,0.4);
      \node[black] at (0.4,0.55) {$\scriptstyle{i^{(\red{\alpha})}}$};
\end{tikzpicture}
}\:
,
\quad&
\mathord{
\begin{tikzpicture}[red,baseline = 1mm]
	\draw[<-] (0.4,0.4) to[out=-90, in=0] (0.1,0);
	\draw[-,black] (-0.2,.188) to (-0.123,.22);
	\draw[-] (0.1,0) to[out = 180, in = -90] (-0.2,0.4);
      \node[black] at (0.4,0.55) {$\scriptstyle{i^{(\red{\alpha})}}$};
\end{tikzpicture}
}
=
\:\mathord{
\begin{tikzpicture}[red,baseline = 1mm]
	\draw[<-] (0.4,0.4) to[out=-90, in=0] (0.1,0);
	\draw[-,black] (.402,.188) to (.325,.22);
	\draw[-] (0.1,0) to[out = 180, in = -90] (-0.2,0.4);
      \node[black] at (0.4,0.55) {$\scriptstyle{i^{(\red{\alpha})}}$};
\end{tikzpicture}
}\:
,
\quad&
\mathord{
\begin{tikzpicture}[red,baseline = 1mm]
	\draw[-] (0.4,0.4) to[out=-90, in=0] (0.1,0);
	\draw[-,black] (-0.2,.188) to (-0.123,.22);
	\draw[->] (0.1,0) to[out = 180, in = -90] (-0.2,0.4);
      \node[black] at (0.4,0.55) {$\scriptstyle{i^{(\red{\alpha})}}$};
\end{tikzpicture}
}
=\:
\mathord{
\begin{tikzpicture}[red,baseline = 1mm]
	\draw[-] (0.4,0.4) to[out=-90, in=0] (0.1,0);
	\draw[-,black] (.402,.188) to (.325,.22);
	\draw[->] (0.1,0) to[out = 180, in = -90] (-0.2,0.4);
      \node[black] at (0.4,0.55) {$\scriptstyle{i^{(\red{\alpha})}}$};
\end{tikzpicture}
}\:
,
\\\\
\mathord{
\begin{tikzpicture}[red,baseline = 1mm]
	\draw[<-] (0.4,0) to[out=90, in=0] (0.1,0.4);
	\draw[-,black] (-0.2,.22) to (-0.123,.188);
	\draw[-] (0.1,0.4) to[out = 180, in = 90] (-0.2,0);
      \node[black] at (-0.2,-0.15) {$\scriptstyle{i^{(\red{\alpha})}}$};
\end{tikzpicture}
}\:=\:
\mathord{
\begin{tikzpicture}[red,baseline = 1mm]
	\draw[<-] (0.4,0) to[out=90, in=0] (0.1,0.4);
	\draw[-,black] (.402,.22) to (.325,.188);
	\draw[-] (0.1,0.4) to[out = 180, in = 90] (-0.2,0);
      \node[black] at (-0.2,-0.15) {$\scriptstyle{i^{(\red{\alpha})}}$};
\end{tikzpicture}
}\:,
\quad&
\mathord{
\begin{tikzpicture}[red,baseline = 1mm]
	\draw[-] (0.4,0) to[out=90, in=0] (0.1,0.4);
	\draw[-,black] (-0.2,.22) to (-0.123,.188);
	\draw[->] (0.1,0.4) to[out = 180, in = 90] (-0.2,0);
      \node[black] at (-0.2,-0.15) {$\scriptstyle{i^{(\red{\alpha})}}$};
\end{tikzpicture}
}\;=\;
\mathord{
\begin{tikzpicture}[red,baseline = 1mm]
	\draw[-] (0.4,0) to[out=90, in=0] (0.1,0.4);
	\draw[-,black] (.402,.22) to (.325,.188);
	\draw[->] (0.1,0.4) to[out = 180, in = 90] (-0.2,0);
      \node[black] at (-0.2,-0.15) {$\scriptstyle{i^{(\red{\alpha})}}$};
\end{tikzpicture}
}\:,
\quad&
\mathord{
\begin{tikzpicture}[red,baseline = 1mm]
	\draw[<-] (0.4,0) to[out=90, in=0] (0.1,0.4);
	\draw[-,black] (-0.2,.22) to (-0.123,.188);
	\draw[-] (0.1,0.4) to[out = 180, in = 90] (-0.2,0);
      \node[black] at (0.4,-0.15) {$\scriptstyle{i^{(\red{\alpha})}}$};
\end{tikzpicture}
}=\:
\mathord{
\begin{tikzpicture}[red,baseline = 1mm]
	\draw[<-] (0.4,0) to[out=90, in=0] (0.1,0.4);
	\draw[-,black] (.402,.22) to (.325,.188);
	\draw[-] (0.1,0.4) to[out = 180, in = 90] (-0.2,0);
      \node[black] at (0.4,-0.15) {$\scriptstyle{i^{(\red{\alpha})}}$};
\end{tikzpicture}
}\:,
\quad&
\mathord{
\begin{tikzpicture}[red,baseline = 1mm]
	\draw[-] (0.4,0) to[out=90, in=0] (0.1,0.4);
	\draw[-,black] (-0.2,.22) to (-0.123,.188);
	\draw[->] (0.1,0.4) to[out = 180, in = 90] (-0.2,0);
      \node[black] at (0.4,-0.15) {$\scriptstyle{i^{(\red{\alpha})}}$};
\end{tikzpicture}
}=\:
\mathord{
\begin{tikzpicture}[red,baseline = 1mm]
	\draw[-] (0.4,0) to[out=90, in=0] (0.1,0.4);
	\draw[-,black] (.402,.22) to (.325,.188);
	\draw[->] (0.1,0.4) to[out = 180, in = 90] (-0.2,0);
      \node[black] at (0.4,-0.15) {$\scriptstyle{i^{(\red{\alpha})}}$};
\end{tikzpicture}
}\:.
\end{array}\label{incpro}
\end{align}

For $i^{({\blue{\beta}})},j^{({\red{\alpha}})},(i')^{({\red{\alpha}})},(j')^{({\blue{\beta}})} \in \kk$, we define (we often omit the upper script ${\red{\alpha}},{\blue{\beta}}$ in crossings)
\begin{align}
\left.
\begin{array}{ll}
\mathord{
\begin{tikzpicture}[baseline = -1mm]
	\draw[->,blue] (0.28,-.28) to (-0.28,.28);
	\draw[->,red] (-0.28,-.28) to (0.28,.28);
\node at (-0.33,-0.43) {$\scriptstyle{j}$};
      \node at (0.33,-0.43) {$\scriptstyle{i}$};
      \node at (-0.33,0.43) {$\scriptstyle{j'}$};
      \node at (0.33,0.43) {$\scriptstyle{i'}$};
\node at (0,-0.01) {$\diamond$};
\end{tikzpicture}
}:=
\mathord{
\begin{tikzpicture}[baseline = -1mm]
	\draw[->,blue] (0.28,-.28) to (-0.28,.28);
	\draw[->,red] (-0.28,-.28) to (0.28,.28);
\node at (-0.33,-0.43) {$\scriptstyle{j}$};
      \node at (0.33,-0.43) {$\scriptstyle{i}$};
      \node at (-0.33,0.43) {$\scriptstyle{j'}$};
      \node at (0.33,0.43) {$\scriptstyle{i'}$};
	\draw[-] (.14,.22) to (.22,.14);
	\draw[-] (.14,-.22) to (.22,-.14);
	\draw[-] (-.14,.22) to (-.22,.14);
	\draw[-] (-.14,-.22) to (-.22,-.14);
\end{tikzpicture}
}
&\text{in the degenerate case and the quantum case when ${\red{\alpha}}\neq{\blue{\beta}}$}\\
\mathord{
\begin{tikzpicture}[baseline = -1mm]
	\draw[->,red] (0.28,-.28) to (-0.28,.28);
	\draw[-,white,line width=3pt] (-0.28,-.28) to (0.28,.28);
	\draw[->,red] (-0.28,-.28) to (0.28,.28);
\node at (-0.33,-0.43) {$\scriptstyle{j}$};
      \node at (0.33,-0.43) {$\scriptstyle{i}$};
      \node at (-0.33,0.43) {$\scriptstyle{j'}$};
      \node at (0.33,0.43) {$\scriptstyle{i'}$};
\node at (0,-0.01) {$\diamond$};
\end{tikzpicture}
}:=
\mathord{
\begin{tikzpicture}[baseline = -1mm]
	\draw[->,red] (0.28,-.28) to (-0.28,.28);
	\draw[-,white,line width=4pt] (-0.28,-.28) to (0.28,.28);
	\draw[->,red] (-0.28,-.28) to (0.28,.28);
    \node at (-0.33,-0.43) {$\scriptstyle{j}$};
      \node at (0.33,-0.43) {$\scriptstyle{i}$};
      \node at (-0.33,0.43) {$\scriptstyle{j'}$};
      \node at (0.33,0.43) {$\scriptstyle{i'}$};
	\draw[-] (.14,.22) to (.22,.14);
	\draw[-] (.14,-.22) to (.22,-.14);
	\draw[-] (-.14,.22) to (-.22,.14);
	\draw[-] (-.14,-.22) to (-.22,-.14);
\end{tikzpicture}
}\:,&\quad
\mathord{
\begin{tikzpicture}[baseline = 0]
	\draw[->,red] (-0.28,-.28) to (0.28,.28);
	\draw[-,white,line width=3pt] (0.28,-.28) to (-0.28,.28);
	\draw[->,red] (0.28,-.28) to (-0.28,.28);
     \node at (-0.33,-0.43) {$\scriptstyle{j}$};
      \node at (0.33,-0.43) {$\scriptstyle{i}$};
      \node at (-0.33,0.43) {$\scriptstyle{j'}$};
      \node at (0.33,0.43) {$\scriptstyle{i'}$};
\node at (0,-0.01) {$\diamond$};
\end{tikzpicture}
}:=
\mathord{
\begin{tikzpicture}[baseline = 0]
	\draw[->,red] (-0.28,-.28) to (0.28,.28);
	\draw[-,line width=4pt,white] (0.28,-.28) to (-0.28,.28);
	\draw[->,red] (0.28,-.28) to (-0.28,.28);
      \node at (-0.33,-0.43) {$\scriptstyle{j}$};
      \node at (0.33,-0.43) {$\scriptstyle{i}$};
      \node at (-0.33,0.43) {$\scriptstyle{j'}$};
      \node at (0.33,0.43) {$\scriptstyle{i'}$};
	\draw[-] (.14,.22) to (.22,.14);
	\draw[-] (.14,-.22) to (.22,-.14);
	\draw[-] (-.14,.22) to (-.22,.14);
	\draw[-] (-.14,-.22) to (-.22,-.14);
\end{tikzpicture}
}
\text{in the quantum case when ${\red{\alpha}}={\blue{\beta}}$.}
\end{array}\right.\label{crossingijab}\end{align}
These natural transformation are defined by first including
the summand $E^{(\red{\alpha})}_j E^{(\blue{\beta})}_i$ into $E^{(\red{\alpha})} E^{(\blue{\beta})}$, then
applying natural transformation $E^{(\red{\alpha})}E^{(\blue{\beta})} \Rightarrow E^{(\blue{\beta})}E^{(\red{\alpha})}$ defined by
the usual crossing (positive or negative in the quantum case), then
projecting
$E^{(\blue{\beta})} E^{(\red{\alpha})}$ onto the summand $E^{(\blue{\beta})}_{j'} E^{(\red{\alpha})}_{i'}$.

There are also sideways and downwards versions of the new crossings,
which may be defined in a similar way, or equivalently by ``rotating'' the upward versions using \eqref{incpro}.

\begin{lemma}\label{essential}
If $\{i,j\}\neq\{i',j'\}$ then the natural transformation (\ref{crossingijab})
is zero.
The same
holds for the rotated versions of these crossings.
\end{lemma}
\begin{proof}
    We give a proof in the degenerate case and the quantum case is similar.
    The case where ${\red{\alpha}}={\blue{\beta}}$ follows from \cite[Lemma 4.1]{BSW20a}. For the case when ${\red{\alpha}}\neq{\blue{\beta}}$, by \eqref{sizzle} and \eqref{mixedrelation2} we have \begin{align}\label{newdotslideab}
\mathord{
\begin{tikzpicture}[baseline = -1mm]
	\draw[->,blue] (0.28,-.28) to (-0.28,.28);
	\draw[->,red] (-0.28,-.28) to (0.28,.28);
      \node at (-0.33,-0.43) {$\scriptstyle{j}$};
      \node at (0.33,-0.43) {$\scriptstyle{i}$};
      \node at (-0.33,0.43) {$\scriptstyle{j'}$};
      \node at (0.33,0.43) {$\scriptstyle{i'}$};
\node at (0,-.01) {$\diamond$};
\node[red] at (-.17,-.19) {$\dt$};
\end{tikzpicture}
}&=
\mathord{
\begin{tikzpicture}[baseline = -1mm]
	\draw[->,blue] (0.28,-.28) to (-0.28,.28);
	\draw[->,red] (-0.28,-.28) to (0.28,.28);
    \node at (-0.33,-0.43) {$\scriptstyle{j}$};
      \node at (0.33,-0.43) {$\scriptstyle{i}$};
      \node at (-0.33,0.43) {$\scriptstyle{j'}$};
      \node at (0.33,0.43) {$\scriptstyle{i'}$};
\node at (0,-.01) {$\diamond$};
\node[red] at (.17,.15) {$\dt$};
\end{tikzpicture},
}&
\mathord{
\begin{tikzpicture}[baseline = -1mm]
	\draw[->,blue] (0.28,-.28) to (-0.28,.28);
	\draw[->,red] (-0.28,-.28) to (0.28,.28);
     \node at (-0.33,-0.43) {$\scriptstyle{j}$};
      \node at (0.33,-0.43) {$\scriptstyle{i}$};
      \node at (-0.33,0.43) {$\scriptstyle{j'}$};
      \node at (0.33,0.43) {$\scriptstyle{i'}$};
\node at (0,-.01) {$\diamond$};
\node[blue] at (-.17,.15) {$\dt$};
\end{tikzpicture}
}&=
\mathord{
\begin{tikzpicture}[baseline = -1mm]
	\draw[->,blue] (0.28,-.28) to (-0.28,.28);
	\draw[->,red] (-0.28,-.28) to (0.28,.28);
    \node at (-0.33,-0.43) {$\scriptstyle{j}$};
      \node at (0.33,-0.43) {$\scriptstyle{i}$};
      \node at (-0.33,0.43) {$\scriptstyle{j'}$};
      \node at (0.33,0.43) {$\scriptstyle{i'}$};
\node at (0,-.01) {$\diamond$};
\node[blue] at (.18,-.19) {$\dt$};
\end{tikzpicture}.
}
\end{align}
Hence the lemma in this case follows from the same reasoning as in the case ${\red{\alpha}}={\blue{\beta}}$.
\end{proof}

\subsubsection{Colored central elements}
Our presentations of the degenerate and quantum Heisenberg categories omit
bubbles from the generating data.  The equivalent bubble presentations
interpret \eqref{dMackey1}--\eqref{dMackey2} and
\eqref{qinver1}--\eqref{qinver2} accordingly; see
\cite[\S\S3.1--3.2]{BSW20a}.  We shall use the resulting negatively dotted
bubbles.

In the degenerate Heisenberg category, we have
\begin{align*}
\mathord{
\begin{tikzpicture}[baseline = 1.25mm]
  \draw[->] (0.2,0.2) to[out=90,in=0] (0,.4);
  \draw[-] (0,0.4) to[out=180,in=90] (-.2,0.2);
\draw[-] (-.2,0.2) to[out=-90,in=180] (0,0);
  \draw[-] (0,0) to[out=0,in=-90] (0.2,0.2);
   \node at (0.2,0.2) {$\dt$};
   \node at (0.7,0.2) {$\scriptstyle{n-k-1}$};
\end{tikzpicture}}
&:=
\left\{\begin{array}{ll}
\det\left(
\mathord{
\begin{tikzpicture}[baseline = 1.25mm]
  \draw[<-] (0,0.4) to[out=180,in=90] (-.2,0.2);
  \draw[-] (0.2,0.2) to[out=90,in=0] (0,.4);
 \draw[-] (-.2,0.2) to[out=-90,in=180] (0,0);
  \draw[-] (0,0) to[out=0,in=-90] (0.2,0.2);
   \node at (-0.2,0.2) {$\dt$};
   \node at (-0.7,0.2) {$\scriptstyle{r-s+k}$};
\end{tikzpicture}
}\:
\right)_{r,s=1,\dots,n}
\phantom{(-1)^{n+1}}&\text{if $k \geq n > 0$,}\\
1_\unit&\text{if $k \geq n=0$,}\\
0&\text{if $k \geq n < 0$,}
\end{array}\right.\\
\mathord{
\begin{tikzpicture}[baseline = 1.25mm]
  \draw[<-] (0,0.4) to[out=180,in=90] (-.2,0.2);
  \draw[-] (0.2,0.2) to[out=90,in=0] (0,.4);
 \draw[-] (-.2,0.2) to[out=-90,in=180] (0,0);
  \draw[-] (0,0) to[out=0,in=-90] (0.2,0.2);
   \node at (-0.2,0.2) {$\dt$};
   \node at (-0.7,0.2) {$\scriptstyle{n+k-1}$};
\end{tikzpicture}
}&:=
\left\{
\begin{array}{ll}
(-1)^{n+1} \det\left(\:\mathord{
\begin{tikzpicture}[baseline = 1.25mm]
  \draw[->] (0.2,0.2) to[out=90,in=0] (0,.4);
  \draw[-] (0,0.4) to[out=180,in=90] (-.2,.2);
\draw[-] (-.2,0.2) to[out=-90,in=180] (0,0);
  \draw[-] (0,0) to[out=0,in=-90] (0.2,0.2);
   \node at (0.2,0.2) {$\dt$};
   \node at (0.7,0.2) {$\scriptstyle{r-s-k}$};
\end{tikzpicture}
}\right)_{r,s=1,\dots,n}&\text{if $-k\geq n > 0$,}\\
-1_\unit&\text{if $-k\geq n=0$,}\\
0&\text{if $-k \geq n < 0$.}
\end{array}\right.
\end{align*}

In the quantum Heisenberg category we have
the $(+)$-bubbles
\begin{equation*}
\begin{aligned}
\mathord{
\begin{tikzpicture}[baseline = 1.25mm]
  \draw[->] (0.2,0.2) to[out=90,in=0] (0,.4);
  \draw[-] (0,0.4) to[out=180,in=90] (-.2,0.2);
\draw[-] (-.2,0.2) to[out=-90,in=180] (0,0);
  \draw[-] (0,0) to[out=0,in=-90] (0.2,0.2);
   \node at (0.5,0.2) {$\scriptstyle{n-k}$};
   \node at (0,0.2) {$+$};
\end{tikzpicture}}
&:=
\left\{\begin{array}{ll}
\mathord{
\begin{tikzpicture}[baseline = 1.25mm]
  \draw[->] (0.2,0.2) to[out=90,in=0] (0,.4);
  \draw[-] (0,0.4) to[out=180,in=90] (-.2,0.2);
\draw[-] (-.2,0.2) to[out=-90,in=180] (0,0);
  \draw[-] (0,0) to[out=0,in=-90] (0.2,0.2);
   \node at (0.2,0.2) {$\dt$};
   \node at (0.6,0.2) {$\scriptstyle{n-k}$};
\end{tikzpicture}}
\hspace{51.8mm}&\text{if $k<n$,}\\
t^{n+1} z^{n-1}\det\left(
\!\mathord{
\begin{tikzpicture}[baseline = 1.25mm]
  \draw[<-] (0,0.4) to[out=180,in=90] (-.2,0.2);
  \draw[-] (0.2,0.2) to[out=90,in=0] (0,.4);
 \draw[-] (-.2,0.2) to[out=-90,in=180] (0,0);
  \draw[-] (0,0) to[out=0,in=-90] (0.2,0.2);
   \node at (-0.2,0.2) {$\dt$};
   \node at (-0.85,0.2) {$\scriptstyle{r-s+k+1}$};
\end{tikzpicture}
}\,
\right)_{r,s=1,\dots,n},&\text{if $k \geq n > 0$,}\\
tz^{-1} 1_\unit&\text{if $k \geq n=0$,}\\
0&\text{if $k \geq n < 0$,}
\end{array}\right.\\
\mathord{
\begin{tikzpicture}[baseline = 1.25mm]
  \draw[<-] (0,0.4) to[out=180,in=90] (-.2,0.2);
  \draw[-] (0.2,0.2) to[out=90,in=0] (0,.4);
 \draw[-] (-.2,0.2) to[out=-90,in=180] (0,0);
  \draw[-] (0,0) to[out=0,in=-90] (0.2,0.2);
   \node at (0,0.2) {$+$};
   \node at (-0.48,0.2) {$\scriptstyle{n+k}$};
\end{tikzpicture}
}&:=
\left\{
\begin{array}{ll}
\mathord{
\begin{tikzpicture}[baseline = 1.25mm]
  \draw[<-] (0,0.4) to[out=180,in=90] (-.2,0.2);
  \draw[-] (0.2,0.2) to[out=90,in=0] (0,.4);
 \draw[-] (-.2,0.2) to[out=-90,in=180] (0,0);
  \draw[-] (0,0) to[out=0,in=-90] (0.2,0.2);
   \node at (-0.2,0.2) {$\dt$};
   \node at (-0.6,0.2) {$\scriptstyle{n+k}$};
\end{tikzpicture}
}&\text{if $-k < n$,}\\
(-1)^{n+1} t^{-n-1} z^{n-1}\det\left(\:\mathord{
\begin{tikzpicture}[baseline = 1.25mm]
  \draw[->] (0.2,0.2) to[out=90,in=0] (0,.4);
  \draw[-] (0,0.4) to[out=180,in=90] (-.2,.2);
\draw[-] (-.2,0.2) to[out=-90,in=180] (0,0);
  \draw[-] (0,0) to[out=0,in=-90] (0.2,0.2);
   \node at (0.2,0.2) {$\dt$};
   \node at (0.85,0.2) {$\scriptstyle{r-s-k+1}$};
\end{tikzpicture}
}\right)_{r,s=1,\dots,n}&\text{if $-k\geq n > 0$,}\\
-t^{-1}z^{-1} 1_\unit
&\text{if $-k \geq n=0$,}\\
0&\text{if $-k\geq n < 0$.}
\end{array}\right.
\end{aligned}
\end{equation*}
(These bubbles are not colored as they belong to just one copy of the Heisenberg category).

 Following \cite{BSW20a}, we define the colored
generating functions in the degenerate case:
\begin{align*}
\begin{tikzpicture}[baseline = 1.25mm]
		\draw[-,red] (0,0.4) to[out=180,in=90] (-.2,0.2);
		\draw[->,red] (0.2,0.2) to[out=90,in=0] (0,.4);
		\draw[-,red] (-.2,0.2) to[out=-90,in=180] (0,0);
		\draw[-,red] (0,0) to[out=0,in=-90] (0.2,0.2);
		\node at (0.0,-.2) {$\scriptstyle{\red\alpha}$};
	\end{tikzpicture}(u) &:= \sum_{r\in\Z}
\mathord{
\begin{tikzpicture}[red,baseline = 1.25mm]
  \draw[-] (0,0.4) to[out=180,in=90] (-.2,0.2);
  \draw[->] (0.2,0.2) to[out=90,in=0] (0,.4);
 \draw[-] (-.2,0.2) to[out=-90,in=180] (0,0);
  \draw[-] (0,0) to[out=0,in=-90] (0.2,0.2);
   \node at (0.2,0.2) {$\dt$};
   \node[black] at (0.45,0.2) {$\scriptstyle{x^r}$};
\end{tikzpicture}
}\: u^{-r-1} \in u^k 1_\unit + u^{k-1}\End_{\Heis_k}(\unit)\llbracket u^{-1}\rrbracket,\\
\begin{tikzpicture}[baseline = 1.25mm]
		\draw[-,red] (0,0.4) to[out=180,in=90] (-.2,0.2);
		\draw[<-,red] (0.2,0.2) to[out=90,in=0] (0,.4);
		\draw[-,red] (-.2,0.2) to[out=-90,in=180] (0,0);
		\draw[-,red] (0,0) to[out=0,in=-90] (0.2,0.2);
		\node at (0.0,-.2) {$\scriptstyle{\red\alpha}$};
	\end{tikzpicture}(u)&:= -\sum_{r \in\Z}
\mathord{
\begin{tikzpicture}[red,baseline = 1.25mm]
  \draw[<-] (0,0.4) to[out=180,in=90] (-.2,0.2);
  \draw[-] (0.2,0.2) to[out=90,in=0] (0,.4);
 \draw[-] (-.2,0.2) to[out=-90,in=180] (0,0);
  \draw[-] (0,0) to[out=0,in=-90] (0.2,0.2);
   \node at (-0.2,0.2) {$\dt$};
   \node[black] at (-0.45,0.2) {$\scriptstyle{x^r}$};
\end{tikzpicture}
} \:u^{-r-1}\in u^{-k} 1_\unit +u^{-k-1}\End_{\Heis_k}(\unit)\llbracket u^{-1}\rrbracket.
\end{align*}
and in the quantum case:
\begin{align*}
	\begin{tikzpicture}[baseline = 1.25mm]
		\draw[-,red] (0,0.4) to[out=180,in=90] (-.2,0.2);
		\draw[->,red] (0.2,0.2) to[out=90,in=0] (0,.4);
		\draw[-,red] (-.2,0.2) to[out=-90,in=180] (0,0);
		\draw[-,red] (0,0) to[out=0,in=-90] (0.2,0.2);
		\node at (0.0,-.2) {$\scriptstyle{\red\alpha}$};
	\end{tikzpicture}
	(u) &:= t_{\red{\alpha}}^{-1}z_{\red{\alpha}} \sum_{r\in\bbZ}
	\mathord{
		\begin{tikzpicture}[baseline = 1.25mm]
			\draw[-,red] (0,0.4) to[out=180,in=90] (-.2,0.2);
			\draw[->,red] (0.2,0.2) to[out=90,in=0] (0,.4);
			\draw[-,red] (-.2,0.2) to[out=-90,in=180] (0,0);
			\draw[-,red] (0,0) to[out=0,in=-90] (0.2,0.2);
			\node at (0,.21) {$+$};
			\node at (0.33,0.2) {$\scriptstyle{r}$};
			\node at (0.0,-.2) {$\scriptstyle{\red\alpha}$};
		\end{tikzpicture}
	}\: u^{-r}
	\in u^k 1_\unit + u^{k-1} \End_{\Heis_k(z,t)}(\unit)[[
	u^{-1} ]]
	,\\
	\begin{tikzpicture}[baseline = 1.25mm]
		\draw[<-,red] (0,0.4) to[out=180,in=90] (-.2,0.2);
		\draw[-,red] (0.2,0.2) to[out=90,in=0] (0,.4);
		\draw[-,red] (-.2,0.2) to[out=-90,in=180] (0,0);
		\draw[-,red] (0,0) to[out=0,in=-90] (0.2,0.2);
		\node at (0.0,-.2) {$\scriptstyle{\red\alpha}$};
	\end{tikzpicture}(u)&:= -t_{\red{\alpha}}z_{\red{\alpha}} \sum_{r\in\Z}
	\mathord{
		\begin{tikzpicture}[baseline = 1.25mm]
			\draw[<-,red] (0,0.4) to[out=180,in=90] (-.2,0.2);
			\draw[-,red] (0.2,0.2) to[out=90,in=0] (0,.4);
			\draw[-,red] (-.2,0.2) to[out=-90,in=180] (0,0);
			\draw[-,red] (0,0) to[out=0,in=-90] (0.2,0.2);
			\node at (0,.21) {$+$};
			\node at (-0.33,0.2) {$\scriptstyle{r}$};
			\node at (0.0,-.2) {$\scriptstyle{\red\alpha}$};
		\end{tikzpicture}
	} \:u^{-r} \in u^{-k}1_\unit +
	u^{-k-1}\End_{\Heis_k(z,t)}(\unit)[[
	u^{-1} ]].
\end{align*}
Then they satisfy \begin{equation*}
	\begin{tikzpicture}[baseline = 1.25mm]
		\draw[-,red] (0,0.4) to[out=180,in=90] (-.2,0.2);
		\draw[->,red] (0.2,0.2) to[out=90,in=0] (0,.4);
		\draw[-,red] (-.2,0.2) to[out=-90,in=180] (0,0);
		\draw[-,red] (0,0) to[out=0,in=-90] (0.2,0.2);
		\node at (0.0,-.2) {$\scriptstyle{\red\alpha}$};
	\end{tikzpicture}
	(u)\; \begin{tikzpicture}[baseline = 1.25mm]
		\draw[<-,red] (0,0.4) to[out=180,in=90] (-.2,0.2);
		\draw[-,red] (0.2,0.2) to[out=90,in=0] (0,.4);
		\draw[-,red] (-.2,0.2) to[out=-90,in=180] (0,0);
		\draw[-,red] (0,0) to[out=0,in=-90] (0.2,0.2);
		\node at (0.0,-.2) {$\scriptstyle{\red\alpha}$};
	\end{tikzpicture}(u)
	= 1_\unit.
\end{equation*}
For convenience, we denote $$\bbO^{(\red{\alpha})}(u):=\begin{tikzpicture}[baseline = 1.25mm]
	\draw[-,red] (0,0.4) to[out=180,in=90] (-.2,0.2);
	\draw[->,red] (0.2,0.2) to[out=90,in=0] (0,.4);
	\draw[-,red] (-.2,0.2) to[out=-90,in=180] (0,0);
	\draw[-,red] (0,0) to[out=0,in=-90] (0.2,0.2);
	\node at (0.0,-.2) {$\scriptstyle{\red\alpha}$};
\end{tikzpicture}
(u)=\big(\begin{tikzpicture}[baseline = 1.25mm]
	\draw[<-,red] (0,0.4) to[out=180,in=90] (-.2,0.2);
	\draw[-,red] (0.2,0.2) to[out=90,in=0] (0,.4);
	\draw[-,red] (-.2,0.2) to[out=-90,in=180] (0,0);
	\draw[-,red] (0,0) to[out=0,in=-90] (0.2,0.2);
	\node at (0.0,-.2) {$\scriptstyle{\red\alpha}$};
\end{tikzpicture}(u)\big)^{-1}.$$

The evaluation $\bbO^{(\red{\alpha})}(u)(L): \unit(L) \to \unit(L)$ of this natural transformation $\bbO^{(\red{\alpha})}(u)\in \calE nd(\unit)((u^{-1}))$
on an irreducible object $L\in \calC$ is
\begin{align*}
	\bbO^{(\red{\alpha})}_L(u)&:=
	\mathord{
		\begin{tikzpicture}[baseline = -1mm]
			\node at (0.08,0) {$\scriptstyle\redanticlock(u)$};
   \node[red] at (.05,-.37){$\scriptstyle \alpha$};
			\draw[-,darkg,thick] (0.68,.2) to (0.68,-.22);
			\node at (0.68,-.37) {$\darkg\scriptstyle{L}$};
		\end{tikzpicture}
	}=\left(
	\mathord{
		\begin{tikzpicture}[baseline = -1mm]
			\node at (0.08,0) {$\scriptstyle\redclock(u)$};
   \node[red] at (.05,-.37){$\scriptstyle \alpha$};
			\draw[-,darkg,thick] (0.68,.2) to (0.68,-.22);
			\node at (0.68,-.37) {$\darkg\scriptstyle{L}$};
		\end{tikzpicture}
	}\right)^{-1}\in \End_{\calC}(L)((u^{-1})).
\end{align*}
For an irreducible object $L \in \calC$,
it was shown in \cite[Lemma 4.17]{BSW20a} that
$$
\bbO^{(\red{\alpha})}_L(u) = m^{(\red{\alpha})}_L(u)/n^{(\red{\alpha})}_L(u),\text{ for each }\alpha=1,\ldots,d.
$$

By the relations \eqref{mixedHecke}, a bubble of color ${\red{\alpha}}$ can slide through any strand of color $\blue{\beta}\neq \red{\alpha}$ freely as follows.
\begin{lemma}
\label{lem:colorbubbleslide}
	We have the following relations:
	\begin{align*}
		&\mathord{\begin{tikzpicture}[baseline=0]
				\draw[<-,thin,red] (0.3,0) to[out=90,in=0] (0,0.3);
				\draw[-,thin,red] (0,0.3) to[out=180,in=90] (-.3,0);
				\draw[-,thin,red] (-.3,0) to[out=-90,in=180] (0,-0.3);
				\draw[-,thin,red] (0,-0.3) to[out=0,in=-90] (0.3,0);
				\draw[->,thin,blue] (0.6,-0.5) to (0.6,.6);
				\node[red] at (-.6,0) {$\color{darkblue}\scriptstyle{x^r}$};
				\node at (-.3,0) {$\dott$};
			\end{tikzpicture}
		}
		=
		\mathord{\begin{tikzpicture}[baseline=0]
				\draw[<-,thin,red] (0.3,0) to[out=90,in=0] (0,0.3);
				\draw[-,thin,red] (0,0.3) to[out=180,in=90] (-.3,0);
				\draw[-,thin,red] (-.3,0) to[out=-90,in=180] (0,-0.3);
				\draw[-,thin,red] (0,-0.3) to[out=0,in=-90] (0.3,0);
				\draw[->,thin,blue] (-1.25,-0.5) to (-1.25,.6);
				\node[red] at (-.6,0) {$\color{darkblue}\scriptstyle{x^r}$};
				\node at (-.3,0) {$\dott$};
			\end{tikzpicture}
		}\:,
		&\mathord{\begin{tikzpicture}[baseline=0]
				\draw[<-,thin,red] (0.3,0) to[out=90,in=0] (0,0.3);
				\draw[-,thin,red] (0,0.3) to[out=180,in=90] (-.3,0);
				\draw[-,thin,red] (-.3,0) to[out=-90,in=180] (0,-0.3);
				\draw[-,thin,red] (0,-0.3) to[out=0,in=-90] (0.3,0);
				\draw[<-,thin,blue] (0.6,-0.5) to (0.6,.6);
				\node[red] at (-.6,0) {$\color{darkblue}\scriptstyle{x^r}$};
				\node at (-.3,0) {$\dott$};
			\end{tikzpicture}
		}
		=
		\mathord{\begin{tikzpicture}[baseline=0]
				\draw[<-,thin,red] (0.3,0) to[out=90,in=0] (0,0.3);
				\draw[-,thin,red] (0,0.3) to[out=180,in=90] (-.3,0);
				\draw[-,thin,red] (-.3,0) to[out=-90,in=180] (0,-0.3);
				\draw[-,thin,red] (0,-0.3) to[out=0,in=-90] (0.3,0);
				\draw[<-,thin,blue] (-1.25,-0.5) to (-1.25,.6);
				\node at (-.6,0) {$\color{darkblue}\scriptstyle{x^r}$};
				\node at (-.3,0) {$\dott$};
			\end{tikzpicture}
		}\:,\\
		&\mathord{\begin{tikzpicture}[baseline=0]
				\draw[->,thin,red] (0.3,0) to[out=90,in=0] (0,0.3);
				\draw[-,thin,red] (0,0.3) to[out=180,in=90] (-.3,0);
				\draw[-,thin,red] (-.3,0) to[out=-90,in=180] (0,-0.3);
				\draw[-,thin,red] (0,-0.3) to[out=0,in=-90] (0.3,0);
				\draw[->,thin,blue] (0.6,-0.5) to (0.6,.6);
				\node at (-.6,0) {$\color{darkblue}\scriptstyle{x^r}$};
				\node at (-.3,0) {$\dott$};
			\end{tikzpicture}
		}
		=
		\mathord{\begin{tikzpicture}[baseline=0]
				\draw[->,thin,red] (0.3,0) to[out=90,in=0] (0,0.3);
				\draw[-,thin,red] (0,0.3) to[out=180,in=90] (-.3,0);
				\draw[-,thin,red] (-.3,0) to[out=-90,in=180] (0,-0.3);
				\draw[-,thin,red] (0,-0.3) to[out=0,in=-90] (0.3,0);
				\draw[->,thin,blue] (-1.25,-0.5) to (-1.25,.6);
				\node at (-.6,0) {$\color{darkblue}\scriptstyle{x^r}$};
				\node at (-.3,0) {$\dott$};
			\end{tikzpicture}
		}\:,
		&\mathord{\begin{tikzpicture}[baseline=0]
				\draw[->,thin,red] (0.3,0) to[out=90,in=0] (0,0.3);
				\draw[-,thin,red] (0,0.3) to[out=180,in=90] (-.3,0);
				\draw[-,thin,red] (-.3,0) to[out=-90,in=180] (0,-0.3);
				\draw[-,thin,red] (0,-0.3) to[out=0,in=-90] (0.3,0);
				\draw[<-,thin,blue] (0.6,-0.5) to (0.6,.6);
				\node at (-.6,0) {$\color{darkblue}\scriptstyle{x^r}$};
				\node at (-.3,0) {$\dott$};
			\end{tikzpicture}
		}
		=
		\mathord{\begin{tikzpicture}[baseline=0]
				\draw[->,thin,red] (0.3,0) to[out=90,in=0] (0,0.3);
				\draw[-,thin,red] (0,0.3) to[out=180,in=90] (-.3,0);
				\draw[-,thin,red] (-.3,0) to[out=-90,in=180] (0,-0.3);
				\draw[-,thin,red] (0,-0.3) to[out=0,in=-90] (0.3,0);
				\draw[<-,thin,blue] (-1.25,-0.5) to (-1.25,.6);
				\node at (-.6,0) {$\color{darkblue}\scriptstyle{x^r}$};
				\node at (-.3,0) {$\dott$};
			\end{tikzpicture}
		}.
	\end{align*}
\end{lemma}

For each quantum color $\alpha$, choose $q_\alpha\in\kk^\times$ with $q_\alpha-q_\alpha^{-1}=z_\alpha$. Define
\begin{equation}
\label{eq:ipm}
	i^{\pm,\alpha} := \left\{\begin{array}{ll}
		q_\alpha^{\pm 2} i&\text{in the quantum case ($z_\alpha\neq 0$),}\\
		i\pm 1&\text{in the degenerate case ($z_\alpha = 0$),}
	\end{array}\right.
\end{equation}
Then we have the following lemma.
\begin{lemma}
\label{lem:choose}
	Suppose that $L\in \calC$ is an irreducible object, and let $M$ be an irreducible subquotient of  $E_i^{(\blue{\beta})}(L)$ for $i\in I^{(\blue{\beta})}.$
	Then for each $\alpha,\beta=1,\cdots,d$, we have
	\begin{align*}
		\bbO^{(\red{\alpha})}_M(u)=\begin{cases}
			\bbO^{(\red{\alpha})}_L(u)\frac{ (u-i^{+,\beta})(u-i^{-,\beta})}{(u-i)^2}
			&\text{if ${\red{\alpha}}={\blue{\beta}}$};\\
			\bbO^{(\red{\alpha})}_L(u)&\text{if ${\red{\alpha}}\neq{\blue{\beta}}$}.
		\end{cases}
	\end{align*}
\end{lemma}
\begin{proof}
	The case when $\alpha=\beta$ follows from \cite[Lemma 4.5]{BSW20a}. The case when $\alpha\neq \beta$ follows from Lemma~\ref{lem:colorbubbleslide}.
\end{proof}

\subsubsection{Weight decomposition of category $\calC$}
\label{sec:wtdecompC}
For each $\red{\alpha}=1,\cdots,d,$ we define the {\em $(\red{\alpha})$-spectrum} $I^{(\red{\alpha})}$ of $\calC$ to be the union of the root sets of the minimal polynomials $m^{(\red{\alpha})}_L(u)$ for all irreducible objects $L$ in $\calC$. By exactness and adjunction, the spectrum $I^{(\red{\alpha})}$ is also
the union of the sets of roots of the polynomials $n^{(\red{\alpha})}_L(u)$ for all irreducible $L \in \mathcal{C}$. Consequently,
\begin{align*}
	E^{(\red{\alpha})} &=
	\bigoplus_{i \in I^{(\red{\alpha})}} E^{(\red{\alpha})}_i,&
	F ^{(\red{\alpha})}&= \bigoplus_{i \in I^{(\red{\alpha})}} F^{(\red{\alpha})}_i,
\end{align*}
where each of the endofunctors
$E^{(\red{\alpha})}_i$ and $F^{(\red{\alpha})}_i$ are non-zero.

Recall $i^{\pm,\alpha}$ from \eqref{eq:ipm}. A subset of $R$ is called $\sigma_\alpha$-stable if it is closed under the automorphisms $\sigma_\alpha:i \mapsto i^{+,\alpha}$ and $\sigma_\alpha^{-1}:i\mapsto i^{-,\alpha}$, assuming moreover that it does not contain $0$ in the quantum case. In particular, \cite[Lemma 4.6]{BSW20a} proves that the $(\red{\alpha})$-spectrum $I^{(\red{\alpha})}$ is $\sigma_\alpha$-stable.
The data $(I^{(\red{\alpha})},\sigma_\alpha)$ naturally give rises to a quiver $Q^{(\red{\alpha})}$ .

We define a quiver $Q$ to be the disjoint union of $Q^{(\red{\alpha})}$:
$Q=\bigoplus_{\red{\alpha}=1}^d Q^{(\red{\alpha})},$
where there are no arrows between $Q^{(\red{\alpha})}$ and $Q^{(\blue{\beta})}$ for any $\red{\alpha}\neq \blue{\beta}.$ We denote by $$I=\bigsqcup_{\red{\alpha}=1}^{d}I^{(\red{\alpha})}$$ the
subsequent partition of the vertex set, and we call $I$ the colored spectrum.

In particular, the quiver $Q$ is obtained as follows
\begin{itemize}[leftmargin=8mm]
	\item  the vertex set is $I$, and
	\item there is an arrow $i\to i^{+,\alpha}$ if and only if $i, i^{+,\alpha}\in I^{(\red{\alpha})}$ for some $\red{\alpha}=1,\ldots,d$.
 \end{itemize}
Each connected component of $Q$ is determined by the corresponding $\sigma_\alpha$-orbit, hence has type $A_{e-1}^{(1)}$ for a finite orbit of size $e\ge2$ and type $A_\infty$ for an infinite orbit.

There is an associated generalized Cartan matrix
$(a^{(\red{\alpha})}_{i,j})_{i,j \in I^{(\red{\alpha})}}$ with $a^{(\red{\alpha})}_{i,i} := 2$ for each $i^{(\red{\alpha})} \in I^{(\red{\alpha})}$, and
$a^{(\red{\alpha})}_{i,j}:=-\delta_{i^{+,\alpha},j}-\delta_{i,j^{+,\alpha}}$
for each $i^{(\red{\alpha})} \neq j^{(\red{\alpha})}$.
Let $\g^{(\red{\alpha})}$ be the Kac--Moody Lie algebra over $\C$
generated by $\{e^{(\red{\alpha})}_i, f^{(\red{\alpha})}_i, h^{(\red{\alpha})}_i\:|\:i^{(\red{\alpha})} \in I^{(\red{\alpha})}\}$ subject to the Serre
relations defined from
the Cartan matrix $(a^{(\red{\alpha})}_{i,j})_{i^{(\red{\alpha})}, j^{(\red{\alpha})} \in I^{(\red{\alpha})}}$.
Let $\g$ be the direct sum of $\g^{(\red{\alpha})}$ for all $\red{\alpha}=1,\ldots,d$.

As in \S~\ref{sec:KMCdef}, for each $\red{\alpha}$ and $i^{(\red{\alpha})}\in\scrI^{(\red{\alpha})}$, we let $\alpha_{i^{(\red{\alpha})}},$ $\alpha_{i^{(\red{\alpha})}}^\vee$ be the simple root and coroot
corresponding to $e_{i^{(\red{\alpha})}}$ and let $\Lambda_{i^{(\red{\alpha})}}$ be the $i$-th fundamental weight associated with $\mathfrak{g}^{(\red{\alpha})}$. Define
$$\Y^{(\red{\alpha})}=\bigoplus\limits_{i^{(\red{\alpha})}\in \scrI^{(\red{\alpha})}} \bbZ \alpha_{i^{(\red{\alpha})}}, {(\Y^{(\red{\alpha})})}^\vee = \bigoplus\limits_{i^{(\red{\alpha})}\in \scrI^{(\red{\alpha})}} \bbZ \alpha_{i^{(\red{\alpha})}}^\vee,
\X^{(\red{\alpha})}=\bigoplus\limits_{i^{(\red{\alpha})}\in \scrI^{(\red{\alpha})}}\bbZ \Lambda_{i^{(\red{\alpha})}}.$$

For an irreducible object $L \in \mathcal{C}$, let
\begin{equation*}
	\wt(L) := \sum_{\red{\alpha}=1}^d\sum_{i \in \scrI^{(\red{\alpha})}} (\phi^{(\red{\alpha})}_i(L)-\epsilon^{(\red{\alpha})}_i(L)) \Lambda_{i^{(\red{\alpha})}} \in \X=\bigoplus_{\red{\alpha}=1}^d \X^{(\red{\alpha})}.
\end{equation*}
For $\lambda\in\X$, let $\mathcal C_\lambda$ be the Serre subcategory of
objects whose irreducible subquotients all have weight $\lambda$.
Irreducibles of different weights have different central characters, and the
block decomposition gives
\begin{equation*}
	\textstyle
	\mathcal{C} =
	\bigoplus_{\lambda \in X} \mathcal{C}_\lambda  .
\end{equation*}
This is the {\em weight-space decomposition} of $\mathcal{C}$.

\begin{lemma}\label{pg}
	For $\lambda \in X$ and $i \in \scrI^{(\red{\alpha})}$, the
	restrictions of $E^{(\red{\alpha})}_i$ and $F^{(\red{\alpha})}_i$ to $\mathcal{C}_\lambda$ give functors
	\begin{align*}
		E^{(\red{\alpha})}_i|_{\mathcal{C}_\lambda}&:\mathcal{C}_\lambda \rightarrow \mathcal{C}_{\lambda+\alpha_{i^{(\red{\alpha})}}},&
		F^{(\red{\alpha})}_i|_{\mathcal{C}_\lambda}&:\mathcal{C}_\lambda \rightarrow \mathcal{C}_{\lambda-\alpha_{i^{(\red{\alpha})}}},
	\end{align*}
\end{lemma}

\begin{proof}
	For $E^{(\red{\alpha})}_i$, this follows from Lemma~\ref{lem:choose}.
	Then it follows for $F^{(\red{\alpha})}_i$ by adjunction.
\end{proof}
\medskip

\subsection{Heisenberg category to Kac--Moody categorification}

Following \cite{BSW20a}, for a finitely generated $V\in\mathcal C$ denote
the endomorphism
$
\begin{tikzpicture}[red,baseline = -1.7mm]
\draw[->] (0.08,-.2) to (0.08,.2);
\node[black] at (0.08,-0.35) {$\scriptstyle{i^{(\red{\alpha})}}$};
\node at (0.08,-0.02) {$\dott$};
\node[black] at (-0.35,-0.02) {$\scriptstyle{x-i^{(\red{\alpha})}}$};
\draw[-,darkg,thick] (0.4,-.2) to (0.4,.2);
\node at (0.55,0) {$\darkg\scriptstyle{V}$};
\end{tikzpicture}
$
is known to be nilpotent, allowing us to define the following
transformation
\begin{equation*}
\begin{tikzpicture}[red,baseline = -0.8mm]
	\draw[->] (0.08,-.25) to (0.08,.3);
      \node[black] at (0.08,-0.4) {$\scriptstyle{i^{(\red{\alpha})}}$};
      \node at (0.08,0.02) {$\dt$};
      \node[black] at (-0.3,0.02) {$\scriptstyle{p(x)}$};
\end{tikzpicture}:
E^{(\red{\alpha})}_i \Rightarrow E^{(\red{\alpha})}_i
\end{equation*}
for any $\red{\alpha}=1,\ldots,d,\ i^{(\red{\alpha})} \in \kk$ and any $p(x) \in \kk[\![x-i^{(\red{\alpha})}]\!]$.
The same definition can be made for dots on downward strings too.

More generally, suppose that we are given some more complicated string
diagram for a natural
transformation between some endofunctors of $\mathcal{C}$,
together with a sequence of $n$
points $P_1,\dots,P_n$ on strings colored $i^{(\red{\alpha_1})}_1,\dots,i^{(\blue{\alpha_n})}_n \in \kk$ in this diagram.
Then for any $p(x_1,\dots,x_n) \in \kk[\![x_1-i^{(\red{\alpha_1})}_1,\dots,x_n-i^{(\blue{\alpha_n})}_n]\!]$
there is a well-defined natural transformation represented diagrammatically by placing dots on these points and connecting them with a dotted arrow from $P_1$ to $P_n$, labeled by the power series $p(x_1,\dots,x_n)$. Here, $x_1$ denotes the variable corresponding to the first dot, closest to the arrow’s tail, while $x_n$ corresponds to the last dot, near the arrowhead. For example, suppose that $n = 2$ and $i^{(\red{\alpha_1})}_1 \neq i^{(\blue{\alpha_2})}_2$.
Set $c :=
(i^{(\blue{\alpha_2})}_2-i^{(\red{\alpha_1})}_1)^{-1}$ so that $(x_2-x_1)^{-1} \in \kk[\![x_1-i^{(\red{\alpha_1})}_1,x_2-i^{(\blue{\alpha_2})}_2]\!]$
has power series expansion
$c
+c^2(x_1-i_1)
- c^2(x_2-i_2) +$ (higher order terms).
Then we have
\begin{align*}
\begin{tikzpicture}[baseline = -0.8mm]
	\draw[->,densely dotted] (0.55,.03) to (0.15,.03);
	\draw[->,red] (0.58,-.25) to (0.58,.3);
	\draw[->,blue] (0.08,-.25) to (0.08,.3);
      \node at (0.08,-0.5) {$\scriptstyle{i^{(\blue{\alpha_2})}_2}$};
      \node[blue] at (0.08,0.02) {$\dt$};
      \node at (0.58,-0.5) {$\scriptstyle{i^{(\red{\alpha_1})}_1}$};
      \node[red] at (0.58,0.02) {$\dt$};
      \node at (1.3,0.04) {$\scriptstyle{(x_2-x_1)^{-1}}$};
\end{tikzpicture}&=
c\begin{tikzpicture}[baseline = -0.8mm]
	\draw[->,red] (0.58,-.25) to (0.58,.3);
	\draw[->,blue] (0.08,-.25) to (0.08,.3);
      \node at (0.08,-0.5) {$\scriptstyle{i^{(\blue{\alpha_2})}_2}$};
      \node at (0.58,-0.5) {$\scriptstyle{i^{(\red{\alpha_1})}_1}$};
\end{tikzpicture}
+c^2\begin{tikzpicture}[baseline = -0.8mm]
	\draw[->,red] (0.58,-.25) to (0.58,.3);
	\draw[->,blue] (0.08,-.25) to (0.08,.3);
      \node at (0.08,-0.5) {$\scriptstyle{i^{(\blue{\alpha_2})}_2}$};
      \node at (0.58,-0.5) {$\scriptstyle{i^{(\red{\alpha_1})}_1}$};
      \node[red] at (0.58,0.02) {$\dt$};
      \node at (1.23,0.04) {$\scriptstyle x-i^{(\red{\alpha_1})}_1$};
\end{tikzpicture}
-c^2\begin{tikzpicture}[baseline = -0.8mm]
	\draw[->,red] (0.58,-.25) to (0.58,.3);
	\draw[->,blue] (0.08,-.25) to (0.08,.3);
      \node at (0.08,-0.5) {$\scriptstyle{i^{(\blue{\alpha_2})}_2}$};
      \node at (0.58,-0.5) {$\scriptstyle{i^{(\red{\alpha_1})}_1}$};
      \node at (0.08,0.02) {$\dt$};
      \node at (-.42,0.04) {$\scriptstyle x-i^{(\blue{\alpha_2})}_2$};
\end{tikzpicture}
+\cdots.
\end{align*}

\begin{lemma}\label{banach1}
	For $i \in I^{(\blue{\beta})}$ and $j\in \scrI^{(\red{\alpha})}$ with $i^{(\blue{\beta})}\neq j^{(\red{\alpha})}$,
	the natural transformations
	\begin{align*}
		\begin{tikzpicture}[baseline = -1mm]
			\draw[->,blue] (0.28,-.28) to (-0.28,.28);
			\draw[<-,red] (-0.28,-.28) to (0.28,.28);
			\node at (-0.33,-0.43) {$\scriptstyle{j}$};
			\node at (0.33,-0.43) {$\scriptstyle{i}$};
			\node at (-0.33,0.43) {$\scriptstyle{i}$};
			\node at (0.33,0.43) {$\scriptstyle{j}$};
			\node at (0,-.01) {$\diamond$};
		\end{tikzpicture}
		:F^{(\red{\alpha})}_j E^{(\blue{\beta})}_i \Rightarrow E^{(\blue{\beta})}_i F^{(\red{\alpha})}_j&
		,&
		\begin{tikzpicture}[baseline = -1mm]
			\draw[<-,red] (0.28,-.28) to (-0.28,.28);
			\draw[->,blue] (-0.28,-.28) to (0.28,.28);
			\node at (-0.33,-0.43) {$\scriptstyle{i}$};
			\node at (0.33,-0.43) {$\scriptstyle{j}$};
			\node at (-0.33,0.43) {$\scriptstyle{j}$};
			\node at (0.33,0.43) {$\scriptstyle{i}$};
			\node at (0,-.01) {$\diamond$};
		\end{tikzpicture}
		: E^{(\blue{\beta})}_i F^{(\red{\alpha})}_j \Rightarrow F^{(\red{\alpha})}_j E^{(\blue{\beta})}_i&
	\end{align*}
	are mutually inverse
	isomorphisms.
	(Here, we have drawn the crossings in the degenerate
	case; in the quantum case they should be interpreted as
	positive or negative crossings, it does not matter which is chosen.)
\end{lemma}

\begin{proof}
	For the case when $\red{\alpha}=\blue{\beta}$, one checks that the compositions both ways around are the identities as shown in \cite[Lemma 4.8]{BSW20a}. When $\red{\alpha}\neq\blue{\beta}$, the lemma follows from \eqref{mixedMackey}.
\end{proof}

Recall that $\calC$ is a locally finite
module category over $\symHeis_d$. Let $E_i^{(\red{\alpha})}$ and $F_i^{(\red{\alpha})}$ be the
eigenfunctors defined \S~\ref{sec:colordiagam} for each $\red{\alpha}=1,\ldots,d$, and recall the various diagrams representing natural transformations between these functors introduced there. Let $I$ be the spectrum of $\calC$
as in \S~\ref{sec:wtdecompC}, and
$\UU(\g)$ be the corresponding Kac--Moody 2-category as in \S~\ref{sec:KMCdef}.

For $i\in \scrI^{(\red{\alpha})}$ and $j\in I^{(\blue{\beta})}$ of the same
color, the required generating 2-morphisms are those of
\cite[Theorem~4.11]{BSW20a}.  The mixed-color crossings constructed above
supply the remaining generators.

\begin{theorem}
\label{thm:HtoK}
	  Associated to $\mathcal{C}$,
	there is a unique $2$-representation
	$\Phi:\UU(\g)\rightarrow
	\mathfrak{Cat}_\kk$
	defined on objects by
	$\lambda\mapsto \mathcal{C}_\lambda$,
	on generating
	1-morphisms by
	$E_i^{(\red{\alpha})} 1_\lambda\mapsto E^{(\red{\alpha})}_i|_{\mathcal{C}_\lambda}$
	and $F^{(\red{\alpha})}_i 1_\lambda\mapsto F^{(\red{\alpha})}_i|_{\mathcal{C}_\lambda}$ for $i^{(\red{\alpha})}\in  \scrI^{(\red{\alpha})}$,
	and on generating 2-morphisms by
	\begin{align*}
		\begin{tikzpicture}[baseline = -0.8mm]
			\draw[->,thick] (0.08,-.25) to (0.08,.3);
			\node at (0.08,-0.4) {$\scriptstyle{i^{(\red{\alpha})}}$};
			\node at (0.08,0.02) {$\bullet$};
			\node at (0.3,0.02) {$\color{gray}\scriptstyle\lambda$};
		\end{tikzpicture}
		&\mapsto
		\begin{tikzpicture}[baseline = -0.8mm]
			\draw[->,red] (0.08,-.25) to (0.08,.3);
			\node at (0.08,-0.4) {$\scriptstyle{i^{(\red{\alpha})}}$};
			\node[red] at (0.08,0.02) {$\dott$};
			\node at (0.5,0.02) {$\scriptstyle{x-i^{(\red{\alpha})}}$};
		\end{tikzpicture}
		,
		&
		\begin{tikzpicture}[baseline = 1mm]
			\draw[<-,thick] (0.4,0.4) to[out=-90, in=0] (0.1,0);
			\draw[-,thick] (0.1,0) to[out = 180, in = -90] (-0.2,0.4);
			\node at (-0.2,0.55) {$\scriptstyle{i^{(\red{\alpha})}}$};
			\node at (0.6,0.2) {$\color{gray}\scriptstyle\lambda$};
		\end{tikzpicture}
		&\mapsto
		\begin{tikzpicture}[baseline = 1mm]
			\draw[<-,red] (0.4,0.4) to[out=-90, in=0] (0.1,0);
			\draw[-,red] (0.1,0) to[out = 180, in = -90] (-0.2,0.4);
			\node at (-0.2,0.55) {$\scriptstyle{i^{(\red{\alpha})}}$};
		\end{tikzpicture}\:,
		&
		\begin{tikzpicture}[baseline = 1mm]
			\draw[<-,thick] (0.4,0) to[out=90, in=0] (0.1,0.4);
			\draw[-,thick] (0.1,0.4) to[out = 180, in = 90] (-0.2,0);
			\node at (-0.2,-0.15) {$\scriptstyle{i^{(\red{\alpha})}}$};
			\node at (0.6,0.2) {$\color{gray}\scriptstyle\lambda$};
		\end{tikzpicture}
		&\mapsto
		\begin{tikzpicture}[baseline = 1mm]
			\draw[<-,red] (0.4,0) to[out=90, in=0] (0.1,0.4);
			\draw[-, red] (0.1,0.4) to[out = 180, in = 90] (-0.2,0);
			\node at (-0.2,-0.15) {$\scriptstyle{i^{(\red{\alpha})}}$};
		\end{tikzpicture},
	\end{align*}

	\begin{align*}
 \begin{tikzpicture}[baseline = 0]
	\draw[->,thick] (0.28,-.3) to (-0.28,.4);
	\draw[->,thick] (-0.28,-.3) to (0.28,.4);
   \node at (-.3,-.45) {$\scriptstyle{j^{(\red{\alpha})}}$};
   \node at (.3,-.45) {$\scriptstyle{i^{(\red{\alpha})}}$};
   \node at (0.5,0.05) {$\color{gray}\scriptstyle{\lambda}$};
\end{tikzpicture}
\mapsto
\left\{
\begin{array}{ll}
\begin{tikzpicture}[baseline = 0]
	\draw[->,red] (0.38,-.4) to (-0.38,.5);
	\draw[->,red] (-0.38,-.4) to (0.38,.5);
      \node at (0,0.05) {$\diamond$};
	\draw[->,densely dotted] (0.26,-.24) to (-0.19,-.24);
   \node at (-.4,-.55) {$\scriptstyle{i}$};
   \node at (.4,-.55) {$\scriptstyle{i}$};
   \node at (-.4,.65) {$\scriptstyle{i}$};
   \node at (.4,.65) {$\scriptstyle{i}$};
      \node[red] at (-0.26,-0.25) {$\dt$};
      \node[red] at (0.26,-0.25) {$\dt$};
      \node at (1.1,-0.2) {$\scriptstyle{(q_{\red{\alpha}}x_2-q_{\red{\alpha}}^{-1}x_1)^{-1}}$};
\end{tikzpicture}
+\begin{tikzpicture}[baseline = 0mm]
	\draw[->,densely dotted] (0.55,.06) to (0.14,.06);
	\draw[->,red] (0.58,-.4) to (0.58,.5);
	\draw[->,red] (0.08,-.4) to (0.08,.5);
      \node at (0.08,-0.55) {$\scriptstyle{i}$};
      \node at (0.58,-0.55) {$\scriptstyle{i}$};
      \node[red] at (0.08,0.05) {$\dt$};
      \node[red] at (0.58,0.05) {$\dt$};
      \node at (1.4,0.1) {$\scriptstyle{(x_2-x_1+1)^{-1}}$};
\end{tikzpicture}
&\text{if $j^{(\red{\alpha})}=i^{(\red{\alpha})}$,}\\
		\begin{tikzpicture}[baseline = 0]
	\draw[->,red] (0.38,-.4) to (-0.38,.5);
	\draw[->,red] (-0.38,-.4) to (0.38,.5);
	\draw[->,densely dotted] (0.26,-.24) to (-0.19,-.24);
      \node at (0,0.05) {$\diamond$};
   \node at (-.4,-.55) {$\scriptstyle{q_{\red{\alpha}}^2i}$};
   \node at (-.4,.65) {$\scriptstyle{i}$};
   \node at (.4,.65) {$\scriptstyle{i+1}$};
   \node at (.4,-.55) {$\scriptstyle{i}$};
      \node[red] at (-0.26,-0.25) {$\dt$};
      \node[red] at (0.26,-0.25) {$\dt$};
      \node at (.77,-0.2) {$\scriptstyle{x_2-x_1}$};
\end{tikzpicture}
&\text{if $j^{(\red{\alpha})}={(i^{(\red{\alpha})})}^+$,}\\
-\begin{tikzpicture}[baseline = 0]
	\draw[->,red] (0.38,-.4) to (-0.38,.5);
      \node at (0,0.05) {$\diamond$};
	\draw[->,red] (-0.38,-.4) to (0.38,.5);
	\draw[->,densely dotted] (0.26,-.24) to (-0.19,-.24);
   \node at (-.4,-.55) {$\scriptstyle{j}$};
   \node at (.4,-.55) {$\scriptstyle{i}$};
   \node at (.4,.65) {$\scriptstyle{j}$};
   \node at (-.4,.65) {$\scriptstyle{i}$};
      \node[red] at (-0.26,-0.25) {$\dt$};
      \node[red] at (0.26,-0.25) {$\dt$};
      \node at (1.57,-0.2) {$\scriptstyle{(x_2-x_1)(x_2-x_1-1)^{-1}}$};
\end{tikzpicture}
&\text{if $j^{(\red{\alpha})} \neq i^{(\red{\alpha})},{(i^{(\red{\alpha})})}^+$}
\end{array}\right.
	\end{align*}

	\begin{align*}
		&\begin{tikzpicture}[baseline = 0]
			\draw[->,thick] (0.28,-.3) to (-0.28,.4);
			\draw[->,thick] (-0.28,-.3) to (0.28,.4);
			\node at (-.3,-.45) {$\scriptstyle{j^{(\blue{\beta})}}$};
			\node at (.3,-.45) {$\scriptstyle{i^{(\red{\alpha})}}$};
			\node at (0.5,0.05) {$\color{gray}\scriptstyle{\lambda}$};
		\end{tikzpicture}
		\quad\mapsto\quad
		\begin{tikzpicture}[baseline = 0]
			\draw[->,red] (0.38,-.4) to (-0.38,.5);
			\draw[->,blue] (-0.38,-.4) to (0.38,.5);
			\node at (0,0.05) {$\diamond$};
			\node at (-.4,-.55) {$\scriptstyle{j}$};
			\node at (.4,-.55) {$\scriptstyle{i}$};
			\node at (-.4,.65) {$\scriptstyle{i}$};
			\node at (.4,.65) {$\scriptstyle{j}$};
		\end{tikzpicture},
	\end{align*}
	for  $\red{\alpha},\blue{\beta}=1,\ldots,d$ and $\red{\alpha}\neq\blue{\beta}$ in the degenerate case, or
    \begin{align*}
		\begin{tikzpicture}[baseline = -0.8mm]
			\draw[->,thick] (0.08,-.25) to (0.08,.3);
			\node at (0.08,-0.4) {$\scriptstyle{i^{(\red{\alpha})}}$};
			\node at (0.08,0.02) {$\bullet$};
			\node at (0.3,0.02) {$\color{gray}\scriptstyle\lambda$};
		\end{tikzpicture}
		&\mapsto
		\begin{tikzpicture}[baseline = -0.8mm]
			\draw[->,red] (0.08,-.25) to (0.08,.3);
			\node at (0.08,-0.4) {$\scriptstyle{i^{(\red{\alpha})}}$};
			\node[red] at (0.08,0.02) {$\dott$};
			\node at (0.5,0.02) {$\scriptstyle{\frac{x}{i^{(\red{\alpha})}}}-1$};
		\end{tikzpicture}
		,
		&
		\begin{tikzpicture}[baseline = 1mm]
			\draw[<-,thick] (0.4,0.4) to[out=-90, in=0] (0.1,0);
			\draw[-,thick] (0.1,0) to[out = 180, in = -90] (-0.2,0.4);
			\node at (-0.2,0.55) {$\scriptstyle{i^{(\red{\alpha})}}$};
			\node at (0.6,0.2) {$\color{gray}\scriptstyle\lambda$};
		\end{tikzpicture}
		&\mapsto
		\begin{tikzpicture}[baseline = 1mm]
			\draw[<-,red] (0.4,0.4) to[out=-90, in=0] (0.1,0);
			\draw[-,red] (0.1,0) to[out = 180, in = -90] (-0.2,0.4);
			\node at (-0.2,0.55) {$\scriptstyle{i^{(\red{\alpha})}}$};
		\end{tikzpicture}\:,
		&
		\begin{tikzpicture}[baseline = 1mm]
			\draw[<-,thick] (0.4,0) to[out=90, in=0] (0.1,0.4);
			\draw[-,thick] (0.1,0.4) to[out = 180, in = 90] (-0.2,0);
			\node at (-0.2,-0.15) {$\scriptstyle{i^{(\red{\alpha})}}$};
			\node at (0.6,0.2) {$\color{gray}\scriptstyle\lambda$};
		\end{tikzpicture}
		&\mapsto
		\begin{tikzpicture}[baseline = 1mm]
			\draw[<-,red] (0.4,0) to[out=90, in=0] (0.1,0.4);
			\draw[-, red] (0.1,0.4) to[out = 180, in = 90] (-0.2,0);
			\node at (-0.2,-0.15) {$\scriptstyle{i^{(\red{\alpha})}}$};
		\end{tikzpicture},
	\end{align*}

	\begin{align*}
 \begin{tikzpicture}[baseline = 0]
	\draw[->,thick] (0.28,-.3) to (-0.28,.4);
	\draw[->,thick] (-0.28,-.3) to (0.28,.4);
   \node at (-.3,-.45) {$\scriptstyle{j^{(\red{\alpha})}}$};
   \node at (.3,-.45) {$\scriptstyle{i^{(\red{\alpha})}}$};
   \node at (0.5,0.05) {$\color{gray}\scriptstyle{\lambda}$};
\end{tikzpicture}
\mapsto
\left\{
\begin{array}{ll}
i^{(\red{\alpha})}\begin{tikzpicture}[baseline = 0]
	\draw[->,red] (0.38,-.4) to (-0.38,.5);
    \draw[wipe] (-0.38,-.4) to (0.38,.5);
	\draw[->,red] (-0.38,-.4) to (0.38,.5);
      \node at (0,0.05) {$\diamond$};
	\draw[->,densely dotted] (0.26,-.24) to (-0.19,-.24);
   \node at (-.4,-.55) {$\scriptstyle{i}$};
   \node at (.4,-.55) {$\scriptstyle{i}$};
   \node at (-.4,.65) {$\scriptstyle{i}$};
   \node at (.4,.65) {$\scriptstyle{i}$};
      \node[red] at (-0.26,-0.25) {$\dt$};
      \node[red] at (0.26,-0.25) {$\dt$};
      \node at (1.1,-0.2) {$\scriptstyle{(x_2-x_1+1)^{-1}}$};
\end{tikzpicture}
+q_{\red{\alpha}}^{-1}i^{(\red{\alpha})}\begin{tikzpicture}[baseline = 0mm]
	\draw[->,densely dotted] (0.55,.06) to (0.14,.06);
	\draw[->,red] (0.58,-.4) to (0.58,.5);
	\draw[->,red] (0.08,-.4) to (0.08,.5);
      \node at (0.08,-0.55) {$\scriptstyle{i}$};
      \node at (0.58,-0.55) {$\scriptstyle{i}$};
      \node[red] at (0.08,0.05) {$\dt$};
      \node[red] at (0.58,0.05) {$\dt$};
      \node at (1.6,0.1) {$\scriptstyle{(q_{\red{\alpha}}x_2-q_{\red{\alpha}}^{-1}x_1)^{-1}}$};
\end{tikzpicture}
&\text{if $j^{(\red{\alpha})}=i^{(\red{\alpha})}$,}\\
		q_{\red{\alpha}}^{-1}(i^{(\red{\alpha})})^{-1}\begin{tikzpicture}[baseline = 0]
	\draw[->,red] (0.38,-.4) to (-0.38,.5);
    \draw[wipe]  (-0.38,-.4) to (0.38,.5);
	\draw[->,red] (-0.38,-.4) to (0.38,.5);
	\draw[->,densely dotted] (0.26,-.24) to (-0.19,-.24);
      \node at (0,0.05) {$\diamond$};
   \node at (-.4,-.55) {$\scriptstyle{i+1}$};
   \node at (-.4,.65) {$q^2\scriptstyle{i}$};
   \node at (.4,.65) {$\scriptstyle{i+1}$};
   \node at (.4,-.55) {$q^2\scriptstyle{i}$};
      \node[red] at (-0.26,-0.25) {$\dt$};
      \node[red] at (0.26,-0.25) {$\dt$};
      \node at (.77,-0.2) {$\scriptstyle{x_2-x_1}$};
\end{tikzpicture}
&\text{if $j^{(\red{\alpha})}={(i^{(\red{\alpha})})}^+$,}\\
-\begin{tikzpicture}[baseline = 0]
	\draw[->,red] (0.38,-.4) to (-0.38,.5);
    \draw[wipe] (-0.38,-.4) to (0.38,.5);
      \node at (0,0.05) {$\diamond$};
	\draw[->,red] (-0.38,-.4) to (0.38,.5);
	\draw[->,densely dotted] (0.26,-.24) to (-0.19,-.24);
   \node at (-.4,-.55) {$\scriptstyle{j}$};
   \node at (.4,-.55) {$\scriptstyle{i}$};
   \node at (.4,.65) {$\scriptstyle{j}$};
   \node at (-.4,.65) {$\scriptstyle{i}$};
      \node[red] at (-0.26,-0.25) {$\dt$};
      \node[red] at (0.26,-0.25) {$\dt$};
      \node at (1.7,-0.2) {$\scriptstyle{(x_2-x_1)(q_{\red{\alpha}}^{-1}x_2-q_{\red{\alpha}}x_1)^{-1}}$};
\end{tikzpicture}
&\text{if $j^{(\red{\alpha})} \neq i^{(\red{\alpha})},{(i^{(\red{\alpha})})}^+$}
\end{array}\right.
	\end{align*}

	\begin{align*}
		&\begin{tikzpicture}[baseline = 0]
			\draw[->,thick] (0.28,-.3) to (-0.28,.4);
			\draw[->,thick] (-0.28,-.3) to (0.28,.4);
			\node at (-.3,-.45) {$\scriptstyle{j^{(\blue{\beta})}}$};
			\node at (.3,-.45) {$\scriptstyle{i^{(\red{\alpha})}}$};
			\node at (0.5,0.05) {$\color{gray}\scriptstyle{\lambda}$};
		\end{tikzpicture}
		\quad\mapsto\quad
		\begin{tikzpicture}[baseline = 0]
			\draw[->,red] (0.38,-.4) to (-0.38,.5);
			\draw[->,blue] (-0.38,-.4) to (0.38,.5);
			\node at (0,0.05) {$\diamond$};
			\node at (-.4,-.55) {$\scriptstyle{j}$};
			\node at (.4,-.55) {$\scriptstyle{i}$};
			\node at (-.4,.65) {$\scriptstyle{i}$};
			\node at (.4,.65) {$\scriptstyle{j}$};
		\end{tikzpicture},
	\end{align*}
	for  $\red{\alpha},\blue{\beta}=1,\ldots,d$ and $\red{\alpha}\neq\blue{\beta}$ in the quantum case
\end{theorem}
\begin{proof}
	We only prove for the degenerate case as the quantum case follows similarly. To that end, we need to verify the defining relations (\ref{KM1})--(\ref{KMrightadj}) and (\ref{KMinverse1})--(\ref{KMinverse3}). By \cite[Theorem 4.11]{BSW20a}, it suffices to verify these relations for crossings of the form $\begin{tikzpicture}[baseline = 0]
			\draw[->,thick] (0.28,-.3) to (-0.28,.4);
            \draw[wipe] (-0.28,-.3) to (0.28,.4);
			\draw[->,thick] (-0.28,-.3) to (0.28,.4);
			\node at (-.3,-.45) {$\scriptstyle{j^{(\blue{\beta})}}$};
			\node at (.3,-.45) {$\scriptstyle{i^{(\red{\alpha})}}$};
			\node at (0.5,0.05) {$\color{gray}\scriptstyle{\lambda}$};
		\end{tikzpicture}$ where $\red{\alpha}\neq \blue{\beta}$.  Suppose that $i^{(\red{\alpha})}\in \scrI^{(\red{\alpha})}$, $j^{(\blue{\beta})}\in I^{(\blue{\beta})}$ where $\red{\alpha}\neq \blue{\beta}$. The verification of \eqref{KM1}  follows from \eqref{newdotslideab}.

  To verify \eqref{KM2}, it follows from Lemma~\ref{essential} that
	\begin{align*}\mathord{
			\begin{tikzpicture}[baseline = -1mm]
				\draw[->,red] (0.28,-.28) to (-0.28,.28);
				\draw[->,blue] (-0.28,-.28) to (0.28,.28);
				\node at (-0.33,-0.43) {$\scriptstyle{j}$};
				\node at (0.33,-0.43) {$\scriptstyle{i}$};
				\node at (-0.33,0.43) {$\scriptstyle{i'}$};
				\node at (0.33,0.43) {$\scriptstyle{j'}$};
				\node at (0,-0.01) {$\diamond$};
			\end{tikzpicture}
		}=
		0,\quad \text{and}\quad
		\mathord{
			\begin{tikzpicture}[baseline = -1mm]
				\draw[->,blue] (0.28,-.28) to (-0.28,.28);
				\draw[->,red] (-0.28,-.28) to (0.28,.28);
				\node at (-0.33,-0.43) {$\scriptstyle{i}$};
				\node at (0.33,-0.43) {$\scriptstyle{j}$};
				\node at (-0.33,0.43) {$\scriptstyle{j'}$};
				\node at (0.33,0.43) {$\scriptstyle{i'}$};
				\node at (0,-0.01) {$\diamond$};
			\end{tikzpicture}
		}=
		0,\quad \text{unless $(i')^{(\red{\alpha})}=i^{(\red{\alpha})}\in \scrI^{(\red{\alpha})}$ and $j^{(\blue{\beta})}=(j')^{(\blue{\beta})}\in I^{(\blue{\beta})}$}.
	\end{align*}
	Hence we have \begin{align*}
		\mathord{
			\begin{tikzpicture}[baseline = -1mm]
				\node at (-.23,.75){$\scriptstyle i$};
				\node at (.23,.75){$\scriptstyle j$};
				\node at (-.23,-.75){$\scriptstyle i$};
				\node at (.23,-.75){$\scriptstyle j$};
				\node at (-.35,0){$\scriptstyle j$};
				\node at (.35,0){$\scriptstyle i$};
				\draw[->,red] (0.23,0) to[out=90,in=-90] (-0.23,.6);
				\draw[->,blue] (-0.23,0) to[out=90,in=-90] (0.23,.6);
				\draw[-,blue] (0.23,-.6) to[out=90,in=-90] (-0.23,0);
				\draw[-,red] (-0.23,-.6) to[out=90,in=-90] (0.23,0);
				\node at (0,.3){$\diamond$};
				\node at (0,-.3){$\diamond$};
			\end{tikzpicture}
		}
		&
		=
		\mathord{
			\begin{tikzpicture}[baseline = -1mm]
				\node at (-.23,.75){$\scriptstyle i$};
				\node at (.23,.75){$\scriptstyle j$};
				\node at (-.23,-.75){$\scriptstyle i$};
				\node at (.23,-.75){$\scriptstyle j$};
				\draw[->,red] (0.23,0) to[out=90,in=-90] (-0.23,.6);
				\draw[->,blue] (-0.23,0) to[out=90,in=-90] (0.23,.6);
				\draw[-,blue] (0.23,-.6) to[out=90,in=-90] (-0.23,0);
				\draw[-,red] (-0.23,-.6) to[out=90,in=-90] (0.23,0);
				\draw[-] (-.2,.4) to (-.13,.45);
				\draw[-] (-.2,-.4) to (-.13,-.45);
				\draw[-] (.2,.4) to (.13,.45);
				\draw[-] (.2,-.4) to (.13,-.45);
			\end{tikzpicture}
		}
		=
		\mathord{
			\begin{tikzpicture}[baseline = -1mm]
				\node at (-.28,-.75){$\scriptstyle i$};
				\node at (.08,-.75){$\scriptstyle j$};
				\draw[->,blue] (0.08,-.6) to (0.08,.6);
				\draw[->,red] (-0.28,-.6) to (-0.28,.6);
			\end{tikzpicture}
		},
		&
		\mathord{
			\begin{tikzpicture}[baseline = -1mm]
				\node at (-.23,.75){$\scriptstyle j$};
				\node at (.23,.75){$\scriptstyle i$};
				\node at (-.23,-.75){$\scriptstyle j$};
				\node at (.23,-.75){$\scriptstyle i$};
				\node at (-.35,0){$\scriptstyle i$};
				\node at (.35,0){$\scriptstyle j$};
				\draw[->,blue] (0.23,0) to[out=90,in=-90] (-0.23,.6);
				\draw[->,red] (-0.23,0) to[out=90,in=-90] (0.23,.6);
				\draw[-,red] (0.23,-.6) to[out=90,in=-90] (-0.23,0);
				\draw[-,blue] (-0.23,-.6) to[out=90,in=-90] (0.23,0);
				\node at (0,.3){$\diamond$};
				\node at (0,-.3){$\diamond$};
			\end{tikzpicture}
		}
		=&
		\mathord{
			\begin{tikzpicture}[baseline = -1mm]
				\node at (-.23,.75){$\scriptstyle j$};
				\node at (.23,.75){$\scriptstyle i$};
				\node at (-.23,-.75){$\scriptstyle j$};
				\node at (.23,-.75){$\scriptstyle i$};
				\draw[->,blue] (0.23,0) to[out=90,in=-90] (-0.23,.6);
				\draw[->,red] (-0.23,0) to[out=90,in=-90] (0.23,.6);
				\draw[-,red] (0.23,-.6) to[out=90,in=-90] (-0.23,0);
				\draw[-,blue] (-0.23,-.6) to[out=90,in=-90] (0.23,0);
				\draw[-] (-.2,.4) to (-.13,.45);
				\draw[-] (-.2,-.4) to (-.13,-.45);
				\draw[-] (.2,.4) to (.13,.45);
				\draw[-] (.2,-.4) to (.13,-.45);
			\end{tikzpicture}
		}
		=
		\mathord{
			\begin{tikzpicture}[baseline = -1mm]
				\node at (-.28,-.75){$\scriptstyle j$};
				\node at (.08,-.75){$\scriptstyle i$};
				\draw[->,red] (0.08,-.6) to (0.08,.6);
				\draw[->,blue] (-0.28,-.6) to (-0.28,.6);
			\end{tikzpicture}
		}.
	\end{align*}
The verification of \eqref{KM3} and \eqref{KMrightadj} follow similarly.

 To verify \eqref{KMinverse1}--\eqref{KMinverse3}, we just need to prove that the natural transformations
	\begin{align*}
		\begin{tikzpicture}[baseline = -1mm]
			\draw[->,red] (0.28,-.28) to (-0.28,.28);
			\draw[<-,blue] (-0.28,-.28) to (0.28,.28);
			\node at (-0.33,-0.43) {$\scriptstyle{j}$};
			\node at (0.33,-0.43) {$\scriptstyle{i}$};
			\node at (-0.33,0.43) {$\scriptstyle{i}$};
			\node at (0.33,0.43) {$\scriptstyle{j}$};
			\node at (0,-.01) {$\diamond$};
		\end{tikzpicture}
		:F_{j} E_i \Rightarrow E_i F_{j}&
		,&
		\begin{tikzpicture}[baseline = -1mm]
			\draw[<-,blue] (0.28,-.28) to (-0.28,.28);
			\draw[->,red] (-0.28,-.28) to (0.28,.28);
			\node at (-0.33,-0.43) {$\scriptstyle{i}$};
			\node at (0.33,-0.43) {$\scriptstyle{j}$};
			\node at (-0.33,0.43) {$\scriptstyle{j}$};
			\node at (0.33,0.43) {$\scriptstyle{i}$};
			\node at (0,-.01) {$\diamond$};
		\end{tikzpicture}
		: E_i F_{j} \Rightarrow F_{j} E_i&
	\end{align*}
	are mutually inverse
	isomorphisms. In fact, we have
	\begin{align*}
		\mathord{
			\begin{tikzpicture}[baseline = -1mm]
				\node at (-.23,.75){$\scriptstyle i$};
				\node at (.23,.75){$\scriptstyle j$};
				\node at (-.23,-.75){$\scriptstyle i$};
				\node at (.23,-.75){$\scriptstyle j$};
				\node at (-.35,0){$\scriptstyle j$};
				\node at (.35,0){$\scriptstyle i$};
				\draw[-,red] (0.23,0) to[out=90,in=-90] (-0.23,.6);
				\draw[->,blue] (-0.23,0) to[out=90,in=-90] (0.23,.6);
				\draw[-,blue] (0.23,-.6) to[out=90,in=-90] (-0.23,0);
				\draw[<-,red] (-0.23,-.6) to[out=90,in=-90] (0.23,0);
				\node at (0,.3){$\diamond$};
				\node at (0,-.3){$\diamond$};
			\end{tikzpicture}
		}
		&
		=
		\mathord{
			\begin{tikzpicture}[baseline = -1mm]
				\node at (-.23,.75){$\scriptstyle i$};
				\node at (.23,.75){$\scriptstyle j$};
				\node at (-.23,-.75){$\scriptstyle i$};
				\node at (.23,-.75){$\scriptstyle j$};
				\draw[-,red] (0.23,0) to[out=90,in=-90] (-0.23,.6);
				\draw[->,blue] (-0.23,0) to[out=90,in=-90] (0.23,.6);
				\draw[-,blue] (0.23,-.6) to[out=90,in=-90] (-0.23,0);
				\draw[<-,red] (-0.23,-.6) to[out=90,in=-90] (0.23,0);
				\draw[-] (-.2,.4) to (-.13,.45);
				\draw[-] (-.2,-.4) to (-.13,-.45);
				\draw[-] (.2,.4) to (.13,.45);
				\draw[-] (.2,-.4) to (.13,-.45);
			\end{tikzpicture}
		}
		=
		\mathord{
			\begin{tikzpicture}[baseline = -1mm]
				\node at (-.28,.75){$\scriptstyle i$};
				\node at (.08,-.75){$\scriptstyle i'$};
				\draw[->,blue] (0.08,-.6) to (0.08,.6);
				\draw[<-,red] (-0.28,-.6) to (-0.28,.6);
			\end{tikzpicture}
		},
		&
		\mathord{
			\begin{tikzpicture}[baseline = -1mm]
				\node at (-.23,.75){$\scriptstyle j$};
				\node at (.23,.75){$\scriptstyle i$};
				\node at (-.23,-.75){$\scriptstyle j$};
				\node at (.23,-.75){$\scriptstyle i$};
				\node at (-.35,0){$\scriptstyle i$};
				\node at (.35,0){$\scriptstyle j$};
				\draw[-,blue] (0.23,0) to[out=90,in=-90] (-0.23,.6);
				\draw[->,red] (-0.23,0) to[out=90,in=-90] (0.23,.6);
				\draw[-,red] (0.23,-.6) to[out=90,in=-90] (-0.23,0);
				\draw[<-,blue] (-0.23,-.6) to[out=90,in=-90] (0.23,0);
				\node at (0,.3){$\diamond$};
				\node at (0,-.3){$\diamond$};
			\end{tikzpicture}
		}
		&
		=
		\mathord{
			\begin{tikzpicture}[baseline = -1mm]
				\node at (-.23,.75){$\scriptstyle j$};
				\node at (.23,.75){$\scriptstyle i$};
				\node at (-.23,-.75){$\scriptstyle j$};
				\node at (.23,-.75){$\scriptstyle i$};
				\draw[-,blue] (0.23,0) to[out=90,in=-90] (-0.23,.6);
				\draw[->,red] (-0.23,0) to[out=90,in=-90] (0.23,.6);
				\draw[-,red] (0.23,-.6) to[out=90,in=-90] (-0.23,0);
				\draw[<-,blue] (-0.23,-.6) to[out=90,in=-90] (0.23,0);
				\draw[-] (-.2,.4) to (-.13,.45);
				\draw[-] (-.2,-.4) to (-.13,-.45);
				\draw[-] (.2,.4) to (.13,.45);
				\draw[-] (.2,-.4) to (.13,-.45);
			\end{tikzpicture}
		}
		=
		\mathord{
			\begin{tikzpicture}[baseline = -1mm]
				\node at (-.28,.75){$\scriptstyle j$};
				\node at (.08,-.75){$\scriptstyle i$};
				\draw[->,red] (0.08,-.6) to (0.08,.6);
				\draw[<-,blue] (-0.28,-.6) to (-0.28,.6);
			\end{tikzpicture}
		}.
	\end{align*}
 This completes the proof.
\end{proof}

\section{Big categorification on representation category of finite classical groups: quantum case}
We apply Section~\ref{sec:HeistoKac} to finite classical groups and study the
associated colored weight functions.  \label{sec:bigquantum}
\subsection{Basic definitions}
\subsubsection{Finite classical groups and Witt towers}
Throughout the paper, $G$ denotes one of the finite classical groups
\[
 \GL_n(q),\qquad \GU_n(q),\qquad \Sp_{2n}(q),\qquad
 \O_{2n+1}(q),\qquad \O_{2n}^{\pm}(q).
\]
Here $\GL_n(q)$ is the group of invertible linear transformations of an
$n$-dimensional vector space over $\bbF_q$, and $\Sp_{2n}(q)$ is the
isometry group of a non-degenerate alternating form on a
$2n$-dimensional vector space over $\bbF_q$.  The groups
$\O_{2n+1}(q)$ and $\O_{2n}^{\pm}(q)$ are the full isometry groups of
non-degenerate quadratic spaces.  In even dimension the signs $+$ and
$-$ distinguish the split and non-split Witt types, while in odd
dimension there is a single type.

For unitary groups we use a field-size convention.  Write $q=q_0^2$ and
let $\GU_n(q)$ denote the isometry group of a non-degenerate Hermitian
form on $\bbF_q^n$ with respect to the involution
$a\mapsto a^{q_0}$.  Thus our $\GU_n(q)$ is the group usually denoted
$\GU_n(q_0)$ when the parameter records the size of the fixed field.
With this convention, $q$ is the size of the field on which the natural
module is defined, and the general-linear factors in the standard Levi
subgroups are uniformly written as $\GL_d(q)$.

We work with full orthogonal groups throughout.  In characteristic two,
orthogonal groups are defined as isometry groups of quadratic forms.  In
particular,
\[
 \O_{2n+1}(q)\cong\Sp_{2n}(q)\qquad(q\text{ even}),
\]
and we use this identification for the odd orthogonal tower; the two
even orthogonal towers are treated directly as quadratic spaces.  All
parabolic constructions below are therefore valid in every
characteristic.

Let $\bbH$ denote a hyperbolic plane.  For each unitary, symplectic, or
orthogonal Witt tower, fix its anisotropic kernel $V_a$ and write the
member of Witt rank $n$ as
\[
 V_n=V_a\perp\bbH^{\oplus n}
    =V_n^+\oplus V_a\oplus V_n^-,
\]
where $\bbH^{\oplus n}$ denotes the orthogonal sum of $n$ hyperbolic
planes.
Here $a=0$ or $1$ for the two unitary towers, $a=0$ for the symplectic
and split even orthogonal towers, $a=1$ for the odd orthogonal tower,
and $a=2$ for the non-split even orthogonal tower.  Choose isotropic
vectors
\[
 v_1,\ldots,v_n,v_{-1},\ldots,v_{-n}
\]
with $\langle v_i,v_{-i}\rangle=1$, and put
\[
 V_r^+=\langle v_1,\ldots,v_r\rangle,
 \qquad
 V_r^-=\langle v_{-1},\ldots,v_{-r}\rangle,
 \qquad
 V_r=V_r^+\oplus V_a\oplus V_r^-.
\]
The stabilizer of the standard maximal isotropic flag has Levi factor
\[
 \GL_1(q)^n\times G(V_a).
\]
Its relative Weyl group in the full isometry group is the signed
permutation group
$W_n\cong(\bbZ/2\bbZ)^n\rtimes\frakS_n$; for $\GL_n(q)$ the Weyl group
is $\frakS_n$.  We use these Weyl groups only through the explicit
parabolic double-coset representatives of \S\ref{prop:TypeA}.

	\subsubsection{Levi subgroups}
For a composition $n=m_1+\cdots+m_t$, let
\[
L_{m_1,\ldots,m_t}
 =\GL_{m_1}(q)\times\cdots\times\GL_{m_t}(q)
 \subseteq\GL_n(q)
\]
be the standard block-diagonal Levi subgroup, let $P_{m_1,\ldots,m_t}$ be the
corresponding block upper-triangular parabolic, and let
$V_{m_1,\ldots,m_t}$ be its unipotent radical.

For a non-linear Witt tower, write $V_j=V_a\oplus\bbH^{\oplus j}$.  If
$n=r+m_1+\cdots+m_t$, the standard parabolic of $G(V_n)$ stabilizing the
successive isotropic blocks of dimensions $m_t,\ldots,m_1$ has Levi factor
\[
L_{r,m_1,\ldots,m_t}
 \cong G(V_r)\times\GL_{m_1}(q)\times\cdots\times\GL_{m_t}(q).
\]
With respect to the fixed Witt basis its Levi embedding is
\[
(B,A_1,\ldots,A_t)\longmapsto
\diag(A_t,\ldots,A_1,B,A_1',\ldots,A_t'),
\]
where $A_i'=J_{m_i}A_i^{-\sfH}J_{m_i}$ in the unitary case and
$A_i'=J_{m_i}A_i^{-\sfT}J_{m_i}$ in the symplectic and orthogonal cases.

We write
\[
 P_{r,d}=L_{r,d}\ltimes V_{r,d},\qquad
 \prescript{t}{}{V}_{r,d}=\operatorname{Rad}_u(P_{r,d}^{-}),
\]
where $P_{r,d}^{-}$ is the opposite Witt parabolic.  The opposition elements,
block swaps and signed-permutation representatives below are the standard ones
obtained from the fixed Witt basis.  These conventions apply uniformly to the
unitary, symplectic and full orthogonal Witt towers, with orthogonal spaces
understood through their quadratic forms.

\subsubsection{Semisimple Elements and Conjugacy Classes}
	A semisimple element of a linear algebraic group is an element diagonalizable over $\Fqb$. For finite classical groups, conjugacy classes of semisimple elements are parameterized by multiplicities of irreducible polynomials, subject to group-specific involution constraints. We use the standard primary-polynomial parametrization; for the unitary, symplectic, and orthogonal groups see \cite{Wall63}, and in odd defining characteristic we use the conventions of \cite{FS89}.

	Let $\Irr(\bbF_q[x])$ denote the set of all monic irreducible polynomials in $\bbF_q[x]$.  Set
\[
\calF := \Irr(\bbF_q[x]) \setminus \{x\},
\]	i.e., the set of monic irreducible polynomials over $\bbF_q$ with non-zero constant term. For $\Delta \in \calF$, let $d_\Delta = \deg(\Delta)$ be the degree of $\Delta$.

	\begin{enumerate}
		\item \textbf{$\GL_n(q)$}:
		A semisimple conjugacy class in $\GL_n(q)$ is uniquely determined by a function
		$m:\calF\to\bbN$ satisfying
		\[
		\sum_{\Delta\in\calF}m_\Delta d_\Delta=n,
		\]
		where $m_\Delta$ is the multiplicity of the irreducible polynomial $\Delta$ in the
		characteristic polynomial of any element of the class.

		\item \textbf{$\GU_n(q)$}:
		Let $q = q_0^2$. Define an involution on $\calF$ by $\Delta \mapsto \overline{\Delta}$, where $\overline{\Delta}$ is the unique monic polynomial whose roots are the $(-q_0)$-th powers of the roots of $\Delta$. Let
		\[
		\calF^{\GU}_1 = \{ \Delta \in \calF \mid \overline{\Delta} = \Delta \}, \quad
		\calF^{\GU}_2 = \{ \Delta \overline{\Delta} \mid \Delta \in \calF, \overline{\Delta} \neq \Delta \}, \quad
		\calF^{\GU} = \calF^{\GU}_1 \cup \calF^{\GU}_2.
		\]
		Elements $\Gamma\in\calF^{\GU}$ are the unitary primary polynomials.  Following the standard convention, put $d_\Gamma:=\deg\Gamma$.
		Semisimple conjugacy classes in $\GU_n(q)$ are uniquely parameterized by
		functions $m:\calF^{\GU}\to\bbN$ satisfying
		\[
		\sum_{\Gamma\in\calF^{\GU}}m_\Gamma d_\Gamma=n.
		\]

		\item {$\Sp_{2n}(q)$ and ${\rm O}^{\pm}_m(q)$}:
		Define an involution on $\calF$ by $\Delta \mapsto \widetilde{\Delta}$, where $\widetilde{\Delta}$ is the unique monic polynomial whose roots are the inverses of the roots of $\Delta$.
        Let $\rm G=\Sp$ or $\rm O.$
        If $2\nmid q$, we let
		\[
		\calF^{\rm G}_0 = \{ X-1, X+1 \}, \quad
		\calF^{\rm G}_1 = \{ \Delta \in \calF \setminus \calF^{\rm G}_0 \mid \widetilde{\Delta} = \Delta \}, \quad
		\calF^{\rm G}_2 = \{ \Delta \widetilde{\Delta} \mid \Delta \in \calF, \widetilde{\Delta} \neq \Delta \}.
		\]
        If $2\mid q$, let
		\[
		\calF^{\rm G}_0 = \{ X-1\}, \quad
		\calF^{\rm G}_1 = \{ \Delta \in \calF \setminus \calF^{\rm G}_0 \mid \widetilde{\Delta} = \Delta \}, \quad
		\calF^{\rm G}_2 = \{ \Delta \widetilde{\Delta} \mid \Delta \in \calF, \widetilde{\Delta} \neq \Delta \}.
		\]
        In both cases, we write $\calF^{\rm G}=\calF^{\rm G}_0\cup\calF^{\rm G}_1\cup\calF^{\rm G}_2$.

        We use the standard notation recalled in \cite[\S4.1.5]{LSZ25}.
        For $\Gamma\in\calF^{\rm G}$ let
        \[
          d_\Gamma:=\deg\Gamma
        \]
        be its degree and let $\delta_\Gamma$ be its reduced degree:
        \[
          \delta_\Gamma:=
          \begin{cases}
            d_\Gamma,&\Gamma\in\calF^{\rm G}_0,\\
            \frac12d_\Gamma,&\Gamma\in\calF^{\rm G}_1\cup\calF^{\rm G}_2.
          \end{cases}
        \]
        For $\Gamma\in\calF^{\rm G}_1\cup\calF^{\rm G}_2$ also put
        \[
          \epsilon_\Gamma:=
          \begin{cases}
            -1,&\Gamma\in\calF^{\rm G}_1,\\
            1,&\Gamma\in\calF^G_2.
          \end{cases}
        \]
        Thus the dimension condition is
        \[
          \sum_{\Gamma\in\calF^{\rm G}}m_\Gamma d_\Gamma=\dim V,
        \]
        and for $\Gamma\in\calF^{\rm G}_1\cup\calF^{\rm G}_2$ the corresponding type-$A$ centralizer factor is
        \[
          \GL_{m_\Gamma}(\epsilon_\Gamma q^{\delta_\Gamma}),
        \]
        with the standard formal convention
        \(\GL_m(-Q)=\mathrm U_m(Q)\), where \(\mathrm U_m(Q)\) denotes the
        unitary group over \(\bbF_{Q^2}\) with fixed field \(\bbF_Q\).
        Here \(Q\) is the centralizer parameter, distinct from the field-size
        parameter used for the ambient \(\GU\)-tower.  In particular,
        $\calF^{\rm G}_1$ gives a unitary factor and $\calF^{\rm G}_2$ a general-linear factor.

        For $\Sp_{2n}(q)$ the right-hand side is $2n$, and the multiplicities of both $X-1$ and $X+1$ are even, since the corresponding generalized eigenspaces are non-degenerate symplectic spaces.
			For orthogonal groups, the multiplicity function is supplemented by the isometry types (equivalently, the relevant discriminant data) of the $\pm1$-primary non-degenerate spaces. We retain this type data as part of the semisimple conjugacy parameter \cite{Wall63,FS89}; the multiplicity function alone does not in general determine a conjugacy class in a fixed full orthogonal group.

		\end{enumerate}

\subsection{Big Categorification}
\subsubsection{Cuspidal modules for $\GL_n(q)$}\label{subsec:cus}

Fix a prime $\ell$ with $\ell\nmid q(q-1)$.  Put
\(\mathbb K=\overline{\mathbb Q}_\ell\), let \(\mathcal O\) be its
valuation ring, and let \(\mathbb F=\overline{\mathbb F}_\ell\) be the
residue field.  We write \(\k\in\{\mathbb K,\mathcal O,\mathbb F\}\).
The ring $\mathcal O$ is the filtered union of the valuation rings of finite
extensions of $\mathbb Q_\ell$.  Every integral construction below involves
only finitely many coefficients: it is first made over a sufficiently large
finite splitting discrete valuation ring and then extended to $\mathcal O$.
In particular, $q$, $q-1$, and $2$ are units, and the required normalizing
roots may be fixed compatibly before scalar extension.  The same convention
applies whenever a block-theoretic statement is used.

The cuspidal representations of general linear groups serve as the fundamental building blocks from which all other representations are constructed. The complex characters of cuspidal representations were constructed by Green~\cite{Green55}.
The construction of cuspidal representations over a field $\mathbb K$ was established by Gelfand~\cite{Gel70}, based on Green's characters. Let $\bfG =\mathrm{GL}_n(\overline\bbF_q)  $ and $F:\bfG\to \bfG$ be the standard Frobenius map such that $G=\bfG^F=\GL_n(q).$
Deligne--Lusztig theory \cite{DL76} shows that every cuspidal character of $G$ is of the form $\pm R_{\mathbf{T}_c}^{\mathbf{G}}(\theta)$, where $F$ is the standard Frobenius map, $R_{\mathbf{T}}^{\mathbf{G}}$ is Deligne--Lusztig induction, $\mathbf{T}_c$ is a Coxeter torus of $G$, and $\theta$ is a regular character of $\mathbf{T}_c^F$.
James~\cite{Jam86} constructed an $\ell$-integral form of these representations.

For each \(\Delta \in \calF\), let \(d_\Delta\) denote the degree associated with \(\Delta\). The irreducible cuspidal modules for the family of group algebras \(\{\bbK \mathrm{GL}_d(q)\}_{d \in \mathbb{N}}\) are parametrized in the standard way by \(\calF\).
For each \(\Delta \in \calF\), let \(\M_\Delta=\M_\Delta^\bbK\) be the corresponding irreducible cuspidal \(\bbK \mathrm{GL}_{d_\Delta}(q)\)-module.
Let \(\M^\calO_\Delta\) denote the \(\calO\)-lattice constructed by James \cite{Jam86}, and let \(\M^\bbF_\Delta\) denote the reduction of \(\M^\calO_\Delta\) modulo \(\ell\). We omit these superscripts when no confusion arises.
We write \(\omega_\Delta\) for the central character of \(\M_\Delta\), so that
\[
 \rho_\Delta(aI_{d_\Delta})=\omega_\Delta(aI_{d_\Delta})\,\id_{\M_\Delta}
 \qquad(a\in\bbF_q^\times).
\]
We use the same notation for its integral and modular realizations.
Moreover, we have
\[
\dim(\M_\Delta)= (q-1)(q^2-1)\cdots (q^{d_\Delta-1}-1).
\]
The reduction need not be irreducible for arbitrary $\Delta$; the defect-zero colors used below are covered by the following lemma.
\begin{lemma}\label{Lem:ellnmidq}
  If $\ell\nmid q^{d_\Delta}-1$, then $\M^\bbF_\Delta$ is a simple projective $\bbF\GL_{d_\Delta}(q)$-module; equivalently, it lies in a defect-zero block.  The integral lattice $\M^\calO_\Delta$ belongs to the corresponding block idempotent over $\calO$.
  \end{lemma}
  \begin{proof}
Let $\chi_\Delta$ be the ordinary cuspidal character afforded by
$\M_\Delta^\bbK$.  Its degree is
\[
 \chi_\Delta(1)=\prod_{i=1}^{d_\Delta-1}(q^i-1),
\]
whereas
\[
 |\GL_{d_\Delta}(q)|
 =q^{d_\Delta(d_\Delta-1)/2}
   \prod_{i=1}^{d_\Delta}(q^i-1).
\]
Thus
\[
 \frac{|\GL_{d_\Delta}(q)|}{\chi_\Delta(1)}
 =q^{d_\Delta(d_\Delta-1)/2}(q^{d_\Delta}-1).
\]
Because $\ell\nmid q$, the ordinary character $\chi_\Delta$ has
$\ell$-defect zero exactly when $\ell\nmid q^{d_\Delta}-1$.  A defect-zero block has a unique simple modular module, it is projective, and the decomposition number of the defect-zero ordinary character is one.  Hence the reduction of the chosen integral lattice has exactly this one composition factor and is therefore the unique simple projective module in the block.  The block idempotent lifts to $\calO$, and $\M_\Delta^\calO$ lies in that block.
  \end{proof}

Let $e$ be the multiplicative order of $q$ in \(\bbF_\ell^\times\). By the definition of multiplicative order, $q^{d_\Delta}\equiv1\pmod\ell$ holds exactly when $e\mid d_\Delta$; hence $\ell\nmid q^{d_\Delta}-1$ is equivalent to $e\nmid d_\Delta$.
We define \(\calF(\ell) \subseteq \calF\) to be the subset consisting of all \(\Delta \in \calF\) with $e\nmid d_\Delta$. In particular, $|\calF(\ell)|=\infty$ if $e>1.$
By Lemma \ref{Lem:ellnmidq}, we know that for any $\Delta\in \calF(\ell)$, $\M^\bbF_{\Delta}$ is a projective-simple module of $\bbF\GL_{d_\Delta}(q).$
For $\k\in\{\bbK,\calO,\bbF\},$
we will write
$$\calF(\k)=\begin{cases}
    \calF&\text{if $\k=\bbK;$}\\
    \calF(\ell)&\text{if $\k=\calO,\bbF.$}\\
\end{cases}$$
\subsubsection{Functors}
Fix a tower and index it by Witt rank rather than ambient dimension:
\[
G_r=\begin{cases}
\GL_r(q),&\GL,\\
\GU_{2r+a}(q),&\GU\text{ of parity }a\in\{0,1\},\\
\Sp_{2r}(q),&\Sp,\\
\O_{2r+1}(q),&\O_{\mathrm{odd}},\\
\O^+_{2r}(q),&\O^+_{\mathrm{even}},\\
\O^-_{2r+2}(q),&\O^-_{\mathrm{even}}.
\end{cases}
\]
In particular the base anisotropic group is \(G_0\), and adding a
\(\GL_d(q)\)-factor increases the tower index from \(r\) to \(r+d\).
This convention is used in every functor and Levi formula below.

We define the category of representations
\[
\k G_\bullet\text{-mod} := \bigoplus_{n \geq 0} \k G_n\text{-mod},
\]
where \(\k G_n\text{-mod}\) denotes the category of finitely generated left \(\k G_n\)-modules.

For \(d \in \bbN_{>0}\), we always identify \(G_r\) as a subgroup of \(G_{r+d}\) via the standard embedding induced by the Witt tower. Let \(P_{r,d}\) be the standard parabolic subgroup of \(G_{r+d}\) stabilizing the added totally isotropic \(d\)-space, with Levi complement \(L_{r,d} \cong G_r \times \GL_d(q)\). Then we have the semidirect product decomposition \(P_{r,d} = L_{r,d} \ltimes V_{r,d}\).

Following the ideas of \cite{CR08,DVV17,DVV19,LLZ}, we define the following functors:
\[
\begin{array}{rcrclcl}
F^\Delta_{r+d_\Delta,r} &=& \k G_{r+d_\Delta} \cdot e_{V_{r,d_\Delta}} \ot_{\GL_{d_\Delta}} \M_\Delta \ot_{\k G_r} - &:& \k G_r\text{-mod} &\to& \k G_{r+d_\Delta}\text{-mod}, \\
E^\Delta_{r,r+d_\Delta} &=& \M_\Delta^* \ot_{\GL_{d_\Delta}} e_{V_{r,d_\Delta}}\cdot \k G_{r+d_\Delta} \ot_{\k G_{r+d_\Delta}} - &:& \k G_{r+d_\Delta}\text{-mod} &\to& \k G_r\text{-mod}.
\end{array}
\]
We remind the reader that $d_\Delta$ denotes the degree of the polynomial $\Delta.$

For \(X \in \k G_r\text{-mod}\) and \(Y \in \k G_{r+d_\Delta}\text{-mod}\), we have
\begin{align*}
F^\Delta_{r+d_\Delta,r}(X) &\cong R_{L_{r,d_\Delta} \leq P_{r,d_\Delta}}^{G_{r+d_\Delta}} (X \boxtimes \M_\Delta), \\
E^\Delta_{r,r+d_\Delta}(Y) &\cong \Hom_{\GL_{d_\Delta}(q)} \left( \M_\Delta, {^*R}_{L_{r,d_\Delta} \leq P_{r,d_\Delta}}^{G_{r+d_\Delta}} (Y) \right)
\end{align*}
where \(R_{L_{r,d_\Delta} \leq P_{r,d_\Delta}}^{G_{r+d_\Delta}}\) (resp. \({^*R}_{L_{r,d_\Delta} \leq P_{r,d_\Delta}}^{G_{r+d_\Delta}}\)) denotes the Harish--Chandra induction functor (resp. restriction functor).

We abbreviate
\(F^\Delta:=\bigoplus_{r\ge0}F^\Delta_{r+d_\Delta,r}\) and
\(E^\Delta:=\bigoplus_{r\ge0}E^\Delta_{r,r+d_\Delta}\); thus
\(\deg F^\Delta=d_\Delta\) and \(\deg E^\Delta=-d_\Delta\).
The dual polynomial is denoted by \(\Delta^*\): it is \(\Delta\) for
\(G=\GL\), \(\overline\Delta\) for \(G=\GU\), and \(\widetilde\Delta\)
for \(G=\Sp,\O\).

\begin{lemma}\label{Lem:isomfunctors}
If \(\Delta\in\mathcal F(\k)\) and \(\Delta^*\ne\Delta\), then
\(F^\Delta\cong F^{\Delta^*}\) and \(E^\Delta\cong E^{\Delta^*}\).
\end{lemma}
The proof appears at the end of \S\ref{sec:decomposition}.  Fix one
representative of each primary \(*\)-orbit and use it in the categorical
notation below.

For \(\dagger\in\{\GU,\rm O,\Sp\}\), set
\(\mathcal F_1^\dagger(\k)=\mathcal F_1^\dagger\cap\mathcal F(\k)\) and
\(\mathcal F_2^\dagger(\k)=\{\Delta\Delta^*\mid
\Delta\in\mathcal F(\k),\ \Delta^*\ne\Delta\}\).  For
\(\dagger\in\{\rm O,\Sp\}\) also set
\(\mathcal F_0^\dagger(\k)=\mathcal F_0^\dagger\cap\mathcal F(\k)\).
For notational uniformity, when the fixed tower is unitary we write
\(\mathcal F_i^G(\k):=\mathcal F_i^{\GU}(\k)\) for \(i=1,2\); for
\(G=\Sp,\O\) we write \(\mathcal F_i^G(\k):=\mathcal F_i^{\rm G}(\k)\).
Thus
\[
 \mathcal F^G(\k)=
 \begin{cases}
  \mathcal F_1^G(\k)\cup\mathcal F_2^G(\k),&G=\GU,\\
  \mathcal F_0^G(\k)\cup\mathcal F_1^G(\k)\cup\mathcal F_2^G(\k),&G=\Sp,\O.
 \end{cases}
\]
Lemma~\ref{Lem:isomfunctors} therefore reduces the study of the creation
functors to the indices \(\Gamma\in\mathcal F^G(\k)\) occurring in the primary decomposition.

We keep the two kinds of indices separate.  The symbol \(\Delta\) always denotes an irreducible polynomial in \(\mathcal F(\k)\), hence the cuspidal label for a general-linear factor.  It determines an index \(\Gamma\in\mathcal F^G(\k)\) by
\[
 \Gamma=\Delta\quad\text{if }\Delta=\Delta^*,
 \qquad
 \Gamma=\Delta\Delta^*\quad\text{if }\Delta\ne\Delta^*.
\]
Conversely, a paired index \(\Gamma\in\mathcal F^G_2(\k)\) has the two irreducible representatives \(\Delta\) and \(\Delta^*\), and we fix one of them only when an explicit cuspidal module is needed.  Thus the Harish--Chandra rank increment is \(d_\Gamma\) for a self-dual index and \(d_\Gamma/2\) for a paired index.  For \(G=\Sp,\O\), the latter equals \(\delta_\Gamma\), whereas for \(\Gamma\in\calF_1^{\rm G}\) one has \(d_\Gamma=2\delta_\Gamma\).

We write the relative Hecke parameter directly as a power of $q$: it is \(q^{d_\Gamma}\) for a self-dual index and \(q^{d_\Gamma/2}\) for a paired index.  Accordingly, define
\[
z^\Gamma:=
\begin{cases}
q^{d_\Gamma/2}-q^{-d_\Gamma/2},&\Gamma\text{ is self-dual},\\[2pt]
q^{d_\Gamma/4}-q^{-d_\Gamma/4},&\Gamma\text{ is paired}.
\end{cases}
\]
and
\begin{align}\label{def:t}
t^\Gamma:=
\begin{cases}
\sqrt{-1}
 & \text{if $\Gamma\in\calI_{-1}(\k)$};\\
\sqrt{-q^{-d_\Gamma/2}}
 & \text{if $G=\GU$ and $\Gamma\in\calF^{\GU}_1(\k)$, or if $G=\Sp,\rm O$ and $\Gamma\in\calF^{\rm G}_1(\k)$};\\
\sqrt{-q^{-1}}
 & \text{if $\Gamma=X-1$ and $G=\O_{\mathrm{odd}},\Sp$};\\
\sqrt{-1}
 & \text{if $\Gamma=X-1$ and $G=\O_{\mathrm{even}}$};\\
\sqrt{-q^{-1}}
 & \text{if $\Gamma=X+1$, $G=\O_{\mathrm{odd}}$ and $2\nmid q$};\\
\sqrt{-1}
 & \text{if $\Gamma=X+1$, $G=\Sp,\O_{\mathrm{even}}$ and $2\nmid q$}.
\end{cases}
\end{align}
When a relation is written using an irreducible cuspidal label \(\Delta\), the symbols \(z^\Delta\) and \(t^\Delta\) mean the parameters attached to the corresponding index \(\Gamma\).  In \eqref{eq:ipm} the two adjacent eigenvalues are
\[
i^{+,\Gamma}=\begin{cases}q^{d_\Gamma}i,&\Gamma\text{ self-dual},\\ q^{d_\Gamma/2}i,&\Gamma\text{ paired},\end{cases}
\qquad
i^{-,\Gamma}=\begin{cases}q^{-d_\Gamma}i,&\Gamma\text{ self-dual},\\ q^{-d_\Gamma/2}i,&\Gamma\text{ paired}.\end{cases}
\]
In our convention, the primary polynomial $\Gamma=X+1$ for $G=\Sp$ or $\O$ occurs only when $q$ is odd.
\subsubsection{Adjunctions}
\label{sec:adj}
For $\Delta\in \calF(\k)$, since $(E^{\Delta}, F^{\Delta})$ is a bi-adjoint pair, there exist four adjunction maps for each pair. Note that these adjunction maps are unique up to a nonzero scalar.
Whenever a formula below contains $G_{r-d_\Delta}$, it is understood for $r\ge d_\Delta$; for $r<d_\Delta$ the corresponding component involving $E^\Delta$ is zero.
Let $\M^*_\Delta:=\Hom_\k(\M_\Delta,\k)\in \text{mod-}\k\GL_{d_\Delta}$ be the dual module of $\M_\Delta$ with \emph{right} action of $\GL_{d_\Delta}.$
Let $\{v_i\}^{\dim(\M_\Delta)}_{i=1}$ be a basis of $\M_\Delta$ and
$\{f^*_i\}^{\dim(\M_\Delta)}_{i=1}$ its dual basis in $\M^*_\Delta$.
The four adjunction maps are represented by the following bimodule maps.

\begin{itemize}
\item[$(a)$]
For a group element $a\in G_{r+d_\Delta}$, we define
$$
\Proj_{L_{r,d_\Delta}}(e_{V_{r,d_\Delta}} \cdot a\cdot e_{V_{r,d_\Delta}})=
\begin{cases}
\overline{a} & \text{if } a\in P_{r,d_\Delta},\\
0 & \text{if } a\notin P_{r,d_\Delta},
\end{cases}
$$
where $\overline{a}$ is the image of $a$ under the quotient map $P_{r,d_\Delta}\twoheadrightarrow L_{r,d_\Delta}\cong G_r\times \GL_{d_\Delta}.$
Extending by $\k$-linearity, this defines a $(\k G_{r},\k G_{r})$-bimodule map
$$
\varepsilon^\Delta_L: \M^*_\Delta\ot_{\GL_{d_\Delta}}e_{V_{r,d_\Delta}}\k G_{r+d_\Delta} e_{V_{r,d_\Delta}}\ot_{\GL_{d_\Delta}}\M_\Delta\to \k G_r,
$$
given by
$$
n^*\otimes e_{V_{r,d_\Delta}}a e_{V_{r,d_\Delta}}\otimes m \mapsto \delta_{a \in P_{r,d_\Delta}} n^*(c m)\cdot b,
$$
where, if $a\in P_{r,d_\Delta}$, its Levi projection is uniquely written as
$\overline a=b\times c$ with $b\in G_r$ and $c\in\GL_{d_\Delta}$;
if $a\notin P_{r,d_\Delta}$, the displayed image is understood to be zero.

\item[$(b)$]
Let $G_{r}=\coprod\limits_{i=1}^{s}P_{r-d_\Delta,d_\Delta}\cdot g_i$ be a decomposition of $G_{r}$ into left $P_{r-d_\Delta,d_\Delta}$-cosets.
We have the following element in $(\k G_r{e_{V_{r-d_\Delta,d_\Delta}}}\ot_{\GL_{d_\Delta}}
\M_\Delta)\ot_{G_{r-d_\Delta}}
({\M^*_\Delta} \ot_{\GL_{d_\Delta}}  e_{V_{r-d_\Delta,d_\Delta}}\k G_{r}):$
$$
Z^\Delta_r=\sum\limits_{i=1}^s\left(g_i^{-1} e_{V_{r-d_\Delta,d_\Delta}}\otimes \left(\sum_{j=1}^{\dim(\M_\Delta)} v_j\otimes f^*_j\right)\otimes e_{V_{r-d_\Delta,d_\Delta}}g_i\right).
$$
This element satisfies $aZ^\Delta_r=Z^\Delta_ra$ for any $a\in G_{r}$ and does not depend on the choice of representatives $\{g_i\}_{i=1}^{s}$.
\smallskip

Then we define the $(\k G_r , \k G_r )$-bimodule map $\eta^\Delta_L$ as follows:
$$
\eta^\Delta_L: \k G_{r}\to (\k G_{r}\cdot e_{V_{r-d_\Delta,d_\Delta}}\ot_{\GL_{d_\Delta}}
\M_\Delta)\ot_{G_{r-d_\Delta}}(\M^*_\Delta\ot_{\GL_{d_\Delta}} e_{V_{r-d_\Delta,d_\Delta}} \k G_{r}),
$$
$$
a\mapsto a\cdot Z^\Delta_r=Z^\Delta_r \cdot a
$$
for any $a\in  \k G_r .$

\item[$(c)$]
Let $B_\Delta$ denote the block component containing $\M_\Delta$, and let
$f_\Delta$ be its central block idempotent.  For $\k=\mathbb K$ or
$\mathbb F$, this is a full matrix algebra with unique simple module
$\M_\Delta$; for $\k=\mathcal O$, the defect-zero block is the corresponding
full matrix algebra over $\mathcal O$, with $\M_\Delta^{\mathcal O}$ as its
standard Morita generator.  This follows over a finite splitting modular
subsystem from Lemma~\ref{Lem:ellnmidq} and then by scalar extension.  We fix the corresponding Morita identification
\[
\Phi_\Delta:\M_\Delta\otimes_\k\M_\Delta^*
   \xrightarrow{\sim}\End_\k(\M_\Delta)
   \xrightarrow{\sim}B_\Delta,
\qquad
m\otimes n^*\longmapsto\bigl(w\mapsto n^*(w)m\bigr).
\]
With the matrix-algebra identification normalized by the action on
$\M_\Delta$, the identity endomorphism corresponds to $f_\Delta$; hence
\[
\Phi_\Delta\left(\sum_{j=1}^{\dim \M_\Delta}v_j\otimes f_j^*\right)=f_\Delta.
\]
(Over $\bbK$ this is the usual primitive-central-idempotent formula; over
$\calO$ and $\bbF$ we use the integral block idempotent and its reduction,
rather than dividing by $|\GL_{d_\Delta}(q)|$.)

Define the $(\k G_r,\k G_r)$-bimodule map
\[
\varepsilon^\Delta_R:
(\k G_r e_{V_{r-d_\Delta,d_\Delta}}\otimes_{\GL_{d_\Delta}}\M_\Delta)
\otimes_{G_{r-d_\Delta}}
(\M_\Delta^*\otimes_{\GL_{d_\Delta}}
 e_{V_{r-d_\Delta,d_\Delta}}\k G_r)
\longrightarrow \k G_r
\]
by
\[
(ae_V\otimes m)\otimes(n^*\otimes e_Vb)
\longmapsto
 -t^\Delta z^\Delta\,ae_V\Phi_\Delta(m\otimes n^*)e_Vb.
\]

\item[$(d)$]
The balanced evaluation map
\[
\operatorname{ev}_\Delta:
\M_\Delta^*\otimes_{\k\GL_{d_\Delta}(q)}\M_\Delta\longrightarrow\k,
\qquad f^*\otimes v\longmapsto f^*(v),
\]
is an isomorphism: the source is free of rank one by projectivity and
Schur's lemma, and the displayed map is nonzero.  In particular,
$f_i^*\otimes v_i=f_j^*\otimes v_j$ in the balanced tensor product for all
$i,j$, because both evaluate to $1$.

We therefore define the $(\k G_r,\k G_r)$-bimodule map
\[
\eta^\Delta_R:\k G_r\longrightarrow
\M_\Delta^*\otimes_{\GL_{d_\Delta}}
 e_{V_{r,d_\Delta}}\k G_{r+d_\Delta}e_{V_{r,d_\Delta}}
 \otimes_{\GL_{d_\Delta}}\M_\Delta
\]
by
\[
a\longmapsto-(t^\Delta z^\Delta)^{-1}
 f_1^*\otimes e_{V_{r,d_\Delta}}a e_{V_{r,d_\Delta}}\otimes v_1.
\]
The formula is independent of the chosen dual-basis index by the preceding
balanced-tensor identity.
\end{itemize}

For $\Gamma\in\mathcal F^G(\k)$, let $\Delta\in\mathcal F(\k)$ be an irreducible representative of $\Gamma$ when an explicit cuspidal label is required.  All categorical maps are henceforth those attached to this fixed factor; in particular
\(
\varepsilon_L^\Gamma,\varepsilon_R^\Gamma,
\eta_L^\Gamma,\eta_R^\Gamma,t^\Gamma,z^\Gamma
\)
mean the corresponding objects for $\Delta$.

\subsubsection{Natural transformations $\mathrm{X}^\Delta$ and $\mathrm{T}^{\Delta,\Delta'}$ for general linear group}

For non-negative integers $r, d, d'$ and $n = r + d + d'$, define the element
\begin{equation*}
    \dot{s}_{r, d, d'} \in \GL_n(q)
\end{equation*}
by its action on the standard basis vectors:
\begin{align*}
    \dot{s}_{r, d, d'}(e_i) =
    \begin{cases}
        e_i & \text{if } 1 \le i \le r, \\[4pt]
        e_{i+d'} & \text{if } r+1 \le i \le r+d, \\[4pt]
        e_{i-d} & \text{if } r+d+1 \le i \le r+d+d'.
    \end{cases}
\end{align*}

We define the natural transformation
\[
\mathrm{X}^{\Delta}_{r+d_\Delta,r} \in \End(F^{\Delta}_{r+d_\Delta,r})
\]
as the $(\GL_{r+d_\Delta}, \GL_r)$-bimodule map:
\[
\begin{aligned}
    \k \GL_{r+d_\Delta}\cdot  e_{V_{r,d_\Delta}} \ot_{\GL_{d_\Delta}} \M_\Delta &\to \k \GL_{r+d_\Delta} \cdot e_{V_{r,d_\Delta}} \ot_{\GL_{d_\Delta}} \M_\Delta, \\
    g e_{V_{r,d_\Delta}} \otimes m &\mapsto q^{rd_\Delta} g e_{V_{r,d_\Delta}} e_{\prescript{t}{}{V}_{r,d_\Delta}} e_{V_{r,d_\Delta}} \otimes m,
\end{aligned}
\]
where $g \in \GL_{r+d_\Delta}(q)$, $m \in \M_\Delta$, and $\prescript{t}{}{V}_{r,d_\Delta}$ denotes the transpose of the subgroup $V_{r,d_\Delta}$.

For any $\Delta, \Delta' \in \mathcal{F}(\k)$, the composite functor $F^{\Delta'} \circ F^\Delta$ is represented by the $(\GL_{r+d_\Delta+d_{\Delta'}}, \GL_r)$-bimodule
\[
\k \GL_{r+d_\Delta+d_{\Delta'}} \cdot e_{V_{r,d_\Delta,d_{\Delta'}}} \ot_{\GL_{d_\Delta} \times \GL_{d_{\Delta'}}} (\M_\Delta \boxtimes \M_{\Delta'}).
\]

Define the natural transformation
\[
\mathrm{T}_{r+d_\Delta+d_{\Delta'},r}^{\Delta,\Delta'}: F^\Delta \circ F^{\Delta'} \to F^{\Delta'} \circ F^\Delta
\]
as the $(\GL_{r+d_\Delta+d_{\Delta'}}(q), \GL_r(q))$-bimodule map:
\[
\begin{aligned}
    \k \GL_{r+d_\Delta+d_{\Delta'}}\cdot e_{V_{r,d_{\Delta'},d_\Delta}} \ot_{\GL_{d_{\Delta'}} \times \GL_{d_\Delta}} (\M_{\Delta'} \boxtimes \M_\Delta) &\to \k \GL_{r+d_\Delta+d_{\Delta'}}\cdot e_{V_{r,d_\Delta,d_{\Delta'}}} \ot_{\GL_{d_\Delta} \times \GL_{d_{\Delta'}}} (\M_\Delta \boxtimes \M_{\Delta'}), \\
    y e_{V_{r,d_{\Delta'},d_\Delta}} \ot (m' \boxtimes m) \mapsto & y \widehat{\mathrm{T}}^{\Delta,\Delta'}_{r,d_\Delta,d_{\Delta'}} \ot (m \boxtimes m'),
\end{aligned}
\]
where $y \in \GL_{r+d_\Delta+d_{\Delta'}}$, $m \in \M_\Delta$, $m'\in \M_{\Delta'}$ and
\[
\widehat{\mathrm{T}}^{\Delta,\Delta'}_{r,d_\Delta,d_{\Delta'}}:=\begin{cases}
    \omega_\Delta(-\id_{d_\Delta}) q^{\frac{d_\Delta ^2}{2}}  e_{V_{r,d_{\Delta},d_\Delta}} \dot{s}_{r,d_\Delta,d_{\Delta'}} e_{V_{r,d_{\Delta},d_\Delta}} & \text{if $\Delta=\Delta'$}\\
    q^{\frac{d_\Delta d_{\Delta'}}{2}}  e_{V_{r,d_{\Delta'},d_\Delta}} \dot{s}_{r,d_\Delta,d_{\Delta'}} e_{V_{r,d_\Delta,d_{\Delta'}}} & \text{if $\Delta\neq \Delta'$}\\

\end{cases}
.\]

We set
\[
\mathrm{X}^\Delta = \bigoplus_{r \in \mathbb{N}} \mathrm{X}^\Delta_{r+d_\Delta,r}, \qquad
\mathrm{T}^{\Delta,\Delta'} = \bigoplus_{r \in \mathbb{N}} \mathrm{T}^{\Delta,\Delta'}_{r+d_\Delta+d_{\Delta'},r}.
\]
Thus $\mathrm{X}^\Delta \in \End(F^\Delta)$ and $\mathrm{T}^{\Delta,\Delta'} \in \Hom(F^\Delta \circ F^{\Delta'}, F^{\Delta'} \circ F^\Delta)$.
The source of $\mathrm T^{\Delta,\Delta'}$ has inducing blocks
$(\Delta',\Delta)$, since functors compose from right to left.
Put $d=d_\Delta$ and $d'=d_{\Delta'}$.  For the displayed block swap,
$\diag(I_r,h',h)\dot s_{r,d,d'}=\dot s_{r,d,d'}\diag(I_r,h,h')$.
Thus right multiplication by $\dot s_{r,d,d'}$, followed by the coefficient
flip, is balanced from this source to the stated target.  With source
$(\Delta,\Delta')$ the right multiplier would instead be
$\dot s_{r,d,d'}^{-1}=\dot s_{r,d',d}$; these are equal only when $d=d'$.

\begin{remark}
For $\Delta=X-1$, this construction is exactly the original construction of
Chuang and Rouquier \cite{CR08}.
\end{remark}

\subsubsection{Natural transformations $\mathrm{X}^\Delta$ and $\mathrm{T}^{\Delta,\Delta'}$ for other classical groups}

An endomorphism of $F^\Delta_{r+d_\Delta, r}$ can be represented by a $(\k G_{r+d_\Delta}, \k G_r)$-bimodule endomorphism of
\[
\k G_{r+d_\Delta}\cdot e_{V_{r,d_\Delta}} \otimes_{\GL_{d_\Delta}} \M_\Delta.
\]

The composite functor $F^{\Delta'}_{r+d_\Delta+d_{\Delta'},r+d_\Delta} F^\Delta_{r+d_\Delta, r}$ is represented by the $(\k G_{r+d_\Delta+d_{\Delta'}}, \k G_r)$-bimodule
\[
\k G_{r+d_\Delta+d_{\Delta'}} e_{V_{r,d_\Delta,d_{\Delta'}}} \otimes_{\GL_{d_\Delta} \times \GL_{d_{\Delta'}}} (\M_\Delta \otimes \M_{\Delta'}).
\]

Let $n=r+d$.  We choose the opposition element $\dot x_{r,d}\in G_n$
to fix $V_r$ pointwise and, for $r<j\le n$, to act on the added
hyperbolic pairs by
\[
\begin{array}{c|cc}
G & \dot x_{r,d}(v_j) & \dot x_{r,d}(v_{-j})\\ \hline
\GU & -v_{-n-r-1+j} & -v_{n+r+1-j}\\
\Sp & -v_{-n-r-1+j} & \phantom{-}v_{n+r+1-j}\\
\O  & -v_{-n-r-1+j} & -v_{n+r+1-j}.
\end{array}
\]
These are compatible choices for the Hermitian, alternating and quadratic
hyperbolic forms, respectively.  In particular, the orthogonal formula is
valid in every characteristic; when $q$ is even its two minus signs are of
course invisible.

Let $n = r + d + d'$. We similarly define $\dot{s}_{r,d,d'} \in G_n$ by:
\[
\dot{s}_{r,d,d'}:
\begin{cases}
v \mapsto v, & \text{for all } v \in V_r, \\
v_j \mapsto v_{j+d'}, & \text{if } r < j \leq r+d, \\
v_{-j} \mapsto v_{-j-d'}, & \text{if } r < j \leq r+d, \\
v_j \mapsto v_{j-d}, & \text{if } r+d < j \leq n, \\
v_{-j} \mapsto v_{-j+d}, & \text{if } r+d < j \leq n.
\end{cases}
\]

Both $\dot{x}_{r,d_\Delta} \in G_{r+d_\Delta}$ and $\dot{s}_{r,d_\Delta,d_{\Delta'}} \in G_{r+d_\Delta+d_{\Delta'}}$ centralize $G_r$: by definition they are the identity on the nondegenerate summand on which $G_r$ acts, and they act only on the added hyperbolic blocks, while the standard embedding of $G_r$ is the identity on those blocks.  For $n = r + d_\Delta$, in the ordered basis
\[
\mathcal{B} = \{v_{-n}, v_{-(n-1)}, \dots, v_{-1}, \mathcal{B}_a, v_1, v_2, \dots, v_n\},
\]
where $\mathcal{B}_a$ is a basis for $V_a$, the matrix of $\dot{x}_{r,d_\Delta}$ on the added two blocks is,
respectively,
\[
\begin{pmatrix} & & -\id_{d_\Delta} \\ & \id_{G_r} & \\ -\id_{d_\Delta} & & \end{pmatrix},
\quad
\begin{pmatrix} & & -\id_{d_\Delta} \\ & \id_{G_r} & \\ \id_{d_\Delta} & & \end{pmatrix},
\quad
\begin{pmatrix} & & -\id_{d_\Delta} \\ & \id_{G_r} & \\ -\id_{d_\Delta} & & \end{pmatrix}
\]
for the unitary, symplectic and orthogonal towers, respectively.  Thus its square on the general-linear block is
$+I$, $-I$, $+I$, respectively, exactly as recorded by $\kappa_d$ below.

For $\Delta \in \mathcal{F}(\k)$, we define a natural transformation for each functor $F^\Delta$, distinguishing two cases:

\textbf{Case 1: $\Delta^* \neq \Delta$.}
In this case, the construction is analogous to the $\GL_n(q)$ case. We define the natural transformation
\[
\mathrm{X}^{\Delta}_{r+d_\Delta,r} \in \End(F^{\Delta}_{r+d_\Delta,r})
\]
as the $(G_{r+d_\Delta}, G_r)$-bimodule map:
\[
\begin{aligned}
\k G_{r+d_\Delta}\cdot  e_{V_{r,d_\Delta}} \ot_{\GL_{d_\Delta}(q)} \M_\Delta &\to \k G_{r+d_\Delta} \cdot e_{V_{r,d_\Delta}} \ot_{\GL_{d_\Delta}(q)} \M_\Delta, \\
g e_{V_{r,d_\Delta}} \ot m &\mapsto   |V_{r,d_\Delta}|\cdot  ge_{V_{r,d_\Delta}} e_{\prescript{t}{}{V}_{r,d_\Delta}} e_{V_{r,d_\Delta}}\ot m,
\end{aligned}
\]
where
$\prescript{t}{}{V}_{r,d_\Delta}$ denotes the transpose of the subgroup $V_{r,d_\Delta}$.

For a paired color $\Gamma=\Delta\Delta^*$, where $\Delta$ is an irreducible representative of $\Gamma$, the categorical dot is
$\mathrm{X}^\Gamma:=\mathrm{X}^\Delta$.  The raw
operator attached to the opposite representative is compared with it by the extra
isomorphism constructed in \S\ref{sec:decomposition}.
\begin{remark}
    In Case 1, the construction is similar to the construction of $\GL_n(q).$
\end{remark}
\textbf{Case 2: $\Delta^*=\Delta$.}
The opposition element does not always define a split extension of the Levi.
Put
\[
\widehat H_{r,d}:=\langle L_{r,d},\dot x_{r,d}\rangle\subseteq G_{r+d}
\qquad(d=d_\Delta).
\]
Then $L_{r,d}\lhd \widehat H_{r,d}$ and
$\widehat H_{r,d}/L_{r,d}\cong C_2$, but in the symplectic case the chosen
lift $\dot x_{r,d}$ has nontrivial square in the Levi.  More precisely, with
$L_{r,d}\cong G_r\times\GL_d(q)$ as above,
\begin{equation*}
 \dot x_{r,d}^{\,2}=(1,\kappa_d),\qquad
 \kappa_d=
 \begin{cases}
 -I_d,&G=\Sp,\\
 I_d,&G=\GU\text{ or }G=\rm O.
 \end{cases}
\end{equation*}
For the fixed color $\Delta$ we abbreviate $\kappa_\Delta:=\kappa_{d_\Delta}$.
Let $\dagger=\mathsf H$ in the unitary case and $\dagger=\mathsf T$ in the
symplectic and orthogonal cases.  Conjugation by $\dot x_{r,d}$ induces on
the general-linear factor the involution
\begin{equation*}
 \theta_d(g)=J_dg^{-\dagger}J_d^{-1}.
\end{equation*}
Accordingly, define the crossed extension
\begin{equation}\label{eq:twisted-extension}
 \widehat H_d^{\dagger}(\kappa_d)
 :=\left\langle \GL_d(q),\sigma_d\ \middle|\
 \sigma_dg\sigma_d^{-1}=\theta_d(g),\quad
 \sigma_d^2=\kappa_d\right\rangle .
\end{equation}
There is a surjective homomorphism
\begin{equation}\label{eq:twisted-extension-projection}
 \pi_{r,d}:\widehat H_{r,d}\twoheadrightarrow
 \widehat H_d^{\dagger}(\kappa_d),\qquad
 (a,b)\longmapsto b,\quad \dot x_{r,d}\longmapsto\sigma_d,
\end{equation}
whose kernel contains the factor $G_r$.

First work over $\mathbb K$.  Since $\Delta$ is self-dual,
$\M_\Delta^{\mathbb K}$ is stable under $\theta_d$, and Schur's lemma makes
the corresponding intertwining space one-dimensional.  After rescaling an
intertwiner, we may impose the square relation below; the two resulting
choices differ by sign.  The resulting representation of the finite crossed
extension $\widehat H_d^{\dagger}(\kappa_d)$ admits an
$\widehat H_d^{\dagger}(\kappa_d)$-stable $\mathcal O$-lattice (take the
$\mathcal O$-span of the orbit of any lattice).  We henceforth choose
$\M_\Delta^{\mathcal O}$ to be such a lattice; this does not change the
ordinary representation, and in the defect-zero colors its reduction is the
same unique simple projective module.  The stable lattice supplies compatible
integral and modular intertwiners.

\begin{definition}\label{def:cuspidal-extension}
For every self-dual color $\Delta$, fix once and for all the compatible
intertwiners $A_\Delta$ over $\mathbb K$, $\mathcal O$, and $\mathbb F$
satisfying
\begin{equation*}
 A_\Delta\rho_\Delta(g)A_\Delta^{-1}
 =\rho_\Delta(\theta_d(g)),\qquad
 A_\Delta^2=\rho_\Delta(\kappa_d)
 =\omega_\Delta(\kappa_d)\id_{\M_\Delta}.
\end{equation*}
We write $\widetilde\rho_\Delta$ for the resulting extension of
$\rho_\Delta$ to $\widehat H_d^{\dagger}(\kappa_d)$, so that
$\widetilde\rho_\Delta(\sigma_d)=A_\Delta$.  Pulling back along
\eqref{eq:twisted-extension-projection}, we use the same notation for the
representation of $\widehat H_{r,d}$.  The other extension is obtained by
$A_\Delta\mapsto-A_\Delta$.
\end{definition}

We define
\[
\mathrm{X}^{\Delta}_{r+d,r}\in\End(F^{\Delta}_{r+d,r})
\]
by the bimodule map
\[
\begin{aligned}
\k G_{r+d}e_{V_{r,d}}\ot_{\GL_d(q)}\M_\Delta
&\longrightarrow
\k G_{r+d}e_{V_{r,d}}\ot_{\GL_d(q)}\M_\Delta,\\
ge_{V_{r,d}}\otimes m
&\longmapsto
 g\,\widehat{\mathrm{X}}^\Delta_{r,d}\otimes A_\Delta m,
\end{aligned}
\]
where
\begin{equation*}
 \widehat{\mathrm{X}}^\Delta_{r,d}
 :=\beta_\Delta q^{rd}e_{V_{r,d}}\dot x_{r,d}e_{V_{r,d}}.
\end{equation*}
Here $\beta_\Delta\in\k^\times$ is independent of $r$.  Put
$\omega=\omega_\Delta(\kappa_d)$.  The group-algebra factor contributes
$\omega$ to the identity double-coset component of the square, and the
coefficient factor contributes $A_\Delta^2=\omega\id$.
Since $\kappa_d^2=1$, we have $\omega^2=1$.  We therefore normalize by
\begin{equation}\label{eq:beta-normalization}
 \frac{\beta_\Delta^2q^{2rd}}{|V_{r,d}|}=-(t^\Delta)^2.
\end{equation}
The quantity $|V_{r,d}|q^{-2rd}$ is independent of $r$, so one scalar
$\beta_\Delta$ works throughout the Witt tower.  Equation
\eqref{eq:beta-normalization} fixes $\beta_\Delta$ up to sign.  We choose
the pair $(A_\Delta,\beta_\Delta)$ once and for all; the resulting
normalization is recorded in Proposition~\ref{prop:all-bottom-bubbles}.

\begin{remark}
\label{lem:extension-sign-independence}
Changing to the other cuspidal extension sends $A_\Delta$ to $-A_\Delta$;
changing simultaneously $\beta_\Delta$ to $-\beta_\Delta$ leaves
$\beta_\Delta A_\Delta$, hence $\mathrm{X}^\Delta$ and $\bbO^\Delta(u)$, unchanged.
\end{remark}

\begin{remark}
When $\kappa_d=I_d$, the crossed extension is split and may be identified
with the usual outer extension after conjugating the outer generator by
$J_d$.  When $G=\Sp$, however, $\kappa_d=-I_d$ and the lift has square
$-I_d$.  Both this square and the coefficient square $A_\Delta^2$ must
be retained when computing the dot.  Their two central-character factors
cancel in \eqref{eq:beta-normalization}; retaining only one gives an
incorrect normalization.
\end{remark}

Similarly, for any $\Delta, \Delta' \in \mathcal{F}(\k)$, we define the natural transformation
\[
\mathrm{T}_{r+d_\Delta+d_{\Delta'},r}^{\Delta,\Delta'}: F^\Delta F^{\Delta'} \to F^{\Delta'} F^\Delta
\]
as the $(G_{r+d_\Delta+d_{\Delta'}}, G_r)$-bimodule map:
\[
\begin{aligned}
\k G_{r+d_\Delta+d_{\Delta'}} \cdot e_{V_{r,d_{\Delta'},d_\Delta}} \otimes_{\GL_{d_{\Delta'}} \times \GL_{d_\Delta}} (\M_{\Delta'} \boxtimes \M_\Delta) &\to \k G_{r+d_\Delta+d_{\Delta'}} \cdot e_{V_{r,d_\Delta,d_{\Delta'}}} \otimes_{\GL_{d_\Delta} \times \GL_{d_{\Delta'}}} (\M_\Delta \boxtimes \M_{\Delta'}), \\
y e_{V_{r,d_{\Delta'},d_\Delta}} \otimes (m' \boxtimes m) \mapsto &y \widehat{\mathrm{T}}^{\Delta,\Delta'}_{r,d_\Delta,d_{\Delta'}} \otimes (m \boxtimes m'),
\end{aligned}
\]
where $y \in G_{r+d_\Delta+d_{\Delta'}}$, $m \in \M_\Delta$, $m'\in \M_{\Delta'}$ and
\[
\widehat{\mathrm{T}}^{\Delta,\Delta'}_{r,d_\Delta,d_{\Delta'}}:=\begin{cases}
    \omega_\Delta(-\id_{d_\Delta}) q^{\frac{d_\Delta ^2}{2}}  e_{V_{r,d_{\Delta},d_\Delta}} \dot{s}_{r,d_\Delta,d_{\Delta'}} e_{V_{r,d_{\Delta},d_\Delta}} & \text{if $\Delta=\Delta'$}\\
    q^{\frac{d_\Delta d_{\Delta'}}{2}}  e_{V_{r,d_{\Delta'},d_\Delta}} \dot{s}_{r,d_\Delta,d_{\Delta'}} e_{V_{r,d_\Delta,d_{\Delta'}}} & \text{if $\Delta\neq \Delta'$}\\

\end{cases}
.\]

We set
\[
\mathrm{X}^\Delta = \bigoplus_{r \in \mathbb{N}} \mathrm{X}^\Delta_{r+d_\Delta,r}, \qquad
\mathrm{T}^{\Delta,\Delta'} = \bigoplus_{r \in \mathbb{N}} \mathrm{T}^{\Delta,\Delta'}_{r+d_\Delta+d_{\Delta'},r}.
\]

The same conjugation calculation for the block transposition holds on both halves of the Witt basis, so the preceding right-multiplier convention and balancing verification apply here as well.

At the level of the categorical action, the superscript $\Gamma$ records the corresponding index in the primary decomposition.  For fixed $\Gamma$, let $\Delta$ be an irreducible representative when the explicit cuspidal label is needed, and set
\[
F^\Gamma:=F^\Delta,\qquad
E^\Gamma:=E^\Delta,\qquad
\mathrm{X}^\Gamma:=\mathrm{X}^\Delta,
\]
and similarly for $\mathrm{T}$, the four adjunction maps, and all their mates.  In a
paired color the extra isomorphism
$Y^{\Delta,\Delta^*}$ identifies this choice with the functor constructed from the other factor.
Thus $\Gamma$ denotes the primary polynomial/color.  The symbol
$\Delta$ is retained only inside the explicit cuspidal bimodule and Howlett--Lehrer
computations.  Whenever a relation below is first verified with superscripts
$\Delta,\Delta'$ on chosen representatives, the corresponding categorical relation
with superscripts $\Gamma,\Gamma'$ is its transport through these fixed identifications.

\begin{remark}
For $\Delta=X-1$, this construction is exactly the original construction of
Dudas--Varagnolo--Vasserot \cite{DVV17,DVV19}.
\end{remark}

\subsubsection{A big quantum Heisenberg action}
Let $\k \in \{\bbK,\calO,\bbF\}.$
Recall that $\ell=\mathrm{char}(\bbF)$ with $\ell\nmid q(q-1)$ and  $\mathrm{char}(\bbK)=0.$
Denote \begin{align*}
   \calI_{-1}(\k):=\begin{cases}
 \calF(\k) \,\,&\text{if}\,\,G=\GL;\\ \calF^{\GU}_2(\k) \,\,&\text{if}\,\,G=\GU;\\\
 \calF^{G}_2(\k)\,\,&\text{if} \,\,G=\Sp, {\rm O},
\end{cases}\end{align*}
and
 \begin{align*}
   \calI_{-2}(\k):=\begin{cases}
 \emptyset \,\,&\text{if}\,\,G=\GL;\\ \calF^{\GU}_1(\k) \,\,&\text{if}\,\,G=\GU;\\
 \calF^{G}_0(\k)\sqcup \calF^{G}_1(\k)\,\,&\text{if} \,\,G=\Sp, {\rm O}.
\end{cases}\end{align*}
Denote  $$\calI(\k):= \calI_{-1}(\k)\sqcup \calI_{-2}(\k),$$
and put $k_\Gamma=-1$ or $-2$ according as
$\Gamma\in\calI_{-1}(\k)$ or $\calI_{-2}(\k)$.  Write
\[
 \symHeis_{\calI(\k)}:=
 \bigodot_{\Gamma\in\calI(\k)}
 \Heis^{(\Gamma)}_{k_\Gamma}(z^\Gamma,t^\Gamma)
\]
for the corresponding symmetric product of colored quantum Heisenberg
categories.

\begin{theorem}\label{Thm:quantumheisenberg}
There exists a strict $\k$-linear monoidal functor
\[
\Psi:\symHeis_{\calI(\k)}
\longrightarrow
\mathcal{E}nd_\k\!\left(\k G_{\bullet}(q)\mod\right)
\]
which
sends the generating objects $\mathord{
\begin{tikzpicture}[baseline = -1mm]
 	\draw[<-,red] (0.68,.28) to (0.68,-.22);
     \node at (0.68,-.37) {$\scriptstyle{\red{\Gamma}}$};
\end{tikzpicture}
}$ and $\mathord{
\begin{tikzpicture}[baseline = -1mm]
 	\draw[->,red] (0.68,.28) to (0.68,-.22);
     \node at (0.68,-.37) {$\scriptstyle{\red\Gamma}$};
\end{tikzpicture}
}$ to endo-functors $F^{\Gamma}$ and $E^{\Gamma}$ of $\k G_{\bullet}(q)\mod $ for any $\Gamma\in \calI(\k)$, respectively. Moreover,
$\Psi$ sends the generating morphisms  to the corresponding natural transformations
in  $ \k G_{\bullet}\mod $ as follows:
\begin{itemize}
\item
$\Psi(\begin{tikzpicture}[baseline = -1mm]
	\draw[->,red] (0.08,-.2) to (0.08,.2);
      \node[red] at (0.08,0) {$\dott$};
      \node at (0.08,-.37) {$\scriptstyle{\red{\Gamma}}$};
\end{tikzpicture})=\mathrm{X}^{\Gamma}\::\;F^{\Gamma} \Rightarrow F^{\Gamma}\,,\quad\Psi(
\begin{tikzpicture}[baseline = -1mm]
	\draw[->,red] (0.2,-.2) to (-0.2,.2);
    \draw[-,white,line width=4pt] (-0.2,-.2) to (0.2,.2);
    \draw[->,red] (-0.2,-.2) to (0.2,.2);
 \node at (0.2,-.37) {$\scriptstyle{\red{\Gamma}}$};
  \node at (-0.2,-.37) {$\scriptstyle{\red{\Gamma}}$};
\end{tikzpicture})=\mathrm{T}^{\Gamma,\Gamma}\::\; (F^{\Gamma})^2 \Rightarrow (F^{\Gamma})^2\,;$
\item
$\Psi(\begin{tikzpicture}[baseline = .75mm]
	\draw[<-,red] (0.3,0) to[out=90, in=0] (0.1,0.3);
	\draw[-,red] (0.1,0.3) to[out = 180, in = 90] (-0.1,0);
  \node at (-0.1,-.2) {$\scriptstyle{\red{\Gamma}}$};
\end{tikzpicture})=\varepsilon^{\Gamma}_R\::\;F^{\Gamma}\otimes E^{\Gamma} \Rightarrow \unit
\,,\quad\Psi(\begin{tikzpicture}[baseline = .75mm]
	\draw[<-,red] (0.3,0.3) to[out=-90, in=0] (0.1,0);
	\draw[-,red] (0.1,0) to[out = 180, in = -90] (-0.1,0.3);
 \node at (-0.1,0.42) {$\scriptstyle{\red{\Gamma}}$};
\end{tikzpicture})=\eta^{\Gamma}_R\::\;\unit\Rightarrow E^{\Gamma} \circ F^{\Gamma}\,;$
\item
$\Psi(\begin{tikzpicture}[baseline = .75mm]
	\draw[-,red](0.1,0.3) to[out=0, in=90](0.3,0);
	\draw[<-,red] (-0.1,0)to[out =90 , in = 180](0.1,0.3);
\node at (0.3,-.2) {$\scriptstyle{\red{\Gamma}}$};
\end{tikzpicture})=\varepsilon^{\Gamma}_L\::\;E^{\Gamma}\circ F^{\Gamma} \Rightarrow \unit
\,,\quad\Psi(\begin{tikzpicture}[baseline = .75mm]
	\draw[-,red] (0.1,0)to[out=0, in=-90](0.3,0.3);
	\draw[<-,red] (-0.1,0.3)to[out =-90, in =180] (0.1,0);
\node at (0.3,0.42) {$\scriptstyle{\red{\Gamma}}$};
\end{tikzpicture})=\eta^{\Gamma}_L\::\;\unit\Rightarrow F^{\Gamma}\circ E^{\Gamma}\,;$

\item
$\Psi(\begin{tikzpicture}[baseline = -1mm]
	\draw[->,blue] (0.2,-.2) to (-0.2,.2);
	\draw[->,red] (-0.2,-.2) to (0.2,.2);
\node at (0.2,-.37) {$\blue{\scriptstyle{\Gamma'}}$};
  \node at (-0.2,-.37) {$\red{\scriptstyle{\Gamma}}$};
\end{tikzpicture})=\mathrm{T}^{\Gamma,\Gamma'}\::\;F^{\Gamma}\otimes F^{\Gamma'} \Rightarrow F^{\Gamma'}\circ F^{\Gamma}\,\, \text{for}\,\, \Gamma\neq\Gamma'.$
\end{itemize}
\end{theorem}
A detailed proof of Theorem~\ref{Thm:quantumheisenberg} will be given in the next section.

\subsection{Proof of the Main Theorem}

\subsubsection{Lemmas on Double Cosets of Weyl Groups}\label{prop:TypeA}
First, we need the following theorem, which is \cite[Theorem 1.3.10]{JK81}.
\begin{theorem}\label{Lem:james}
Let $n=r+d=r'+d'$, and let $\mathfrak{S}_n$ denote the symmetric group on $n$ letters. Define the standard parabolic subgroups
\[
Q_{r,d} = \mathfrak{S}_{[1,r]} \times \mathfrak{S}_{[r+1,n]}, \qquad
Q_{r',d'} = \mathfrak{S}_{[1,r']} \times \mathfrak{S}_{[r'+1,n]}.
\]
Then the following statements hold:
\begin{enumerate}
    \item The double coset space
    \[
    Q_{r', d'} \setminus \mathfrak{S}_n / Q_{r,d}
    \]
    is in bijection with the set of integers
    \[
    a \in \bigl\{ \max(0,r+r'-n),\dots,\min(r,r') \bigr\}.
    \]
    The bijection is encoded by the $2\times 2$ intersection matrix
    \[
    M(a) =
    \begin{pmatrix}
        a & r-a \\
        r'-a & n-r-r'+a
    \end{pmatrix},
    \]
    whose entries count the sizes of the indicated intersections:
    $|\sigma(A)\cap B|=a$,
    $|\sigma(A)\cap B^\complement|=r-a$,
    $|\sigma(A^\complement)\cap B|=r'-a$,
    and $|\sigma(A^\complement)\cap B^\complement|=n-r-r'+a$,
    where $A=\{1,\dots,r\}$ and $B=\{1,\dots,r'\}$.

    \item Each double coset contains a unique shortest-length representative $\sigma_a$ with length $l(\sigma_a) = (r-a)(r'-a)$, given explicitly by:
    \[
    \sigma_a(i) =
    \begin{cases}
        i, & 1\leq i\leq a, \\
        r' + i - a, & a+1\leq i\leq r, \\
        a + i - r, & r+1\leq i\leq r+r'-a, \\
        i, & r+r'-a+1\leq i\leq n.
    \end{cases}
    \]

    \item The intersection of the parabolic subgroups satisfies:
    \[
    Q_{r',d'}^{\sigma_a} \cap Q_{r,d} \cong \mathfrak{S}_{[1,a]} \times \mathfrak{S}_{[a+1,r]} \times \mathfrak{S}_{[r+1,r+r'-a]} \times \mathfrak{S}_{[r+r'-a+1,n]}.
    \]
\end{enumerate}
\end{theorem}
\begin{proof}
This is \cite[Theorem~1.3.10]{JK81}.
\end{proof}

Let $W_n\cong (\mathbb Z/2\mathbb Z)^n\rtimes\mathfrak S_n$ be the signed-permutation group.  For the symplectic and odd-orthogonal towers it is the usual relative Weyl group; for a full even orthogonal group it is the extended relative Weyl group realized by permutations and swaps of hyperbolic pairs. Thus
\[
w(e_i)=\varepsilon_i e_{\sigma(i)},\qquad
\varepsilon_i\in\{\pm1\},\quad \sigma\in\mathfrak S_n.
\]
Let $n=r+d=r'+d'$ and set
\[
K_{r,d}=W_{[1,r]}\times\mathfrak S_{[r+1,n]},
\qquad
K_{r',d'}=W_{[1,r']}\times\mathfrak S_{[r'+1,n]}.
\]
Here the $W$-factor acts by signed permutations, whereas the symmetric-group factor acts without changing signs.

\begin{theorem}\label{Prop:TypeBC}
The double cosets $K_{r',d'}\backslash W_n/K_{r,d}$ are parametrized by pairs $(a,t)$ satisfying
\[
\max(0,r+r'-n)\le a\le \min(r,r'),
\qquad
m:=r+r'-a\le t\le n.
\]
For such a pair, a convenient representative is
\[
w_{a,t}(i)=
\begin{cases}
 i, & 1\le i\le a,\\
 r'+i-a, & a+1\le i\le r,\\
 a+i-r, & r+1\le i\le m,\\
 -(t+m+1-i), & m+1\le i\le t,\\
 i, & t+1\le i\le n.
\end{cases}
\]
Moreover,
\[
K_{r,d}\cap w_{a,t}^{-1}K_{r',d'}w_{a,t}
=
W_{[1,a]}
\times\mathfrak S_{[a+1,r]}
\times\mathfrak S_{[r+1,m]}
\times\mathfrak S_{[m+1,t]}
\times\mathfrak S_{[t+1,n]}.
\]
\end{theorem}

\begin{proof}
Put
\[
A=[1,r],\qquad B=[r+1,n],\qquad
C=[1,r'],\qquad D=[r'+1,n].
\]
We first classify the double cosets.  For $w\in W_n$, the right coset
$wK_{r,d}$ is completely determined by the signed $d$-subset
\[
S_w:=w(B)\subset\{\pm1,\ldots,\pm n\},
\]
where no two elements have the same absolute value.  Indeed,
$W_{[1,r]}$ acts only on the complementary coordinates, while
$\mathfrak S_{[r+1,n]}$ merely permutes the elements of $B$; conversely,
if $S_w=S_{w'}$, then $w^{-1}w'\in K_{r,d}$.

Take orbits under the left action of
$K_{r',d'}=W_C\times\mathfrak S_D$.  The factor $W_C$ may arbitrarily
permute the coordinates with absolute value in $C$ and change their
signs, whereas $\mathfrak S_D$ may permute the coordinates in $D$ but
preserves their signs.  Hence the orbit of $S_w$ is determined by the two
integers
\[
b:=\#\{x\in S_w:|x|\le r'\},
\qquad
s:=\#\{x\in S_w:x<0,\ |x|>r'\}.
\]
These invariants are complete: every orbit contains the unique canonical
signed subset
\[
\{a+1,\ldots,r'\}
\sqcup
\{-(m+1),\ldots,-t\}
\sqcup
\{t+1,\ldots,n\},
\]
where
\[
a:=r'-b,\qquad
m:=r+r'-a,\qquad
t:=m+s.
\]
Since $S_w$ has $d$ elements while $C$ and $D$ have cardinalities
$r'$ and $d'$, respectively, one has
\[
\max(0,d-d')\le b\le\min(d,r').
\]
Substituting $a=r'-b$ gives
\[
\max(0,r+r'-n)\le a\le\min(r,r').
\]
Moreover,
\[
d-b=n-r-r'+a=n-m,
\]
so $0\le s\le d-b$ is equivalent to $m\le t\le n$.
The displayed element $w_{a,t}$ sends $B$ precisely to this canonical
signed subset.  Thus the pairs $(a,t)$ give a complete and nonredundant
set of double-coset parameters.

It remains to compute the intersection subgroup.  Decompose
\[
[1,n]=I_1\sqcup I_2\sqcup I_3\sqcup I_4\sqcup I_5
\]
with
\[
I_1=[1,a],\quad I_2=[a+1,r],\quad I_3=[r+1,m],
\quad I_4=[m+1,t],\quad I_5=[t+1,n].
\]
By construction, $w_{a,t}$ maps $I_1$ and $I_3$ into $C$, and maps
$I_2,I_4,I_5$ into $D$; the only negative signs occur on $I_4$.
Therefore an element of $K_{r,d}$ lying in
$w_{a,t}^{-1}K_{r',d'}w_{a,t}$ must preserve the five intervals above.
Indeed, the two parabolics already force preservation of the common
refinement of the partitions $A\sqcup B$ and
$w_{a,t}^{-1}(C)\sqcup w_{a,t}^{-1}(D)$, while a permutation mixing
$I_4$ with $I_5$ would, after conjugation by $w_{a,t}$, introduce a sign
change in the unsigned $D$-block.

Finally, signs are allowed exactly on $I_1$: the right parabolic permits
sign changes on $A=I_1\sqcup I_2$, whereas the conjugated left parabolic
permits them on $w_{a,t}^{-1}(C)=I_1\sqcup I_3$.  Their intersection
therefore permits arbitrary signed permutations on $I_1$ and only
ordinary permutations on the remaining four blocks.  Hence
\[
K_{r,d}\cap w_{a,t}^{-1}K_{r',d'}w_{a,t}
=
W_{I_1}\times\mathfrak S_{I_2}\times\mathfrak S_{I_3}
\times\mathfrak S_{I_4}\times\mathfrak S_{I_5},
\]
which is the asserted formula.
\end{proof}

\subsubsection{Decompositions of bimodules}\label{sec:decomposition}

The Harish--Chandra Mackey formula gives a direct-sum decomposition of the
representing bimodule by parabolic double cosets.  After tensoring with the
cuspidal modules, its surviving summands can be identified explicitly; see \cite{DM}.

\begin{theorem}\label{thm:HC-Mackey}
Let \(P=LU\) and \(Q=MV\) be standard parabolic subgroups of \(G\), with
compatible Levi complements \(L\) and \(M\). Choose representatives \(\dot w\)
from the corresponding relative Weyl group, so that
\(L\cap{}^wM\) and \(M\cap{}^{w^{-1}}L\) are Levi complements in
\(L\cap{}^wQ\) and \(M\cap{}^{w^{-1}}P\), respectively.
Then for every \(X\in \k M\text{-mod}\) there is
a natural isomorphism
\[
{}^*\!R^G_{L\le P}\,R^G_{M\le Q}(X)
\cong
\bigoplus_{w\in P\backslash G/Q}
R^L_{L_w\le L\cap{}^wQ}
\circ \operatorname{ad}(w)\circ
{}^*\!R^M_{M_w\le M\cap{}^{w^{-1}}P}(X),
\]
where
\[
L_w:=L\cap{}^wM,\qquad M_w:=M\cap{}^{w^{-1}}L,
\]
and \(\operatorname{ad}(w):M_w\xrightarrow{\sim}L_w\) is conjugation by a
chosen compatible representative \(\dot w\in G\) of the double coset.
\end{theorem}

For the rest of this subsection put
\[
d=d_\Delta,\qquad d'=d_{\Delta'},\qquad
n=r+d=r'+d',
\]
and write
\[
H=\GL_d(q),\qquad H'=\GL_{d'}(q),
\]
\[
P=P_{r,d}=L_{r,d}U,\qquad
P'=P_{r',d'}=L_{r',d'}U',
\]
with
\[
U=V_{r,d},\qquad U'=V_{r',d'}.
\]
Thus
\[
L_{r,d}=G_r\times H,
\qquad
L_{r',d'}=G_{r'}\times H'.
\]

We use the shorthand
\[
\mathscr B_{\Delta',\Delta}(r',r)
:=
\M_{\Delta'}^*
\otimes_{\k H'}
e_{U'}\k G_ne_U
\otimes_{\k H}
\M_\Delta.
\]
This is a \((\k G_{r'},\k G_r)\)-bimodule representing the functor
\[
E^{\Delta'}_{r',n}F^\Delta_{n,r}.
\]

\begin{lemma}
\label{lem:finite-double-coset-bimodule}
Let \(\mathcal D\) be the set of compatible representatives chosen above for
\[
P'\backslash G_n/P.
\]
For \(w\in\mathcal D\), set
\[
\mathscr B_w
:=
\M_{\Delta'}^*
\otimes_{\k H'}
e_{U'}\,\k[P'\dot wP]\,e_U
\otimes_{\k H}
\M_\Delta,
\]
where \(\k[P'\dot wP]\) denotes the \(\k\)-span of the indicated double
coset. Then
\begin{equation}\label{eq:bimodule-double-coset-direct-sum}
\mathscr B_{\Delta',\Delta}(r',r)
=
\bigoplus_{w\in\mathcal D}\mathscr B_w
\end{equation}
as an internal direct sum of
\((\k G_{r'},\k G_r)\)-sub-bimodules.

Moreover, for every \(X\in\k G_r\text{-mod}\), the functor
\[
\mathscr B_w\otimes_{\k G_r}X
\]
is naturally isomorphic to the \(w\)-summand in the Harish--Chandra Mackey
decomposition of \(E^{\Delta'}F^\Delta(X)\).
\end{lemma}

\begin{proof}
The disjoint double-coset spans are stable under left multiplication by
\(L_{r',d'}\) and right multiplication by \(L_{r,d}\).  Since
\(e_{U'}\in\k P'\) and \(e_U\in\k P\) preserve each span,
\[
e_{U'}\k G_ne_U
=\bigoplus_{w\in\mathcal D}e_{U'}\k[P'\dot wP]e_U.
\]
Tensoring with \(\M_{\Delta'}^*\) over \(\k H'\) and with
\(\M_\Delta\) over \(\k H\) gives
\eqref{eq:bimodule-double-coset-direct-sum}; only distributivity over a
finite direct sum is used, not semisimplicity.

For a fixed \(w\), write \(p'=l'u'\in P'\) and \(p=lu\in P\).
The identities \(e_{U'}u'=e_{U'}\) and \(ue_U=e_U\), together with
Levi invariance of the averages, remove the unipotent factors in
\(e_{U'}p'\dot wp e_U\).  The remaining balanced tensors realize the
\(w\)-term of Theorem~\ref{thm:HC-Mackey}, applied to
\(X\boxtimes\M_\Delta\), followed by
\(\Hom_{H'}(\M_{\Delta'},-)\).  This identifies
\(\mathscr B_w\otimes_{\k G_r}X\) with the asserted Mackey summand,
naturally in \(X\).
\end{proof}

The following observation will be used repeatedly.

For \(r\ge d\), define the \((\k G_r,\k G_r)\)-bimodule
\begin{equation*}
\mathscr C_\Delta(r)
:=
\left(
\k G_r e_{V_{r-d,d}}
\otimes_{\k H}\M_\Delta
\right)
\otimes_{\k G_{r-d}}
\left(
\M_\Delta^*
\otimes_{\k H}
e_{V_{r-d,d}}\k G_r
\right).
\end{equation*}
It represents
\[
F^\Delta_{r,r-d}E^\Delta_{r-d,r}.
\]

More generally, if \(a_0:=r-d'=r'-d\ge0\), define
\begin{equation*}
\mathscr C_{\Delta,\Delta'}(r',r)
:=
\left(
\k G_{r'}e_{V_{a_0,d}}\otimes_{\k H}\M_\Delta
\right)
\otimes_{\k G_{a_0}}
\left(
\M_{\Delta'}^*\otimes_{\k H'}e_{V_{a_0,d'}}\k G_r
\right).
\end{equation*}
This represents \(F^\Delta_{r',a_0}E^{\Delta'}_{a_0,r}\).

\begin{lemma}
\label{lem:finite-bimodule-decomposition}
Choose a basis and its dual for each coefficient module.
In (1) and (2), write \(v_1\in\M_\Delta\) and \(f_1^*\in\M_\Delta^*\)
for the first dual pair.  In (3), \(v_1\in\M_\Delta\) and
\(f_1^*\in\M_{\Delta'}^*\) belong to the two different coefficient modules.
\begin{enumerate}
\item[(1)]
Suppose \(\Delta'=\Delta\) and either \(G=\GL\) or
\(\Delta\ne\Delta^*\). Then
\[
\mathscr B_{\Delta,\Delta}(r,r)
\cong \k G_r\oplus\mathscr C_\Delta(r),
\]
where the second summand is omitted when \(r<d\). Under the double-coset
 decomposition the two summands are generated respectively by
\[
\xi_1:=f_1^*\otimes e_{V_{r,d}}\otimes v_1,
\qquad
\xi_T:=f_1^*\otimes
\widehat{\mathrm{T}}^{\Delta,\Delta}_{r-d,d,d}\otimes v_1.
\]

\item[(2)]
Suppose \(\Delta'=\Delta=\Delta^*\) and \(G\ne\GL\). Then
\[
\mathscr B_{\Delta,\Delta}(r,r)
\cong \k G_r\oplus\k G_r\oplus\mathscr C_\Delta(r),
\]
where again the last summand is omitted if \(r<d\). The three summands are
 generated respectively by
\begin{align*}
\xi_1&:=f_1^*\otimes e_{V_{r,d}}\otimes v_1,\\
\xi_X&:=f_1^*\otimes\widehat{\mathrm{X}}^\Delta_{r,d}\otimes
\widetilde\rho_\Delta(\dot x_{r,d})v_1,\\
\xi_T&:=f_1^*\otimes\widehat{\mathrm{T}}^{\Delta,\Delta}_{r-d,d,d}\otimes v_1.
\end{align*}
The first two summands are each isomorphic to the regular
\((\k G_r,\k G_r)\)-bimodule.

\item[(3)]
Suppose \(\Delta'\ne\Delta\) and \(\Delta'\ne\Delta^*\), and put
\(a_0=r-d'=r'-d\). If \(a_0<0\), then
\(\mathscr B_{\Delta',\Delta}(r',r)=0\). If \(a_0\ge0\), then
\[
\mathscr B_{\Delta',\Delta}(r',r)
\cong\mathscr C_{\Delta,\Delta'}(r',r),
\]
and this bimodule is generated by
\[
\xi_T^{\Delta,\Delta'}
:=f_1^*\otimes\widehat{\mathrm{T}}^{\Delta',\Delta}_{a_0,d',d}\otimes v_1.
\]
\end{enumerate}
\end{lemma}

\begin{proof}
We apply cuspidality to the summands of
Lemma~\ref{lem:finite-double-coset-bimodule}.

\smallskip
\noindent\textbf{Step 1: the general linear case.}
For \(w=\sigma_a\) in Theorem~\ref{Lem:james}, put \(c=n-r-r'+a\).
The two cuspidal factors are cut into blocks \((r'-a,c)\) and
\((r-a,c)\).  If both entries of either pair are positive, a proper
Harish--Chandra restriction occurs and the summand vanishes.
If \(c>0\), this forces \(r=r'=a\), \(d=d'=c\), and \(w=1\).
The term is \(X\otimes_\k\Hom_H(\M_{\Delta'},\M_\Delta)\), hence is
\(X\) precisely when \(\Delta'=\Delta\), by Schur's lemma.
If \(c=0\), then \(a=r-d'=r'-d=:a_0\); the exchange term is
\(F^\Delta_{r',a_0}E^{\Delta'}_{a_0,r}(X)\), represented by
\(\mathscr C_{\Delta,\Delta'}(r',r)\), and is absent if \(a_0<0\).

\smallskip
\noindent\textbf{Step 2: the unitary, symplectic and orthogonal cases.}
For \(w=w_{a,t}\) in Theorem~\ref{Prop:TypeBC}, put \(m=r+r'-a\).
The two block decompositions are
\[
(b_1,b_2,b_3)=(r'-a,t-m,n-t),\qquad
(c_1,c_2,c_3)=(r-a,t-m,n-t),
\]
with sums \(d,d'\).  Cuspidality permits at most one nonzero entry in
each triple, leaving exactly the following cases.

\smallskip
\noindent\emph{(i) The identity double coset.}
If \(b_3=d\), \(c_3=d'\), then \(a=r=r'\), \(m=t=r\), and
\(w_{r,r}=1\).  As in Step~1, this contributes \(\k G_r\) exactly
when \(\Delta'=\Delta\).

\smallskip
\noindent\emph{(ii) The opposition/duality double coset.}
If \(b_2=d\), \(c_2=d'\), then \(a=r=r'\), \(t=n\), and
\(w_{r,n}=x_{r,d}\).  Conjugation by \(\dot x_{r,d}\) fixes \(G_r\)
and sends \(\M_\Delta\) to \(\M_{\Delta^*}\).  Thus the term is
\(X\otimes_\k\Hom_H(\M_{\Delta'},\M_{\Delta^*})\), and survives
exactly when \(\Delta'=\Delta^*\).  In the self-dual case, the chosen
extension in Definition~\ref{def:cuspidal-extension} spans this
one-dimensional intertwining space, giving a second copy of \(\k G_r\).

\smallskip
\noindent\emph{(iii) The block-exchange double coset.}
If \(b_1=d\), \(c_1=d'\), then \(t=m=n\), \(a=a_0=r-d'=r'-d\),
and \(w_{a_0,n}=s_{a_0,d',d}\).  This gives the same exchange bimodule
as in Step~1, again omitted for \(a_0<0\).  No other summand survives.

\smallskip
\noindent\textbf{Step 3: identification of the explicit generators.}
Using \(e_U\k Pe_U\cong\k G_r\otimes\k H\), contraction
\[
f^*\otimes e_U(g,h)e_U\otimes m\longmapsto f^*(hm)g
\]
sends \(\xi_1\) to \(1\), since \(f_1^*(v_1)=1\).
For the opposition summand, the twisted identification is
\(\k G_r\otimes\Hom_H(\M_\Delta,{}^{\dot x}\M_\Delta)\);
its Hom-factor is spanned by \(\widetilde\rho_\Delta(\dot x_{r,d})\),
so the stated \(\xi_X\) generates it.

The exchange bimodule is generated by
\((e_{V_{a_0,d}}\otimes v_1)\otimes(f_1^*\otimes e_{V_{a_0,d'}})\).
Indeed, over each coefficient ring the matrix-unit identities
\(E_{j1}v_1=v_j\) and \(f_1^*E_{1j}=f_j^*\), moved into the adjacent
group-algebra factors by balancing, generate every coefficient tensor.
The Mackey identification sends this generator, with the fixed
normalization of \(\mathrm T^{\Delta',\Delta}\), to
\[
f_1^*\otimes\widehat{\mathrm T}^{\Delta',\Delta}_{a_0,d',d}\otimes v_1
=\xi_T^{\Delta,\Delta'}.
\]
Here the underlying representative is \(\dot s_{a_0,d',d}\), between
\(e_{V_{r',d'}}\) and \(e_{V_{r,d}}\); for \(\Delta'=\Delta\) this is
\(\xi_T\).  Together with
\eqref{eq:bimodule-double-coset-direct-sum}, these identifications prove
all three assertions.
\end{proof}

\subsubsection{Extra isomorphisms}
Recall
\[
\Delta^*=\begin{cases}
\Delta,&G=\GL,\\
\overline\Delta,&G=\GU,\\
\widetilde\Delta,&G=\Sp,\O.
\end{cases}
\]
Assume $\Delta^*\neq\Delta$ and put $d=d_\Delta=d_{\Delta^*}$,
$V=V_{r,d}$, $V^-={}^tV$, $e=e_V$, $e^-=e_{V^-}$, and
$x=\dot x_{r,d}$.  Conjugation by $x$ identifies the two primary Levi
data.  Choose an isomorphism
\[
a_{\Delta,\Delta^*}:\M_\Delta\xrightarrow{\sim}\M_{\Delta^*}
\]
satisfying
\begin{equation}\label{eq:a-Gamma-dual}
a_{\Delta,\Delta^*}\rho_\Delta(h)
 =\rho_{\Delta^*}(x^{-1}hx)\,a_{\Delta,\Delta^*}
\qquad(h\in\GL_d(q)).
\end{equation}
It is unique up to a nonzero scalar.  We take
$a_{\Delta^*,\Delta}=a_{\Delta,\Delta^*}^{-1}$ with the inverse conjugation
convention.

We define the comparison maps directly on the induced bimodules by
\begin{align}
Y^{\Delta,\Delta^*}_{r+d,r}(ge\otimes m)
 &:=g\,e e^-x e\otimes a_{\Delta,\Delta^*}(m),
 \label{eq:Y-direct}\\
Y^{\Delta^*,\Delta}_{r+d,r}(ge\otimes m')
 &:=|V|\,g\,e e^-x^{-1}e\otimes a_{\Delta^*,\Delta}(m').
 \label{eq:Y-direct-reverse}
\end{align}
These formulas are balanced over $\GL_d(q)$.  Indeed, since the Levi
normalizes both $V$ and $V^-$, the idempotents $e,e^-$ commute with its
$\GL_d(q)$-factor.  Thus, for $h\in\GL_d(q)$,
\begin{align*}
Y^{\Delta,\Delta^*}(ge h\otimes m)
 &=g e h e^-x e\otimes a_{\Delta,\Delta^*}(m)\\
 &=g e e^-x e\,(x^{-1}hx)\otimes a_{\Delta,\Delta^*}(m)\\
 &=g e e^-x e\otimes
   \rho_{\Delta^*}(x^{-1}hx)a_{\Delta,\Delta^*}(m)\\
 &=Y^{\Delta,\Delta^*}(ge\otimes\rho_\Delta(h)m),
\end{align*}
where the last equality is \eqref{eq:a-Gamma-dual}.  The reverse map is
checked in the same way.  They are also
$G_{r+d}$--$G_r$ bimodule maps since $x$ centralizes $G_r$.  Thus they define
natural transformations $F^\Delta\to F^{\Delta^*}$ and
$F^{\Delta^*}\to F^\Delta$.
Conceptually, the factor $ee^-$ is the same opposite-unipotent averaging
that occurs in the definition of $\mathrm{X}$, while $x$ only transports the
$\Gamma$-primary Levi datum to the $\Delta^*$-primary datum.

We use
$\begin{tikzpicture}[baseline = -1mm]
\draw[-,red] (0.08,-.2) to (0.08,0);
\draw[->,blue] (0.08,0) to (0.08,.3);
\node at (0.08,0) {$\dott$};
\node at (0.08,-.37) {$\scriptstyle{\red{\Delta}}$};
\node at (0.08,0.47) {$\scriptstyle{\blue{\Delta^*}}$};
\end{tikzpicture}$
to denote $Y^{\Delta,\Delta^*}$.

\begin{lemma}\label{Lem:YYX}
With the normalization \eqref{eq:Y-direct}--\eqref{eq:Y-direct-reverse},
\[
Y^{\Delta^*,\Delta}Y^{\Delta,\Delta^*}=\mathrm{X}^\Delta,
\qquad
Y^{\Delta,\Delta^*}Y^{\Delta^*,\Delta}=\mathrm{X}^{\Delta^*}.
\]
Equivalently, the first identity is
$\begin{tikzpicture}[baseline = -1mm]
\draw[-,red] (0.08,-.4) to (0.08,-.2);
\draw[-,blue] (0.08,-.2) to (0.08,.2);
\draw[->,red] (0.08,.2) to (0.08,.5);
\node at (0.08,-.2) {$\dott$};
\node at (0.08,.2) {$\dott$};
\node at (0.08,-.57) {$\scriptstyle{\red{\Delta}}$};
\node at (0.08,.67) {$\scriptstyle{\red{\Delta}}$};
\node at (.2,.0) {$\scriptstyle{\,\,\blue{\Delta^*}}$};
\end{tikzpicture}
=
\begin{tikzpicture}[baseline = -1mm]
\draw[->,red] (0.08,-.2) to (0.08,.2);
\node[red] at (0.08,0) {$\dott$};
\node at (0.08,-.37) {$\scriptstyle{\red{\Delta}}$};
\end{tikzpicture}$.
\end{lemma}

\begin{proof}
The cuspidal intertwiners cancel.  Since
$xex^{-1}=e^-$ and $xe^-x^{-1}=e$, the first composite is multiplication by
\[
|V|\,ee^-x\,ee^-x^{-1}e
 =|V|\,ee^-(e^-e)e
 =|V|\,ee^-e,
\]
which is exactly the defining operator for $\mathrm{X}^\Delta$ in the case
$\Delta^*\ne\Delta$.  The second identity is identical with
$\Delta$ and $\Delta^*$ interchanged.
\end{proof}

\begin{proof}[Proof of Lemma~\ref{Lem:isomfunctors}]
Assume $\Delta^*\ne\Delta$.  By Lemma~\ref{Lem:YYX}, the two composites of
the comparison maps are $\mathrm{X}^\Delta$ and $\mathrm{X}^{\Delta^*}$, respectively.
Proposition~\ref{prop:invertible} shows that these dots are invertible, hence
so are the comparison maps.  Explicitly,
\[
\bigl(Y^{\Delta,\Delta^*}\bigr)^{-1}
=(\mathrm{X}^\Delta)^{-1}Y^{\Delta^*,\Delta}
=Y^{\Delta^*,\Delta}(\mathrm{X}^{\Delta^*})^{-1}.
\]
Therefore $F^\Delta\cong F^{\Delta^*}$, and taking either adjoint gives
$E^\Delta\cong E^{\Delta^*}$.
\end{proof}

\subsubsection{Affine Hecke relations and mixed affine Hecke relations}\label{subsec:repdatumN}

Write
\[
M\odot N:=R^{\GL_{d+d'}(q)}_{\GL_d(q)\times\GL_{d'}(q)}(M\boxtimes N)
\]
for Harish--Chandra induction on the general-linear factors.  The following
lemma gives the Howlett--Lehrer relations used below; see \cite{HL80}.

\begin{lemma}\label{lem:HLintertwiners}
Let \(\Delta,\Delta',\Delta''\in\calF(\k)\), and let the operators \(\mathrm{T}\) be
the normalized block intertwiners defined above.
\begin{enumerate}
\item[(1)] On \(\M_\Delta\odot\M_\Delta\),
\[
(\mathrm{T}^{\Delta,\Delta})^2=z^\Delta \mathrm{T}^{\Delta,\Delta}+1,
\]
and on three equal blocks the usual braid relation holds.
\item[(2)] If \(\Delta\ne\Delta'\), the opposite block intertwiners are
mutual inverses:
\[
\mathrm{T}^{\Delta',\Delta}\mathrm{T}^{\Delta,\Delta'}=1,
\qquad
\mathrm{T}^{\Delta,\Delta'}\mathrm{T}^{\Delta',\Delta}=1.
\]
\item[(3)] For arbitrary three colors, the normalized block intertwiners
satisfy the braid relation
\[
(\mathrm{T}^{\Delta',\Delta''}\otimes1)
(1\otimes \mathrm{T}^{\Delta,\Delta''})
(\mathrm{T}^{\Delta,\Delta'}\otimes1)
=
(1\otimes \mathrm{T}^{\Delta,\Delta'})
(\mathrm{T}^{\Delta,\Delta''}\otimes1)
(1\otimes \mathrm{T}^{\Delta',\Delta''}),
\]
with source and target obtained by the evident permutations of the three
inducing blocks.  No equality among \(d_\Delta,d_{\Delta'},d_{\Delta''}\) is
assumed.
\end{enumerate}
\end{lemma}

\begin{proof}
By transitivity of Harish--Chandra induction, the \emph{block intertwiners}
occurring in this lemma are already defined in the intermediate Levi
\[
 G_r\times\GL_{d_\Delta+d_{\Delta'}+d_{\Delta''}}(q),
\]
and act trivially on the factor $G_r$.  Thus the relations among crossings follow from
the general-linear Howlett--Lehrer calculation \cite{HL80}; no
connectedness assumption on the ambient full isometry group is involved.  In
particular, this reduction is unchanged for even orthogonal groups in
characteristic two.

For (1), the relative Weyl group on the equal general-linear blocks is of type
$A_1$ in rank two and $A_2$ in rank three.  The equal-cuspidal-block Hecke parameter is \(q^{d_\Delta}\).
Our $\mathrm{T}^{\Delta,\Delta}$ is the symmetric normalization of that generator, so
\[
 \mathrm{T}^2=(q^{d_\Delta/2}-q^{-d_\Delta/2})\mathrm{T}+1
      =z^\Delta\mathrm{T}+1,
\]
and the braid relation holds.  For distinct cuspidal blocks the block
transposition has no equal-color Hecke factor, so its normalized intertwiner
is invertible and the opposite transposition is its inverse, proving (2).
Finally both sides of (3) are the normalized standard intertwiner for the
same permutation of three inducing blocks, which proves (3).

The idempotents and normalized intertwiners are defined over $\mathcal O$
since $\ell\nmid q(q-1)$.  After extension of scalars
to $\mathbb K$ the preceding Howlett--Lehrer identities hold.  The relevant
$\mathcal O$-bimodules are direct summands of finite free group algebras
(tensored with the chosen free cuspidal lattices), hence are
$\mathcal O$-torsion-free; the identities therefore already hold over
$\mathcal O$, and reduction gives the same identities over $\mathbb F$.
\end{proof}

\begin{lemma}\label{lem:quantum12relations}
The natural transformations  $\mathrm{X}^\Delta=\begin{tikzpicture}[baseline = -1mm]
	\draw[->,red] (0.08,-.2) to (0.08,.2);
      \node at (0.08,0) {$\dott$};
      \node at (0.08,-.37) {$\scriptstyle{\red{\Delta}}$};
      \end{tikzpicture}$,
$\mathrm{T}^{\Delta,\Delta}=\begin{tikzpicture}[baseline = -1mm]
	\draw[->,red] (0.2,-.2) to (-0.2,.2);
	\draw[-,white,line width=4pt] (-0.2,-.2) to (0.2,.2);
	\draw[->,red] (-0.2,-.2) to (0.2,.2);
\node at (0.2,-.37) {$\scriptstyle{\red{\Delta}}$};
  \node at (-0.2,-.37) {$\scriptstyle{\red{\Delta}}$};
\end{tikzpicture}\:,$
$\mathrm{T}^{\Delta,\Delta'}=\begin{tikzpicture}[baseline = -1mm]
	\draw[->,red] (0.2,-.2) to (-0.2,.2);
	\draw[->,blue] (-0.2,-.2) to (0.2,.2);
\node at (0.2,-.37) {$\scriptstyle{\red{\Delta}}$};
  \node at (-0.2,-.37) {$\scriptstyle{\blue{\Delta'}}$};
\end{tikzpicture}\:$
satisfy the defining same-color affine-Hecke relations (quadratic, braid, and dot--slide relations) and, for distinct colors, the invertible mixed crossing, mixed dot--slide, and three-color braid relations required by the colored Heisenberg presentation.

\end{lemma}

\begin{proof}
We first treat upward strands.  The quadratic, braid, and dot--slide
relations are the normalized Howlett--Lehrer affine-Hecke relations.  The
quadratic and braid relations follow from Lemma~\ref{lem:HLintertwiners};
for completeness, we record the normalization of the dot--slide relation.

Put $d=d_\Delta$.  On two consecutive equal
$\GL_d$-blocks, let $s=\dot s_{r,d,d}$ be the block transposition.  Denote by
$x_1$ the opposition element defining the dot on the first added block and by
$x_2$ the corresponding opposition element after that block has been moved
past the second one.  With the embeddings fixed above one has the elementary
block-Weyl identity
\begin{equation}\label{eq:block-weyl-dot-slide}
 s x_1s^{-1}=x_2.
\end{equation}
In the general-linear case $x_i$ means the opposite-unipotent correspondence
$e_Ve_{{}^tV}e_V$ rather than an opposition matrix; the same identity is the
usual rank-two parabolic convolution identity.  In the self-dual classical
case, the factor $A_\Delta$ in Definition~\ref{def:cuspidal-extension} is
carried from the first tensor factor to the second by the block swap without
a scalar factor.

Insert the unipotent averages on both sides of
\eqref{eq:block-weyl-dot-slide}.  The standard rank-two root-subgroup
calculation gives, for the \emph{unnormalized} equal-color intertwiner
$\mathsf T_\Delta$,
\begin{equation}\label{eq:unnormalized-block-dot-slide}
 \mathsf T_\Delta(1\otimes \mathrm{X}^\Delta)\mathsf T_\Delta
   =q^d(\mathrm{X}^\Delta\otimes1).
\end{equation}
This is the block form of the calculation in
\cite[Proposition~4.2(c),(f)]{LLZ}; see also
\cite[Proposition~2.9]{LSZ25}.  For a degree-$d$ cuspidal
$\GL_d$-block, the equal-block Howlett--Lehrer endomorphism algebra is the
Hecke algebra with parameter $q^d$; see \cite[\S2.5]{BDK01}.  Thus, after
applying the cuspidal projector to the relative-rank-one convolution, the
scalar in \eqref{eq:unnormalized-block-dot-slide} is precisely
$q^d$.  The block-normalizer identity
\eqref{eq:block-weyl-dot-slide} is independent of this scalar.  Our crossing
is the symmetric normalization
$\mathrm{T}^{\Delta,\Delta}=q^{-d/2}\mathsf T_\Delta$.
Consequently \eqref{eq:unnormalized-block-dot-slide} becomes
\[
 \mathrm{T}^{\Delta,\Delta}(1\otimes \mathrm{X}^\Delta)\mathrm{T}^{\Delta,\Delta}
   =\mathrm{X}^\Delta\otimes1,
\]
which is the required same-color dot--slide relation.  Together with
$\mathrm{T}^2=z^\Delta\mathrm{T}+1$, where
$z^\Delta=q^{d/2}-q^{-d/2}$, this is the affine-Hecke
presentation in our normalization.

Now suppose \(\Delta\ne\Delta'\).  The defining opposition/opposite-parabolic
correspondence for the $\Delta$-dot is transported by the block transposition
to the corresponding $\Delta$-correspondence on the other side.  Since the
two cuspidal blocks are nonisomorphic, the mixed intertwiner has no
same-color rank-one Hecke summand and hence no correction term.  Thus
\[
\mathrm{T}^{\Delta,\Delta'}(\mathrm{X}^\Delta\otimes1)
=(1\otimes \mathrm{X}^\Delta)\mathrm{T}^{\Delta,\Delta'},
\qquad
\mathrm{T}^{\Delta,\Delta'}(1\otimes \mathrm{X}^{\Delta'})
=(\mathrm{X}^{\Delta'}\otimes1)\mathrm{T}^{\Delta,\Delta'}.
\]
Invertibility of the mixed crossing and the three-color braid relation are
parts~(2) and~(3) of Lemma~\ref{lem:HLintertwiners}.

The remaining orientations are obtained by taking left and right mates with
respect to Lemma~\ref{lem:adjointpair}.  Taking mates preserves equality of
natural transformations, so the sideways and downward relations follow from
the upward ones.  This proves all affine and mixed affine-Hecke relations used
in Theorem~\ref{Thm:quantumheisenberg}.
\end{proof}

\smallskip

\subsubsection{Adjunctions}
\begin{lemma}
\label{lem:adjointpair}
The maps $\varepsilon^\Delta_R$, $\eta^\Delta_R$, $\varepsilon^\Delta_L$, $\eta^\Delta_L$ defined above make $(E^\Delta,F^\Delta)$ into a biadjoint pair.
\end{lemma}
\begin{proof}
Harish--Chandra induction and restriction are biadjoint, and tensoring with the
finite projective module \(\M_\Delta\) is biadjoint to contraction with
\(\M_\Delta^*\).  Thus the abstract functors \(F^\Delta\) and \(E^\Delta\)
are biadjoint.  It remains to check that the four concrete maps defined above
are the corresponding units and counits with the stated normalization.

On the representing bimodules the four triangle identities are
\begin{align*}
(1_{F^\Delta}\varepsilon_L^\Delta)
(\eta_L^\Delta1_{F^\Delta})&=1_{F^\Delta},&
(\varepsilon_L^\Delta1_{E^\Delta})
(1_{E^\Delta}\eta_L^\Delta)&=1_{E^\Delta},\\
(\varepsilon_R^\Delta1_{F^\Delta})
(1_{F^\Delta}\eta_R^\Delta)&=1_{F^\Delta},&
(1_{E^\Delta}\varepsilon_R^\Delta)
(\eta_R^\Delta1_{E^\Delta})&=1_{E^\Delta}.
\end{align*}
For the left pair, insert the coset sum \(Z_r^\Delta\) appearing in
\(\eta_L^\Delta\).  In the subsequent application of
\(\varepsilon_L^\Delta\), every term whose middle group element lies outside
\(P_{r-d_\Delta,d_\Delta}\) is killed by the Levi projection.  On the unique
surviving coset, the dual-basis tensor
\(\sum_jv_j\otimes f_j^*\) contracts to the identity of \(\M_\Delta\).
With the canonical coset-sum normalization used in \(Z_r^\Delta\), no extra
scalar remains, so the first two composites are identities.

For the right pair, the class inserted by the unit is
\(f_1^*\otimes v_1\), with \(f_1^*(v_1)=1\).  The two coefficient
contractions are
\[
 \Phi_\Delta(m\otimes f_1^*)v_1=m,
 \qquad
 f_1^*\Phi_\Delta(v_1\otimes n^*)=n^*.
\]
These identities hold over all three coefficient rings.  The scalar in \(\eta_R^\Delta\) is
\(-(t^\Delta z^\Delta)^{-1}\), while the scalar in \(\varepsilon_R^\Delta\)
is \(-t^\Delta z^\Delta\); these two factors cancel in each triangle
composite, giving the last two identities.
Thus the four displayed maps make \((E^\Delta,F^\Delta)\) a biadjoint pair
with the stated normalization.
\end{proof}
\smallskip

\subsubsection{Rightward crossings}
Define the rightward crossings by
\begin{align*}
\Rcross^{\Gamma,\Gamma}:=\mathord{
\begin{tikzpicture}[baseline = -.5mm]
	\draw[thin,red,->] (-0.28,-.3) to (0.28,.4);
	\draw[line width=4pt,white,-] (0.28,-.3) to (-0.28,.4);
	\draw[<-,thin,red] (0.28,-.3) to (-0.28,.4);
\node at (0.28,-.5) {$\scriptstyle{\red{\Gamma}}$};
  \node at (-0.28,-.5)
{$\scriptstyle{\red{\Gamma}}$};
\end{tikzpicture}
}:=
\mathord{
\begin{tikzpicture}[baseline = 0]
	\draw[->,thin,red] (0.3,-.5) to (-0.3,.5);
	\draw[line width=4pt,-,white] (-0.2,-.2) to (0.2,.3);
	\draw[-,thin,red] (-0.2,-.2) to (0.2,.3);
        \draw[-,thin,red] (0.2,.3) to[out=50,in=180] (0.5,.5);
        \draw[->,thin,red] (0.5,.5) to[out=0,in=90] (0.8,-.5);
        \draw[-,thin,red] (-0.2,-.2) to[out=230,in=0] (-0.6,-.5);
        \draw[-,thin,red] (-0.6,-.5) to[out=180,in=-90] (-0.85,.5);
         \node at (0.3,-.7) {$\scriptstyle{\red{\Gamma}}$};
  \node at (0.8,-.7) {$\scriptstyle{\red{\Gamma}}$};
  \end{tikzpicture}}.
\end{align*}
By the definition, $\Rcross^{\Gamma,\Gamma}=(1_{E^\Gamma}\otimes1_{F^\Gamma}\otimes\varepsilon^\Gamma_R)\circ(1_{E^\Gamma}\otimes \mathrm{T}^{\Gamma,\Gamma}\otimes 1_{E^\Gamma})\circ(\eta^\Gamma_R\otimes 1_{F^\Gamma}\otimes 1_{E^\Gamma}).$
If $\Gamma\neq \Gamma',$
we also define the rightward mixed crossings:
\begin{align*}
\Rcross^{\Gamma,\Gamma'}:=\mathord{
\begin{tikzpicture}[baseline = -.5mm]
	\draw[thin,blue,->] (-0.28,-.3) to (0.28,.4);
	\draw[<-,thin,red] (0.28,-.3) to (-0.28,.4);
\node at (0.28,-.5) {$\scriptstyle{\red{\Gamma}}$};
  \node at (-0.28,-.5) {$\scriptstyle{\blue{\Gamma'}}$};
\end{tikzpicture}
}:=
\mathord{
\begin{tikzpicture}[baseline = 0]
	\draw[->,thin,blue] (0.3,-.5) to (-0.3,.5);
	\draw[-,thin,red] (-0.2,-.2) to (0.2,.3);
        \draw[-,thin,red] (0.2,.3) to[out=50,in=180] (0.5,.5);
        \draw[->,thin,red] (0.5,.5) to[out=0,in=90] (0.8,-.5);
        \draw[-,thin,red] (-0.2,-.2) to[out=230,in=0] (-0.6,-.5);
        \draw[-,thin,red] (-0.6,-.5) to[out=180,in=-90] (-0.85,.5);
\node at (0.3,-.7) {$\scriptstyle{\blue{\Gamma'}}$};
  \node at (0.8,-.7) {$\scriptstyle{\red{\Gamma}}$};
\end{tikzpicture}
}.
\end{align*}
By the definition, $\Rcross^{\Gamma,\Gamma'}=(1_{E^\Gamma}\otimes1_{F^{\Gamma'}}\otimes \varepsilon^\Gamma_R)\circ(1_{E^\Gamma}\otimes \mathrm{T}^{\Gamma,\Gamma'}\otimes 1_{E^\Gamma})\circ(\eta^\Gamma_R\otimes 1_{F^{\Gamma'}}\otimes 1_{E^\Gamma})$.

\begin{lemma}\label{lcross}
Let \(\Gamma,\Gamma'\in\calI(\k)\), and let \(\Delta,\Delta'\) be irreducible representatives of \(\Gamma,\Gamma'\), respectively.  Assume
\(n=r+d_\Delta=r'+d_{\Delta'}\), and put
\(a_0=r-d_{\Delta'}=r'-d_\Delta\).  If \(a_0<0\), the source functor
and the rightward crossing are zero.  If \(a_0\ge0\), the color-level rightward crossing
\[
\Rcross^{\Gamma,\Gamma'}:
F^{\Gamma'}_{r,r-d_{\Delta'}}E^\Gamma_{r'-d_\Delta,r'}
\Longrightarrow
E^\Gamma_{r,n}F^{\Gamma'}_{n,r'}
\]
is represented by the \((\k G_r,\k G_{r'})\)-bimodule map
\begin{align*}
&\k G_r e_{V_{r-d_{\Delta'},d_{\Delta'}}}
 \ot_{\GL_{d_{\Delta'}}}\M_{\Delta'}
 \ot_{\k G_{r'-d_\Delta}}
 \M_\Delta^*\ot_{\GL_{d_\Delta}}
 e_{V_{r'-d_\Delta,d_\Delta}}\k G_{r'}\\
&\hspace{18mm}\longrightarrow
\M_\Delta^*\ot_{\GL_{d_\Delta}}e_{V_{r,d_\Delta}}
 \k G_n e_{V_{r',d_{\Delta'}}}
 \ot_{\GL_{d_{\Delta'}}}\M_{\Delta'},
\end{align*}
which on elementary tensors is
\begin{align*}
&a e_{V_{r-d_{\Delta'},d_{\Delta'}}}\otimes m\otimes n^*
 \otimes e_{V_{r'-d_\Delta,d_\Delta}}b\\
&\qquad\longmapsto
n^*\otimes e_{V_{r,d_\Delta}}
 a\widehat{\mathrm{T}}^{\Delta,\Delta'}_{a_0,d_\Delta,d_{\Delta'}}
 b e_{V_{r',d_{\Delta'}}}\otimes m,
\end{align*}
where \(a\in G_r\), \(b\in G_{r'}\),
\(m\in\M_{\Delta'}\) and \(n^*\in\M_\Delta^*\).
\end{lemma}

\begin{proof}
Expand the defining composite
\[
(1_{E^\Gamma}1_{F^{\Gamma'}}\varepsilon_R^\Gamma)
(1_{E^\Gamma}\mathrm{T}^{\Gamma,\Gamma'}1_{E^\Gamma})
(\eta_R^\Gamma1_{F^{\Gamma'}}1_{E^\Gamma}).
\]
The unit inserts the class \((f_1^\Delta)^*\otimes v_1^\Delta\) in the
one-dimensional balanced tensor product.  The middle map is the block
intertwiner \(\widehat{\mathrm T}^{\Delta,\Delta'}\).  The final counit applies the matrix coefficient
\(\Phi_\Delta(v_1^\Delta\otimes n^*)\).
Transporting this matrix coefficient through the swap gives
\[
 (f_1^\Delta)^*\Phi_\Delta(v_1^\Delta\otimes n^*)=n^*,
\]
because its value at \(w\) is
\((f_1^\Delta)^*(v_1^\Delta)n^*(w)=n^*(w)\).
The factors \(-(t^\Delta z^\Delta)^{-1}\) and \(-t^\Delta z^\Delta\)
therefore cancel, giving the displayed formula.
Conjugation by the block transposition moves the \(\GL_{d_{\Delta'}}\)-action from the source vector to the target vector, and the \(\GL_{d_\Delta}\)-action from the source covector to the target covector.
The common \(G_{a_0}\)-factor commutes with the swap.  These observations
verify all three balancing relations and the \(G_r\)--\(G_{r'}\)
bimodule structure; the outer projectors are
\(e_{V_{r,d_\Delta}}\) and \(e_{V_{r',d_{\Delta'}}}\), as required.
\end{proof}

\subsubsection{Leftward crossings}
Define the leftward crossings by
\begin{align*}
\Lcross^{\Gamma,\Gamma}:=\mathord{
\begin{tikzpicture}[baseline = -.5mm]
	\draw[<-,red] (-0.28,-.3) to (0.28,.4);
	\draw[line width=4pt,white,-] (0.28,-.3) to (-0.28,.4);
	\draw[->,red] (0.28,-.3) to (-0.28,.4);
\node at (-0.28,-.5) {$\scriptstyle{\red{\Gamma}}$};
  \node at (0.28,-.5) {$\scriptstyle{\red{\Gamma}}$};
\end{tikzpicture}
}:=
\mathord{
\begin{tikzpicture}[baseline = 0]
	\draw[-,red] (-0.2,.2) to (0.2,-.3);
	\draw[-,line width=4pt,white] (0.3,.5) to (-0.3,-.5);
	\draw[<-,red] (0.3,.5) to (-0.3,-.5);
        \draw[-,red] (0.2,-.3) to[out=130,in=180] (0.5,-.5);
        \draw[-,red] (0.5,-.5) to[out=0,in=270] (0.8,.5);
        \draw[-,red] (-0.2,.2) to[out=130,in=0] (-0.5,.5);
        \draw[->,red] (-0.5,.5) to[out=180,in=-270] (-0.8,-.5);
\node at (-0.3,-.7) {$\scriptstyle{\red{\Gamma}}$};
  \node at (-0.8,-.7) {$\scriptstyle{\red{\Gamma}}$};
\end{tikzpicture}
}.
\end{align*}
By the definition, $\Lcross^{\Gamma,\Gamma}=(\varepsilon^\Gamma_L\otimes 1_{F^\Gamma}\otimes 1_{E^\Gamma} )\circ(1_{E^\Gamma}\otimes \mathrm{T}^{\Gamma,\Gamma} \otimes 1_{E^\Gamma})\circ(1_{E^\Gamma}\otimes 1_{F^\Gamma}\otimes\eta^\Gamma_L).$
We also define the leftward mixed crossings:
$$
\Lcross^{\Gamma,\Gamma'}:=\mathord{
\begin{tikzpicture}[baseline = 0]
	\draw[->,red] (0.28,-.3) to (-0.28,.4);
	\draw[<-,blue] (-0.28,-.3) to (0.28,.4);
\node at (-0.28,-.5) {$\scriptstyle{\blue{\Gamma'}}$};
  \node at (0.28,-.5) {$\scriptstyle{\red{\Gamma}}$};
\end{tikzpicture}
}
:=
\mathord{
\begin{tikzpicture}[baseline = 0]
	\draw[<-, red] (0.3,.5) to (-0.3,-.5);
	\draw[-,blue] (-0.2,.2) to (0.2,-.3);
        \draw[-,blue] (0.2,-.3) to[out=130,in=180] (0.5,-.5);
        \draw[-,blue] (0.5,-.5) to[out=0,in=270] (0.9,.5);
        \draw[-,blue] (-0.2,.2) to[out=130,in=0] (-0.6,.5);
        \draw[->,blue] (-0.6,.5) to[out=180,in=-270] (-0.9,-.5);
\node at (-0.3,-.7) {$\scriptstyle{\red{\Gamma}}$};
  \node at (-0.9,-.7) {$\scriptstyle{\blue{\Gamma'}}$};
\end{tikzpicture}
}.$$
By the definition, $\MixLcross^{\Gamma,\Gamma'}=(\varepsilon^{\Gamma'}_L\otimes 1_{F^\Gamma}\otimes 1_{E^{\Gamma'}} )\circ(1_{E^{\Gamma'}}\otimes \mathrm{T}^{\Gamma,\Gamma'} \otimes 1_{E^{\Gamma'}})\circ(1_{E^{\Gamma'}}\otimes 1_{F^\Gamma}\otimes\eta^{\Gamma'}_L)$.

\begin{lemma}\label{Lem:Rcrossing}
Let \(\Gamma\in\calI(\k)\) and let \(\Delta\) be an irreducible representative of \(\Gamma\).  In the mixed
case let \(\Gamma'\in\calI(\k)\) and let \(\Delta'\) be an irreducible representative of \(\Gamma'\).
\begin{enumerate}
\item[(a)] If \(\Gamma\in\calI_{-1}(\k)\), then
\(\Lcross^{\Gamma,\Gamma}:E^\Gamma F^\Gamma\Rightarrow F^\Gamma E^\Gamma\)
is represented in rank \(r\) by
\[
\M_\Delta^*\ot_{\GL_{d_\Delta}}e_{V_{r,d_\Delta}}
\k G_{r+d_\Delta}e_{V_{r,d_\Delta}}
\ot_{\GL_{d_\Delta}}\M_\Delta
\longrightarrow \mathscr C_\Delta(r),
\]
with
\[
n^*\otimes e_{V_{r,d_\Delta}}\otimes m\longmapsto0,
\qquad
n^*\otimes\widehat{\mathrm{T}}^{\Delta,\Delta}_{r-d_\Delta,d_\Delta,d_\Delta}
\otimes m
\longmapsto
\sum_j n^*(v_j^\Delta)e_{V_{r-d_\Delta,d_\Delta}}
\otimes m\otimes(f_j^\Delta)^*\otimes
 e_{V_{r-d_\Delta,d_\Delta}}.
\]
\item[(b)] If \(\Gamma\in\calI_{-2}(\k)\), the same formula holds and, in
addition,
\[
n^*\otimes\widehat{\mathrm{X}}^\Delta_{r,d_\Delta}\otimes m\longmapsto0.
\]
\item[(c)] If \(\Gamma\ne\Gamma'\) and
\(n=r+d_\Delta=r'+d_{\Delta'}\), then the color-level mixed leftward crossing
\(\Lcross^{\Gamma,\Gamma'}\) is represented by the map
\begin{align*}
&\M_{\Delta'}^*\ot_{\GL_{d_{\Delta'}}}e_{V_{r',d_{\Delta'}}}
 \k G_n e_{V_{r,d_\Delta}}\ot_{\GL_{d_\Delta}}\M_\Delta
 \longrightarrow \mathscr C_{\Delta,\Delta'}(r',r),\\
&n^*\otimes
 \widehat{\mathrm{T}}^{\Delta',\Delta}_{r-d_{\Delta'},d_{\Delta'},d_\Delta}
 \otimes m\\
&\qquad\longmapsto
\sum_{j=1}^{\dim\M_{\Delta'}}
 n^*(v_j^{\Delta'})e_{V_{r'-d_\Delta,d_\Delta}}
 \otimes m\otimes(f_j^{\Delta'})^*
 \otimes e_{V_{r-d_{\Delta'},d_{\Delta'}}}.
\end{align*}
Here the contraction basis in the mixed case belongs to \(\M_{\Delta'}\),
which is forced by \(n^*\in\M_{\Delta'}^*\).  The open-generator formulas
are omitted when the smaller tower index is negative.
\end{enumerate}
\end{lemma}

\begin{proof}
Substitute the definitions of the color-level \(\Lcross\), \(\mathrm{T}\), and the
left unit and counit, using irreducible representatives \(\Delta\) and \(\Delta'\) in the explicit bimodules.  The counit contains the Levi projection
\(\Proj_{L_{r,d_\Delta}}\) in (a) and (b), and
\(\Proj_{L_{r',d_{\Delta'}}}\) in (c).  On the closed double-coset generator
\(e_{V_{r,d_\Delta}}\), this projection vanishes after the crossing, giving
the first zero in (a).  In level \(-2\), the extra generator
\(\widehat{\mathrm{X}}^\Delta_{r,d_\Delta}\) is also killed; this is the zero projection
computed explicitly in Lemma~\ref{lrdots}.  On the open Bruhat generator
\(\widehat{\mathrm{T}}\), the Levi projection survives and the coevaluation--evaluation
pair contracts to
\(\sum_j v_j^\Delta\otimes(f_j^\Delta)^*\), which is the identity of
\(\M_\Delta\).  This gives (a) and (b).

For distinct colors, Lemma~\ref{lem:finite-bimodule-decomposition}(3) shows that no \(\unit\)-summand occurs.  The covector belongs to
\(\M_{\Delta'}^*\), so the contraction must use a basis of
\(\M_{\Delta'}\); carrying this dual-basis tensor through the unique open
double coset gives exactly the formula in (c).  Thus all three maps are the
claimed mates of the upward crossings.
\end{proof}
\subsubsection{Mackey formula relations}

If $\Gamma\in \calI_{-2}(\k)$, let $\Delta$ be an irreducible representative of $\Gamma$.  We define
$$
\RX^\Gamma:=\mathord{
\begin{tikzpicture}[baseline = -0.9mm]
	\draw[<-,thin,red] (0.4,0.2) to[out=-90, in=0] (0.1,-.2);
	\draw[-,thin,red] (0.1,-.2) to[out = 180, in = -90] (-0.2,0.2);
      \node at (0.38,0) {$\dott$};
      \node at (-0.2, 0.35) {$\scriptstyle{\red{\Gamma}}$};
  \end{tikzpicture}
}$$
By definition, $\RX^\Gamma=(1_{E^\Gamma}\otimes \mathrm{X}^\Gamma )\circ\eta^\Gamma_R$.
For the same $\Gamma$, we also define
$$
\LX^\Gamma:=\mathord{
\begin{tikzpicture}[baseline = 1mm]
	\draw[-,thin,red] (0.4,0) to[out=90, in=0] (0.1,0.4);
	\draw[->,thin,red] (0.1,0.4) to[out = 180, in = 90] (-0.2,0);
      \node at (0.35,0.2) {$\dott$};
\node at (0.4, -0.15) {$\scriptstyle{\red{\Gamma}}$};
\end{tikzpicture}
}.
$$
By definition, $\LX^\Gamma=\varepsilon^\Gamma_L\circ(1_{E^\Gamma}\otimes \mathrm{X}^\Gamma)$.

\begin{lemma}\label{lrdots}
\begin{itemize}
Assume that $\Gamma\in \calI_{-2}(\k)$ and let $\Delta$ be an irreducible representative of $\Gamma$.
\item[$(a)$]
 The natural transformation $\LX^\Gamma$ is given by the $( \k G_r , \k G_r )$-bimodule map
\begin{align*} \M^*_\Delta\otimes_{\GL_{d_\Delta}}e_{V_{r,d_\Delta}} \cdot  \k G_{r+d_\Delta}  \cdot e_{V_{r,d_\Delta}}\otimes_{\GL_{d_\Delta}} \M_\Delta  &\to  \k G_r ,\\
  n^*\otimes e_{V_{r,d_\Delta}}
\otimes m&\mapsto 0,\\
 n^*\otimes \widehat{\mathrm{X}}^\Delta_{r,d_\Delta}\otimes m &\mapsto\,-(t^\Gamma)^{2}\omega_\Delta(\kappa_\Delta)\cdot n^*(A_\Delta m),\\
n^*\otimes \widehat{\mathrm{T}}^{\Delta,\Delta}_{r-d_\Delta,d_\Delta,d_\Delta}\otimes m&\mapsto 0,\end{align*}
for $m\in\M_\Delta$ and $n^*\in \M_\Delta^*$; the third formula is used only when $r\ge d_\Delta$.
\item[$(b)$]
The natural transformation $\RX^\Gamma$ is  given by the $( \k G_r , \k G_r )$-bimodule map
\begin{align*} \k G_r  &\to \M^*_\Delta\ot_{\GL_{d_\Delta}}e_{V_{r,d_\Delta}} \cdot  \k G_{r+d_\Delta}  \cdot e_{V_{r,d_\Delta}}\ot_{\GL_{d_\Delta}} \M_\Delta ,\\
1&\mapsto  -({t^\Gamma z^\Gamma})^{-1} f^*_1 \otimes \widehat{\mathrm{X}}^\Delta_{r,d_\Delta}\otimes\tilde{\rho}_\Delta(\dot{x}_{r,d_\Delta})\cdot v_1 .\end{align*}
\end{itemize}
\end{lemma}

\begin{proof}
For (a), apply
\(\LX^\Gamma=\varepsilon_L^\Gamma\circ(1_{E^\Gamma}\otimes\mathrm X^\Gamma)\)
to the three generators.  Write \(V=V_{r,d_\Delta}\),
\(\dot x=\dot x_{r,d_\Delta}\), and
\(\widehat X=\widehat{\mathrm X}^\Delta_{r,d_\Delta}\).
The identity generator maps to
\(n^*\otimes\widehat X\otimes A_\Delta m\), which the Levi projection
kills because \(\dot x\) represents the nonidentity double coset.

For the dot generator, expand
\(e_V\dot x e_V\dot x e_V
=|V|^{-1}\sum_{v\in V}e_V\dot x v\dot x e_V\).
Only \(v=1\) has support on \(P_{r,d_\Delta}\), so
\[
 \Proj_{L_{r,d_\Delta}}(\widehat X^2)
 =\frac{\beta_\Delta^2q^{2rd_\Delta}}{|V|}
       (1_{G_r}\times\kappa_\Delta).
\]
Contracting the coefficient factors and using
\eqref{eq:beta-normalization} gives
\[
 \LX^\Gamma(n^*\otimes\widehat X\otimes m)
 =-(t^\Gamma)^2\omega_\Delta(\kappa_\Delta)n^*(A_\Delta m).
\]
This is the value on the bare generator.  On
\(\xi_X=f_1^*\otimes\widehat X\otimes A_\Delta v_1\), both square
factors occur:
\[
 \LX^\Gamma(\xi_X)
 =-(t^\Gamma)^2\omega_\Delta(\kappa_\Delta)f_1^*(A_\Delta^2v_1)
 =-(t^\Gamma)^2.
\]
Finally,
\(\widehat{\mathrm T}^{\Delta,\Delta}_{r-d_\Delta,d_\Delta,d_\Delta}
\widehat X\) has nonidentity relative Bruhat support, hence zero Levi
projection.  This proves (a).

For (b), apply \(1_{E^\Gamma}\otimes\mathrm X^\Gamma\) to the right
unit \(\eta_R^\Gamma(1)=-(t^\Gamma z^\Gamma)^{-1}
f_1^*\otimes e_V\otimes v_1\).  It replaces \(e_V\otimes v_1\) by
\(\widehat X\otimes A_\Delta v_1\), giving the asserted formula.
\end{proof}

\begin{lemma}\label{lem:quantumMackey}
\begin{itemize}

\item[$(a)$](level $-1$)
If $\Gamma\in \calI_{-1}(\k)$,
the natural transformation\begin{align*}
\begin{array}{rl}
\left[\!\!\!\!\!
\begin{array}{l}\,\,\,\mathord{
\begin{tikzpicture}[baseline = 0]
	\draw[<-,thin,red] (-0.2,-.3) to (0.4,.4);
	\draw[-,line width=4pt,white] (0.4,-.3) to (-0.2,.4);
	\draw[->,thin,red] (0.4,-.3) to (-0.2,.4);
    \node at (0.4,-.45) {${\scriptstyle{\red{\Gamma}}}$};
  \node at (-0.2,-.45) {${\scriptstyle{\red{\Gamma}}}$};
  \end{tikzpicture}
}\\\,\,\,\mathord{
\begin{tikzpicture}[baseline = 1mm]
	\draw[-,thin,red] (0.4,0) to[out=90, in=0] (0.1,0.4);
	\draw[->,thin,red] (0.1,0.4) to[out = 180, in = 90] (-0.2,0);
    \node at (0.4,-.2) {${\scriptstyle{\red{\Gamma}}}$};
  \node at (-0.2,-.2) {${\scriptstyle{\red{\Gamma}}}$};    \end{tikzpicture}
}
\end{array}
\right]
:
E^\Gamma F^\Gamma \Rightarrow
F^\Gamma E^\Gamma \oplus \unit
\end{array}
\end{align*}

is an isomorphism of functors. It has a two-sided inverse
\begin{align*}
\begin{array}{rl}
\left[
\:\mathord{
\begin{tikzpicture}[baseline = 0]
	\draw[->,thin,red] (-0.2,-.3) to (0.4,.4);
	\draw[-,line width=4pt,white] (0.4,-.3) to (-0.2,.4);
	\draw[<-,thin,red] (0.4,-.3) to (-0.2,.4);
 \node at (0.4,-0.45) {${\scriptstyle{\red{\Gamma}}}$};
  \node at (-0.2,-0.45) {${\scriptstyle{\red{\Gamma}}}$};
\end{tikzpicture}
}
\:\:\:\:-t^\Gamma z^\Gamma
\mathord{
\begin{tikzpicture}[baseline = -0.9mm]
	\draw[<-,thin,red] (0.4,0.2) to[out=-90, in=0] (0.1,-.2);
	\draw[-,thin,red] (0.1,-.2) to[out = 180, in = -90] (-0.2,0.2);
    \node at (0.4,0.3) {${\scriptstyle{\red{\Gamma}}}$};
  \node at (-0.2,0.3) {${\scriptstyle{\red{\Gamma}}}$};
  \end{tikzpicture}
}
\right]
:F^\Gamma E^\Gamma \oplus
\unit
\Rightarrow E^\Gamma F^\Gamma.
&
\end{array}
\end{align*}

\item[$(b)$](level $-2$)
If $\Gamma\in \calI_{-2}(\k)$, then the natural transformation\begin{align*}
\begin{array}{rl}
\left[\!\!\!\!\!
\begin{array}{l}\,\,\,\mathord{
\begin{tikzpicture}[baseline = 0]
	\draw[<-,thin,red] (-0.2,-.3) to (0.4,.4);
	\draw[-,line width=4pt,white] (0.4,-.3) to (-0.2,.4);
	\draw[->,thin,red] (0.4,-.3) to (-0.2,.4);
    \node at (0.4,-.45) {${\scriptstyle{\red{\Gamma}}}$};
  \node at (-0.2,-.45) {${\scriptstyle{\red{\Gamma}}}$};
  \end{tikzpicture}
}\\\,\,\,\mathord{
\begin{tikzpicture}[baseline = 1mm]
	\draw[-,thin,red] (0.4,0) to[out=90, in=0] (0.1,0.4);
	\draw[->,thin,red] (0.1,0.4) to[out = 180, in = 90] (-0.2,0);
    \node at (0.4,-.2) {${\scriptstyle{\red{\Gamma}}}$};
  \node at (-0.2,-.2) {${\scriptstyle{\red{\Gamma}}}$};
  \end{tikzpicture}
}
\\
\,\,\,\mathord{
\begin{tikzpicture}[baseline = 1mm]
	\draw[-,thin,red] (0.4,0) to[out=90, in=0] (0.1,0.4);
	\draw[->,thin,red] (0.1,0.4) to[out = 180, in = 90] (-0.2,0);
      \node at (0.35,0.2) {$\dott$};
      \node at (0.4,-.2) {${\scriptstyle{\red{\Gamma}}}$};
  \node at (-0.2,-.2) {${\scriptstyle{\red{\Gamma}}}$};
  \end{tikzpicture}
}
\end{array}
\right]
:
E^\Gamma F^\Gamma \Rightarrow
F^\Gamma E^\Gamma \oplus \unit^{\oplus 2}
\end{array}
\end{align*}

is an isomorphism of functors. It has a two-sided inverse
\begin{align*}
\begin{array}{rl}
\left[
\:\mathord{
\begin{tikzpicture}[baseline = 0]
	\draw[->,thin,red] (-0.2,-.3) to (0.4,.4);
	\draw[-,line width=4pt,white] (0.4,-.3) to (-0.2,.4);
	\draw[<-,thin,red] (0.4,-.3) to (-0.2,.4);
\node at (0.4,-0.45) {${\scriptstyle{\red{\Gamma}}}$};
  \node at (-0.2,-0.45) {${\scriptstyle{\red{\Gamma}}}$};
\end{tikzpicture}
}
\:\:\:\:-t^\Gamma z^\Gamma
\mathord{
\begin{tikzpicture}[baseline = -0.9mm]
	\draw[<-,thin,red] (0.4,0.2) to[out=-90, in=0] (0.1,-.2);
	\draw[-,thin,red] (0.1,-.2) to[out = 180, in = -90] (-0.2,0.2);
\node at (0.4,0.3) {${\scriptstyle{\red{\Gamma}}}$};
  \node at (-0.2,0.3) {${\scriptstyle{\red{\Gamma}}}$};
\end{tikzpicture}
}
\:\:\:\:\:\frac{z^\Gamma}{t^\Gamma}
\mathord{
\begin{tikzpicture}[baseline = -0.9mm]
	\draw[<-,thin,red] (0.4,0.2) to[out=-90, in=0] (0.1,-.2);
	\draw[-,thin,red] (0.1,-.2) to[out = 180, in = -90] (-0.2,0.2);
      \node at (0.38,0) {$\dott$};
\node at (0.4,0.3) {${\scriptstyle{\red{\Gamma}}}$};
  \node at (-0.2,0.3) {${\scriptstyle{\red{\Gamma}}}$};
\end{tikzpicture}
}
\right]
:F^\Gamma E^\Gamma\oplus
\unit^{\oplus 2}
\Rightarrow E^\Gamma F^\Gamma.
&
\end{array}
\end{align*}

\item[$(c)$]
If $\Gamma\neq \Gamma'$, then the natural transformations
\begin{align*}
\begin{array}{rl}
\mathord{
\begin{tikzpicture}[baseline = -4]
	\draw[->,thin,blue] (-0.28,-.3) to (0.28,.4);
	\draw[<-,thin,red] (0.28,-.3) to (-0.28,.4);
    \node at (-0.28,-.5) {$\scriptstyle{\blue{\Gamma'}}$};
  \node at (0.28,-.5) {$\scriptstyle{\red{\Gamma}}$};
  \end{tikzpicture}
}
:F^{\Gamma'} E^\Gamma
\Rightarrow
 E^{\Gamma}  F^{\Gamma'}
\end{array}
~\text{and}~
\begin{array}{rl}
\mathord{
\begin{tikzpicture}[baseline = -4]
	\draw[->,blue] (0.28,-.3) to (-0.28,.4);
	\draw[<-,red] (-0.28,-.3) to (0.28,.4);
    \node at (-0.28,-.5) {$\scriptstyle{\red{\Gamma}}$};
  \node at (0.28,-.5) {$\scriptstyle{\blue{\Gamma'}}$};
  \end{tikzpicture}
}
:E^\Gamma F^{\Gamma'}
\Rightarrow
 F^{\Gamma'} E^\Gamma
\end{array}
\end{align*}
are two-sided inverses.
\end{itemize}
\end{lemma}

\begin{proof}
Work rank by rank on the representing \((\k G_r,\k G_{r'})\)-bimodules,
with \(r=r'\) in the same-color case.
Lemma~\ref{lem:finite-bimodule-decomposition} gives the Mackey decomposition
before any choice of maps.  In level \(-1\) it is
\[
E^\Gamma F^\Gamma\cong F^\Gamma E^\Gamma\oplus\unit,
\]
whereas in level \(-2\) it is
\[
E^\Gamma F^\Gamma\cong F^\Gamma E^\Gamma\oplus\unit^{\oplus2}.
\]
For distinct indices there is only the summand corresponding to the open double coset.

The displayed morphisms agree with these summands.  On a nonzero summand corresponding to the open double coset, put \(a_0=r-d_{\Delta'}=r'-d_\Delta\), with \(\Delta,\Delta'\) irreducible representatives of the two indices.  On an elementary tensor, Lemmas~\ref{lcross} and
\ref{Lem:Rcrossing} give
\[
 (ae_{V_{a_0,d_{\Delta'}}}\otimes m)\otimes
       (n^*\otimes e_{V_{a_0,d_\Delta}}b)
 \longmapsto
 \sum_j (ae_{V_{a_0,d_{\Delta'}}}\otimes m)\otimes
       \bigl(n^*(v_j^\Delta)(f_j^\Delta)^*
                    \otimes e_{V_{a_0,d_\Delta}}b\bigr),
\]
which is the original tensor.  Thus the leftward mate composed with the
rightward mate is the identity.  In the opposite order, the rightward mate
inserts the same coevaluation tensor and the leftward mate contracts it by
\[
 \sum_j n^*(v_j^\Delta)(f_j^\Delta)^*=n^*.
\]
Hence the reverse composite is also the identity on the summand corresponding to the open double coset. The composites with the \(\unit\)-summands vanish by the corresponding Levi projections.  In particular, this argument does not apply the
distinct-color inverse relation of Lemma~\ref{lem:HLintertwiners}(2)
to a same-color crossing.

On the \(\unit\)-summands, use the basis \((\xi_1,\xi_X)\) of
Lemma~\ref{lem:finite-bimodule-decomposition}(2).  The undotted counit gives
\[
 \varepsilon_L^\Gamma(\xi_1)=1,
 \qquad
 \varepsilon_L^\Gamma(\xi_X)=0,
\]
the second equality being the nonidentity-Bruhat projection.  Lemma~\ref{lrdots}
gives for the dotted counit
\[
 \LX^\Gamma(\xi_1)=0,
 \qquad
 \LX^\Gamma(\xi_X)=-(t^\Gamma)^2.
\]
Here $\xi_X$ includes the coefficient $A_\Delta v_1$; hence the second
entry uses $\omega_\Delta(\kappa_\Delta)A_\Delta^2=\id$, not only the
Levi scalar on the bare group-algebra generator.
Thus the restriction of the forward map in (b) to these \(\unit\)-summands is the diagonal matrix
\[
 \begin{pmatrix}1&0\\0&-(t^\Gamma)^2\end{pmatrix}.
\]
On the other hand, the explicit right unit is
\(\eta_R^\Gamma(1)=-(t^\Gamma z^\Gamma)^{-1}\xi_1\), while
Lemma~\ref{lrdots}(b) gives
\(\RX^\Gamma(1)=-(t^\Gamma z^\Gamma)^{-1}\xi_X\).  Hence the two displayed
inverse columns are
\[
 -t^\Gamma z^\Gamma\,\eta_R^\Gamma(1)=\xi_1,
 \qquad
 \frac{z^\Gamma}{t^\Gamma}\RX^\Gamma(1)
 =-(t^\Gamma)^{-2}\xi_X,
\]
which are exactly the inverse diagonal matrix.  This proves both composites
are the identity on the two \(\unit\)-summands.  In level \(-1\), only \(\xi_1\)
occurs, and the same calculation reduces to
\(-t^\Gamma z^\Gamma\eta_R^\Gamma(1)=\xi_1\).  Hence the maps in (a) and (b)
are two-sided inverses on every Mackey summand.

Finally, if \(\Gamma\ne\Gamma'\),
Lemma~\ref{lem:finite-bimodule-decomposition}(3) leaves only the summand corresponding to the open double coset; the two mixed crossings are \(\Rcross^{\Gamma,\Gamma'}\) and \(\Lcross^{\Gamma',\Gamma}\).  The same computation on this summand therefore proves that they are inverse.  This proves (c) and
completes the Mackey relations.
\end{proof}

\begin{lemma}\label{lem:curl}
We have $\mathord{
\begin{tikzpicture}[baseline = -0.5mm]
	\draw[<-,red] (0,0.6) to (0,0.3);
	\draw[-,red] (-0.3,-0.2) to [out=180,in=-90](-.5,0);
	\draw[-,red] (-0.5,0) to [out=90,in=180](-.3,0.2);
	\draw[-,red] (-0.3,.2) to [out=0,in=90](0,-0.3);
	\draw[-,red] (0,-0.3) to (0,-0.6);
	\draw[-,line width=4pt,white] (0,0.3) to [out=-90,in=0] (-.3,-0.2);
	\draw[-,red] (0,0.3) to [out=-90,in=0] (-.3,-0.2);
    \node at (0,-0.75) {${\scriptstyle{\red{\Gamma}}}$};
  \end{tikzpicture}
}=
0$.
\end{lemma}

\begin{proof}
Expand the curl using its defining cup, crossing, and cap.  On the representing
bimodule, the cup lands in the closed double-coset generator.  The leftward
crossing sends that generator to zero by
Lemma~\ref{Lem:Rcrossing}(a) in level \(-1\), and by
Lemma~\ref{Lem:Rcrossing}(b) in level \(-2\).  Therefore the whole composite
is zero.  Equivalently, in the explicit Levi-projection calculation the only
possible term is the projection of the crossing class \(\widehat{\mathrm{T}}\), which is
zero (cf. Lemma~\ref{lrdots}).
\end{proof}

\smallskip

\subsubsection{Invertibility of $\mathrm{X}^\Gamma$}\label{sub:invertible}
\begin{definition}Let $\k=\bbK$ or $\bbF$ and  $M\in  \calC=\k G_\bullet\mod $. We say that $M$ is $F^\Gamma$-cuspidal if $E^\Gamma(M)=0$.
\end{definition}

\begin{proposition}\label{Prop:dim2}
Let $\k=\bbK$ or $\bbF$ be an algebraically closed field.  Let
$\Gamma\in\calI(\k)$ and let $M$ be an irreducible $F^\Gamma$-cuspidal
module.
\begin{itemize}
\item[$(1)$] We have
\[
\dim_\k\End_{\mathcal C}(F^\Gamma M)=
\begin{cases}
1,&\Gamma\in\calI_{-1}(\k),\\
2,&\Gamma\in\calI_{-2}(\k).
\end{cases}
\]
The polynomial evaluation homomorphism
$\k[u]\longrightarrow\End_{\mathcal C}(F^\Gamma M)$,
$u\longmapsto\mathrm X^\Gamma(M)$, is surjective.
\item[$(2)$] If $\Gamma\in\calI_{-1}(\k)$, then
$\mathrm X^\Gamma(M)=\id_{F^\Gamma M}$.  If
$\Gamma\in\calI_{-2}(\k)$, then
\[
 (\mathrm X^\Gamma(M))^2+\mu_\Gamma(M)\mathrm X^\Gamma(M)
              +(t^\Gamma)^2\id_{F^\Gamma M}=0
\]
for some $\mu_\Gamma(M)\in\k$.  In either case
$\mathrm X^\Gamma(M)$ is invertible.
\end{itemize}
\end{proposition}
\begin{proof}
Since $E^\Gamma M=0$, biadjunction and the abstract decomposition of
Lemma~\ref{lem:finite-bimodule-decomposition} give
\[
 \End_{\mathcal C}(F^\Gamma M)
 \cong\Hom_{\mathcal C}(E^\Gamma F^\Gamma M,M)
 \cong\begin{cases}\k,&\Gamma\in\calI_{-1}(\k),\\
                     \k^2,&\Gamma\in\calI_{-2}(\k),
       \end{cases}
\]
as vector spaces, independently of the inverse matrices in Lemma~\ref{lem:quantumMackey}.  Put $X=\mathrm X^\Gamma(M)$ and define
$\tau_M(a)=\varepsilon_L^\Gamma(M)\circ E^\Gamma(a)\circ\eta_R^\Gamma(M)$.
The explicit adjunctions give the nonzero value
$\tau_M(1)=-(t^\Gamma z^\Gamma)^{-1}$.

In level $-1$, the dot is right multiplication by
$|V|e_Ve_{V^-}e_V$, where $V=V_{r,d_\Delta}$ and $V^-$ is its opposite.
Since $V^-\cap P_{r,d_\Delta}=1$ and $|V^-|=|V|$, only the identity in
the middle average survives Levi projection, giving the Levi identity.
Hence $\tau_M(X)=\tau_M(1)$.  As $X$ is scalar, $X=1$, proving both assertions.

In level $-2$, the identity and opposition projections of
Lemma~\ref{lrdots} give the explicit closures
\begin{equation}\label{eq:closed-bubbles-dim2}
 \coloredbubble{0}{\Gamma}=-(t^\Gamma z^\Gamma)^{-1},\qquad
 \coloredbubble{1}{\Gamma}=0,\qquad
 \coloredbubble{2}{\Gamma}=t^\Gamma/z^\Gamma.
\end{equation}
These are the direct evaluations $\tau_M(1),\tau_M(X),\tau_M(X^2)$
obtained from the adjunction maps.  Applying $\tau_M$ to $a1+bX=0$ and to its
product with $X$ gives $a=b=0$.  Thus $1,X$ form a basis, proving
polynomial surjectivity.  Write $X^2=c1+dX$; its closure gives
$c=-(t^\Gamma)^2$.  Setting $\mu_\Gamma(M)=-d$ proves the relation and
$X^{-1}=-(t^\Gamma)^{-2}(X+\mu_\Gamma(M)1)$.
\end{proof}
\begin{remark}
    For $\Gamma\in\calI_{-2}(\k)$, the coefficient $\mu_{\Gamma}(M)$ is
    $-\frac{\mathord{
\begin{tikzpicture}[baseline = -1mm]
  \draw[-,thin] (0,0.2) to[out=180,in=90] (-.2,0);
  \draw[->,thin] (0.2,0) to[out=90,in=0] (0,.2);
 \draw[-,thin] (-.2,0) to[out=-90,in=180] (0,-0.2);
  \draw[-,thin] (0,-0.2) to[out=0,in=-90] (0.2,0);
      \node at (0.2,0) {$\dott$};
      \node at (0.4,0) {$\color{darkblue}\scriptstyle 3$};
      \node at (0.0,-.33) {$\scriptstyle{\Gamma}$};
      \draw[-,darkg,thick] (0.55,.35) to (0.55,-.3);
     \node at (0.55,-.45) {$\darkg\scriptstyle{M}$};
\end{tikzpicture}
}}{ \mathord{\begin{tikzpicture}[baseline = -1mm]
  \draw[-,thin] (0,0.2) to[out=180,in=90] (-.2,0);
  \draw[->,thin] (0.2,0) to[out=90,in=0] (0,.2);
 \draw[-,thin] (-.2,0) to[out=-90,in=180] (0,-0.2);
  \draw[-,thin] (0,-0.2) to[out=0,in=-90] (0.2,0);
  \node at (0.4,0) {$\color{darkblue}\scriptstyle 2$};
      \node at (0.2,0) {$\dott$};
      \node at (0.0,-.33) {$\scriptstyle{\Gamma}$};
      \draw[-,darkg,thick] (0.55,.35) to (0.55,-.3);
     \node at (0.55,-.45) {$\darkg\scriptstyle{M}$};
\end{tikzpicture}
}}$.  Lemma~\ref{Lem:Key} gives the corresponding character formula.
\end{remark}

\begin{proposition}\label{prop:invertible}
The natural transformation
$\mathrm{X}^\Gamma\in \End(F^\Gamma)$ is invertible.
\end{proposition}
\begin{proof}
We use rank induction as in \cite[Proposition~2.18]{LSZ25}, with the
cuspidal calculation supplied by Proposition~\ref{Prop:dim2} for both
levels.  We use only the same-color polynomial dot-slide and
Hecke quadratic relations of Lemma~\ref{lem:quantum12relations}.

First work over either coefficient field.  By exactness of $F^\Gamma$ and
naturality of $\mathrm X^\Gamma$, it suffices to treat simple modules:
a composition series then makes the dot a block triangular linear map
with invertible diagonal blocks.  Let $M$ be simple at rank $r$ and
assume the result at smaller ranks.  If $E^\Gamma M=0$, use
Proposition~\ref{Prop:dim2}.  Otherwise choose a simple submodule
$N\subset E^\Gamma M$.  Biadjunction gives a nonzero map
$F^\Gamma N\longrightarrow M$, which is surjective since $M$ is simple.
The rank of $N$ is $r-d_\Delta<r$.

On $(F^\Gamma)^2N$, label the two dot operators by
\[
 X_1=F^\Gamma(\mathrm X^\Gamma(N)),\qquad
 X_2=\mathrm X^\Gamma(F^\Gamma N).
\]
The induction hypothesis makes $X_1$ invertible.  The Hecke quadratic
relation gives $T^{-1}=T-z^\Gamma$, and the same-color dot-slide relation
$TX_1T=X_2$ therefore makes $X_2$ invertible.  Exactness gives a surjection
$(F^\Gamma)^2N\twoheadrightarrow F^\Gamma M$ intertwining $X_2$ and
$\mathrm X^\Gamma(M)$.  Its kernel is $X_2$-stable and finite-dimensional,
so it is also $X_2^{-1}$-stable.  The induced operator on the quotient is
invertible.  This completes the rank induction.  The inverses are natural,
since the original transformations are natural.

For the integral form, at each rank the representing induction bimodule
is finite free over $\calO$.  Its reduction is one of the preceding
invertible operators over $\bbF$.  Thus its determinant is a unit in
$\calO$, so its inverse exists over $\calO$ and is again a bimodule map.
\end{proof}

\begin{proof}[Proof of Theorem~\ref{Thm:quantumheisenberg}]
The defining relations are verified in the order in which they were developed
above.  The same-color and mixed affine-Hecke relations are
Lemma~\ref{lem:quantum12relations}.  The four cups and caps satisfy the
triangle identities by Lemma~\ref{lem:adjointpair}.  The inversion/Mackey
relations are Lemma~\ref{lem:quantumMackey}; the sideways crossings used there
are the explicit mates computed in Lemmas~\ref{lcross} and
\ref{Lem:Rcrossing}.  The last columns of the inverse matrices in
Lemma~\ref{lem:quantumMackey} also give the additional quantum normalization,
equivalently the initial bubble values used in Proposition~\ref{Prop:dim2}.
The curl relation is Lemma~\ref{lem:curl}, and the dot is invertible by
Proposition~\ref{prop:invertible}.  All defining relations involve only finitely many colors.  Hence the
preceding assignments define the claimed strict monoidal functor on
$\symHeis_{\calI(\k)}$.
\end{proof}

\subsection{Application to ordinary representations of finite classical groups}
\subsubsection{Jordan decompositions compatible with categorifications}\label{subsec:jordan}

For the linear and unitary towers we use the standard Jordan decomposition,
chosen compatibly with Harish--Chandra induction.  For the symplectic and
orthogonal towers we use the corresponding Harish--Chandra-compatible primary
Jordan parametrization.  When $q$ is odd this is the parametrization of
\cite[Theorem~4.4]{LSZ25}, with the full-orthogonal conventions used there.
When $q$ is even, we use the corresponding parametrization in terms of quadratic forms: only $X-1$ remains and the same two-runner symbol graph and bubble recursion apply.  Thus in every defining characteristic an
irreducible ordinary character has a semisimple parameter $s$ and primary
labels $\mu=(\mu_\Gamma)_\Gamma$, consisting of partitions on the type-$A$
factors and Lusztig symbols on the classical factors.  Denote
the parameter set by $\Psi(s)$ and the character by $\rho_{s,\mu}$.

For the odd orthogonal tower (with \(q\) odd),
\[
 \O_{2n+1}(q)=\SO_{2n+1}(q)\times\langle-I_{2n+1}\rangle.
\]
Accordingly the Jordan parameter has one additional sign
\(\varepsilon\in\{\pm1\}\), characterized by
\(\rho_{s,\mu,\varepsilon}(-I_{2n+1})=\varepsilon\,\rho_{s,\mu,\varepsilon}(1)\),
and
\[
 \Psi(s)\times\{\pm1\}\xrightarrow{\sim}\mathcal E(\O_{2n+1}(q),(s)),
 \qquad(\mu,\varepsilon)\longmapsto\rho_{s,\mu,\varepsilon}.
\]
This is the extra \(\{\pm1\}\)-factor of
\cite[\S4.1.7]{LSZ25}.  We suppress \(\varepsilon\) only when it plays no role.

\subsubsection{Harish--Chandra induction}

Let $M$ be a split
$F$-stable Levi subgroup of $G$. Then $M^*$ is a split $F^*$-stable Levi subgroup
of $G^*$. Let $s$ be an $F^*$-stable semisimple element of $M^*$.
The Jordan decomposition of characters
 commutes with Lusztig induction, and in particular Harish--Chandra induction.
 Hence the computation of Harish--Chandra induction $R_{M}^G$
 reduces to Harish--Chandra induction
 $R_{C_{M^*}(s)^*}^{C_{G^*}(s)^*}$ with unipotent characters.

Under Lusztig's parametrization, the functor \(F^\Gamma\) changes only the
\(\Gamma\)-primary unipotent coordinate, apart from the extension data in the
full orthogonal towers, which are transported as described below.  Write
\(\Delta\) for an irreducible representative of \(\Gamma\).  If
\(G_r=\O_{2r+1}(q)\), then under the standard Levi embedding
\[
 L_{r,d_\Delta}\cong\O_{2r+1}(q)\times\GL_{d_\Delta}(q)
 \hookrightarrow\O_{2(r+d_\Delta)+1}(q)
\]
one has
\[
 -I_{2(r+d_\Delta)+1}
 \longleftrightarrow
 (-I_{2r+1},-I_{d_\Delta}).
\]
Hence the inducing module \(\rho_{s,\mu,\varepsilon}\boxtimes\M_\Delta\)
has central sign \(\varepsilon\,\omega_\Delta(-I_{d_\Delta})\).  Since the left-hand
scalar matrix is central in the larger odd orthogonal group, every irreducible
constituent of \(F^\Gamma(\rho_{s,\mu,\varepsilon})\) has sign
\begin{equation}\label{eq:Oodd-sign-transport}
 \varepsilon' = \varepsilon\,\omega_\Delta(-I_{d_\Delta}).
\end{equation}
The same rule holds for \(E^\Gamma\).

For $\Gamma\in\mathcal F^G(\k)$, with irreducible representative $\Delta$, put
\[
 \sigma_\Gamma=
 \begin{cases}
  1,&\Delta\ne\Delta^*,\\
  (-1)^{d_\Delta},&\Delta=\Delta^*.
 \end{cases}
\]
When $q$ is odd, let $\zeta$ be the quadratic character of
$\bbF_q^\times$ and define
\[
 \Gamma^-(X):=(-1)^{d_\Gamma}\Gamma(-X).
\]
Thus $\Gamma^-$ is the index for which, for irreducible representatives,
\(\M_{\Delta^-}\cong(\zeta\circ\det)\otimes\M_\Delta\), where
\(\Delta^-(X)=(-1)^{d_\Delta}\Delta(-X)\).  If $\Delta=\Delta^*$, set
\[
 \eta_\Gamma:=\omega_\Delta(\delta I_{d_\Delta})\in\{\pm1\},
 \qquad \delta\in\bbF_q^\times\text{ nonsquare},
\]
and set $\eta_\Gamma=1$ if $\Delta\ne\Delta^*$.  The value is independent
of the choice of $\delta$, since the central character of a self-dual
cuspidal module has order at most two.  We choose the signs in
Definition~\ref{def:cuspidal-extension} compatibly on the orbits of
$\Gamma\mapsto\Gamma^-$.

The following lemma records the extra symmetries.
\begin{lemma}\label{lem:extra-symmetries}
Let $\rho\in\Irr(G_n)$ and let $\Gamma\in\calI(\bbK)$.
\begin{enumerate}
\item If $q$ is odd and $G_n$ is orthogonal, then
\[
 \bbO^\Gamma(u)(\mathrm{det}_n\otimes\rho)
 =\bbO^\Gamma(\sigma_\Gamma u)(\rho).
\]
If $q$ is even and $G_n$ belongs to a full even orthogonal tower, put
$\delta_n(g)=(-1)^{D_n(g)}$, where
$D_n(g)=\operatorname{rank}(g-1)\pmod2$.  Then the same formula holds with
$\mathrm{det}_n$ replaced by $\delta_n$.

\item Suppose that $q$ is odd and $G_n$ is orthogonal.  If $\sp_n$ denotes
the spinor character, then
\[
 \bbO^\Gamma(u)(\sp_n\otimes\rho)
 =\bbO^{\Gamma^-}(u)(\rho).
\]

\item Suppose that $q$ is odd and that $G_n$ is symplectic or full even
orthogonal.  Let $c_n$ be the diagonal automorphism defined by conjugation
with a similitude of nonsquare multiplier.  Then
\[
 \bbO^\Gamma(u)(\rho^{c_n})
 =\bbO^\Gamma(\eta_\Gamma u)(\rho).
\]
\end{enumerate}
Equivalently, on an $r$-dot bubble the first and third symmetries act by
$\sigma_\Gamma^r$ and $\eta_\Gamma^r$, respectively, while the spinor
symmetry carries the $\Gamma$-bubble to the $\Gamma^-$-bubble.  These
formulas also determine every product of the extra symmetries.
\end{lemma}
\begin{proof}
The functorial identifications are those of \cite[\S3]{LSZ25}.  Their
extension to the present colors follows directly from the restriction of
the three characters and automorphisms to a Witt Levi.  For odd $q$,
\[
 \mathrm{det}_{n+d}|_{G_n\times\GL_d(q)}
   =\mathrm{det}_n\boxtimes1,
 \qquad
 \sp_{n+d}|_{G_n\times\GL_d(q)}
   =\sp_n\boxtimes(\zeta\circ\det).
\]
In even characteristic the unipotent radicals lie in the kernel of the
Dickson invariant, and on the Levi the two opposite general-linear blocks
give
\[
 \operatorname{rank}(A-1)+\operatorname{rank}(A^{-T}-1)
 =2\operatorname{rank}(A-1).
\]
Hence
$\delta_{n+d}|_{\O(V_n)\times\GL_d(q)}=\delta_n\boxtimes1$.

For a paired color, the dot is formed from the two unipotent radicals and is
unchanged by the determinant or Dickson twist.  For a self-dual color, the
opposition representative has determinant $(-1)^d$; in even characteristic
its Dickson invariant is $d\pmod2$.  This gives $\sigma_\Gamma$.  The
spinor restriction formula carries $\M_\Delta$ to $\M_{\Delta^-}$, and the
compatible choices of the cuspidal intertwiners carry $\mathrm{X}^\Gamma$ to
$\mathrm{X}^{\Gamma^-}$.

For the diagonal automorphism choose compatible similitudes at ranks $n$
and $n+d$.  Their quotient on the added hyperbolic block has general-linear
component $\delta I_d$.  It fixes the paired-color dot.  On the opposition
dot for a self-dual color it acts by
$\omega_\Delta(\delta I_d)=\eta_\Gamma$.  Closing the $r$-th power of the
dot gives the stated factors and hence the substitutions in the bubble
series.  For $\Gamma=X-1,X+1$ one has, respectively,
\[
 \sigma_\Gamma=-1,\qquad
 \Gamma^- =X+1,X-1,\qquad
 \eta_\Gamma=1,-1,
\]
which recovers \cite[Proposition~4.11]{LSZ25}.  The same calculation shows
that $\sigma_\Gamma=-1$ only for $X\pm1$ when $q$ is odd, and only for
$X-1$ when $q$ is even.  This proves the stated transformation laws.  A
transformation law alone does not imply that a twist is separated when the
bubble series is even; separation is established separately in the $q$-odd
parametrization used below.
\end{proof}

When $q$ is odd, for a paired index the scalar in
\eqref{eq:Oodd-sign-transport} is independent of the choice of irreducible representative: the
opposite factor has inverse central character, and at \(-I_{d_\Delta}\) the
inverse value is the same.  In even defining characteristic the Dickson
character transports the same two-runner coordinate of the Lusztig symbol.  Suppressing
the semisimple label and the extension transport, the branching formula is
\[
 F^\Gamma(\rho_{\mu})=\sum_{\nu}\rho_{\nu},
\]
where the sum runs over the primary-label tuples \(\nu\) satisfying
\(\nu_{\Gamma'}=\mu_{\Gamma'}\) for every \(\Gamma'\ne\Gamma\), while
\(\nu_\Gamma\) is obtained from \(\mu_\Gamma\) by one rank-one
Harish--Chandra branching step in the local \(\Gamma\)-primary factor.
When $q$ is odd, for the two
rank-one functors of \cite[(4.21)--(4.22)]{LSZ25}, this specializes to
multiplication of the sign by \(1\) and by \(\zeta(-1)\), respectively.
Equivalently, after passing
to the standard local Fock-space coordinate, one adds a single Fock node; for
a level-$1$ type-$A$ factor this is the usual addition of one box, while for a
level-$2$ factor it is read in the corresponding two-runner, bipartition, or
symbol coordinate.
\begin{lemma}\label{Lem:cuspidal}
Work over $\bbK$, using the Lusztig-symbol convention above when $q$ is
even.  Let $\rho\in\Irr(G_n)$.
\begin{itemize}
    \item[$(1)$] If $G_n=\GL_n(q)$, then $\rho$ is
 $F^\Gamma$-cuspidal for every $\Gamma\in\calI(\bbK)$ if and only if
 $n=0$ and $\rho={\bf 1}_{\GL_0(q)}$.

 \item[$(2)$] If $G_n$ belongs to the unitary, symplectic, or orthogonal
 tower, then $\rho$ is $F^\Gamma$-cuspidal for every
 $\Gamma\in\calI(\bbK)$ if and only if $\rho$ is a cuspidal character of
 $G_n$ (for the fixed Witt tower).
 \end{itemize}
\end{lemma}
\begin{proof}
By adjunction, \(E^\Gamma(\rho)\ne0\) if and only if \(\rho\) occurs as a
composition factor of \(F^\Gamma(\sigma)\) for some representation \(\sigma\)
of smaller rank.  The branching rule above says exactly that this happens when
the \(\Gamma\)-primary label admits the corresponding removable rank-one
move.  Hence simultaneous \(F^\Gamma\)-cuspidality is equivalent to the
absence of every such removable primary factor.

For \(\GL_n(q)\), every non-zero-rank simple module has a non-empty primary
label and hence is obtained from a smaller rank by removing the last creation
step associated with one of the indices $\Gamma$ in its primary decomposition; the only simultaneous highest
weight object in the full colored tower is therefore the vacuum
\(\mathbf1_{\GL_0(q)}\).  For the other classical Witt towers, the absence of
all general-linear rank-one removals is precisely the generalized
Harish--Chandra cuspidality condition.  Thus the simultaneous
\(F^\Gamma\)-cuspidal simples are exactly the ordinary cuspidal characters,
as claimed.
\end{proof}

\subsubsection{Kac--Moody algebra attached to the dot spectrum}
For \(\k\in\{\bbK,\bbF\}\) and \(\Gamma\in\calI(\k)\), let
\(m^\Gamma_\rho(u)\) and \(n^\Gamma_\rho(u)\) be the monic minimal
polynomials of the dots on \(F^\Gamma\rho\) and \(E^\Gamma\rho\),
respectively; the polynomial of the endomorphism of the zero object is \(1\).
Let \(I_\Gamma(\k)\) be the union of the root sets of \(m^\Gamma_\rho(u)\),
as \(\rho\) runs through the irreducible objects of
\(\k G_\bullet\mathrm{-mod}\).  When \(\k=\bbK\) we write simply
\(I_\Gamma\).  Here the upward and downward functors are named
\(F^\Gamma\) and \(E^\Gamma\), respectively, so their names are reversed
relative to Section~\ref{sec:HeistoKac}; the upward and downward minimal
polynomials are still denoted by \(m\) and \(n\).

For the full color set put
\[
 \widehat X:=\prod_{\Gamma\in\calI(\k)}X_\Gamma,
 \qquad
 X_\Gamma:=\bigoplus_{i\in I_\Gamma(\k)}\mathbb Z\Lambda_i.
\]
Thus each color coordinate has finite support in its local spectrum, but no
finite-support condition is imposed across the set of colors.  For an
irreducible object $L$, write $\epsilon_i^\Gamma(L)$ and
$\phi_i^\Gamma(L)$ for the multiplicities of $i$ in the minimal polynomials
of the $\Gamma$-dot on $E^\Gamma L$ and $F^\Gamma L$, respectively,
so $\epsilon_i^\Gamma(L)=\operatorname{ord}_{u=i}n_L^\Gamma(u)$ and
$\phi_i^\Gamma(L)=\operatorname{ord}_{u=i}m_L^\Gamma(u)$, and define
\[
 \operatorname{wt}_\Gamma(L)
 :=\sum_{i\in I_\Gamma(\k)}
   (\phi_i^\Gamma(L)-\epsilon_i^\Gamma(L))\Lambda_i,
 \qquad
 \operatorname{wt}(L):=(\operatorname{wt}_\Gamma(L))_\Gamma
 \in\widehat X.
\]
For $\lambda\in\widehat X$, let $\mathcal C_\lambda$ be the Serre
subcategory generated by the simples of full weight $\lambda$.  If two
simples have different full weights, they differ in some color $\Gamma$;
the one-color weight-decomposition argument of Section~\ref{sec:HeistoKac}
then gives different bubble central characters.  Hence they lie in different
blocks, and therefore
\[
 \k G_\bullet\mathrm{-mod}
 =\bigoplus_{\lambda\in\widehat X}\mathcal C_\lambda.
\]
Moreover, Lemma~\ref{pg}, applied to any finite color set containing
$\Gamma$, shows that $E_i^\Gamma$ and $F_i^\Gamma$ change only the
$\Gamma$-coordinate, by $+\alpha_i$ and $-\alpha_i$, respectively,
after the preceding exchange of functor names and negation of the
Section~\ref{sec:HeistoKac} weight.  Thus the present weight is
$\sum_i\operatorname{ord}_{u=i}(m_L^\Gamma/n_L^\Gamma)\Lambda_i$ in
each color.  We write
\(
\frakU_{\widehat X}(\bigoplus_\Gamma
\fraks\frakl'_{I_\Gamma(\k)})
\)
for the strict 2-category presented by the object set $\widehat X$,
the local generators $E_i^\Gamma,F_i^\Gamma$, and all Kac--Moody
relations supported on finite sets of colors.  Equivalently, it is the
filtered 2-colimit of the finite-color 2-categories under the canonical
generator-preserving 2-functors; every 1- and 2-morphism has finite color
support.  Its
root lattice is the direct sum of the local root lattices, whereas its object
lattice is the product $\widehat X$.

\begin{theorem}\label{ThmB}
Let \(\k\in\{\bbK,\bbF\}\).  With the spectra \(I_\Gamma(\k)\), the
generalized eigenfunctors of the colored dots define a strict \(\k\)-linear
Kac--Moody 2-representation
\[
\Psi':\frakU_{\widehat X}\!\left(
\bigoplus_{\Gamma\in\calI(\k)}
\fraks\frakl'_{I_\Gamma(\k)}\right)
\longrightarrow
\mathfrak{Cat}_\k,\qquad \lambda\longmapsto\mathcal C_\lambda.
\]
\end{theorem}
\begin{proof}
Theorem~\ref{Thm:quantumheisenberg} makes
\(\k G_\bullet\mathrm{-mod}\) a locally finite module category over the
colored symmetric product of quantum Heisenberg categories.  Section~\ref{sec:HeistoKac}
is stated for finitely many colors.  Its construction extends to the full
color set because every generating 2-morphism and every defining relation involves only finitely
many colors (in fact at most three).  The full weight decomposition above
makes the object part of the action well defined.  For each finite
\(J\subset\calI(\k)\), the theorem of Section~\ref{sec:HeistoKac} gives
the required 2-representation on the corresponding coordinates; these
structures agree under inclusions \(J\subset J'\), since they use the same
eigenfunctors, weight shifts, and local diagrams.  Consequently the local
actions define the displayed 2-representation on the full weight set
\(\widehat X\) over either coefficient field.
\end{proof}

\subsubsection{Deligne--Lusztig extension and outer-class input}
For a self-dual cuspidal character
\(\chi_\theta=\pm R_{T_c}^{\GL_d}(\theta)\), with \(T_c\) a Coxeter torus,
we use the non-connected Deligne--Lusztig extension
\[
 R_{T_c\sigma}^{\GL_d\sigma}(\widetilde\theta),
\]
whose value is the trace of \(g\sigma\) on the corresponding
Deligne--Lusztig cohomology with the \(\widetilde\theta\)-isotypic torus
factor; see \cite{DM94,Shu22}.  On a supported component, for \(u\) unipotent
in the connected centralizer of \(\sigma\), the character formula has the
form
\begin{equation}\label{eq:DL-extension-Green}
 \widetilde\chi_\theta(\sigma u)
   =\widetilde\theta(\sigma)\,Q_\sigma(u),
\end{equation}
where \(Q_\sigma\) is the corresponding Green function.  After fixing the
extension sign, \(\widetilde\theta(\sigma)\) is fixed; hence the values used
below on unipotent classes are independent of the regular torus character
\(\theta\), equivalently of the self-dual color.  This Deligne--Lusztig
statement is valid for every finite field, including characteristic two.
The assumption that \(q\) is odd enters below only through Feit's concrete
outer-class description and the matrix model used for the numerical sum.

For that odd-characteristic calculation put
\[
 \tau(g)=g^{-\mathsf T},\qquad
 H_{2n}^{\mathsf T}(q)=\GL_{2n}(q)\rtimes\langle\tau\rangle.
\]

We use \cite[Lemmas~3E(ii) and~3F]{Feit} for the two outer-class facts below and \cite[Theorem~10J]{Feit}, together with \cite{Shu22}, for the extension-character support formula.

\begin{lemma}\label{3E}
Assume \(q\) is odd.  The elements \(\tau x\) with
\(x=-x^{\mathsf T}\in\GL_{2n}(q)\) form the distinguished order-four
outer class.  Moreover, if \(x\) is nonsingular skew-symmetric,
\(c\in\mathbb F_q^\times\), and \(v\ne0\), then
\(\tau(x+cvv^{\mathsf T})\) is conjugate to \(\sigma_0u\), where
\(\sigma_0\) is in this order-four class and \(u\) is a commuting
transvection.
\end{lemma}

\begin{lemma}\label{thm:ext}
Assume \(q\) is odd and let \(\widetilde\chi_\theta\) be an extension of a
\(\tau\)-stable cuspidal character of \(\GL_{2n}(q)\).  On the outer
classes in Feit's support,
\[
 \widetilde\chi_\theta(\sigma x)
   =\widetilde\theta(\sigma)\,\zeta_\theta(x)
\]
for the corresponding cuspidal character \(\zeta_\theta\) of
\(Z_G(\sigma)\); on the outer classes excluded by the support criterion the
extension character vanishes.
\end{lemma}

\begin{remark}
When \(q\) is even, the non-connected Deligne--Lusztig construction gives the
cuspidal extension, and \eqref{eq:DL-extension-Green} remains valid.  The only linear classical elementary divisor is then \(X-1\).  Feit's outer-class
description in Lemmas~\ref{3E}--\ref{thm:ext} is used only for odd $q$.
\end{remark}

\subsubsection{Bubble normalization}
\label{sec:normalization}

This subsection compares the categorical normalization with the numerical
character calculation.  Assume throughout that $\Delta^*=\Delta$.
Thus $\Gamma=\Delta$, and we use $\Delta$ throughout the character and matrix calculations in this subsection.
We write $\widetilde\rho_\Delta$ and $\widetilde\chi_\Delta$ for the
chosen extension of Definition~\ref{def:cuspidal-extension} and its
character.  By Remark~\ref{lem:extension-sign-independence}, all formulas
below are unchanged if the opposite extension is used together with the
opposite sign of $\beta_\Delta$.  Put
\[
d=d_\Delta,\qquad
\dot{x}=\dot{x}_{r,d},\qquad
P=P_{r,d},\qquad
V=V_{r,d},\qquad e=e_V,
\]
and let $\sigma=\sigma_d$ be the generator of the crossed extension
\eqref{eq:twisted-extension}.  For $a\in P$, write
\[
a=b(a)\times c(a)\in G_r\times\GL_d(q)
\]
for its Levi projection.

\begin{lemma}\label{Lem:Key}
Let $V_\chi\in\Irr(G_r)$ be irreducible and $F^\Delta$-cuspidal, with
character $\chi$, and assume the corresponding index $\Gamma=\Delta$ lies in $\calI_{-2}(\bbK)$.  Then
\begin{equation*}
\bbO^\Delta(u)(V_\chi)
=u^2+\mu_\Delta(V_\chi)u+(t^\Delta)^2,
\end{equation*}
where
\begin{equation}\label{eq:mu-exact}
\mu_\Delta(V_\chi)
=-\frac{\beta_\Delta q^{rd}\omega_\Delta(\kappa_\Delta)}
{|V|\chi(1)\chi_\Delta(1)}
\sum_{\substack{u,u'\in V\\
\dot{x}u\dot{x}u'\dot{x}\in P}}
\chi\!\left(b(\dot{x}u\dot{x}u'\dot{x})\right)
\widetilde\chi_\Delta\!\left(
 c(\dot{x}u\dot{x}u'\dot{x})\sigma\right).
\end{equation}
\end{lemma}

\begin{proof}
By Proposition~\ref{Prop:dim2}, on $F^\Delta(V_\chi)$ one has
\begin{equation}\label{eq:X-quadratic-correct}
(\mathrm{X}^\Delta)^2
=-\mu_\Delta(V_\chi)\mathrm{X}^\Delta-(t^\Delta)^2\id.
\end{equation}
Close the free $\Delta$-strand by the adjunction cup and cap.  By
definition, the closure of $(\mathrm{X}^\Delta)^m$ is the clockwise $\Delta$-bubble
with $m$ dots.  In the following displayed identities, every bubble is
understood as its scalar action on $V_\chi$.  Multiplying
\eqref{eq:X-quadratic-correct} by $\mathrm{X}^\Delta$ and then closing gives
\[
\coloredbubble{3}{\Delta}
=-\mu_\Delta(V_\chi)\coloredbubble{2}{\Delta}
-(t^\Delta)^2\coloredbubble{1}{\Delta}.
\]
The bubble relations give
\[
\coloredbubble{1}{\Delta}=0,
\qquad
\coloredbubble{2}{\Delta}=\frac{t^\Delta}{z^\Delta},
\]
hence
\begin{equation}\label{eq:mu-bubble-three}
\mu_\Delta(V_\chi)
=-\frac{z^\Delta}{t^\Delta}\coloredbubble{3}{\Delta}.
\end{equation}

The three-dot clockwise $\Delta$-bubble on $V_\chi$ is computed from
the representing bimodules.  The right unit contributes
$-(t^\Delta z^\Delta)^{-1}$, every dot contributes
$\beta_\Delta q^{rd}e\dot{x}e$, and
\[
e\dot{x}e\dot{x}e\dot{x}e
=\frac1{|V|^2}\sum_{u,u'\in V}
e\dot{x}u\dot{x}u'\dot{x}e.
\]
The left counit kills the summands for which
$\dot{x}u\dot{x}u'\dot{x}\notin P$ and takes the Levi projection on the
remaining summands.  Therefore
\[
\coloredbubble{3}{\Delta}
=-\frac{(\beta_\Delta q^{rd})^3}
{t^\Delta z^\Delta |V|^2\chi(1)\chi_\Delta(1)}
\sum_{\substack{u,u'\in V\\
\dot{x}u\dot{x}u'\dot{x}\in P}}
\chi\!\left(b(\dot{x}u\dot{x}u'\dot{x})\right)
\widetilde\chi_\Delta\!\left(
 c(\dot{x}u\dot{x}u'\dot{x})(\sigma)^3\right).
\]
Here the factor $\chi_\Delta(1)^{-1}$ is the contraction on the cuspidal
factor.  Since $\sigma^2=\kappa_\Delta$ in the crossed extension,
\[
\widetilde\chi_\Delta\!\left(c\sigma^3\right)
=\omega_\Delta(\kappa_\Delta)
\widetilde\chi_\Delta\!\left(c\sigma\right).
\]
Substituting this identity into \eqref{eq:mu-bubble-three} and using
\eqref{eq:beta-normalization} gives exactly \eqref{eq:mu-exact}.
\end{proof}

\medskip
\noindent
For the bottom matrix calculation, put $r=0$ and write $a=\dim(V_a)$.  The block-matrix formulas below
are written in terms of Gram matrices for the underlying sesquilinear or
bilinear form.  They apply to the unitary cases in every characteristic and
to the orthogonal cases in odd characteristic.  In characteristic two, the
$X-1$ normalization in Proposition~\ref{prop:all-bottom-bubbles} follows from
the rank-one Witt-parabolic calculation.  The matrix calculation below is
used only in the indicated odd-characteristic orthogonal cases.
  When $a=0$, set
\[
c_G=\begin{cases}
1,&G=\Sp_{2d},\\
-1,&G=\GU_{2d}\text{ or }\O^+_{2d},
\end{cases}
\qquad
\dagger=\begin{cases}
\mathsf H,&G=\GU_{2d},\\
\mathsf T,&G=\Sp_{2d}\text{ or }\O^+_{2d},
\end{cases}
\]
and define
\begin{equation}\label{eq:Y-G-space}
\mathscr Y_G(d):=
\{Y\in M_d(\bbF_q)\mid Y^\dagger J_d=c_GJ_dY\},
\qquad
\mathscr Y_G(d)^\times:=\mathscr Y_G(d)\cap\GL_d(q).
\end{equation}
Then the map
\[
Y\longmapsto u(Y):=\begin{pmatrix}I_d&Y\\0&I_d\end{pmatrix}
\]
identifies $\mathscr Y_G(d)$ with $V_{0,d}$.  Thus the unitary,
symplectic, and split-orthogonal form conditions are respectively
$Y^\mathsf HJ_d=-J_dY$, $Y^\mathsf TJ_d=J_dY$, and
$Y^\mathsf TJ_d=-J_dY$.
For $a\in\{1,2\}$ write
\[
u=\begin{pmatrix}
I_d&X_u&Y_u\\
0&I_a&Z_u\\
0&0&I_d
\end{pmatrix}.
\]
Take $J_a=(1)$ when $a=1$.  For $G=\O^-_{2d+2}(q)$, fix a
nonsquare $\delta\in\bbF_q^\times$ and take
\[
J_a=J_2^-:=\diag(1,-\delta).
\]
With $\dagger=\mathsf{H}$ in the unitary case and
$\dagger=\mathsf{T}$ in the orthogonal case, the form equations are
\begin{equation}\label{eq:bottom-form-eqs}
X_u=-J_d^{-1}Z_u^\dagger J_a,
\qquad
Y_u^\dagger J_d+J_dY_u+Z_u^\dagger J_aZ_u=0.
\end{equation}
Set
\[
W_u=Y_u-X_uZ_u.
\]
Then
\begin{equation}\label{eq:W-adjoint}
W_u=-J_d^{-1}Y_u^\dagger J_d,
\end{equation}
so $W_u$ is invertible if and only if $Y_u$ is invertible.

\begin{lemma}\label{lem:xuxvx-all-cases}
Let $\dot{x}=\dot{x}_{0,d}$, $V=V_{0,d}$ and $P=P_{0,d}$.
\begin{enumerate}
\item If $a=0$, with $c_G$ and $\mathscr Y_G(d)$ as in
\eqref{eq:Y-G-space}, one has
\[
\dot{x}u(Y_u)\dot{x}u(Y_v)\dot{x}\in P
\iff
Y_u\in\GL_d(q)\text{ and }Y_v=c_GY_u^{-1},
\]
and
\begin{equation}\label{eq:bottom-product-a0}
\dot{x}u(Y_u)\dot{x}u(c_GY_u^{-1})\dot{x}
=
\begin{pmatrix}
-c_GY_u^{-1}&c_GI_d\\
0&-Y_u
\end{pmatrix}.
\end{equation}
With our Levi convention this gives
\begin{equation}\label{eq:bottom-Levi-a0}
c\!\left(\dot{x}u(Y_u)\dot{x}u(c_GY_u^{-1})\dot{x}\right)
=-c_GY_u^{-1}.
\end{equation}
In particular, the upper-left block in the symplectic case is
$-Y_u^{-1}$.

\item If $a\in\{1,2\}$, then
\[
\bigl(\exists v\in V:\ \dot{x}u\dot{x}v\dot{x}\in P\bigr)
\iff W_u\in\GL_d(q).
\]
When this holds, the contributing $v$ is unique and satisfies
\begin{equation}\label{eq:bottom-v-solution}
Y_v=-W_u^{-1},\qquad
Z_v=-Z_uW_u^{-1},\qquad
X_v=Y_u^{-1}X_u.
\end{equation}
Moreover,
\begin{equation}\label{eq:bottom-product-a12}
\dot{x}u\dot{x}v\dot{x}
=
\begin{pmatrix}
W_u^{-1}&Y_u^{-1}X_u&-I_d\\
0&I_a-Z_uY_u^{-1}X_u&Z_u\\
0&0&-Y_u
\end{pmatrix}.
\end{equation}
Thus the Levi components relevant to Lemma~\ref{Lem:Key} are
\begin{equation}\label{eq:bottom-Levi-a12}
c(\dot{x}u\dot{x}v\dot{x})=W_u^{-1},\qquad
b(\dot{x}u\dot{x}v\dot{x})=I_a-Z_uY_u^{-1}X_u.
\end{equation}
\end{enumerate}
\end{lemma}

\begin{proof}
For $a=0$ write
\[
\dot{x}_c=\begin{pmatrix}0&-I_d\\cI_d&0\end{pmatrix},
\qquad
c=\begin{cases}1,&G=\Sp,\\-1,&G=\GU,\O^+.
\end{cases}
\]
A direct multiplication gives
\[
\dot{x}_cu(Y_u)\dot{x}_cu(Y_v)\dot{x}_c
=
\begin{pmatrix}
-Y_v&cI_d\\
cY_uY_v-I_d&-Y_u
\end{pmatrix}.
\]
Thus the lower-left block vanishes exactly when
$Y_v=cY_u^{-1}$, proving \eqref{eq:bottom-product-a0}.

For $a\in\{1,2\}$ direct block multiplication gives
\[
\dot{x}u\dot{x}v\dot{x}
=
\begin{pmatrix}
-Y_v&X_v&-I_d\\
Z_uY_v-Z_v&I_a-Z_uX_v&Z_u\\
-(Y_uY_v-X_uZ_v+I_d)&Y_uX_v-X_u&-Y_u
\end{pmatrix}.
\]
Membership in $P$ is equivalent to
\[
Z_v=Z_uY_v,
\qquad
Y_uX_v=X_u,
\qquad
(Y_u-X_uZ_u)Y_v=-I_d.
\]
Using \eqref{eq:W-adjoint}, these equations have a solution exactly when
$W_u$ is invertible, and then the solution is
\eqref{eq:bottom-v-solution}.  Substitution proves
\eqref{eq:bottom-product-a12}.
\end{proof}

\begin{lemma}\label{lem:Ominus-rank-two-sum}
Assume \(q\) is odd, \(G_a=\O^-_2(q)\), \(n\ge2\), and
\(\Delta=\Delta^*\in\calF_1^{\rm G}(\bbK)\) has degree \(d=2n\).  Put
\(P_j=\prod_{i=1}^j(q^{2i}-1)\), \(P_0=1\), and identify the split crossed
extension with \(\GL_d(q)\rtimes\langle\tau\rangle\) as above.  Normalize
its sign by
\begin{equation}\label{eq:Ominus-extension-normalization}
 \widetilde\chi_\Delta(\sigma_0)=(q^n-1)P_{n-1}.
\end{equation}
Then
\begin{equation}\label{eq:Ominus-total-sum}
\Sigma_\Delta^-:=
\sum_{\substack{u,u'\in V_{0,d}\\
\dot{x}u\dot{x}u'\dot{x}\in P_{0,d}}}
\widetilde\chi_\Delta\!\left(
 c(\dot{x}u\dot{x}u'\dot{x})\sigma\right)
=q^{n^2+n}(q^n-1)\chi_\Delta(1).
\end{equation}
Consequently
\begin{equation}\label{eq:beta-Ominus-sign}
 \beta_\Delta=q^{n^2+n},
\end{equation}
and \(\mu_\Delta({\bf1}_{\O^-_2})=q^{-n}-1\).
\end{lemma}

\begin{proof}
Put \(J_2^-=\operatorname{diag}(1,-\delta)\).  For \(u\in V_{0,d}\), set
\[
 S_Z=\tfrac12Z^TJ_2^-Z,\qquad A=J_dY+S_Z,\qquad C=J_dW=A+S_Z.
\]
Equation~\eqref{eq:bottom-form-eqs} gives \(A^T=-A\), and
\(J_d^2=I_d\) gives \(C^{-1}=W^{-1}J_d\).
Lemma~\ref{lem:xuxvx-all-cases} therefore rewrites the sum as
\[
 \Sigma_\Delta^-=
 \sum_{\substack{A^T=-A,\ Z\in M_{2,d}(q)\\
 C=A+\frac12Z^TJ_2^-Z\in\GL_d(q)}}
 \widetilde\chi_\Delta(C^{-1}\tau).
\]
Let \(\zeta_\Delta\) be the cuspidal character of \(\Sp_{2n}(q)\) in
Feit's extension formula for the order-four \(2\)-part \(\sigma_0\).
Lemma~\ref{thm:ext}, with \eqref{eq:Ominus-extension-normalization}, gives
\(\widetilde\chi_\Delta(\sigma_0x)=\zeta_\Delta(x)\).
The required Coxeter Green-function values are
\begin{align}
\zeta_\Delta(1)&=(q^n-1)P_{n-1},\notag\\
\zeta_\Delta(u_1)&=-P_{n-1},
\label{eq:Green-rank-one}\\
\zeta_\Delta(u_{2,-})&=(q^{n-1}+1)P_{n-2}.
\label{eq:Green-rank-two-minus}
\end{align}
Here \(u_1\) is a transvection and \(u_{2,-}\) is the anisotropic rational
class of type \((2^2,1^{2n-4})\); see \cite[Corollary~10K]{Feit} for the
first two values and \cite[\S38]{Lus76Green} for the third.
By \eqref{eq:DL-extension-Green}, these values depend only on \(q\) and
the unipotent class, not on \(\Delta\).  This Green-function statement
also holds in characteristic two, although the present matrix realization
uses odd characteristic.  In particular, the rank-one value has the
opposite sign from the identity value.

The number of nondegenerate alternating matrices and the cuspidal degree
satisfy
\[
 N_A=\frac{|\GL_{2n}(q)|}{|\Sp_{2n}(q)|}
 =q^{n(n-1)}\prod_{i=1}^n(q^{2i-1}-1),\qquad
 N_AP_{n-1}=q^{n(n-1)}\chi_\Delta(1),
\]
since \(\chi_\Delta(1)=\prod_{j=1}^{2n-1}(q^j-1)\).

\smallskip
\noindent\emph{Rank zero.}
For \(Z=0\), \(C=A\) and \(C^{-1}\tau\) lies in the order-four class, so
\begin{equation}\label{eq:Sigma0}
\Sigma_0
=N_A(q^n-1)P_{n-1}
=q^{n(n-1)}(q^n-1)\chi_\Delta(1).
\end{equation}

\smallskip
\noindent\emph{Rank one.}
Now \(S_Z=cvv^T\) with \(c\ne0\).  Even corank prevents this rank-one
perturbation from making a singular alternating \(A\) invertible; for
invertible \(A\), the determinant lemma gives
\(\det(A+cvv^T)=\det A\), since \(v^TA^{-1}v=0\).
There are \((q+1)(q^{2n}-1)\) rank-one matrices \(Z\).
Lemma~\ref{3E}(ii) identifies the odd part of each supported element as a
transvection, and \eqref{eq:Green-rank-one} yields
\begin{equation}\label{eq:Sigma1}
\Sigma_1
=-q^{n(n-1)}(q+1)(q^{2n}-1)\chi_\Delta(1).
\end{equation}

\smallskip
\noindent\emph{Rank two with $A$ nondegenerate.}
Write the rows of \(Z\) as \(z_1,z_2\), and put \(P=A^{-1}\),
\(r=z_1Pz_2^T\).  Thus
\(ZPZ^T=\left(\begin{smallmatrix}0&r\\-r&0\end{smallmatrix}\right)\),
and Sylvester's identity gives
\[
 \det C=\det A\det\left(I_2+\tfrac12J_2^-ZPZ^T\right)
       =\det A(1-\delta r^2/4)\ne0,
\]
because \(\delta\) is a nonsquare.
If \(r=0\), then \(K=PS_Z\) satisfies \(K^2=0\) and
\[
 (C^{-1}\tau)^2=C^{-1}C^T=-(I+K)^{-1}(I-K)=-(I-2K).
\]
The odd unipotent part has type \((2^2,1^{2n-4})\).  For
\(N=-PZ^TJ_2^-Z\), its rational form on \(\bbF_q^{2n}/\ker N\) is
\[
 b_N(\bar v,\bar w)=v^TANw=-(Zv)^TJ_2^-(Zw).
\]
Since \(Z\) induces an isomorphism of this quotient with \(\bbF_q^2\),
the form is anisotropic, giving \(u_{2,-}\).
For each \(A\), there are \((q^{2n}-1)(q^{2n-1}-q)\) such matrices
\(Z\): choose \(z_1\ne0\), then
\(z_2\in z_1^\perp\setminus\langle z_1\rangle\).
Consequently \eqref{eq:Green-rank-two-minus} gives
\begin{align}
\Sigma_2
&=N_A(q^{2n}-1)(q^{2n-1}-q)
  (q^{n-1}+1)P_{n-2}\notag\\
&=q^{n(n-1)}(q^{2n}-1)(q^n+q)
  \chi_\Delta(1).
\label{eq:Sigma2}
\end{align}

For \(r\ne0\), the nonzero eigenvalues of \(K\) are \(\pm\kappa\),
where \(\kappa^2=\delta r^2/4\) and \(\kappa^q=-\kappa\).
Set \(z=(1-\kappa)/(1+\kappa)\); then \(z^{q+1}=1\), \(z\ne1\),
and \((C^{-1}\tau)^2\) has eigenvalues
\(-z,-z^{-1},-1,\ldots,-1\).
If the \(2\)-part of \(z\) is nontrivial, the \(2\)-part of
\(C^{-1}\tau\) has mixed-type connected centralizer and lies outside the
support of Lemma~\ref{thm:ext}.
Otherwise \(C^{-1}\tau=\sigma_0x\), with \(x\ne1\) odd-order semisimple
in \(\Sp_{2n}(q)\) and fixed space of dimension at least \(2n-2\).
The Coxeter torus defining \(\zeta_\Delta\) acts irreducibly, so \(x\)
is not conjugate into it; the Deligne--Lusztig formula gives
\(\zeta_\Delta(x)=0\).  Thus all \(r\ne0\) terms vanish.

\smallskip
\noindent\emph{Singular $A$.}
If \(C\) is invertible, even corank and \(\rank S_Z\le2\) force
\(\operatorname{corank}A=\rank Z=2\).
For \(R=\ker A\), the map \(Z|_R:R\to\bbF_q^2\) is an isomorphism.
Adjust a complement \(U\) by vectors in \(R\) so that \(Z|_U=0\),
without changing \(A|_U\).  In suitable bases,
\[
 A=\begin{pmatrix}A_0&0\\0&0\end{pmatrix},\qquad
 Z=\begin{pmatrix}0&I_2\end{pmatrix},\qquad
 C=\begin{pmatrix}A_0&0\\0&\frac12J_2^-\end{pmatrix}.
\]
Here \(A_0\) is nondegenerate alternating.  The \(2\)-part of
\(C^{-1}\tau\) has an order-four block on \(U\) and an involutory
orthogonal block on \(R\), hence mixed centralizer
\(\Sp_{2n-2}\times\O^-_2\).  Lemma~\ref{thm:ext} again gives zero.

Combining \eqref{eq:Sigma0}--\eqref{eq:Sigma2} gives
\begin{align*}
\Sigma_\Delta^-
&=q^{n(n-1)}\chi_\Delta(1)
\Bigl[(q^n-1)-(q+1)(q^{2n}-1)
 +(q^{2n}-1)(q^n+q)\Bigr]\\
&=q^{n^2+n}(q^n-1)\chi_\Delta(1),
\end{align*}
proving \eqref{eq:Ominus-total-sum}.
Finally, \(|V_{0,d}|=q^{\binom d2+2d}=q^{2n^2+3n}\),
\(\kappa_\Delta=I_d\), and \((t^\Delta)^2=-q^{-n}\), so
\eqref{eq:beta-normalization} gives
\(\beta_\Delta^2=q^{-n}|V_{0,d}|=q^{2n^2+2n}\).
The sign fixed by \eqref{eq:Ominus-extension-normalization} and
\eqref{eq:beta-Ominus-sign}, together with Lemma~\ref{Lem:Key}, yields
\[
 \mu_\Delta({\bf1}_{\O^-_2})
 =-\frac{\beta_\Delta}{|V_{0,d}|\chi_\Delta(1)}\Sigma_\Delta^-
 =q^{-n}-1.
\]
\end{proof}

\begin{proposition}
\label{prop:normalization}
Let $\Delta=\Delta^*$ be a level-$2$ color and let $G_a$ be the bottom
member of the corresponding Witt tower.  Put $d=d_\Delta$ and
$V=V_{0,d}$.  The following are equivalent.
\begin{enumerate}
\item The chosen normalization gives
\[
\bbO^\Delta(u)({\bf1}_{G_a})
=(u-1)\bigl(u-(t^\Delta)^2\bigr).
\]
\item
\[
\mu_\Delta({\bf1}_{G_a})=-\bigl(1+(t^\Delta)^2\bigr).
\]
\item The explicit rank-zero class sum satisfies
\begin{equation}\label{eq:bottom-class-sum-criterion}
\frac{\beta_\Delta\omega_\Delta(\kappa_\Delta)}{|V|\chi_\Delta(1)}
\sum_{\substack{u,u'\in V\\
\dot{x}u\dot{x}u'\dot{x}\in P_{0,d}}}
\widetilde\chi_\Delta\!\left(
 c(\dot{x}u\dot{x}u'\dot{x})\sigma\right)
=1+(t^\Delta)^2.
\end{equation}
\end{enumerate}
Moreover, the double sum in \eqref{eq:bottom-class-sum-criterion}
reduces to a single sum.  If $a=0$, it is
\begin{equation}\label{eq:bottom-single-sum-a0}
\sum_{Y\in\mathscr Y_G(d)^\times}
\widetilde\chi_\Delta(-c_GY^{-1}\sigma).
\end{equation}
If $a\in\{1,2\}$, it is
\begin{equation}\label{eq:bottom-single-sum-a12}
\sum_{\substack{u\in V_{0,d}\\W_u\in\GL_d(q)}}
\widetilde\chi_\Delta(W_u^{-1}\sigma).
\end{equation}
\end{proposition}

\begin{proof}
By Lemma~\ref{Lem:Key},
\[
\bbO^\Delta(u)({\bf1}_{G_a})
=u^2+\mu_\Delta({\bf1}_{G_a})u+(t^\Delta)^2.
\]
Comparing with
\[
(u-1)(u-(t^\Delta)^2)
=u^2-\bigl(1+(t^\Delta)^2\bigr)u+(t^\Delta)^2
\]
proves the equivalence of (1) and (2).  For the trivial character of
$G_a$ and $r=0$, formula \eqref{eq:mu-exact} becomes
\[
\mu_\Delta({\bf1}_{G_a})
=-\frac{\beta_\Delta\omega_\Delta(\kappa_\Delta)}{|V|\chi_\Delta(1)}
\sum_{\substack{u,u'\in V\\
\dot{x}u\dot{x}u'\dot{x}\in P_{0,d}}}
\widetilde\chi_\Delta\!\left(
 c(\dot{x}u\dot{x}u'\dot{x})\sigma\right),
\]
so (2) and (3) are equivalent.

It remains only to remove the redundant variable $u'$.  For $a=0$,
write $u=u(Y)$ with $Y\in\mathscr Y_G(d)$.  By
Lemma~\ref{lem:xuxvx-all-cases}(1), a term contributes exactly when
$Y\in\mathscr Y_G(d)^\times$, in which case $u'$ is uniquely determined by
$Y_{u'}=c_GY^{-1}$; the general-linear Levi component is
\eqref{eq:bottom-Levi-a0}.  This gives
\eqref{eq:bottom-single-sum-a0}.  For $a\in\{1,2\}$,
Lemma~\ref{lem:xuxvx-all-cases}(2) says that a term contributes exactly
when $W_u$ is invertible, determines $u'$ uniquely, and
\eqref{eq:bottom-Levi-a12} gives the general-linear Levi component.
This proves \eqref{eq:bottom-single-sum-a12}.
\end{proof}

\begin{proposition}
\label{prop:all-bottom-bubbles}
Work over $\bbK$, and write $G_a=G_0$ for the bottom member of the fixed
Witt tower.  Then, for every $\Gamma\in\calI(\bbK)$, one has
\[
\bbO^\Gamma(u)({\bf1}_{G_a})=
\begin{cases}
 u-1,
   &\Gamma\in\calI_{-1}(\bbK),\\[2pt]
 (u-1)(u+q^{-d_\Gamma/2}),
   &G=\GU,\ \Gamma\in\calF_1^{\GU}(\bbK),\
    (G,\Gamma)\neq(\GU_{\mathrm{odd}},X-1),\\[2pt]
 (u-1)(u+q^{-d_\Gamma/2}),
   &G=\Sp,\O,\ \Gamma\in\calF_1^{\rm G}(\bbK),\\[2pt]
 (u+q^{1/2})(u-q^{-1}),
   &G=\GU_{\mathrm{odd}},\ \Gamma=X-1,\\[2pt]
 (u-1)(u+q^{-1}),
   &G=\Sp,\ \Gamma=X-1,\\[2pt]
 (u-1)(u+1),
   &G=\Sp,\ \Gamma=X+1,\\[2pt]
 (u-1)(u+q^{-1}),
   &G=\O_{\mathrm{odd}},\ \Gamma=X\pm1,\\[2pt]
 (u-1)(u+1),
   &G=\O^+_{\mathrm{even}},\ \Gamma=X\pm1,\\[2pt]
 (u-q)(u+q^{-1}),
   &G=\O^-_{\mathrm{even}},\ \Gamma=X-1,\\[2pt]
 (u-1)(u+1),
   &G=\O^-_{\mathrm{even}},\ \Gamma=X+1.
\end{cases}
\]
Equivalently, apart from the two anisotropic exceptions
\[
 (G,\Gamma)=(\GU_{\mathrm{odd}},X-1),\qquad
 (G,\Gamma)=(\O^-_{\mathrm{even}},X-1),
\]
for every level-$2$ index the value on the bottom object has the uniform form
\[
 \bbO^\Gamma(u)({\bf1}_{G_a})
 =(u-1)\bigl(u-(t^\Gamma)^2\bigr).
\]
\end{proposition}

\begin{proof}
The level-$1$ formula is the defining degree-one normalization.  It remains
to discuss level $-2$ colors.

First suppose either $G=\GU$ and $\Gamma\in\calF_1^{\GU}(\bbK)$, or $G=\Sp,\O$ and $\Gamma\in\calF_1^{\rm G}(\bbK)$.  The corresponding local Jordan factor is unitary.  The characteristic-zero rank-one unitary model has
eigenvalues on the cuspidal source
\[
 (-q^{d_\Gamma/2})^c,\qquad (-q^{d_\Gamma/2})^{-1-c}.
\]
for the unitary charge $c$; see \cite[Theorem~4.12]{DVV19}.  On the bottom
object the local primary multiplicity is zero, hence $c=0$, except in the
odd unitary tower for $\Gamma=X-1$.  Therefore, in every non-exceptional
self-dual case, the two roots are
\[
 1,\qquad -q^{-d_\Gamma/2},
\]
which gives
\[
 \bbO^\Gamma(u)({\bf1}_{G_a})
 =(u-1)(u+q^{-d_\Gamma/2}).
\]
In the nonsplit even orthogonal tower this formula also agrees with the
independent Deligne--Lusztig class-sum calculation whenever
$d_\Gamma=2n$ with $n\ge2$: Lemma~\ref{lem:Ominus-rank-two-sum}
gives \(\mu_\Gamma({\bf1}_{\O^-_2})=q^{-n}-1\), and
Proposition~\ref{prop:normalization} gives the same polynomial.  The cases
outside this auxiliary verification are already covered by the local-unitary
normalization used in the first paragraph of the proof.

Now let $G=\GU_{\mathrm{odd}}$ and $\Gamma=X-1$.  The bottom object is
${\bf1}_{\GU_1(q)}$, so the local unitary factor is the cuspidal
$\GU_1$-vertex, corresponding to charge $c=1$ (equivalently $c=-2$).
With the field-size convention $q=q_0^2$ used in this paper, \cite[Theorem~4.12]{DVV19} gives the unordered pair of eigenvalues
\[
 -q^{1/2},\qquad q^{-1},
\]
and hence
\[
 \bbO^{X-1}(u)({\bf1}_{\GU_1(q)})
 =(u+q^{1/2})(u-q^{-1}).
\]

It remains to treat the classical colors.  The $X-1$ calculation is
characteristic-free.  On the bottom Harish--Chandra induction let
\[
 h=e_U\dot x e_U.
\]
The rank-one double-coset multiplication gives
\[
 |U|h^2=(|U|-1)h+1,
\]
so the two eigenvalues of $h$ are $1$ and $-|U|^{-1}$.  For the
symplectic bottom vertex one has $|U|=q$ and
\eqref{eq:beta-normalization} gives $\beta_{X-1}=1$; hence the roots are
$1,-q^{-1}$.  The same formula holds for the odd orthogonal tower (and in
characteristic two is transported through
$\O_{2n+1}(q)\cong\Sp_{2n}(q)$).

For the split and nonsplit even orthogonal towers the standard rank-one
Witt-parabolic orders give, respectively,
$|U|=|V_{0,1}|=1$ and $|U|=|V_{0,1}|=q^2$.  Since
$\kappa_1=I_1$ and
$(t^{X-1})^2=-1$, equation \eqref{eq:beta-normalization} gives respectively
$\beta_{X-1}=1$ and $\beta_{X-1}=q$ with our chosen sign.  Therefore
\[
 \bbO^{X-1}(u)({\bf1}_{\O^+_0})=(u-1)(u+1),
 \qquad
 \bbO^{X-1}(u)({\bf1}_{\O^-_2})=(u-q)(u+q^{-1}),
\]
for every $q$, including characteristic two.

When $q$ is odd, the remaining classical color $X+1$ is the standard
normalization of \cite[Theorem~6.5]{DVV17} and
\cite[Lemmas~4.6--4.7 and Corollary~4.8]{LSZ25}; it gives the displayed
$X+1$ rows.  This exhausts
$\calI_{-1}(\bbK)\sqcup\calI_{-2}(\bbK)$.
\end{proof}

\subsubsection{Colored weight functions and separation of characters}\label{sec:completeinvariants}

Fix a simple object \(\rho\) which is \(F^\Gamma\)-cuspidal for every
\(\Gamma\in\calI(\bbK)\), and let \((\bbK G_\bullet,\rho)\mathrm{-mod}\) be the
Serre subcategory generated from \(\rho\) by the eigenfunctors of
Theorem~\ref{ThmB}.  The following recursion is the part of the colored-weight
calculation that follows formally from the categorical action.

\begin{theorem}\label{TheoremC}
Let \(L\) be a simple object in
\((\bbK G_\bullet,\rho)\mathrm{-mod}\).  Choose a chain of simple
subquotients
\[
\rho=L_0,L_1,\dots,L_m=L,
\qquad
L_j\text{ a subquotient of }F^{\Gamma_j}_{i_j}(L_{j-1}).
\]
Then for every color \(\Gamma\),
\[
{
\bbO_L^\Gamma(u)=\bbO_\rho^\Gamma(u)
\prod_{\substack{1\le j\le m\\\Gamma_j=\Gamma}}
\frac{(u-i_j^+)(u-i_j^-)}{(u-i_j)^2}.}
\]
Here \(i_j^\pm\) are the neighboring vertices of \(i_j\) in the
rank-two Hecke quiver.  The resulting rational function is independent of the
chosen chain.
\end{theorem}
\begin{proof}
At one eigenfunctor step, the bubble--dot relation in the
Heisenberg-to-Kac--Moody construction changes only the color being added and
gives
\[
\frac{\bbO_{L_j}^{\Gamma_j}(u)}
     {\bbO_{L_{j-1}}^{\Gamma_j}(u)}
=
\frac{(u-i_j^+)(u-i_j^-)}{(u-i_j)^2};
\]
all other colored bubble series are unchanged.  Iterating gives the displayed
product.  Path independence is not an additional combinatorial assertion: the
product equals the intrinsically defined central bubble series
\(\bbO_L^\Gamma(u)\), so two chains ending at the same simple object must give
the same rational function.
\end{proof}

\begin{remark}
For a level-$1$ factor, the product in Theorem~\ref{TheoremC} is the usual
charged-content boundary product. For a level-$2$ factor the same formula is
read in the two-runner/bipartition coordinate, and for the classical
$X\pm1$ factors equivalently on the symbol crystal.
\end{remark}

\subsubsection{Jordan reduction of the colored bubbles}
\label{subsec:jordan-bubbles}
Fix the Jordan decompositions from \S\ref{subsec:jordan}, compatibly with
Harish--Chandra induction.  Each color changes exactly one primary Jordan
factor; the dictionary is Table~\ref{tab:jordan-colored-bubbles}.  Here
$\delta_\Gamma$ is the reduced degree defined above, and
$\mathrm U_m(Q)$ denotes the unitary group over $\mathbb F_{Q^2}$ with fixed
field $\mathbb F_Q$.

\begin{table}[H]
\centering
\small
\renewcommand{\arraystretch}{1.18}
\caption{Jordan dictionary for the colored bubbles}
\label{tab:jordan-colored-bubbles}
\begin{tabular}{@{}p{.11\textwidth}p{.23\textwidth}p{.27\textwidth}p{.30\textwidth}@{}}
\hline
Tower & color $\Gamma$ & $\Gamma$-primary Jordan factor & local datum read by $\bbO^\Gamma$ \\
\hline
$\GL$
& $\Gamma\in\calF$ (level $-1$)
& $\GL_{m_\Gamma(s)}(q^{d_\Gamma})$
& partition $\lambda_\Gamma$; one-runner charged coordinate \\
$\GU$
& $\Gamma\in\calF^{\GU}_2$ (level $-1$)
& $\GL_{m_\Gamma(s)}(q^{d_\Gamma/2})$
& partition $\lambda_\Gamma$; one-runner charged coordinate \\
$\GU$
& $\Gamma\in\calF^{\GU}_1$ (level $-2$)
& $\mathrm U_{m_\Gamma(s)}(q^{d_\Gamma/2})$
& partition $\lambda_\Gamma$ in its two-runner Fock realization \\
$\Sp,\O$
& $\Gamma\in\calF^G_2$ (level $-1$)
& $\GL_{m_\Gamma(s)}(q^{\delta_\Gamma})$
& partition $\lambda_\Gamma$; one-runner charged coordinate \\
$\Sp,\O$
& $\Gamma\in\calF^G_1$ (level $-2$)
& $\mathrm U_{m_\Gamma(s)}(q^{\delta_\Gamma})$
& partition $\lambda_\Gamma$ in its two-runner Fock realization \\
$\Sp,\O$
& $\Gamma=X\pm1$ (level $-2$; only $X-1$ if $q$ is even)
& corresponding classical factor
& Lusztig symbol $\Lambda_\Gamma$ and the parameters attached to its cuspidal source \\
\hline
\end{tabular}
\end{table}

For $q$ even only the $X-1$ classical color occurs.  In the full odd
orthogonal tower the additional central sign is transported separately by
\eqref{eq:Oodd-sign-transport}.  The table lists the Jordan labels for the elementary-divisor factors used in the separation argument.

For a color \(\Gamma\), let \(\Delta\) be an irreducible representative as above. If \(\lambda\) is a partition and
\(\xi\in\bbK^\times\), define
\begin{equation}\label{eq:local-residue-function}
\mathsf R^{\Gamma}_{\lambda,\xi}(u):=
\frac{\displaystyle\prod_{b\in A(\lambda)}(u-\xi q^{d_\Delta\ct(b)})}
     {\displaystyle\prod_{b\in R(\lambda)}(u-\xi q^{d_\Delta\ct(b)})},
\end{equation}
where \(A(\lambda)\) and \(R(\lambda)\) are the addable and removable boxes.
For a level-$2$ Harish--Chandra series we use the product of two such functions,
one for each runner of its charged bipartition. This is the rescaled
\((q,\pm)\)-residue function of \cite[\S4.3.3]{LSZ25}.

\begin{lemma}\label{lem:Fgamma-cuspidal-seeds}
Let \(\rho^0\) be a simple object which is \(F^\Gamma\)-cuspidal.
\begin{enumerate}
\item If \(\Gamma\in\calI_{-1}(\bbK)\), the \(\Gamma\)-primary partition is
empty and
\[
 m_{\rho^0}^\Gamma(u)=u-1.
\]
\item Suppose that \(\Gamma\in\calI_{-2}(\bbK)\) is nonclassical.  The \(\Gamma\)-primary factor is unitary, and its
cuspidal unipotent label is the triangular partition
\(\lambda_t=(t,t-1,\ldots,1)\) for a unique \(t\ge0\).  There is a common
sign \(\varepsilon\in\{\pm1\}\) such that
\[
 m_{\rho^0}^\Gamma(u)
 =\bigl(u-\varepsilon(-q^{d_\Gamma/2})^t\bigr)
  \bigl(u-\varepsilon(-q^{d_\Gamma/2})^{-1-t}\bigr).
\]
The Hecke relation and the Jordan degree ratio do not determine
\(\varepsilon\).  Nevertheless the unordered root pair determines \(t\),
since the quotient of the two roots is
\((-q^{d_\Gamma/2})^{\pm(2t+1)}\), and the roots themselves determine the initial root on each runner.
\item For the classical colors \(X\pm1\) with odd \(q\), lower
simultaneously in the two classical colors.  The resulting source is both
\(F^{X-1}\)- and \(F^{X+1}\)-cuspidal, and its exact colored functions are
those of \cite[Theorem~5.13 and Table~4]{LSZ25}.  If \(q\) is even, only
\(X-1\) occurs.  Fix the standard pair of initial roots in the two-runner realization of the Lusztig symbol, \(\{\xi^0_1,\xi^0_2\}\), with
\(\xi^0_1\in q^{\mathbb Z}\) and
\(\xi^0_2\in-q^{\mathbb Z}\), determined by the cuspidal Lusztig symbol and the Witt tower.  Then there is a sign \(\epsilon(\rho^0)\in\{\pm1\}\) such that
\[
 m_{\rho^0}^{X-1}(u)
 =\bigl(u-\epsilon(\rho^0)\xi^0_1\bigr)
  \bigl(u-\epsilon(\rho^0)\xi^0_2\bigr).
\]
The rank-one Hecke relation determines the pair only up to simultaneous multiplication by this sign.  The calculation at the bottom of the Witt tower fixes the sign only for the cuspidal source there.
\end{enumerate}
\end{lemma}

\begin{proof}
By the Jordan-compatible branching rule, \(E^\Gamma\) removes one node from
the crystal corresponding to the \(\Gamma\)-factor.  In level \(-1\), every nonempty
partition has a removable node, so an \(F^\Gamma\)-cuspidal source has empty
\(\Gamma\)-primary partition; Proposition~\ref{Prop:dim2} gives
\(X^\Gamma(\rho^0)=1\).

For a level \(-2\) index \(\Gamma\notin\{X-1,X+1\}\), the corresponding factor is unitary and its
highest-weight vertices are the triangular \(2\)-cores \(\lambda_t\).  In
the standard unitary normalization the two roots are \((-q^{d_\Gamma/2})^t\) and
\((-q^{d_\Gamma/2})^{-1-t}\); see \cite[Theorem~4.12]{DVV19}.  Comparing our fixed
categorical dot with this standard rank-one operator determines the two roots only up to simultaneous multiplication by a sign.  This is exactly the ambiguity left by the product-and-ratio
calculation: the product and quotient of the roots are fixed, but simultaneous
negation is not.  This simultaneous sign cancels in the quotient, so the unordered roots still determine \(t\); no additional sign is introduced into the invariant.

For the classical colors with odd \(q\), the complete simultaneous cuspidal
calculation, including these signs, is precisely
\cite[Theorem~5.13 and Table~4]{LSZ25}.  For even \(q\), the same
product-and-ratio calculation leaves the same simultaneous sign ambiguity for the \(X-1\) source.
\end{proof}

\begin{lemma}\label{lem:jordan-bubble}
Let $\rho=\rho_{s,\mu}\in\mathcal E(G_n,(s))$, with the usual extension
label in the odd orthogonal case and the $X-1$ Lusztig-symbol label when $q$ is
even.  Fix $\Gamma\in\calI(\bbK)$.  Let $\rho^0_\Gamma$ be the
$F^\Gamma$-cuspidal source of the local $\Gamma$-Harish--Chandra series, and
let $\xi_{\Gamma,a}$ be the roots of $m_{\rho^0_\Gamma}^\Gamma(u)$.  Write the corresponding labeling data as
$(\lambda_\Gamma,\xi_{\Gamma,1})$ in level $-1$ and as
$((\lambda_{\Gamma,1},\xi_{\Gamma,1}),
  (\lambda_{\Gamma,2},\xi_{\Gamma,2}))$ in level $-2$.  Then
\begin{align}
\bbO^\Gamma(u)(\rho)
 &=\mathsf R^\Gamma_{\lambda_\Gamma,\xi_{\Gamma,1}}(u),
 &&\Gamma\in\calI_{-1}(\bbK),\label{eq:jordan-bubble-one}\\
\bbO^\Gamma(u)(\rho)
 &=\mathsf R^\Gamma_{\lambda_{\Gamma,1},\xi_{\Gamma,1}}(u)
   \mathsf R^\Gamma_{\lambda_{\Gamma,2},\xi_{\Gamma,2}}(u),
 &&\Gamma\in\calI_{-2}(\bbK).\label{eq:jordan-bubble-two}
\end{align}
\end{lemma}

\begin{proof}
Choose the $F^\Gamma$-cuspidal source for the $\Gamma$-factor.  When $\Gamma\notin\{X-1,X+1\}$ and the level is $-2$, Lemma~\ref{lem:Fgamma-cuspidal-seeds} determines the standard pair only up to simultaneous multiplication by a sign, but $\xi_{\Gamma,1},\xi_{\Gamma,2}$ are defined here to be the roots of $m_{\rho^0_\Gamma}^\Gamma(u)$ itself, so no comparison sign is chosen.  For the
$q$-even classical color $X-1$ they are
$\epsilon(\rho^0_\Gamma)\xi^0_1$ and
$\epsilon(\rho^0_\Gamma)\xi^0_2$ as in
Lemma~\ref{lem:Fgamma-cuspidal-seeds}.  Compatibility
of Jordan decomposition with Harish--Chandra induction identifies each
subsequent $\Gamma$-colored branching step with adding one node in the crystal corresponding to the $\Gamma$-factor, while the labels for all other elementary-divisor factors remain fixed.

At one such step Theorem~\ref{TheoremC} multiplies the bubble by
\[
 \frac{(u-i^+)(u-i^-)}{(u-i)^2},
\]
exactly the change of the addable/removable boundary in
\eqref{eq:local-residue-function}.  Iterating from $m_{\rho^0_\Gamma}^\Gamma(u)$
gives \eqref{eq:jordan-bubble-one} and \eqref{eq:jordan-bubble-two}.
\end{proof}

\begin{lemma}\label{lem:local-injectivity}
For $\Gamma\notin\{X-1,X+1\}$, the function $\bbO^\Gamma(u)(\rho)$ determines the charged partition labeling the $\Gamma$-factor in level $-1$, and in level $-2$ it determines the charged bipartition together with the triangular partition labeling its cuspidal source, up to interchange of the two runners.  For $q$ even and $\Gamma=X-1$, it determines the two-runner labeling of the Lusztig symbol together with the sign $\epsilon(\rho^0_\Gamma)$, modulo the involution interchanging the two runners.
\end{lemma}

\begin{proof}
In level $-1$ this is the interlacing argument of
\cite[Corollary~4.19]{LSZ25}.  When $\Gamma\notin\{X-1,X+1\}$ and the level is $-2$, the two roots attached to the chosen cuspidal source lie in distinct $q^{d_\Gamma\mathbb Z}$-orbits: their quotient is
$(-q^{d_\Gamma/2})^{\pm(2t+1)}$, which is not in
$q^{d_\Gamma\mathbb Z}$.  A simultaneous sign change does
not alter this quotient.  Hence the zeros and poles of
\eqref{eq:jordan-bubble-two} split canonically between the two $q^{d_\Gamma\mathbb Z}$-orbits containing these roots.  The one-runner interlacing argument recovers each charged partition and the initial root on the corresponding runner; the quotient of these two initial roots then determines $t$.  Thus the comparison sign is determined by the function $\bbO^\Gamma(u)(\rho)$ and need not be fixed separately.

Now let $q$ be even and $\Gamma=X-1$.  The two standard initial roots lie in the disjoint orbits $q^{\mathbb Z}$ and $-q^{\mathbb Z}$, and the roots for the chosen cuspidal source are obtained by multiplying both by $\epsilon(\rho^0_\Gamma)$.  Hence the zero and pole sets distinguish the two runners according to these orbits.  Replacing this sign by its negative interchanges the two orbits and sends $\bbO^{X-1}(u)$ to $\bbO^{X-1}(-u)$.  If the two functions coincide, one-runner injectivity
forces the two-runner labeling to be fixed by the involution interchanging the two runners; this is the same Lusztig-symbol label.  Thus this sign is determined by the weight function $\bbO^{X-1}(u)$ and no distinct Lusztig-symbol label is lost.
\end{proof}

\begin{proposition}
\label{prop:recover-primary-jordan}
For every $\Gamma\in\calI(\bbK)$, the rational function
$\bbO^\Gamma(u)(\rho)$ determines the Jordan label attached to $\Gamma$ up to the
standard finite symmetries at the classical $X\pm1$ factors.  Taken over all
colors, the family $\{\bbO^\Gamma(u)(\rho)\}_\Gamma$ determines the full
Jordan parameter.
\end{proposition}

\begin{proof}
For every $\Gamma\notin\{X-1,X+1\}$, Lemmas~\ref{lem:jordan-bubble} and~\ref{lem:local-injectivity} recover the corresponding partition or bipartition and its cuspidal core.  In level $-2$, the simultaneous sign of the two roots is not an omitted datum: it is determined by the roots of the colored function, while the core parameter $t$ is recovered from their quotient.

For odd $q$, the two classical colors $X\pm1$ are handled directly by
\cite[Theorem~5.13, Table~4 and Theorem~4.22]{LSZ25}, after the labels for the already recovered non-$X\pm1$ factors are fixed.  These results also separate
the determinant, spinor, diagonal, and odd-orthogonal extension symmetries,
with the central sign transported by \eqref{eq:Oodd-sign-transport}.  For even
$q$, Lemmas~\ref{lem:Fgamma-cuspidal-seeds} and
\ref{lem:local-injectivity} retain the sign attached to the $X-1$ roots instead of suppressing it.  Hence all Jordan labels for the elementary-divisor factors are recovered.
\end{proof}

\begin{theorem}
\label{thm:invariants}
Assume that \(\operatorname{char}\bbK=0\). Then
\[
   \{\bbO^\Gamma(u)(-)\}_{\Gamma\in\calI(\bbK)}
\]
is a complete invariant of irreducible ordinary characters in the tower.
\end{theorem}

\begin{proof}
If two irreducible characters have the same colored weight functions,
Proposition~\ref{prop:recover-primary-jordan} gives the same semisimple
primary data, the same unipotent label on every centralizer factor, and the
same full-orthogonal extension data.  Their complete Jordan parameters
therefore coincide, so the two characters are isomorphic.
\end{proof}

\subsubsection{Center of the group algebra}\label{subsec:center}
\begin{theorem}\label{Thm:centersurjective}
For each rank \(n\), let
\(A_n\) be the unital subalgebra of \(Z(\bbK G_n)\) generated by the
coefficients of all evaluated colored bubble series.  Then
\[
A_n=Z(\bbK G_n).
\]
Equivalently, evaluation is surjective on centers in every finite rank.
\end{theorem}
\begin{proof}
An endomorphism of the monoidal unit evaluates to a natural endomorphism of
the identity functor.  Hence every colored bubble acts by an element of
\(Z(\bbK G_n)\), and \(A_n\subseteq Z(\bbK G_n)\).  Since \(\bbK G_n\) is split semisimple, central characters give an
algebra isomorphism
\[
Z(\bbK G_n)\xrightarrow{\sim}
\prod_{\rho\in\Irr(G_n)}\bbK,
\qquad
z\longmapsto(\omega_\rho(z))_\rho.
\]
By Theorem~\ref{thm:invariants}, for every two distinct simples
\(\rho\ne\sigma\) some coefficient \(b_{\rho,\sigma}\in A_n\) of a colored
bubble has
\(\omega_\rho(b_{\rho,\sigma})\ne
\omega_\sigma(b_{\rho,\sigma})\).  Therefore the primitive central idempotent
of \(\rho\) is obtained by finite Lagrange interpolation:
\[
e_\rho=
\prod_{\sigma\ne\rho}
\frac{b_{\rho,\sigma}-
      \omega_\sigma(b_{\rho,\sigma})1}
     {\omega_\rho(b_{\rho,\sigma})-
      \omega_\sigma(b_{\rho,\sigma})}
\in A_n.
\]
The primitive central idempotents \(e_\rho\) form a basis of
\(Z(\bbK G_n)\).  Hence \(A_n=Z(\bbK G_n)\), as required.
\end{proof}

\begin{remark}\label{Rem:computebubbles}
The wreath-product analogue is Theorem~\ref{Thm:expofbubbles} in
\S\ref{sub:coloredweight}.  For double covers of symmetric groups, a spin
analogue should instead involve the super Kac--Moody $2$-categories of
\cite{BE17}; the same phenomenon is expected for towers of spin groups.
\end{remark}

\section{Categorification on representations of wreath products $ H \wr \frakS_n$: degenerate case}
\label{sec:degenerate}
\subsection{Wreath products $ H \wr \frakS_n$}
Let $H$ be a finite group, $H_n=H\wr\frak S_n$ for $n\geq1$, and
$H_0=\{1\}$.  Write $H^{(i)}$ for the $i$-th factor of
$H^n\triangleleft H_n$.

Let $$\k H_\bullet\mod:=\bigoplus\limits_{n\in \bbN}\k H _n\mod.$$

Throughout this section, $\k$ is a splitting field for $H$ containing the
required $|H|$-th roots of unity, and $|H|$ is invertible in $\k$.
Let $\widehat H$ be the set of isomorphism classes of irreducible
$\k H$-modules and put $d=|\widehat H|$.  For $\chi\in\widehat H$, define
$$\Ind^{\chi}:\k H _n\mod\to \k H _{n+1}\mod,$$
$$M\mapsto \Ind^{ H _{n+1}}_{ H _n\times  H ^{(n+1)}}(M\boxtimes V_\chi);$$
$$\Res^{\chi}:\k H _{n+1}\mod\to \k H _n\mod,$$
$$N\mapsto \Hom_{\k H ^{(n+1)}}( V_\chi,\Res^{ H _{n+1}}_{ H _n\times  H ^{(n+1)}}(N));$$
where $V_\chi$ represents $\chi$.  Fix an enumeration
$\widehat H=\{\chi_1,\ldots,\chi_d\}$.

\begin{remark}\label{rem:frobenius-wreath}
Over $\bbK$, the action is the color decomposition of the Frobenius
Heisenberg action for $F=\bbK H$ from \cite[Theorem~7.4]{RS17}; see also
\cite[Theorem~1.4(b)]{Sav19}.  Since $\bbK H$ is split semisimple,
\cite[Theorem~8.3]{BSW21} identifies its Karoubi envelope with the symmetric
product indexed by $\Irr_{\bbK}(H)$ after choosing one primitive idempotent
in each simple block; compare \cite[Corollary~4.10]{Gan20}.  The objects
$(\uparrow,e_\chi)$ and $(\downarrow,e_\chi)$ then act by the bimodules
defining $F^\chi$ and $E^\chi$.  For a non-semisimple Frobenius algebra the
corresponding category has color-changing tokens and is not a symmetric
product of ordinary Heisenberg categories.
\end{remark}

The group algebra $\k H $ decomposes as a direct product of matrix algebras
$$\k H \cong M_{n_1}(\k)\times \cdots\times M_{n_d}(\k)$$
where each distinct irreducible representation $\chi_i$ is realized as an irreducible representation
of one of the matrix algebras $M_{n_i}(\k)$. Let $f_{\chi_1},\cdots, f_{\chi_d}$ denote the distinct  orthogonal central idempotents of $\k H $ such that
$$1=\sum_{i=1}^{d}f_{\chi_i}.$$

The above is not, in general, a minimal decomposition of 1 as a sum of orthogonal idempotents, since
the idempotents $f_{\chi_i}$ themselves are not minimal if $\dim(V_{\chi_i}) > 1$. For each $i$ and $s = 1, . . . , n_i$, let $e_{\chi_i,s}$
denote the matrix unit of $M_{n_i}(\k)$ whose $(s, s)$ entry is equal to 1 and whose other entries are 0. Then
$e_{\chi_i,s}$ are minimal orthogonal idempotents in $\k H $, with
$$f_{\chi_i} =\sum_{s=1}^{n_i}e_{\chi_i,s}.$$
We fix $e_{\chi_i}:=e_{\chi_i,1}$. Then we have $\k H \cdot e_{\chi_i}\cong V_{\chi_i}.$
Its dual module $V^*_{\chi_i}=\Hom_{\k H }(V_{\chi_i},\k H )$,  which is a right $\k H $ module, can be identified with $e_{\chi_i}\cdot \k H $ via the canonical isomorphism
$$\Hom_{\k H }(\k H  \cdot e_{\chi_i},\k H )\cong e_{\chi_i}\cdot \k H .$$
Then for each $\chi_i\in \hat H ,$
the functor $F^{\chi_i}$ is represented by the $( H _{r+1}, H _{r})$-bimodule $$\k H _{r+1}\cdot e^{(r+1)}_{\chi_i}\cong \k H _{r+1}\otimes_{ \k H ^{(r+1)}} (\k H ^{(r+1)}\cdot  e^{(r+1)}_{\chi_i}).$$
Similarly, the functor $E^{\chi_i}$ is represented by $( H _{r}, H _{r+1})$-bimodule $$e^{(r+1)}_{\chi_i}\cdot \k H _{r+1}\cong (e^{(r+1)}_{\chi_i}\cdot \k H ^{(r+1)}  )\otimes_{ \k H ^{(r+1)}}\k H _{r+1}.$$

\subsubsection{Jucys--Murphy elements}
Let $s_1,\ldots,s_{n-1}$ be the simple reflections of $\mathfrak S_n$.
The wreath-product Jucys--Murphy elements of Pushkarev \cite{Pus99} and Wang
\cite{W04a,W04b} are
\begin{align*}
	L_j=\sum\limits_{i<j}\sum_{g\in  H }g^{(i)}(g^{-1})^{(j)}(i,j).
\end{align*}
where $(i,j)$ interchanges $i$ and $j$.  Put
$M^{(i,j)}=\sum_{g\in H}g^{(i)}(g^{-1})^{(j)}$.  For $H=1$ these are the
usual Jucys--Murphy elements.
\begin{lemma}[\cite{Pus99}] The Jucys--Murphy elements $L_j$ ($j=1,2,\cdots, n$) pairwise commute. Moreover, each $L_j$ commutes with elements in  $ H ^n$. Furthermore, we have the relation
	$$s_iL_i-L_{i+1} s_i=-M^{(i,i+1)}.$$
\end{lemma}
It is known by elementary Clifford theory that
for $1 \leq k < l \leq d$, $$\Ind^{\k H _2}
_{\k H ^2}(V_{\chi_k}\boxtimes V_{\chi_l}) = (V_{\chi_k}\boxtimes V_{\chi_l})\oplus (V_{\chi_l}\boxtimes V_{\chi_k})$$ is a simple $\k H _2$-module where
$s_1=(1 ,2)$ acts as the permuting operator $P$.
\begin{lemma}[\cite{WW08}]
	\begin{itemize}\label{Lemma:ortho}

		\item[$(1)$] $M^{(1,2)} = 0$ when acting on a simple $\k H ^2$-module $V_{\chi_i}\boxtimes V_{\chi_j}$ for $i\neq j$.

		\item[$(2)$] $M^{(1,2)}= c_kP$ when acting on the $\k H ^2$-module $V_{\chi_k}^{\boxtimes 2} $, where the scalar $c_k=\frac{|H|}{\chi_k(1)}$.
	\end{itemize}
\end{lemma}

\subsubsection{Adjunctions}
\label{sec:adjwreath}
Since $(E^{\chi_i},F^{\chi_i})$ is a bi-adjoint pair, there exist four adjunction maps for each pair. Note that these adjunction maps are unique up to a scalar.
Fix matrix units $E^{\chi_i}_{ab}$ in the $\chi_i$-block, with
$e_{\chi_i}=E^{\chi_i}_{11}$ and $1\leq a,b\leq n_i$.  We choose the
following bimodule representatives.
\begin{itemize}
	\item[$(a)$]
	Since $ e_{\chi_i}\cdot \k H  \cdot e_{\chi_i}\cong V^\vee_{\chi_i}\otimes_{\k H } V_{\chi_i}\cong  \Hom_{\k H }(V_{\chi_i},V_{\chi_i})=\k,$ we know that $e_{\chi_i}\cdot \k H  \cdot e_{\chi_i}=\k\cdot e_{\chi_i}.$ We denote by $\delta: \k H  \to \k$  such that $e_{\chi_i}g e_{\chi_i}=\delta(g)\cdot e_{\chi_i}$.
	We define  $(\k H  _{r},\k H  _{r})$-bimodule
	maps
	$$\varepsilon^{\chi_i}_L: e^{(r+1)}_{\chi_i}\cdot \k H  _{r+1}\cdot e^{(r+1)}_{\chi_i} \to  \k H  _r,$$
	\begin{align*}
		e^{(r+1)}_{\chi}\cdot g\cdot e^{(r+1)}_{\chi}\mapsto
		\begin{cases}
			\delta(h')h &\text{if $g=h\times h' \in H_r\times H^{(r+1)}$;}\\
			0&\text{otherwise.}
		\end{cases}
	\end{align*}

	\item[$(b)$]
	Let $H_r=\coprod_{a=1}^{r}(H_{r-1}\times H^{(r)})g_a$ be a decomposition
	of $H_r$ into left cosets.  Put
	\[
	Z^{\chi_i}_r=\sum_{a=1}^{r}\sum_{b=1}^{n_i}
	g_a^{-1}(E^{\chi_i}_{b1})^{(r)}\otimes_{\k H_{r-1}}
	(E^{\chi_i}_{1b})^{(r)}g_a.
	\]
	Then $Z^{\chi_i}_r\in \k H_r e^{(r)}_{\chi_i}\otimes_{\k H_{r-1}}
	e^{(r)}_{\chi_i}\k H_r$, it is independent of the chosen representatives,
	and $aZ^{\chi_i}_r=Z^{\chi_i}_ra$ for $a\in H_r$.
	\smallskip

	Then we define $( \k H _r , \k H _r )$-bimodule maps as follows:
	$$\eta^{\chi_i}_L: \k H _r \to  \k H _{r}\cdot e^{(r)}_{\chi_i}\otimes_{\k H _{r-1}} e^{(r)}_{\chi_i}\cdot \k H _r  ,\quad a\mapsto aZ_r^{\chi_i}=Z^{\chi_i}_ra$$
	for any $a\in  \k H _r .$

	\item[$(c)$]

	We define by $\varepsilon^{\chi_i}_R$ the following $( \k H _r , \k H _r )$-bimodule maps:
	$$\varepsilon^{\chi_i}_R: \k H _{r}\cdot e^{(r)}_{\chi_i}\otimes_{\k H _{r-1}} e^{(r)}_{\chi_i}\cdot \k H _r  \to  \k H _r ,\quad g  e^{(r)}_{\chi_i} \otimes e^{(r)}_{\chi_i} h\mapsto g e^{(r)}_{\chi_i} h,$$
	for any $g,h\in  \k H _r $.
	\item[$(d)$] Finally we define by $\eta^{\chi_i}_R$  the following $( \k H _r , \k H _r )$-bimodule maps:
	$$\eta^{\chi_i}_R:  \k H _r \to  e^{(r+1)}_{\chi_i}\cdot \k H _{r+1}\cdot e^{(r+1)}_{\chi_i} ,\quad g\mapsto  g e^{(r+1)}_{\chi_i} $$
	for all $g\in  \k H _r .$
\end{itemize}
The matrix-unit identities
$E_{1b}^{\chi_i}E_{c1}^{\chi_i}=\delta_{bc}E_{11}^{\chi_i}$ and
$\sum_bE_{b1}^{\chi_i}E_{1b}^{\chi_i}=f_{\chi_i}$ give the two triangle
identities for the left adjunction; the right adjunction is immediate from
multiplication by $e_{\chi_i}$.  Thus these are the stated adjunction maps.
\smallskip
\subsubsection{Dots}

Define $c_i:=\frac{|H|}{\chi_i(1)}$.
This normalization is also the one suggested by the Frobenius-Heisenberg description above. Over $\bbK$, let
\[
\operatorname{tr}_H\!\left(\sum_{h\in H}a_hh\right)=a_1
\]
be the coefficient-of-the-identity Frobenius trace on $\bbK H$. For a rank-one primitive idempotent $e_{\chi_i}$ in the $\chi_i$-block one has
$\operatorname{tr}_H(e_{\chi_i})=\chi_i(1)/|H|=c_i^{-1}$. Lemma~5.3 of \cite{BSW21} shows that the Frobenius Heisenberg category is unchanged up to isomorphism when the Frobenius trace is changed. After normalizing the trace on the one-dimensional corner $\bbK e_{\chi_i}$, the ordinary Heisenberg dot corresponds to the normalized Jucys--Murphy operator $c_i^{-1}L_{r+1}$.
We define $\mathrm{X}_r^{\chi_i}:F_r^{\chi_i}\to F_r^{\chi_i}$ by right multiplication by the normalized Jucys--Murphy element $c_i^{-1}e^{(r+1)}_{\chi_i}L_{r+1}e^{(r+1)}_{\chi_i}.$ More precisely, $\mathrm{X}_r^{\chi_i}:F_r^{\chi_i}\to F_r^{\chi_i}$ is given by the following $(\k H _{r+1},\k H _r)$-bimodule maps:
$$\k H _{r+1}\cdot e^{(r+1)}_{\chi_i}\to \k H _{r+1}\cdot e^{(r+1)}_{\chi_i},$$
$$g e^{(r+1)}_{\chi_i}\mapsto c_i^{-1}\cdot gL_{r+1}e^{(r+1)}_{\chi_i}.$$
Moreover, $\mathrm{X}_r^{\chi_i}:F_r^{\chi_i}\to F_r^{\chi_i}$ is also given by the following $(\k H _{r+1},\k H _r)$-bimodule maps:
$$\k H _{r+1}\otimes_{ \k H ^{(r+1)}}V_{\chi_i}\to \k H _{r+1}\otimes_{ \k H ^{(r+1)}}V_{\chi_i},$$
$$g\otimes m\mapsto c_i^{-1}\cdot gL_{r+1}\otimes m.$$
\subsubsection{Crossings}
We define $\mathrm{T}_r^{\chi_i,\chi_j}:F_{r+1}^{\chi_j}F_r^{\chi_i}\to F_{r+1}^{\chi_i}F_r^{\chi_j}$ by right multiplication by the elements $e^{(r+2)}_{\chi_i}(r+1,r+2)e^{(r+2)}_{\chi_j}.$ More precisely, $\mathrm{T}_r^{\chi_i,\chi_j}:F_{r+1}^{\chi_j}F_r^{\chi_i}\to F_{r+1}^{\chi_i}F_r^{\chi_j}$ is given by the following $(\k H _{r+2},\k H _{r+2})$-bimodule maps:
$$\k H _{r+2}\cdot e^{(r+1)}_{\chi_i}e^{(r+2)}_{\chi_j}\to \k H _{r+2}\cdot e^{(r+1)}_{\chi_j}e^{(r+2)}_{\chi_i},$$
$$ge^{(r+1)}_{\chi_i}e^{(r+2)}_{\chi_j} \mapsto g(r+1,r+2)e^{(r+1)}_{\chi_j}e^{(r+2)}_{\chi_i}.$$
Moreover, $\mathrm{T}_n^{\chi_i,\chi_j}:F_{r+1}^{\chi_j}F_r^{\chi_i}\to F_{r+1}^{\chi_i}F_r^{\chi_j}$ is also given by the following $(\k H _{r+2},\k H _{r})$-bimodule maps:
$$\k H _{r+2}\otimes_{\k( H ^{(r+1)}\times  H ^{(r+2)})}(V_{\chi_i}\otimes V_{\chi_j})\to \k H _{r+2}\otimes_{\k( H ^{(r+1)}\times H ^{(r+2)})}(V_{\chi_j}\otimes V_{\chi_i}),$$
$$g\otimes m_1\otimes m_2\mapsto g(r+1,r+2)\otimes m_2\otimes m_1.$$

\subsubsection{Main theorem}
Write $\Irr(H)=\{\chi_1,\ldots,\chi_d\}$.  We identify the $d$ colors of $\symHeis_d$ with these irreducible characters and take $k_i=-1$ for every $i$; thus
\[
\symHeis_d=\Heis^{(\chi_1)}_{-1}\odot\cdots\odot\Heis^{(\chi_d)}_{-1}.
\]
\begin{theorem}\label{Thm:doubleheisenberg}
There exists a strict $\k$-linear monoidal functor
$$
\Psi:\symHeis_d \longrightarrow \mathcal{E}nd_\k
\left(\k H _{\bullet} \mod\right)
$$ which
sends the generating objects $\mathord{
\begin{tikzpicture}[baseline = -1mm]
	\draw[<-,red] (0.68,.28) to (0.68,-.22);
	\node at (0.68,-.37) {$\scriptstyle{\red{\chi_i}}$};
\end{tikzpicture}
}$ and $\mathord{
\begin{tikzpicture}[baseline = -1mm]
	\draw[->,red] (0.68,.28) to (0.68,-.22);
	\node at (0.68,-.37) {$\scriptstyle{\red{\chi_i}}$};
\end{tikzpicture}
}$ to endo-functors $F^{\chi_i}$ and $E^{\chi_i}$ of $ \k H _{\bullet}\mod $, respectively. Moreover,
$\Psi$ sends the generating morphisms  to the corresponding natural transformations
in  $ \k H _{\bullet}\mod $ as follows:
\begin{itemize}
\item
$\Psi(\begin{tikzpicture}[baseline = -1mm]
	\draw[->,red] (0.08,-.2) to (0.08,.2);
	\node[red] at (0.08,0) {$\dott$};
	\node at (0.08,-.37) {$\scriptstyle{\red{\chi_i}}$};
\end{tikzpicture})=\mathrm{X}^{\chi_i}\::\;F^{\chi_i} \Rightarrow F^{\chi_i}\,,\quad\Psi(
\begin{tikzpicture}[baseline = -1mm]
	\draw[->,red] (0.2,-.2) to (-0.2,.2);
	\draw[->,red] (-0.2,-.2) to (0.2,.2);
	\node at (0.2,-.37) {$\scriptstyle{\red{\chi_i}}$};
	\node at (-0.2,-.37) {$\scriptstyle{\red{\chi_i}}$};
\end{tikzpicture})=\mathrm{T}^{\chi_i,\chi_i}\::\; (F^{\chi_i})^2 \Rightarrow (F^{\chi_i})^2\,,$
\item
$\Psi(\begin{tikzpicture}[baseline = .75mm]
	\draw[<-,red] (0.3,0) to[out=90, in=0] (0.1,0.3);
	\draw[-,red] (0.1,0.3) to[out = 180, in = 90] (-0.1,0);
	\node at (-0.1,-.2) {$\scriptstyle{\red{\chi_i}}$};
\end{tikzpicture})=\varepsilon^{\chi_i}_R\::\;F^{\chi_i}\otimes E^{\chi_i} \Rightarrow \unit
\,,\quad\Psi(\begin{tikzpicture}[baseline = .75mm]
	\draw[<-,red] (0.3,0.3) to[out=-90, in=0] (0.1,0);
	\draw[-,red] (0.1,0) to[out = 180, in = -90] (-0.1,0.3);
	\node at (-0.1,0.42) {$\scriptstyle{\red{\chi_i}}$};
\end{tikzpicture})=\eta^{\chi_i}_R\::\;\unit\Rightarrow E^{\chi_i} \otimes F^{\chi_i}\,,$
\item
$\Psi(\begin{tikzpicture}[baseline = .75mm]
	\draw[-,red](0.1,0.3) to[out=0, in=90](0.3,0);
	\draw[<-,red] (-0.1,0)to[out =90 , in = 180](0.1,0.3);
	\node at (0.3,-.2) {$\scriptstyle{\red{\chi_i}}$};
\end{tikzpicture})=\varepsilon^{\chi_i}_L\::\;E^{\chi_i}\otimes F^{\chi_i} \Rightarrow \unit
\,,\quad\Psi(\begin{tikzpicture}[baseline = .75mm]
	\draw[-,red] (0.1,0)to[out=0, in=-90](0.3,0.3);
	\draw[<-,red] (-0.1,0.3)to[out =-90, in =180] (0.1,0);
	\node at (0.3,0.42) {$\scriptstyle{\red{\chi_i}}$};
\end{tikzpicture})=\eta^{\chi_i}_L\::\;\unit\Rightarrow F^{\chi_i}\otimes E^{\chi_i}\,,$

\item
$\Psi(\begin{tikzpicture}[baseline = -1mm]
	\draw[->,blue] (0.2,-.2) to (-0.2,.2);
	\draw[->,red] (-0.2,-.2) to (0.2,.2);
	\node at (0.2,-.37) {$\blue{\scriptstyle{\chi_j}}$};
	\node at (-0.2,-.37) {$\red{\scriptstyle{\chi_i}}$};
\end{tikzpicture})=\mathrm{T}^{\chi_i,\chi_j}\::\;F^{\chi_i}\otimes F^{\chi_j} \Rightarrow F^{\chi_j}\otimes F^{\chi_i}.$
\end{itemize}
\end{theorem}
\subsection{Proof of Theorem~\ref{Thm:doubleheisenberg}}
It remains to verify the defining relations of $\symHeis_d$ for the maps defining $\Psi$.

\subsubsection{Double cosets decomposition}\label{sub:doublecosets}

\begin{lemma}\label{doublecoset}
\begin{itemize}
\item[$(a)$] The set of the distinguished representatives for double coset $ H _{n}\backslash  H _{n+1}/ H _n$ is $\{x,(n,n+1)|x\in  H ^{(n+1)}\}.$
\item[$(b)$] The set of the distinguished representatives for double coset $( H _{n}\times H ^{(n+1)} ) \backslash  H _{n+1}/( H _n\times H ^{(n+1)})$ is $\{1,(n,n+1)\}.$
\end{itemize}
\end{lemma}
\begin{proof}
The two decompositions follow from the standard coset representatives of
$\mathfrak S_n$ in $\mathfrak S_{n+1}$.
\end{proof}

\begin{proposition}\label{mkd}
\begin{itemize}
\item[$(a)$] As an $( \k H _r , \k H _r )$-bimodule, $ e^{(r+1)}_{\chi_i} \cdot  \k H _{r+1}
\cdot e^{(r+1)}_{\chi_i} $ is generated by $ e^{(r+1)}_{\chi_i} , e^{(r+1)}_{\chi_i}(r,r+1)e^{(r+1)}_{\chi_i} $,i.e., we have the following $( \k H _r , \k H _r )$-bimodule isomorphism
\begin{align*} e^{(r+1)}_{\chi_i}\cdot  \k H _{r+1} \cdot e^{(r+1)}_{\chi_i} \cong  \k H _r\cdot  e^{(r+1)}_{\chi_i}  \oplus  \k H _r  \cdot  e^{(r+1)}_{\chi_i}(r,r+1)e^{(r+1)}_{\chi_i}\cdot \k H _r.
\end{align*}
\item[$(b)$]For $i\neq j$, as an $( \k H _r , \k H _r )$-bimodule, $ e^{(r+1)}_{\chi_i} \cdot  \k H _{r+1}
\cdot e^{(r+1)}_{\chi_j} $ is generated by $ e^{(r+1)}_{\chi_i}(r,r+1)e^{(r+1)}_{\chi_j} $,
equivalently, we have the following $( \k H _r , \k H _r )$-bimodule isomorphisms
\begin{align*} e^{(r+1)}_{\chi_i}\cdot  \k H _{r+1} \cdot e^{(r+1)}_{\chi_j} \cong   \k H _r  \cdot  e^{(r+1)}_{\chi_i}(r,r+1)e^{(r+1)}_{\chi_j}\cdot \k H _r.\end{align*}
\end{itemize}
\end{proposition}
\begin{proof}

$(a)$ By Lemma \ref{doublecoset}$(b)$, we have the following double coset decomposition$$ H _{r+1}= H _{r}\times  H ^{(r+1)}\sqcup  H _{r}\times  H ^{(r+1)}\cdot (r,r+1)\cdot  H _{r}\times  H ^{(r+1)}.$$
So we have the following $(\k H _r,\k H _r)$-bimodule decomposition  of $e^{(r+1)}_{\chi}\cdot   \k H _{r+1} \cdot  e^{(r+1)}_{\chi}$
$$e^{(r+1)}_{\chi}\cdot \k(H_r\times H^{(r+1)}) \cdot e^{(r+1)}_{\chi} \oplus  e^{(r+1)}_{\chi}\cdot \k( H_r\times H^{(r+1)}(r,r+1)H_r\times H^{(r+1)})\cdot  e^{(r+1)}_{\chi} .$$

Since $ e^{(r+1)}_{\chi} $ commutes with $ H _r$ and $e^{(r+1)}_{\chi} \cdot \k H ^{(r+1)}\cdot  e^{(r+1)}_{\chi}=\k\cdot e^{(r+1)}_{\chi}$, we have $$ e^{(r+1)}_{\chi} \cdot \k(H_r\times H^{(r+1)})\cdot e^{(r+1)}_{\chi} = \k H _r \cdot  e^{(r+1)}_{\chi}.$$

Finally,  $$ e^{(r+1)}_{\chi}\cdot \k( H_r\times H^{(r+1)}\cdot(r,r+1)\cdot H_r\times H^{(r+1)})\cdot  e^{(r+1)}_{\chi} = \k H _r  \cdot e^{(r+1)}_{\chi} (r,r+1)
e^{(r+1)}_{\chi} \cdot  \k H _r . $$
\smallskip

$(b)$ Similarly, for $i\neq j,$  we have the following $(\k H _r,\k H _r)$-bimodule decomposition  of $e^{(r+1)}_{\chi_i}\cdot   \k H _{r+1} \cdot  e^{(r+1)}_{\chi_j}$:
$$e^{(r+1)}_{\chi_i}\cdot  \k( H _{r}\times  H ^{(r+1)}) \cdot e^{(r+1)}_{\chi_j} \oplus  e^{(r+1)}_{\chi_i}\cdot  \k( H _{r}\times  H ^{(r+1)}(r,r+1) H _{r}\times  H ^{(r+1)})\cdot  e^{(r+1)}_{\chi_j} .$$
The only difference is that $e^{(r+1)}_{\chi_i} \cdot \k H ^{(r+1)}\cdot  e^{(r+1)}_{\chi_j}\cong\Hom_{ H }(V_{\chi_i},V_{\chi_j})=0.$
The remaining proof is similar to $(a).$
\end{proof}

\smallskip

\subsubsection{Degenerate-affine-Hecke-type relations}

\begin{lemma}\label{lem:12relations}

The natural transformations  $\mathrm{X}^{\chi_i}=\begin{tikzpicture}[baseline = -1mm]
\draw[->,red] (0.08,-.2) to (0.08,.2);
\node at (0.08,0) {$\dott$};\node at (0.08,-.3) {${\scriptstyle{\red{\chi_i}}}$};

\end{tikzpicture}$,
$\mathrm{T}^{\chi_i,\chi_i}=\begin{tikzpicture}[baseline = -1mm]
\draw[->,red] (0.2,-.2) to (-0.2,.2);
\draw[->,red] (-0.2,-.2) to (0.2,.2);
\node at (0.2,-.37) {${\scriptstyle{\red{\chi_i}}}$};
\node at (-0.2,-.37) {${\scriptstyle{\red{\chi_i}}}$};
\end{tikzpicture}\:,$
$\mathrm{T}^{\chi_i,\chi_j}=\begin{tikzpicture}[baseline = -1mm]
\draw[->,blue] (0.2,-.2) to (-0.2,.2);
\draw[->,red] (-0.2,-.2) to (0.2,.2);
\node at (0.2,-.37) {${\scriptstyle{\blue{\chi_j}}}$};
\node at (-0.2,-.37) {${\scriptstyle{\red{\chi_i}}}$};
\end{tikzpicture}\:$ ($i\neq j$)
satisfy degenerate affine Hecke relations \eqref{dAHA} for all $1\leq i\leq d$ and mixed degenerate affine-Hecke-type relations \eqref{mixedHecke} for all $1\leq i\neq j \leq d$.

\end{lemma}
\begin{proof}The only non-trivial relations to check are:$\mathord{
\begin{tikzpicture}[baseline = -1mm]
	\draw[<-,red] (0.25,.3) to (-0.25,-.3);
	\draw[->,red] (0.25,-.3) to (-0.25,.3);
	\node at (-0.12,-0.145) {$\dt$};\node at (0.2,-.37) {${\scriptstyle{\red{\chi_i}}}$};
	\node at (-0.2,-.37) {${\scriptstyle{\red{\chi_i}}}$};
\end{tikzpicture}
}
=
\mathord{
\begin{tikzpicture}[baseline = -1mm]
	\draw[<-,red] (0.25,.3) to (-0.25,-.3);
	\draw[->,red] (0.25,-.3) to (-0.25,.3);
	\node at (0.12,0.135) {$\dt$};\node at (0.2,-.37) {${\scriptstyle{\red{\chi_i}}}$};
	\node at (-0.2,-.37) {${\scriptstyle{\red{\chi_i}}}$};
	\end{tikzpicture}}
	+\:\mathord{
\begin{tikzpicture}[baseline = -1mm]
	\draw[->,red] (0.08,-.3) to (0.08,.3);
	\draw[->,red] (-0.28,-.3) to (-0.28,.3);\node at (0.2,-.37) {${\scriptstyle{\red{\chi_i}}}$};
	\node at (-0.2,-.37) {${\scriptstyle{\red{\chi_i}}}$};
\end{tikzpicture}
}$
and $\mathord{
\begin{tikzpicture}[baseline = -1mm]
	\draw[<-,red] (0.25,.3) to (-0.25,-.3);
	\draw[->,blue] (0.25,-.3) to (-0.25,.3);
	\node at (-0.12,-0.145) {$\dt$};\node at (0.2,-.37) {${\scriptstyle{\blue{\chi_j}}}$};
	\node at (-0.2,-.37) {${\scriptstyle{\red{\chi_i}}}$};
\end{tikzpicture}
}
=
\mathord{
\begin{tikzpicture}[baseline = -1mm]
	\draw[<-,red] (0.25,.3) to (-0.25,-.3);
	\draw[->,blue] (0.25,-.3) to (-0.25,.3);
	\node at (0.12,0.135) {$\dt$};\node at (0.2,-.37) {${\scriptstyle{\blue{\chi_j}}}$};
	\node at (-0.2,-.37) {${\scriptstyle{\red{\chi_i}}}$};
	\end{tikzpicture}}$

	By a direct computation, the natural transformation
$\mathord{
	\begin{tikzpicture}[baseline = -1mm]
		\draw[<-,red] (0.25,.3) to (-0.25,-.3);
		\draw[->,blue] (0.25,-.3) to (-0.25,.3);
		\node at (-0.12,-0.145) {$\dt$};\node at (0.2,-.37) {${\scriptstyle{\blue{\chi_j}}}$};
		\node at (-0.2,-.37) {${\scriptstyle{\red{\chi_i}}}$};
	\end{tikzpicture}
}
-
\mathord{
	\begin{tikzpicture}[baseline = -1mm]
		\draw[<-,red] (0.25,.3) to (-0.25,-.3);
		\draw[->,blue] (0.25,-.3) to (-0.25,.3);
		\node at (0.12,0.135) {$\dt$};\node at (0.2,-.37) {${\scriptstyle{\blue{\chi_j}}}$};
		\node at (-0.2,-.37) {${\scriptstyle{\red{\chi_i}}}$};
\end{tikzpicture}}$ (for all $1\leq i,j\leq d$)
can be represented by
$$\k H _{r+2}\otimes_{\k( H ^{(r+1)}\times H ^{(r+2)})}(V_{\chi_i}\otimes V_{\chi_j})\to \k H _{r+2}\otimes_{\k( H ^{(r+1)}\times H ^{(r+2)})}(V_{\chi_i}\otimes V_{\chi_j}),$$
$$g\otimes m_1\otimes m_2\mapsto g(s_{r+1}\frac{L_{r+1}}{c_i}-\frac{L_{r+2}}{c_i}s_{r+1})\otimes m_2\otimes m_1=g\frac{M^{(r+1,r+2)}}{c_i}m_2\otimes m_1.$$

By Lemma \ref{Lemma:ortho}, we know that $M^{(r+1,r+2)}$ acts on $V_{\chi_i}\boxtimes V_{\chi_j}$ by $0$  and on $V_{\chi_i}^{\boxtimes 2}$  by $c_iP$, hence the lemma is proved.
\end{proof}

\subsubsection{Rightward crossings}
Define the rightward crossings by
\begin{align*}
\Rcross^{\chi_i,\chi_i}:=\mathord{
	\begin{tikzpicture}[baseline = -.5mm]
		\draw[thin,red,->] (-0.28,-.3) to (0.28,.4);
		\draw[<-,thin,red] (0.28,-.3) to (-0.28,.4);\node at (0.28,-.5) {${\scriptstyle{\red{\chi_i}}}$};
		\node at (-0.28,-.5) {${\scriptstyle{\red{\chi_i}}}$};
	\end{tikzpicture}
}:=
\mathord{
	\begin{tikzpicture}[baseline = 0]
		\draw[->,thin,red] (0.3,-.5) to (-0.3,.5);
		\draw[-,thin,red] (-0.2,-.2) to (0.2,.3);
		\draw[-,thin,red] (0.2,.3) to[out=50,in=180] (0.5,.5);
		\draw[->,thin,red] (0.5,.5) to[out=0,in=90] (0.8,-.5);
		\draw[-,thin,red] (-0.2,-.2) to[out=230,in=0] (-0.6,-.5);
		\draw[-,thin,red] (-0.6,-.5) to[out=180,in=-90] (-0.85,.5);
		\node at (0.8,-.6) {${\scriptstyle{\red{\chi_i}}}$};
		\node at (0.3,-.6) {${\scriptstyle{\red{\chi_i}}}$};
	\end{tikzpicture}
}\:\quad\text{ and}\quad
&
\Rcross^{\chi_i,\chi_j}:=\mathord{
	\begin{tikzpicture}[baseline = -.5mm]
		\draw[thin,red,->] (-0.28,-.3) to (0.28,.4);
		\draw[<-,thin,blue] (0.28,-.3) to (-0.28,.4);\node at (0.28,-.5) {${\scriptstyle{\blue{\chi_j}}}$};
		\node at (-0.28,-.5) {${\scriptstyle{\red{\chi_i}}}$};
	\end{tikzpicture}
}:=
\mathord{
	\begin{tikzpicture}[baseline = 0]
		\draw[->,thin,red] (0.3,-.5) to (-0.3,.5);
		\draw[-,thin,blue] (-0.2,-.2) to (0.2,.3);
		\draw[-,thin,blue] (0.2,.3) to[out=50,in=180] (0.5,.5);
		\draw[->,thin,blue] (0.5,.5) to[out=0,in=90] (0.8,-.5);
		\draw[-,thin,blue] (-0.2,-.2) to[out=230,in=0] (-0.6,-.5);
		\draw[-,thin,blue] (-0.6,-.5) to[out=180,in=-90] (-0.85,.5);
		\node at (0.8,-.6) {${\scriptstyle{\blue{\chi_j}}}$};
		\node at (0.3,-.6) {${\scriptstyle{\red{\chi_i}}}$};
	\end{tikzpicture}
}\:.
\end{align*}
In other words, we define $\Rcross^{\chi_i,\chi_j}=
(1_{E^{\chi_j}}\otimes1_{F^{\chi_i}}\otimes\varepsilon^{\chi_j}_R)
\circ(1_{E^{\chi_j}}\otimes \mathrm{T}^{\chi_i,\chi_j}\otimes 1_{E^{\chi_j}})\circ(\eta^{\chi_j}_R\otimes 1_{F^{\chi_i}}\otimes 1_{E^{\chi_j}})$.

\begin{lemma}\label{lcrosswreath}
For $r\geq1$, the natural transformation
$\Rcross^{\chi_i,\chi_j}:F^{\chi_i}_{r-1}E^{\chi_j}_{r-1}
\Rightarrow E^{\chi_j}_rF^{\chi_i}_r$ is the $(\k H_r,\k H_r)$-bimodule map
\[
 \k H_r e^{(r)}_{\chi_i}\otimes_{\k H_{r-1}}e^{(r)}_{\chi_j}\k H_r
 \longrightarrow e^{(r+1)}_{\chi_j}\k H_{r+1}e^{(r+1)}_{\chi_i}
\]
given by
\[
 ae^{(r)}_{\chi_i}\otimes e^{(r)}_{\chi_j}b
 \longmapsto e^{(r+1)}_{\chi_j}a(r,r+1)b e^{(r+1)}_{\chi_i}.
\]
At $r=0$ its source is zero.
\end{lemma}
\begin{proof}
The source is generated as a bimodule by
$e^{(r)}_{\chi_i}\otimes e^{(r)}_{\chi_j}$.  Expanding the right unit,
upward crossing, and right counit sends this generator to
$e^{(r+1)}_{\chi_j}(r,r+1)e^{(r+1)}_{\chi_i}$, which gives the formula.
\end{proof}

\subsubsection{Leftward crossings}
We define by $\Lcross^{\chi_i,\chi_j}=(\varepsilon^{\chi_j}_L\otimes 1_{F^{\chi_i}}\otimes 1_{E^{\chi_j}} )\circ(1_{E^{\chi_j}}\otimes \mathrm{T}^{\chi_j,\chi_i} \otimes 1_{E^{\chi_j}})\circ(1_{E^{\chi_j}}\otimes 1_{F^{\chi_i}}\otimes\eta^{\chi_j}_L).$
In terms of diagrams, we have
\begin{align*}
\Lcross^{\chi_i,\chi_i}:=\mathord{
	\begin{tikzpicture}[baseline = -.5mm]
		\draw[<-,red] (-0.28,-.3) to (0.28,.4);
		\draw[->,red] (0.28,-.3) to (-0.28,.4);\node at (0.28,-.5) {${\scriptstyle{\red{\chi_i}}}$};
		\node at (-0.28,-.5) {${\scriptstyle{\red{\chi_i}}}$};
	\end{tikzpicture}
}:=
\mathord{
	\begin{tikzpicture}[baseline = 0]
		\draw[-,red] (-0.2,.2) to (0.2,-.3);
		\draw[<-,red] (0.3,.5) to (-0.3,-.5);
		\draw[-,red] (0.2,-.3) to[out=130,in=180] (0.5,-.5);
		\draw[-,red] (0.5,-.5) to[out=0,in=270] (0.8,.5);
		\draw[-,red] (-0.2,.2) to[out=130,in=0] (-0.5,.5);
		\draw[->,red] (-0.5,.5) to[out=180,in=-270] (-0.8,-.5);
		\node at (-0.8,-.6) {${\scriptstyle{\red{\chi_i}}}$};
		\node at (-0.3,-.6) {${\scriptstyle{\red{\chi_i}}}$};
	\end{tikzpicture}
}\:
\quad\text{ and}\quad
&\Lcross^{\chi_i,\chi_j}:=\mathord{
	\begin{tikzpicture}[baseline = -.5mm]
		\draw[<-,blue] (-0.28,-.3) to (0.28,.4);
		\draw[->,red] (0.28,-.3) to (-0.28,.4);\node at (0.28,-.5) {${\scriptstyle{\red{\chi_i}}}$};
		\node at (-0.28,-.5) {${\scriptstyle{\blue{\chi_j}}}$};
	\end{tikzpicture}
}:=
\mathord{
	\begin{tikzpicture}[baseline = 0]
		\draw[-,blue] (-0.2,.2) to (0.2,-.3);
		\draw[<-,red] (0.3,.5) to (-0.3,-.5);
		\draw[-,blue] (0.2,-.3) to[out=130,in=180] (0.5,-.5);
		\draw[-,blue] (0.5,-.5) to[out=0,in=270] (0.8,.5);
		\draw[-,blue] (-0.2,.2) to[out=130,in=0] (-0.5,.5);
		\draw[->,blue] (-0.5,.5) to[out=180,in=-270] (-0.8,-.5);
		\node at (-0.8,-.6) {${\scriptstyle{\blue{\chi_j}}}$};
		\node at (-0.3,-.6) {${\scriptstyle{\red{\chi_i}}}$};
	\end{tikzpicture}
}\:.
\end{align*}

\begin{lemma}\label{rcrosswreath}
For $r\geq1$, the natural transformation
$\Lcross^{\chi_i,\chi_j}:E^{\chi_j}_rF^{\chi_i}_r
\Rightarrow F^{\chi_i}_{r-1}E^{\chi_j}_{r-1}$ is the
$(\k H_r,\k H_r)$-bimodule map
\[
 e^{(r+1)}_{\chi_j}\k H_{r+1}e^{(r+1)}_{\chi_i}
 \longrightarrow \k H_r e^{(r)}_{\chi_i}\otimes_{\k H_{r-1}}
 e^{(r)}_{\chi_j}\k H_r
\]
with
\[
 e^{(r+1)}_{\chi_j}(r,r+1)e^{(r+1)}_{\chi_i}
 \longmapsto e^{(r)}_{\chi_i}\otimes e^{(r)}_{\chi_j},
 \qquad e^{(r+1)}_{\chi_i}\longmapsto0\quad(i=j).
\]
At $r=0$ this map is zero.
\end{lemma}
\begin{proof}
Insert the matrix-unit formula for $Z_r^{\chi_j}$ in the defining mate.
The parabolic projection kills the closed double coset.  On the open double
coset only the identity coset representative and the matrix unit
$E^{\chi_j}_{11}$ survive, giving the two displayed values.
\end{proof}

\subsubsection{Mackey formula relations}
\begin{lemma}\label{lem:Mackey}
\begin{itemize}
	\item[$(a)$]
	The natural transformation\begin{align*}
		\begin{array}{rl}
			\left[\!\!\!\!\!
			\begin{array}{l}\,\,\,\mathord{
					\begin{tikzpicture}[baseline = 0]
						\draw[<-,thin,red] (-0.28,-.3) to (0.28,.4);
						\draw[->,thin,red] (0.28,-.3) to (-0.28,.4);\node at (0.28,-.4) {${\scriptstyle{\red{\chi_i}}}$};
						\node at (-0.28,-.4) {${\scriptstyle{\red{\chi_i}}}$};
					\end{tikzpicture}
				}\\\,\,\,\mathord{
					\begin{tikzpicture}[baseline = 1mm]
						\draw[-,thin,red] (0.4,0) to[out=90, in=0] (0.1,0.4);
						\draw[->,thin,red] (0.1,0.4) to[out = 180, in = 90] (-0.2,0);\node at (0.4,-.3) {${\scriptstyle{\red{\chi_i}}}$};
						\node at (-0.2,-.3) {${\scriptstyle{\red{\chi_i}}}$};
					\end{tikzpicture}
				}
			\end{array}
			\right]
			:
			E^{\chi_i}F^{\chi_i} \Rightarrow
			F^{\chi_i}E^{\chi_i} \oplus \unit
		\end{array}
	\end{align*}

	is an isomorphism of functors. It has a two-sided inverse
	\begin{align*}
		\begin{array}{rl}
			\left[
			\:\mathord{
				\begin{tikzpicture}[baseline = 0]
					\draw[->,thin,red] (-0.28,-.3) to (0.28,.4);
					\draw[<-,thin,red] (0.28,-.3) to (-0.28,.4);\node at (0.28,-.4) {${\scriptstyle{\red{\chi_i}}}$};
					\node at (-0.28,-.4) {${\scriptstyle{\red{\chi_i}}}$};
				\end{tikzpicture}
			}
			\:\:\:\:
			\mathord{
				\begin{tikzpicture}[baseline = -0.9mm]
					\draw[<-,thin,red] (0.4,0.2) to[out=-90, in=0] (0.1,-.2);
					\draw[-,thin,red] (0.1,-.2) to[out = 180, in = -90] (-0.2,0.2);
					\node at (0.4,.3) {${\scriptstyle{\red{\chi_i}}}$};
					\node at (-0.2,.3) {${\scriptstyle{\red{\chi_i}}}$};
				\end{tikzpicture}
			}
			\right]
			:F^{\chi_i}E^{\chi_i}\oplus
			\unit
			\Rightarrow E^{\chi_i}F^{\chi_i}.
			&
		\end{array}
	\end{align*}

	\item[$(b)$]
	For $i\ne j$, the natural transformations
	\begin{align*}
		\begin{array}{rl}
			\mathord{
				\begin{tikzpicture}[baseline = 0]
					\draw[->,thin,blue] (-0.28,-.3) to (0.28,.4);
					\draw[<-,thin,red] (0.28,-.3) to (-0.28,.4);\node at (0.28,-.4) {${\scriptstyle{\red{\chi_i}}}$};
					\node at (-0.28,-.4) {${\scriptstyle{\blue{\chi_j}}}$};
				\end{tikzpicture}
			}
			:F^{\chi_j} E^{\chi_i}
			\Rightarrow
			E^{\chi_i}  F^{\chi_j}
		\end{array}
		~\text{and}~
		\begin{array}{rl}
			\mathord{
				\begin{tikzpicture}[baseline = 0]
					\draw[->,blue] (0.28,-.3) to (-0.28,.4);
					\draw[<-,red] (-0.28,-.3) to (0.28,.4);\node at (0.28,-.4) {${\scriptstyle{\blue{\chi_j}}}$};
					\node at (-0.28,-.4) {${\scriptstyle{\red{\chi_i}}}$};
				\end{tikzpicture}
			}
			:E^{\chi_i} F^{\chi_j}
			\Rightarrow
			F^{\chi_j} E^{\chi_i}
		\end{array}
	\end{align*}
	are two-sided inverse of each other.

\end{itemize}
\end{lemma}
\begin{proof}
Both statements follow from Lemmas \ref{lcrosswreath}, \ref{rcrosswreath} and a straightforward computation on generators.
\end{proof}

\begin{proof}[Proof of Theorem~\ref{Thm:doubleheisenberg}]
The degenerate affine Hecke relations \eqref{dAHA} and the mixed relations
\eqref{mixedHecke} are Lemma~\ref{lem:12relations}.  The four maps in
\S\ref{sec:adjwreath} are adjunction maps, so they satisfy
\eqref{eq:rightadj}.  Finally, Lemma~\ref{lem:Mackey} gives
\eqref{dMackey1}--\eqref{dMackey2} and \eqref{mixedMackey}.
\end{proof}

							\subsection{Application to representation theory of wreath product groups}
							\subsubsection{Partitions and $\ell$-cores}
                            Write $b(x,y)$ for the box in $\Z_{>0}\times \Z_{>0}$ at position $(x,y)$. A \emph{partition} of $n$ is a non-increasing sequence of non-negative integers
							$\lambda = (\lambda_1 \geqslant \lambda_2 \geqslant \cdots)$ with $\sum_i \lambda_i=n$,
							to which one
							associates the so-called \emph{Young diagram}
							$Y(\lambda)=\{b(x,y)\in\bbZ_{>0}\times\bbZ_{>0}\,\mid\,y\leqslant \lambda_x\}.$
							We write $\scrP=\bigsqcup_n\scrP_n$ for the set
							of all partitions, where $\scrP_n$ is the set of partitions of $n$.
							For $\lambda\in \scrP$, we denote by $|\lambda|$ the \emph{weight}
							of $\lambda$ and by $\lambda^*$ the partition  conjugate to $\lambda$.

							\smallskip

							Let $X=\{x_1,\cdots, x_l\}$ be a finite set. An \emph{$X$-partition} of $n$ is an $l$-tuple $\tuple\lambda=(\lambda^{(x_1)},\ldots,\lambda^{(x_l)})$ of partitions whose weights add up to $n$,
							and its Young diagram is the set $$Y(\tuple\lambda)=\bigsqcup_{x\in X}Y(\lambda^{(x)})\times\{x\}.$$
							The integer $|\tuple\lambda|=\sum\limits_{x\in X}|\lambda^{(x)}|$ is called the weight of the $X$-partition.
							We write $\scrP(X)=\bigsqcup_n\scrP_n(X)$ for the set of all $X$-partitions, where
							$\scrP_n(X)$ is the set of $X$-partitions of $n$.

							\smallskip

							Let  $\lambda$ be a partition.
							Then it is uniquely determined by
							the set $\beta(\lambda)=\{\lambda_u+1-u\,\mid
							\,u\geqslant 1\}$ of the  so-called \emph{$\beta$-numbers}.
							For a positive integer $l$,  an \emph{$l$-hook} of $\lambda$ is a pair $(x,x+l)$ such
							that $x+l\in \beta(\lambda)$ and $x\not\in \beta(\lambda)$, and
							an \emph{$l$-core} of $\lambda$
							is the partition obtained by recursively removing $l$-hooks
							$(x,x+l)$
							(i.e., replacing $x+l$ with $x$ in the set of $\beta$-numbers).

							\smallskip
							The \emph{content} of the box in the $i$-th row from the top and $j$-th column from the left is defined to be $j-i \in \mathbb{Z}$. If we label the boxes of a Young diagram  by positive integers, we call it a \emph{Young tableau}.

							Suppose that $\lambda \vdash n$. If a Young diagram $Y(\lambda)$ is labelled by $1,\ldots,n$ such that the numbers strictly increase in each column from top to bottom and in each row from left to right, then the corresponding Young tableau is called a \emph{standard Young tableau} (of shape $\lambda$). For a standard Young tableau $T$ of shape $\lambda$, an element $r \in \{1,\ldots,n\}$ labels a unique box of $T$; we write $c_T(r)$ for the content of the box in $T$ labelled $r$. For example, if $T$ is the standard Young tableau of shape $(3,3,1)$ below,
							\begin{align*}
								\ytableausetup{centertableaux}\begin{ytableau}
									1 & 3 & 4 \\
									2 & 6 & 7 \\
									5
								\end{ytableau}
							\end{align*}
							\noindent
							then $c_T(5) = 1-3 = -2$, $c_T(6) = 2-2 = 0$, and $c_T(7) = 3-2 = 1$.

							\subsubsection{Content functions and their modulo $\ell$ reductions}\label{sub:contentfunction}
							Let $\lambda\in \scrP$ be a partition. Recall that we denote by $Y(\lambda)$ the Young diagram associated with $\lambda$. A box $b$ is said to be addable (resp. removable) to $Y(\lambda)$ if $b\notin Y(\lambda)$ (resp. $b\in Y(\lambda)$) and $Y(\lambda)\cup \{b\}$ (resp. $Y(\lambda)\backslash \{b\}$) is still a Young diagram.
							Then we define the content function of $\lambda$ by
							$$\bfO_{\lambda}(u)=\frac{\prod\limits_{\text{$b$  addable }}(u-\text{cont}(b))}{\prod\limits_{\text{$b$  removable }}(u-\text{cont}(b))}.$$

							\begin{example}
								Let $\lambda=(4,1,1)$.
								The Young diagram of $\lambda$ with its  contents is
								$$
								\begin{array}{|c|c|c|c|}
									\hline
									0& 1  &2 &3   \\
									\hline
									-1 \\
									\cline{1-1}
									-2 \\
									\cline{1-1}
								\end{array}.$$
								Then we have  $\bfO_{\lambda}(u)=\frac
								{(u-4)(u)(u+3)}{(u-3)(u+2)}.$
							\end{example}

							\begin{lemma}
								We have the following identity:

								$$\bfO_{\lambda}(u)=u\prod\limits_{b\in \,Y(\lambda)}(1-(u-\mathrm{cont}(b))^{-2}).$$
							\end{lemma}
							\begin{proof}
								The identity trivially holds when $\lambda=\emptyset$.
								We proceed by an induction on $n=|\lambda|.$
								Let $b_{\min}$ denote the removable box of $\lambda$ with minimal content and set $\mu=\lambda\setminus \{b_\min\}$ to be the partition obtained from $\lambda$ by removing the box $b_\min.$ By the induction hypothesis, $\mu$ satisfies
								$\bfO_{\mu}(u)=u\prod\limits_{b\in Y(\mu)}(1-(u-\mathrm{cont}(b))^{-2}).$
                                Compare the addable and removable boxes of $\lambda$ with those of $\mu.$ If $\lambda=(\lambda_1,\cdots,\lambda_t)$ with $\lambda_t\neq 0,$ then $b_\min=b(t,\lambda_t)$ and $\mu=(\lambda_1,\cdots,\lambda_t-1).$
								There are four cases based on the shape of $\lambda$ as follows:
								\begin{itemize}
									\item[$(1)$] $t=1$ or $t>1$ with $ \lambda_{t-1}>\lambda_{t}$ and $\lambda_t>1$;
									\item[$(2)$] $t=1$ or $t>1$ with $\lambda_{t-1}>\lambda_{t}$ and $\lambda_t=1$;
									\item[$(3)$] $t>1$ with $ \lambda_{t-1}=\lambda_{t}$ and $\lambda_t>1$;
									\item[$(4)$] $t>1$ with $\lambda_{t-1}=\lambda_{t}$ and $\lambda_t=1$.
								\end{itemize}

								In all cases, the differences in addable and removable boxes are captured by sets $X$ and $Y$ as follows:
								$$ \{\text{addable boxes of $\lambda$}\}\setminus X=\{\text{addable boxes of $\mu$}\}\setminus\{b_\min\}$$
								and $$ \{\text{removable boxes of $\lambda$}\}\setminus\{b_\min\}=\{\text{removable boxes of $\mu$}\}\setminus Y.$$
								\begin{center}
									\begin{table}[h]\label{diag:extrasymm}
										\begin{tabular}{c|c|c}
											\hline
											Cases&$X$&$Y$\\\hline
											$(1)$&$ \{b(t,\lambda_t+1)\}$&$\{b(t,\lambda_t-1)\}$\\\hline
											$(2)$&$\{b(t+1,\lambda_{t}),b(t,\lambda_t+1)\}$&$\emptyset$\\\hline    $(3)$&$\emptyset$&$\{b(t,\lambda_{t}-1),b(t-1,\lambda_t)\}$\\\hline
											$(4)$&$\{b(t+1,\lambda_t)\}$&$\{b(t-1,\lambda_t)\}$\\\hline
										\end{tabular}
									\end{table}
								\end{center}
								Note that $\mathrm{cont}(b(t,\lambda_t\pm1))=\mathrm{cont}(b_\min)\mp1,$ and $\mathrm{cont}(b(t\pm1,\lambda_t))=\mathrm{cont}(b_\min)\pm1.$ Since $Y(\lambda)=Y(\mu)\sqcup b_\min$, in all cases, we have:
								\begin{align*}
									\bfO_{\lambda}(u)=&\bfO_{\mu}(u)\cdot \frac{(u-\mathrm{cont}(b_\min)-1)(u-\mathrm{cont}(b_\min)+1)}{(u-\mathrm{cont}(b_\min))^2}\\
									=&(u\prod\limits_{b\in Y(\mu)}(1-(u-\mathrm{cont}(b))^{-2}))\cdot (1-(u-\mathrm{cont}(b_\min))^{-2})\\
									=&u\prod\limits_{b\in Y(\lambda)}(1-(u-\mathrm{cont}(b))^{-2}).
								\end{align*} This completes the proof.
							\end{proof}

									Since $\bbO_{\mu}(u)=\frac{\prod_{i=1}^n(u-x_i)}{\prod_{j=1}^{n-1}(u-y_j)} \in \bbZ(u)$ for any $\mu,$ we can define its modulo $\ell$ reduction by $$\overline{\bfO}_{\mu}(u)=\frac{\prod_{i=1}^n(u-\overline{x_i})}{\prod_{j=1}^{n-1}(u-\overline{y_j})}$$
									where $\overline{(\bullet)}:\bbZ\to \bbF_\ell$ is the modulo $\ell$ reduction homomorphism. Hence for any partition $\mu,$ we have $\overline{\bbO}_{\mu}(u)\in \bbF_\ell(u).$

									For a partition $\lambda$ and $i\in\bbZ/\ell\bbZ$ we define
									\begin{align*}
										&N_i(\lambda ) = \sharp\{ \text{ addable $i$-nodes of $\lambda$}\}
										- \sharp\{ \text{ removable $i$-nodes of $\lambda$} \}.
									\end{align*}

									Then by the definition, we have
									$$\overline{\bfO}_{\lambda}(u)=\prod_{i\in \bbZ/\ell\bbZ}(u-i)^{{N_i}(\lambda)}.$$

									We know that $\lambda$ and $\mu$ have the same $\ell$-core if and only if $\lambda$ and $\mu$ have the same modulo $\ell$ content multisets.
									Fayers \cite{F06} showed that $\lambda$ and $\mu$ have the same modulo $\ell$ content multisets if and only if  $N_i(\lambda)=N_{i}(\mu)$ for all $i\in \bbZ/\ell\bbZ.$ So we have the following lemma.

									\begin{lemma}[cf.~\cite{F06}]
										Let $\lambda$ and $\mu$ be two partitions of $n$.
										Then $\overline{\bfO}_{\lambda}(u)=\overline{\bfO}_{\mu}(u)$ if and only if $\lambda$ and $\mu$ have the same $\ell$-core.
									\end{lemma}

									\subsubsection{Fock spaces}

									The \emph{Fock space} $\mathcal{F}$
									is a $\bbC$-vector space with a linear basis parametrized by the set of all partitions. Thus we write
									\[
									\mathcal{F}=\bigoplus_{{\boldsymbol{\lambda}}\in\calP
									}\bbC{\boldsymbol{\lambda}}.
									\]
									The Fock space $\mathcal{F}$ may be endowed with a structure of
									$\widehat{\mathfrak{sl}}_{\ell}$ and $\mathfrak{sl}_{\infty}$-modules. Let $\lambda$ be a partition (identified with its Young diagram).

									\begin{theorem}[\cite{Ug99}]\label{thm:Fock}
										The Fock space
										$\mathcal{F}$ has a structure of an integrable
										$\widehat{\mathfrak{sl}}'_{\ell}$-module $\mathcal{F}_{\ell}$ defined
										by
										\begin{gather*}
											e_{i}{\lambda}=\sum_{\text{res}(\lambda/\mu)=i}{\mu},\quad{f_{i}{\lambda}=\sum_{\text{res}(\mu/\lambda
													)=i}\mu},\\[5pt]
											\qquad h_{i}{\lambda}={N_{i}(\lambda)}
											{\lambda},
										\end{gather*}
										for $i\in\mathbb{Z}/\ell\mathbb{Z}$.
									\end{theorem}

									The following result is implicit in \cite[Proposition~3.5]{Ug99}.

									\begin{proposition}[\cite{Ug99}]
										\label{Prop_compa_action}The $\widehat{\mathfrak
											{sl}}_{\ell}'$ and $\mathfrak{sl}_{\infty}$-module
										structures $\mathcal{F}_{\ell}$ and $\mathcal{F}_{\infty
										}$ are compatible in the sense that we may write
										the action of $e_{i}, f_{i}$ and $h_{i}$, for $i\in\mathbb{Z}/\ell\mathbb{Z}$, as follows:
										\begin{align*}
											e^{(\ell)}_{i} &  =\sum_{j\in\mathbb{Z},j\equiv i(\text{mod }\ell)}  e^{(\infty)}_{j},\\
											f^{(\ell)}_{i} &  =\sum_{j\in\mathbb{Z},j\equiv i(\text{mod }\ell)} f^{(\infty)}_{j},\\
											h^{(\ell)}_{i} &  =\sum_{j\in\mathbb{Z},j\equiv i(\text{mod }\ell)}h^{(\infty)}_{j}.
										\end{align*}
									\end{proposition}

									\begin{remark}
										The infinite sums in the proposition reduce in fact to
										finite ones since the number of nodes in ${\lambda}$ is finite.
									\end{remark}

									The empty multipartition $\mathbf{\emptyset}$ is a highest weight vector in
									$\mathcal{F}_{\ell}$ and $\mathcal{F}_{\infty}$ of
									weight $\Lambda_{0}$. We
									then define $V_{\ell}(0)$ and $V_{\infty}(0)$ as the corresponding
									highest weight modules.
									By the previous proposition,
									it follows that $V_{\infty}(0)$ is endowed
									with the structure of a
									$\widehat{\mathfrak{sl}}'_{\ell}$-module and
									$V_{\ell}(0)$ coincides with the $\widehat{\mathfrak{sl}}_{\ell}$-submodule of $V_{\infty
									}(0)$ generated by the highest weight vector $\mathbf{\emptyset}$.

The following lemma gives an explicit formula for bubbles.
									\begin{lemma}
										\begin{equation*}
											\mathord{
												\begin{tikzpicture}[baseline = -1mm]
													\draw[red, <-] (0.08,-.4) to (0.08,.4);
													\node at (-0.6,0) {$\blue\anticlock \scriptstyle (u)$};
													\node at (0.08,-.5) {${\scriptstyle{\red{\alpha}}}$};
													\node at (-0.7,-.45) {${\scriptstyle{\blue{\beta}}}$}      ;\end{tikzpicture}
											}=\left\{
											\begin{array}{ll}
												\mathord{
													\begin{tikzpicture}[baseline = -1mm]
														\draw[red, <-] (0.08,-.4) to (0.08,.4);      \node at (0.8,0) {$\blue\anticlock \scriptstyle (u)$};
														\node at (.08,0) {$\dt$};
														\node at (-.6,0) {$\scriptstyle{1-(u-x)^{-2}}$};
														\node at (0.08,-.5) {${\scriptstyle{\red{\alpha}}}$};
														\node at (0.7,-.45) {${\scriptstyle{\blue{\beta}}}$}      ;\end{tikzpicture}}&
												\text{if $\red \alpha= \blue \beta$,}\\\\
												\mathord{
													\begin{tikzpicture}[baseline = -1mm]
														\draw[red, <-] (0.08,-.4) to (0.08,.4);
														\node at (0.8,0) {$\blue\anticlock \scriptstyle (u)$};
														\node at (0.08,-.5) {${\scriptstyle{\red{\alpha}}}$};
														\node at (0.7,-.45) {${\scriptstyle{\blue{\beta}}}$};\end{tikzpicture}} &\text{if $\red \alpha\neq \blue \beta$,}
											\end{array}\right.
										\end{equation*}
									\end{lemma}

									\begin{proof}
										It follows from Lemma~\ref{lem:colorbubbleslide} directly.
									\end{proof}

        \subsection{Invariants of blocks}\label{sub:coloredweight}
        { }

									Recall that $\hat H $ is the set of isomorphism classes of irreducible representations of $\k H .$
									We write $\calP_n(\hat H )$ for the set of
									$\hat H $-partitions of $n$. Let $\mu \in \calP_n(\hat H )$. A Young $\hat H $-tableau of shape $\tuple \mu$ is obtained by taking the Young $\hat H $-diagram $Y(\tuple \mu)$ and filling its $|\tuple\mu |$ boxes (bijectively) with the numbers $1, 2, \ldots, |\tuple\mu|$. A Young $\hat H $-tableau is said to be standard if the numbers in the boxes strictly increase along each row and each column of all Young diagrams occurring in $\tuple\mu$. Let $\mathrm{tab}_ H (n, \tuple\mu)$, where $\tuple\mu \in \calP_n(\hat H )$, denote the set of all standard Young $\hat H $-tableaux of shape $\tuple\mu.$

									The complex irreducible representations of $H_n$ are parametrized by $\calP_n(\hat H )$.
									Pushkarev \cite{Pus99}, extending the Okounkov--Vershik method for $S_n$ \cite{OV96} to wreath products, gave a spectral explanation for this correspondence; see also the systematic treatment of Mishra--Srinivasan \cite{MS16}.  Namely, an internal analysis of the irreducible representations of $ H _n$ yields spectral objects parametrizing the irreducible representations, together with a bijection between these spectral objects and $\calP_n(\hat H )$. This approach is inductive in nature and has the following advantages:

									A natural byproduct of the theory yields a parametrization of the bases of
									irreducible $ H ^n$-modules using standard Young $\hat H $-tableaux and bases of irreducible $ H $-modules. More precisely, for $\mu \in \calP_n(\hat H )$, we have a Gelfand--Tsetlin decomposition
									$$V^{\tuple\lambda}\cong \prod\limits_{T\in \mathrm{tab}(n,\tuple\lambda)} V_T,$$
									where each $V_T$ is closed under the action of $ H ^n = H  \times\cdots \times  H $ ($n$ factors) and, as a
									$ H ^n$-module, is isomorphic to the irreducible $ H ^n$-module
									$$ V_T=V^{r_T(1)}\boxtimes \cdots\boxtimes V^{r_T(n)}.$$

		\begin{theorem}\label{cwf}
			For any irreducible module $\rho= \rho_{\tuple \lambda}$ with $\tuple \lambda = (\lambda^{(\chi_1)},\cdots, \lambda^{(\chi_d)})$, we have
		$$\bbO^{\chi_i} (u)(\rho) = \bfO_{\lambda^{(\chi_i)}} (u^{(i)}).$$
			In particular, $\{\bbO^{\chi_i}(u)(-)\}_{i}$ are complete invariants of irreducible modules of $\bbK H _{\bullet}\mod.$
            Equivalently, the evaluated colored bubbles generate $Z(\bbK H_n)$ in every rank $n$.
									\end{theorem}
									\begin{proof}

										Assume that $|\tuple\lambda|=n$ and choose a $T\in \mathrm{tab}_H(n,\tuple\lambda).$
										We choose a GZ subspace $V_T$ of $V^{\tuple\lambda}$ such that $$V_T=V^{r_T(1)}\boxtimes \cdots \boxtimes V^{r_T(n)}$$ as $ H ^n$-modules.

                                        Choose a nonzero vector $v\in V_T$. By Schur's lemma, $\bbO^{\chi_i}(u)\cdot v=f(u)\cdot v$.
										It is enough to show that $f(u)=\bfO_{\lambda_i} (u)$.
										By the choice of $V_T$, we know that
										$$ E^{\chi_j}(V_T)\cong\begin{cases} V_T &\text{if $\chi_j=r_T(n)$}\\
											0&\text{if $\chi_j\neq r_T(n)$}
										\end{cases}.$$
										So $$\mathord{
											\begin{tikzpicture}[baseline = -1mm]
												\node at (-0.6,0) {$\anticlock \scriptstyle (u)$};
												\node at (-0.7,-.45) {${\scriptstyle{\chi_i}}$}      ;\end{tikzpicture}
										}\cdot v=\mathord{
											\begin{tikzpicture}[baseline = -1mm]
												\draw[<-] (0,-.4) to (0,.4);
												\node at (-0.6,0) {$\anticlock \scriptstyle (u)$};
												\node at (0.,-.5) {${\scriptstyle{r_T(n)}}$};
												\node at (-0.7,-.45) {${\scriptstyle{\chi_i}}$}
												;\end{tikzpicture}}\cdot v=\left\{
										\begin{array}{ll}
											\mathord{
												\begin{tikzpicture}[baseline = -1mm]
													\draw[ <-] (0.08,-.4) to (0.08,.4);      \node at (0.8,0) {$\anticlock \scriptstyle (u)$};
													\node at (.08,0) {$\dt$};
													\node at (-.6,0) {$\scriptstyle{1-(u-x)^{-2}}$};
													\node at (0.08,-.5) {${\scriptstyle{r_T(n)}}$};
													\node at (0.7,-.45) {${\scriptstyle{\chi_i}}$}      ;\end{tikzpicture}}\cdot v&
											\text{if $r_T(n)= \chi_i$,}\\\\
											\mathord{
												\begin{tikzpicture}[baseline = -1mm]
													\draw[ <-] (0.08,-.4) to (0.08,.4);
													\node at (0.8,0) {$\anticlock \scriptstyle (u)$};
													\node at (0.08,-.5) {${\scriptstyle{r_T(n)}}$};
													\node at (0.7,-.45) {${\scriptstyle{\chi_i}}$};\end{tikzpicture}} \cdot v&\text{if $r_T(n)\neq\chi_i$,}
										\end{array}\right.$$
										Since $\begin{tikzpicture}[baseline = -1mm]
											\draw[<-] (0.08,-.2) to (0.08,.2);
											\node at (0.08,-.3) {$\scriptstyle{\chi_i}$};
											\node at (.08,0) {$\dt$};		\end{tikzpicture}$ acts on $V_T$ by $c_T(n),$ so
\[
\bbO^{\chi_i}(u)^{[n]}\cdot v=
\begin{cases}
\bigl(1-(u-c_T(n))^{-2}\bigr)\,\bbO^{\chi_i}(u)^{[n-1]}\cdot v,
& r_T(n)=\chi_i,\\
\bbO^{\chi_i}(u)^{[n-1]}\cdot v,
& r_T(n)\neq\chi_i.
\end{cases}
\]
										Then by induction, using $\bbO^{\chi_i}(u)^{[0]}=u$, the series $\bbO^{\chi_i}(u)$ acts on $v$ by
$$u\prod\limits_{1\leq j\leq n,\, r_T(j)=\chi_i}(1-(u-c_T(j))^{-2}),$$
										i.e., $\bbO^{\chi_i} (u)(\rho) = \bfO_{\lambda_i} (u).$
									\end{proof}

                            For wreath products the bubbles admit an explicit formula; compare Remark~\ref{Rem:computebubbles} for the finite-classical case.
									For each $k\geq0$, the generalized Gelfand--Tsetlin algebra $GZ_k$ is the subalgebra of $\mathbb{K}[H_k]$ generated by $\mathbb{K}[H^k]$ and the centers $Z_1,\ldots,Z_k$ of $\mathbb{K}[H_1],\ldots,\mathbb{K}[H_k]$:
									\begin{equation*}
										GZ_k \overset{\mathrm{def}}{=} \langle \mathbb{K}[H^k], Z_1,\ldots,Z_k \rangle_{\mathbb{K}}.
									\end{equation*}
									We set
									\begin{equation*}
										\mathscr J_k \overset{\mathrm{def}}{=} \langle \mathbb{K}[H^k], L_1, L_2,\ldots,L_k \rangle_{\mathbb{K}}.
									\end{equation*}

																		The following theorem is due to Pushkarev \cite{Pus99}.
\begin{theorem}[\cite{Pus99}]
										For all $k \in \mathbb{N}$, we have $\mathscr J_k = GZ_k$.
									\end{theorem}

\begin{theorem}\label{Thm:expofbubbles}
										The elements $\bbO_n^{(\chi_i)}(u) \in Z(\bbK H _n)((u^{-1}))$ have the following forms:

										$$\bbO_n^{(\chi_i)}(u)=u\prod\limits_{1\leq j\leq n}\left(1-(u-c_i^{-1}L_j)^{-2}\cdot f^{(j)}_{\chi_i}\right).$$
									\end{theorem}

									\begin{proof}

										Set
										\[
										Z_{n,i}(u):=u\prod\limits_{1\leq j\leq n}\left(1-(u-c_i^{-1}L_j)^{-2}\cdot f^{(j)}_{\chi_i}\right)\in \bbK H_n((u^{-1})).
										\]
										On a Gelfand--Tsetlin summand $V_T$ of an irreducible module $V_{\tuple\lambda}$, the idempotent $f^{(j)}_{\chi_i}$ selects precisely the positions of color $\chi_i$, while $c_i^{-1}L_j$ acts there by $c_T(j)$. Hence
										\[
										Z_{n,i}(u)|_{V_T}=u\prod_{r_T(j)=\chi_i}\left(1-(u-c_T(j))^{-2}\right)=\bfO_{\lambda^{(\chi_i)}}(u).
										\]
										The right-hand side depends only on the multipartition $\tuple\lambda$, not on $T$. Thus $Z_{n,i}(u)$ acts by a scalar on every irreducible $\bbK H_n$-module; since $\bbK H_n$ is split semisimple, $Z_{n,i}(u)\in Z(\bbK H_n)[[u^{-1}]]$. Theorem~\ref{cwf} shows that $Z_{n,i}(u)$ and $\bbO_n^{(\chi_i)}(u)$ have the same scalar action on every irreducible module, and therefore they are equal.

									\end{proof}

			\subsection{Kac--Moody categorifications in characteristic 0}\label{sub:wreath-KM-char0}
            For each $i=1,\ldots,d$, the spectrum is $I_\infty^{(\chi_i)}=\bbZ$; set $I_\infty:=\bigsqcup_{i=1}^{d}I_\infty^{(\chi_i)}.$

									We identify the weight lattice $\X_{\infty}$ of $\fraks\frakl_{\infty}$
									with
									$$\mathrm{Core}_{\infty}=\left\{(a_i)_{i\in \bbZ}| a_i\in \bbZ, a_i=0\,\, \text{for all but finitely many $i\in \bbZ$}, \sum_{i\in\bbZ,a_i\neq0}a_i=0\right\}$$ via
									$\psi(\sum_i a_i\Lambda_i)=(a_i)_{i\in \bbZ}.$

									Since $N_i(\lambda ) = \sharp\{ \text{ addable $i$-nodes of $\lambda$}\}
									- \sharp\{ \text{ removable $i$-nodes of $\lambda$} \}$, by the definition, $$\bfO_{\lambda}(u)=\frac{\prod\limits_{\text{$b$  addable }}(u-\text{cont}(b))}{\prod\limits_{\text{$b$  removable }}(u-\text{cont}(b))}=\prod_{i\in \bbZ}(u-i)^{N_i(\lambda)}.$$
									The weight space of standard basis
									$|\tuple\lambda\rangle
									$ of Fock space is exactly
									$\mathrm{Wt}(\lambda)=\sum_i N_i\Lambda_i.$

									\begin{lemma}
										We have
										\begin{align*}E^{\chi_i}_j(V^{\tuple \lambda})=\begin{cases}
												V^{\tuple\mu } & \text{if $\tuple\lambda \setminus \tuple \mu $ has content $j$ and color $\chi_i$;}\\
												0,&\text{otherwise;}
											\end{cases}
										\end{align*} and
										\begin{align*}
											F^{\chi_i}_j(V^{\tuple \mu})=\begin{cases}
												V^{\tuple\lambda } & \text{if $\tuple\lambda \setminus \tuple \mu $ has content $j$ and color $\chi_i$;}\\
												0,&\text{otherwise.}
											\end{cases}
									\end{align*}\end{lemma}
									\begin{proof}
										By the adjoint properties, we only need to show the first formula.
										By \cite{MS16}, $$\Res^{ H _n}_{ H _{n-1}\times H ^{(n)}}(V^{\tuple\lambda})=\sum_{|\tuple\lambda\setminus\tuple\mu|=1}V^{\tuple\mu}\boxtimes V^{\chi_{{\mathrm{color}(\tuple\lambda\setminus\tuple\mu)}}}.$$ On a summand of color $\chi_a$, the normalized Jucys--Murphy operator $c_a^{-1}L_n$ (equivalently, the categorical dot $\mathrm{X}^{\chi_a}$) acts by the scalar $\mathrm{cont}(\tuple\lambda\setminus\tuple\mu)$. Then the lemma follows from the definition of $E^{\chi_i}_j.$
									\end{proof}

                                    The symmetric-product Heisenberg action and the preceding lemma give the following theorem.

									\begin{theorem}
										\begin{itemize}[leftmargin=8mm]
											\item[$\mathrm{(a)}$]
											The operators $[F_j^{(\chi_i)}],[E_j^{(\chi_i)}] $ for $j\in I_\infty^{(\chi_i)}$, $i=1,\cdots,d$ produce a representation of $\fraks\frakl'_{I_\infty}$ on $[\bbK H _\bullet\mod]$.
											\item[$\mathrm{(b)}$] The map
											$$|\tuple\lambda\rangle
											\mapsto [V^{\tuple\lambda}]$$ gives an
											$\fraks\frakl'_{I_\infty}$-module isomorphism $$
											\calF^{\otimes d}\simto [\bbK H _\bullet\mod]$$
											such that $\calF^{\otimes d}_m\simto [\bbK H _m\mod].$
										\end{itemize}
									\end{theorem}

								\subsection{Kac--Moody categorifications in characteristic $\ell$}\label{sub:wreath-KM-charl}
                                Assume that $\ell \nmid |H|$.  When $H$ is trivial, the wreath-product tower is the symmetric-group tower.
								For each color $\chi_i$, identify the modular residue spectrum with $I_\ell^{(\chi_i)}=\mathbb Z/\ell\mathbb Z$, and set $I_\ell:=\bigsqcup_{i=1}^{d}I_\ell^{(\chi_i)}$. Define $\boldsymbol{\Lambda}_0:=\sum_{i=1}^d\Lambda_0^{(\chi_i)}$, the highest weight for the direct-sum algebra $\fraks\frakl'_{I_\ell}\cong(\widehat{\mathfrak{sl}}'_\ell)^{\oplus d}$.

								We have the following reduction diagram
								$$\xymatrix{
									&\symHeis_d^{\bbK}\ar[d]&\symHeis_d^{\calO}\ar[r]^{\otimes_{\calO} \bbF} \ar[l]_{\otimes_{\calO}\bbK} &    \symHeis_d^{\bbF} \ar[d]  &\\ \bbK H  _\bullet\mod\,\,\,\circlearrowleft&\frakU(\fraks\frakl'_{I_\infty})_{\bbK}&&  \frakU(\fraks\frakl'_{I_\ell})_{\bbF}&\circlearrowright\,\,\,\bbF H  _\bullet\mod }
								$$

								It was shown in \cite{WW08} that $$\sum_{m\geq0}|\text{conj}_{\ell'}( H _m)|t^m=\prod_{m\geq1,\,\ell\nmid m}(\frac{1}{1-t^m})^{|\text{conj}( H )|}$$
								Since $|\Irr(\bbF H _m)|=|\text{conj}_{\ell'}( H _m)|,$
								we have $$\sum_{m\geq0}|\Irr(\bbF H _m)|
								t^m=\prod_{m\geq1,\,\ell\nmid m}(\frac{1}{1-t^m})^{|\Irr( \bbF H )|}.$$

								On the other hand, if $L(\Lambda_0)=\bigoplus_{m\geq0}L(\Lambda_0)_m$ denotes the basic
$\widehat{\mathfrak{sl}}'_\ell$-module graded by partition size, then
$$\sum_{m\geq0}\dim L(\Lambda_0)_m\,t^m=\prod_{m\geq1,\,\ell\nmid m}\frac{1}{1-t^m}.$$

								\begin{theorem}
                                \label{thm:Lieactwreath}
									There is a $\frakU(\fraks\frakl'_{I_\ell})$-categorical action on the category $\bbF H _{\bullet}\mod.$
									\begin{itemize}
										\item[$\mathrm{(a)}$]
										The operators $[F_j^{(\chi_i)}],[E_j^{(\chi_i)}] $ for $j\in I_\ell^{(\chi_i)}$, $i=1,\cdots,d$ produce a representation of $\fraks\frakl'_{I_\ell}$ on $[\bbF H _\bullet\mod]$.
										\item[$\mathrm{(b)}$] The linear map on $\calF_\ell^{\otimes d}$
										$$|\tuple\lambda\rangle
										\mapsto d_{\bbK,\bbF}([V^{\tuple\lambda}])$$ restricts to an
										$\fraks\frakl'_{I_\ell}$-module isomorphism $$
										L(\Lambda_0)^{\otimes d}\simto [\bbF H _\bullet\mod].$$
										\item[$\mathrm{(c)}$]
										$\bbF H _\bullet\mod$ is a minimal categorification of $\fraks\frakl'_{I_\ell}$.
									\end{itemize}
								\end{theorem}

								\begin{proof}
                                    Part $(a)$ follows from the Heisenberg-to-Kac--Moody categorification.

                                    For $(b)$, reduction modulo $\ell$ commutes with induction, restriction, and the dot, so the displayed map restricts to a nonzero $\fraks\frakl'_{I_\ell}$-homomorphism on $L(\Lambda_0)^{\otimes d}$ and hence is injective. The Hilbert series above show that it is surjective.

                                    For $(c)$, Theorem~\ref{Thm:doubleheisenberg} gives a $\frakU(\fraks\frakl'_{I_\ell})$-categorical action on $\bbF H _{\bullet}\mod$.
            Note that $[{\bf1}_{ H _0}]\in [\bbF H _{\bullet}\mod]$ is a highest vector,  so the $\frakU(\fraks\frakl'_{I_\ell})$-action generated by ${\bf1}_{H_0}$ induces an embedding of categories:
                    $$\frakU(\fraks\frakl'_{I_\ell})\cdot {\bf1}_{H_0}\cong \calL(\boldsymbol{\Lambda}_0)\hookrightarrow \bbF H _{\bullet}\mod.$$
                        By part~(b), the above minimal 2-subrepresentation has the same Grothendieck group as $\bbF H _\bullet\mod$; using the uniqueness of the minimal categorification (cf. \cite{R08}), this identifies the generated 2-representation with the whole category:
                    $$\frakU(\fraks\frakl'_{I_\ell})\cdot{\bf1}_{H_0}\cong \calL(\boldsymbol{\Lambda}_0)\cong \bbF H _{\bullet}\mod.$$
                                    So $\bbF H _\bullet\mod$ is a minimal categorification of $\fraks\frakl'_{I_\ell}$.
                                \end{proof}

										\begin{corollary}\label{weight=block}
											The decomposition of $K_0(\bbF H _n\mod)$ in blocks coincides with its simultaneous decomposition by the component sizes $n_{\chi_i}=|\lambda^{(\chi_i)}|$ and the Kac--Moody weight spaces.\end{corollary}

										\begin{remark}
											In particular, for symmetric groups, Lascoux, Leclerc and Thibon \cite{LLT} proved that
											the decomposition of $K_0(\bbF H _n\mod)$ in blocks coincides with its decomposition in weight
											spaces.  The above corollary is a generalization of their work.
										\end{remark}

										\begin{corollary}

											$\rho_{\tuple \lambda}$ and
											$\rho_{\tuple \mu}$ are in the same block if and only if ${\tuple \lambda}$ and ${\tuple \mu}$ have the same component sizes and the same componentwise $\ell$-cores.
											In particular,
											$\{(n_{\chi_i},\overline{\bbO}^{\chi_i}(u)(-))\}_{i}$ are complete invariants of blocks of $\bbF H _{n}\mod.$
										\end{corollary}
										\begin{proof}
                                            The modules $V_{\tuple \lambda}$ and $V_{\tuple \mu}$ have the same component sizes and lie in the same Kac--Moody weight space if and only if their component sizes and componentwise $\ell$-cores agree. The result therefore follows from Corollary~\ref{weight=block}; the last assertion follows since the valuations of $\overline{\bbO}^{\chi_i}(u)$ recover the residue multiset in each color.
                                        \end{proof}

										Wan and Wang showed that, at the Grothendieck-group level, the modular representation theory of wreath products carries an action of $\widehat{\mathfrak{sl}}_{\ell}^{\oplus d}$ \cite{WW08}; earlier modular branching results for generalized symmetric groups were obtained by Tsuchioka \cite{Tsu07}.  Our construction gives a Kac--Moody 2-categorical realization of the corresponding product-crystal structure in the present setting.

										\begin{corollary}
											Assume $\ell\nmid |H|$. The modular branching graph of the tower
$H\wr\mathfrak S_\bullet$ is isomorphic, as a colored directed graph, to the Cartesian product of
$d=|\widehat H|$ copies of the Kleshchev branching graph for $\mathfrak S_\bullet$. Equivalently, it is the crystal $B(\Lambda_0)^{\times d}$ for the direct-sum algebra $\fraks\frakl'_{I_\ell}$, with the $d$ factors distinguished by the colors $\chi_1,\ldots,\chi_d$.
										\end{corollary}

										\begin{proof}
											This follows from the minimal categorification above and the corresponding uniqueness theorem; compare \cite{R08}.  For the corresponding classical and modular branching results, see \cite{Pus99,MS16,Tsu07,WW08}.
										\end{proof}


\begin{thebibliography}{BSW20b}

\bibitem[BE17]{BE17}
J.~Brundan and A.~Ellis.
\newblock Super {K}ac-{M}oody 2-categories.
\newblock {\em Proc. Lond. Math. Soc. (3)}, 115(5):925--973, 2017.
\newblock URL: \url{https://doi-org.proxy.bib.uottawa.ca/10.1112/plms.12055},
  \href {https://doi.org/10.1112/plms.12055} {\path{doi:10.1112/plms.12055}}.

\bibitem[BDK01]{BDK01}
J.~Brundan, R.~Dipper, and A.~Kleshchev.
\newblock {Quantum linear groups and representations of $\GL_n(\mathbb F_q)$}.
\newblock {Mem. Amer. Math. Soc.}, 149(706), 2001.

\bibitem[Bru18]{B18}
J.~Brundan.
\newblock On the definition of {H}eisenberg category.
\newblock {\em Algebr. Comb.}, 1(4):523--544, 2018.
\newblock \href {https://doi.org/10.5802/alco.26} {\path{doi:10.5802/alco.26}}.

\bibitem[BSW20a]{BSW20a}
J.~Brundan, A.~Savage, and B.~Webster.
\newblock Heisenberg and {K}ac-{M}oody categorification.
\newblock {\em Selecta Math. (N.S.)}, 26(5):Paper No. 74, 62, 2020.
\newblock \href {https://doi.org/10.1007/s00029-020-00602-5}
  {\path{doi:10.1007/s00029-020-00602-5}}.

\bibitem[BSW20b]{BSW20b}
J.~Brundan, A.~Savage, and B.~Webster.
\newblock On the definition of quantum {H}eisenberg category.
\newblock {\em Algebra Number Theory}, 14(2):275--321, 2020.
\newblock \href {https://doi.org/10.2140/ant.2020.14.275}
  {\path{doi:10.2140/ant.2020.14.275}}.

\bibitem[BSW23]{BSW23}
J.~Brundan, A.~Savage, and B.~Webster.
\newblock The degenerate {H}eisenberg category and its {G}rothendieck ring.
\newblock {\em Ann. Sci. \'Ec. Norm. Sup\'er. (4)}, 56(5):1517--1563, 2023.

\bibitem[CR08]{CR08}
J.~Chuang and R.~Rouquier.
\newblock Derived equivalences for symmetric groups and
  $\mathfrak{sl}_2$-categorification.
\newblock {\em Ann. of Math. (2)}, 167(1):245--298, 2008.
\newblock \href {https://doi.org/10.4007/annals.2008.167.245}
  {\path{doi:10.4007/annals.2008.167.245}}.

\bibitem[DL76]{DL76}
P.~Deligne and G.~Lusztig.
\newblock Representations of reductive groups over finite fields.
\newblock {\em Ann. of Math. (2)}, 103(1):103--161, 1976.
\newblock \href {https://doi.org/10.2307/1971021} {\path{doi:10.2307/1971021}}.

\bibitem[DM91]{DM}
F.~Digne and J.~Michel.
\newblock {\em Representations of Finite Groups of Lie Type}, volume~2 of {\em
  London Mathematical Society Student Texts}.
\newblock Cambridge University Press, Cambridge, 1991.

\bibitem[DM94]{DM94}
F.~Digne and J.~Michel.
\newblock Groupes r\'eductifs non connexes.
\newblock {\em Ann. Sci. \'Ec. Norm. Sup\'er. (4)}, 27(3):345--406, 1994.
\newblock \href{https://doi.org/10.24033/asens.1696}
  {\path{doi:10.24033/asens.1696}}.

\bibitem[DVV17]{DVV17}
O.~Dudas, M.~Varagnolo, and E.~Vasserot.
\newblock Categorical actions on unipotent representations of finite classical
  groups.
\newblock In {\em Categorification and Higher Representation Theory}, volume
  683 of {\em Contemporary Mathematics}, pages 41--104. American Mathematical
  Society, Providence, RI, 2017.
\newblock \href {https://doi.org/10.1090/conm/683/13700}
  {\path{doi:10.1090/conm/683/13700}}.

\bibitem[DVV19]{DVV19}
O.~Dudas, M.~Varagnolo, and E.~Vasserot.
\newblock Categorical actions on unipotent representations of finite unitary
  groups.
\newblock {\em Publ. Math. Inst. Hautes Études Sci.}, 129:129--197, 2019.
\newblock \href {https://doi.org/10.1007/s10240-019-00104-x}
  {\path{doi:10.1007/s10240-019-00104-x}}.

\bibitem[Fay06]{F06}
M.~Fayers.
\newblock Weights of multipartitions and representations of {A}riki-{K}oike
  algebras.
\newblock {\em Adv. Math.}, 206(1):112--144, 2006.
\newblock \href {https://doi.org/10.1016/j.aim.2005.07.017}
  {\path{doi:10.1016/j.aim.2005.07.017}}.

\bibitem[Fei87]{Feit}
W.~Feit.
\newblock Extensions of cuspidal characters of $gl_m(q)$.
\newblock {\em Publ. Math. Debrecen}, 34(3-4):273--297, 1987.
\newblock \href {https://doi.org/10.5486/pmd.1987.34.3-4.12}
  {\path{doi:10.5486/pmd.1987.34.3-4.12}}.

\bibitem[FH59]{FH59}
H.~K. Farahat and G.~Higman.
\newblock The centres of symmetric group rings.
\newblock {\em Proc. Roy. Soc. London Ser. A}, 250:212--221, 1959.
\newblock \href {https://doi.org/10.1098/rspa.1959.0060}
  {\path{doi:10.1098/rspa.1959.0060}}.

\bibitem[FS89]{FS89}
P.~Fong and B.~Srinivasan.
\newblock The blocks of finite classical groups.
\newblock {\em J. Reine Angew. Math.}, 396:122--191, 1989.
\newblock \href {https://doi.org/10.1515/crll.1989.396.122}
  {\path{doi:10.1515/crll.1989.396.122}}.

\bibitem[FS90]{FS90}
P.~Fong and B.~Srinivasan.
\newblock Brauer trees in classical groups.
\newblock {\em J. Algebra}, 131(1):179--225, 1990.
\newblock \href {https://doi.org/10.1016/0021-8693(90)90192-8}
  {\path{doi:10.1016/0021-8693(90)90192-8}}.

\bibitem[Gel70]{Gel70}
S.~Gelfand.
\newblock Representations of the full linear group over a finite field.
\newblock {\em Math. USSR-Sb.}, 12(1):13--39, 1970.
\newblock \href {https://doi.org/10.1070/SM1970v012n01ABEH000907}
  {\path{doi:10.1070/SM1970v012n01ABEH000907}}.

\bibitem[Gre55]{Green55}
J.~Green.
\newblock The characters of the finite general linear groups.
\newblock {\em Trans. Amer. Math. Soc.}, 80(2):402--447, 1955.
\newblock \href {https://doi.org/10.2307/1992997} {\path{doi:10.2307/1992997}}.

\bibitem[Gro99]{Go}
I.~Grojnowski.
\newblock Affine $\widehat{\mathfrak{sl}}_p$ controls the representation theory
  of the symmetric group and related hecke algebras.
\newblock 1999.
\newblock arXiv:math/9907129v1 [math.RT].
\newblock URL: \url{https://arxiv.org/abs/math/9907129}, \href
  {https://doi.org/10.48550/arXiv.math/9907129}
  {\path{doi:10.48550/arXiv.math/9907129}}.

\bibitem[HL80]{HL80}
R.~B. Howlett and G.~I. Lehrer.
\newblock Induced cuspidal representations and generalized hecke rings.
\newblock {\em Invent. Math.}, 58(1):37--64, 1980.
\newblock \href {https://doi.org/10.1007/BF01402273}
  {\path{doi:10.1007/BF01402273}}.

\bibitem[Jam86]{Jam86}
Gordon James.
\newblock The irreducible representations of the finite general linear groups.
\newblock {\em Proc. Lond. Math. Soc.}, 3-52(2):236--268, 1986.
\newblock \href {https://doi.org/10.1112/plms/s3-52.2.236}
  {\path{doi:10.1112/plms/s3-52.2.236}}.

\bibitem[JK81]{JK81}
G.~James and A.~Kerber.
\newblock {\em The Representation Theory of the Symmetric Group}, volume~16 of
  {\em Encyclopedia of Mathematics and its Applications}.
\newblock Addison-Wesley Publishing Co., Reading, Mass., 1981.
\newblock \href {https://doi.org/10.1017/CBO9781107340732}
  {\path{doi:10.1017/CBO9781107340732}}.

\bibitem[JW23]{JW23}
N.~Jing and Y.~Wu.
\newblock Characters of $\mathrm{GL}_n(\mathbb{F}_q)$ and vertex operators.
\newblock {\em arXiv preprint arXiv:2309.15330}, 2023.

\bibitem[KL10]{KL10}
M.~Khovanov and A.~Lauda.
\newblock A categorification of quantum {${\rm sl}(n)$}.
\newblock {\em Quantum Topol.}, 1(1):1--92, 2010.
\newblock \href {https://doi.org/10.4171/QT/1} {\path{doi:10.4171/QT/1}}.

\bibitem[LLT96]{LLT}
A.~Lascoux, B.~Leclerc, and J.~Thibon.
\newblock Hecke algebras at roots of unity and crystal bases of quantum affine
  algebras.
\newblock {\em Comm. Math. Phys.}, 181(1):205--263, 1996.
\newblock URL: \url{http://projecteuclid.org/euclid.cmp/1104287629}.

\bibitem[LLZ23]{LLZ}
P.~Li, Y.~Liu, and J.~Zhang.
\newblock Categorical actions and derived equivalences for finite
  odd-dimensional orthogonal groups.
\newblock 2023.
\newblock arXiv:2111.03698v3 [math.RT].
\newblock URL: \url{https://arxiv.org/abs/2111.03698v3}.

\bibitem[LSZ25]{LSZ25}
P.~Li, P.~Shan, and J.~Zhang.
\newblock Categorical action for finite classical groups and its applications:
  characteristic 0.
\newblock {\em Adv. Math.}, 471:Paper No. 110275, 84, 2025.
\newblock \href {https://doi.org/10.1016/j.aim.2025.110275}
  {\path{doi:10.1016/j.aim.2025.110275}}.

\bibitem[Lus76]{Lus76Green}
G.~Lusztig.
\newblock On the Green polynomials of classical groups.
\newblock {\em Proc. London Math. Soc. (3)}, 33(3):443--475, 1976.
\newblock \href {https://doi.org/10.1112/plms/s3-33.3.443}
  {\path{doi:10.1112/plms/s3-33.3.443}}.

\bibitem[Mac95]{Mac95}
I.~Macdonald.
\newblock {\em Symmetric functions and {H}all polynomials}.
\newblock Oxford Mathematical Monographs. The Clarendon Press, Oxford
  University Press, New York, second edition, 1995.
\newblock With contributions by A. Zelevinsky, Oxford Science Publications.

\bibitem[MS16]{MS16}
A.~Mishra and M.~Srinivasan.
\newblock The {O}kounkov-{V}ershik approach to the representation theory of
  {$G\wr S_n$}.
\newblock {\em J. Algebraic Combin.}, 44(3):519--560, 2016.
\newblock \href {https://doi.org/10.1007/s10801-016-0679-5}
  {\path{doi:10.1007/s10801-016-0679-5}}.

\bibitem[OV96]{OV96}
A.~Okounkov and A.~Vershik.
\newblock A new approach to representation theory of symmetric groups.
\newblock {\em Selecta Math. (N.S.)}, 2(4):581--605, 1996.
\newblock \href {https://doi.org/10.1007/PL00001384}
  {\path{doi:10.1007/PL00001384}}.

\bibitem[Pus97]{Pus99}
I.~Pushkarev.
\newblock On the theory of representations of the wreath products of finite
  groups and symmetric groups.
\newblock {\em Zap. Nauchn. Sem. S.-Peterburg. Otdel. Mat. Inst. Steklov.
  (POMI)}, 240:229--244, 294--295, 1997.
\newblock \href {https://doi.org/10.1007/BF02175835}
  {\path{doi:10.1007/BF02175835}}.

\bibitem[Rou08]{R08}
R.~Rouquier.
\newblock 2-kac-moody algebras.
\newblock {\em preprint}, 2008.
\newblock \href {https://arxiv.org/abs/0812.5023} {\path{arXiv:0812.5023}}.

\bibitem[Ryb23]{Ry}
C.~Ryba.
\newblock Stable centres of wreath products.
\newblock {\em Algebr. Comb.}, 6(2):413--455, 2023.
\newblock \href {https://doi.org/10.5802/alco.264}
  {\path{doi:10.5802/alco.264}}.

\bibitem[Ugl00]{Ug99}
D.~Uglov.
\newblock Canonical bases of higher-level {$q$}-deformed {F}ock spaces and
  {K}azhdan-{L}usztig polynomials.
\newblock In {\em Physical combinatorics ({K}yoto, 1999)}, volume 191 of {\em
  Progr. Math.}, pages 249--299. Birkh\"auser Boston, Boston, MA, 2000.

\bibitem[Wan04a]{W04a}
W~Wang.
\newblock The {F}arahat-{H}igman ring of wreath products and {H}ilbert schemes.
\newblock {\em Adv. Math.}, 187(2):417--446, 2004.
\newblock \href {https://doi.org/10.1016/j.aim.2003.09.003}
  {\path{doi:10.1016/j.aim.2003.09.003}}.

\bibitem[Wan04b]{W04b}
W.~Wang.
\newblock Vertex algebras and the class algebras of wreath products.
\newblock {\em Proc. London Math. Soc. (3)}, 88(2):381--404, 2004.
\newblock \href {https://doi.org/10.1112/S0024611503014382}
  {\path{doi:10.1112/S0024611503014382}}.

\bibitem[WW08]{WW08}
J.~Wan and W.~Wang.
\newblock Modular representations and branching rules for wreath {H}ecke
  algebras.
\newblock {\em Int. Math. Res. Not. IMRN}, pages Art. ID rnn128, 31, 2008.
\newblock \href {https://doi.org/10.1093/imrn/rnn128}
  {\path{doi:10.1093/imrn/rnn128}}.

\bibitem[WW19]{WW19}
J.~Wan and W.~Wang.
\newblock Stability of the centers of group algebras of $\mathrm{GL}_n(q)$.
\newblock {\em Adv. Math.}, 349:749--780, 2019.
\newblock \href {https://arxiv.org/abs/1805.08796} {\path{arXiv:1805.08796}},
  \href {https://doi.org/10.1016/j.aim.2019.04.016}
  {\path{doi:10.1016/j.aim.2019.04.016}}.

\bibitem[BSW21]{BSW21}
J.~Brundan, A.~Savage, and B.~Webster.
\newblock Foundations of Frobenius {H}eisenberg categories.
\newblock {\em J. Algebra}, 578:115--185, 2021.
\newblock \href{https://doi.org/10.1016/j.jalgebra.2021.02.025}
  {\path{doi:10.1016/j.jalgebra.2021.02.025}}.

\bibitem[FJW00]{FJW00}
I.~B. Frenkel, N.~Jing, and W.~Wang.
\newblock Vertex representations via finite groups and the {M}c{K}ay correspondence.
\newblock {\em Int. Math. Res. Not.}, 2000(4):195--222, 2000.
\newblock \href{https://doi.org/10.1155/S107379280000012X}
  {\path{doi:10.1155/S107379280000012X}}.

\bibitem[Kho14]{Kho14}
M.~Khovanov.
\newblock Heisenberg algebra and a graphical calculus.
\newblock {\em Fund. Math.}, 225:169--210, 2014.
\newblock \href{https://doi.org/10.4064/fm225-1-8}
  {\path{doi:10.4064/fm225-1-8}}.

\bibitem[LS13]{LS13}
A.~Licata and A.~Savage.
\newblock Hecke algebras, finite general linear groups, and {H}eisenberg categorification.
\newblock {\em Quantum Topol.}, 4(2):125--185, 2013.
\newblock \href{https://doi.org/10.4171/QT/37}
  {\path{doi:10.4171/QT/37}}.

\bibitem[LS21]{LS21}
S.~N. Likeng and A.~Savage.
\newblock Group partition categories.
\newblock {\em J. Comb. Algebra}, 5(4):369--406, 2021.
\newblock \href{https://doi.org/10.4171/JCA/55}
  {\path{doi:10.4171/JCA/55}}.

\bibitem[Gan20]{Gan20}
R.~Gandhi.
\newblock Decomposing Frobenius {H}eisenberg categories.
\newblock {\em J. Algebra Appl.}, 19(5):2050094, 2020.
\newblock \href{https://doi.org/10.1142/S0219498820500942}
  {\path{doi:10.1142/S0219498820500942}}.

\bibitem[Sav19]{Sav19}
A.~Savage.
\newblock Frobenius {H}eisenberg categorification.
\newblock {\em Algebr. Comb.}, 2(5):937--967, 2019.
\newblock \href{https://doi.org/10.5802/alco.73}
  {\path{doi:10.5802/alco.73}}.

\bibitem[RS15]{RS15}
D.~Rosso and A.~Savage.
\newblock Towers of graded superalgebras categorify the twisted {H}eisenberg double.
\newblock {\em J. Pure Appl. Algebra}, 219(11):5040--5067, 2015.
\newblock \href{https://doi.org/10.1016/j.jpaa.2015.03.016}
  {\path{doi:10.1016/j.jpaa.2015.03.016}}.

\bibitem[RS17]{RS17}
D.~Rosso and A.~Savage.
\newblock A general approach to {H}eisenberg categorification via wreath product algebras.
\newblock {\em Math. Z.}, 286(1--2):603--655, 2017.
\newblock \href{https://doi.org/10.1007/s00209-016-1776-9}
  {\path{doi:10.1007/s00209-016-1776-9}}.

\bibitem[Shu22]{Shu22}
Cheng Shu.
\newblock The character table of $\operatorname{GL}_n(\mathbb{F}_q) \rtimes
  \langle \sigma \rangle$.
\newblock {\em Adv. Math.}, 403:108357, 85 pp., 2022.
\newblock \href {https://doi.org/10.1016/j.aim.2022.108357}
  {\path{doi:10.1016/j.aim.2022.108357}}.

\bibitem[Tsu07]{Tsu07}
S.~Tsuchioka.
\newblock A modular branching rule for the generalized symmetric groups.
\newblock {\em J. Algebra}, 316(1):459--470, 2007.
\newblock \href{https://doi.org/10.1016/j.jalgebra.2006.12.007}
  {\path{doi:10.1016/j.jalgebra.2006.12.007}}.

\bibitem[Wal63]{Wall63}
G.~E. Wall.
\newblock On the conjugacy classes in the unitary, symplectic and orthogonal groups.
\newblock {\em J. Austral. Math. Soc.}, 3:1--62, 1963.
\newblock \href{https://doi.org/10.1017/S1446788700027622}
  {\path{doi:10.1017/S1446788700027622}}.

\bibitem[Zel81]{Zel81}
A.~V. Zelevinsky.
\newblock {\em Representations of Finite Classical Groups: A Hopf Algebra
  Approach}, volume 869 of {\em Lecture Notes in Mathematics}.
\newblock Springer, Berlin, 1981.

\end{thebibliography}
\end{document}